\documentclass{gtart_a}
\pdfoutput=1

\title[Pseudoholomorphic punctured spheres in
$\mathbb{R}{\times}(S^{1}{\times}S^{2})$]
{Pseudoholomorphic punctured spheres in
$\mathbb{R}{\times}(S^{1}{\times}S^{2})$:\\Properties and existence}

\author{Clifford Henry Taubes}
\givenname{Clifford Henry}
\surname{Taubes}
\address{Department of Mathematics\\Harvard University\\\newline
Cambridge MA 02138\\USA}
\email{chtaubes@math.harvard.edu}

\volumenumber{10}
\issuenumber{}
\publicationyear{2006}
\papernumber{23}
\lognumber{0731}
\startpage{785}
\endpage{928}
\doi{}
\MR{}
\Zbl{}

\keyword{pseudoholomorphic}
\keyword{punctured sphere}
\keyword{almost complex structure}
\keyword{symplectic form}
\keyword{moduli space}
\subject{primary}{msc2000}{53D30}
\subject{secondary}{msc2000}{53C15}
\subject{secondary}{msc2000}{53D05}
\subject{secondary}{msc2000}{57R17}

\arxivreference{}  
\arxivpassword{}   

\received{6 April 2004}
\revised{}
\accepted{9 May 2006}
\published{24 July 2006}
\publishedonline{24 July 2006}
\proposed{Rob Kirby}
\seconded{Peter Ozsv\'ath, Yasha Eliashberg}
\corresponding{}
\editor{}
\version{}

\AtBeginDocument{\let\bar\wbar\let\hat\what\def\notin{\not\in}}

\def\setobjecttype#1{}
\def\unfrac#1#2{#1/#2}

\newcommand{\step}[1]{\medskip\hypertarget{\thesubsection:#1}{{\bf #1}}\qua\ignorespaces}

\newcommand{\substep}[1]{\medskip\hypertarget{\thesubsection:#1}{{\bfseries\itshape #1}}\qua\ignorespaces}
\newcommand{\substepp}[1]{\medskip\hypertarget{\thesubsection:#1X}{{\bfseries\itshape #1}}\qua\ignorespaces}

\def\refstep#1{\hyperlink{\thesubsection:#1}{#1}}
\def\refstepp#1{\hyperlink{\thesubsection:#1X}{#1}}
\def\refsteps#1#2{\hyperlink{#1:#2}{#2}}

\def\itaubes#1{\addtocounter{equation}{1}
\begin{itemize}
\leftskip25pt
\item
\noindent\llap{\hbox to 53.5pt{\rm\hypertarget{eq:#1}%
{(\theequation)}\hss}}\ignorespaces}

\def\enditaubes{\end{itemize}}

\def\qtaubes#1{\addtocounter{equation}{1}\par\leftskip38pt
\noindent\llap{\hbox to 38pt{\rm\hypertarget{eq:#1}%
{(\theequation)}\hss}}\ignorespaces}

\def\endqtaubes{\par\medskip\leftskip0pt}

\def\eqreft#1#2{\hyperlink{eq:#1.#2}{(#1--#2)}}

\DeclareMathOperator{\kernel}{kernel}
\DeclareMathOperator{\Hom}{Hom}
\DeclareMathOperator{\cokernel}{cokernel}
\DeclareMathOperator{\Deg}{Deg}
\DeclareMathOperator{\End}{End}
\DeclareMathOperator{\constant}{constant}
\DeclareMathOperator{\Ree}{Re}
\DeclareMathOperator{\spann}{span}
\DeclareMathOperator{\Crit}{Crit}
\DeclareMathOperator{\dist}{dist}

\makeatletter
\def\cnewtheorem#1[#2]#3{\newtheorem{#1}{#3}[section]
\expandafter\let\csname c@#1\endcsname\c@theorem}

\newtheorem{theorem}{Theorem}[section]
\cnewtheorem{prop}[theorem]{Proposition}
\cnewtheorem{lemma}[theorem]{Lemma}
\cnewtheorem{corollary}[theorem]{Corollary}
\newtheorem{constraint}{Constraint}
\newtheorem{case}{Case}
\theoremstyle{definition}
\cnewtheorem{deff}[theorem]{Definition}

\makeatother  

\makeautorefname{constraint}{Constraint}
\makeautorefname{subsection}{Subsection}
\makeautorefname{case}{Case}
\makeautorefname{deff}{Definition}

\numberwithin{equation}{section}

\newcommand{\mc}{\ensuremath{\mathfrak{c}}}
\newcommand{\mf}{\ensuremath{\mathfrak{f}}}

\renewcommand{\thesubsection}{\thesection.\Alph{subsection}}

\begin{document}

\begin{asciiabstract}
This is the first of at least two articles that describe the moduli
spaces of pseudoholomorphic, multiply punctured spheres in R x (S^1 x
S^2) as defined by a certain natural pair of almost complex structure
and symplectic form. This article proves that all moduli space
components are smooth manifolds. Necessary and sufficient conditions
are also given for a collection of closed curves in S^1 x S^2
to appear as the set of |s| --> infinity limits of the constant s in R
slices of a pseudoholomorphic, multiply punctured sphere.
\end{asciiabstract}

\begin{htmlabstract}
This is the first of at least two articles that describe the moduli
spaces of pseudoholomorphic, multiply punctured spheres in
<b>R</b>&times;(S<sup>1</sup>&times; S<sup>2</sup>) as defined by a
certain natural pair of almost complex structure and symplectic
form. This article proves that all moduli space components are
smooth manifolds. Necessary and sufficient conditions are also given
for a collection of closed curves in S<sup>1</sup> &times;  S<sup>2</sup>
to appear as the set of |s|&rarr;&infin; limits
of the constant s&isin;<b>R</b> slices of a
pseudoholomorphic, multiply punctured sphere.
\end{htmlabstract}

\begin{abstract}
This is the first of at least two articles that describe the moduli
spaces of pseudoholomorphic, multiply punctured spheres in $\mathbb{R}
 \times  (S^1 \times  S^2)$ as defined by a
certain natural pair of almost complex structure and symplectic
form. This article proves that all moduli space components are
smooth manifolds. Necessary and sufficient conditions are also given
for a collection of closed curves in $S^1 \times  S^2$
to appear as the set of $| s| \to \infty $ limits
of the constant $s \in \mathbb{R}$ slices of a
pseudoholomorphic, multiply punctured sphere.
\end{abstract}

\maketitle

\section{Introduction}\label{sec:1}

This is the first of at least two articles that describe the moduli
spaces of multiply punctured, pseudoholomorphic spheres for a very
natural symplectic form and compatible almost complex structure on
$\mathbb{R}\times(S^{1}\times S^{2})$.  In this regard, the
symplectic form and attending almost complex structure arise when
considering 4 dimensional, compact Riemannian manifolds with an
associated self-dual harmonic 2--form.  To elaborate, if the metric
is suitably generic, then the zero locus of the harmonic form is an
embedded union of circles and the harmonic 2--form defines a
symplectic structure on the complement of this locus (see, for example,
Honda \cite{Ho} or Gay and Kirby \cite{GK}). In addition, the complement of any given
component of the zero locus is in an open set that is diffeomorphic
to $(0, \infty)\times(S^{1}\times S^{2})$.  As explained in
\cite{T1}, the given self-dual 2--form can be modified on the
complement of its zero locus so as to give a symplectic form on this
complement that restricts to any of these $(0,
\infty)\times(S^{1}\times S^{2})$ subsets as either the symplectic
form from $\mathbb{R}\times(S^{1}\times S^{2})$ considered here, or
that of its push-forward by a free, symplectic
$\mathbb{Z}/2\mathbb{Z}$ action.

With the preceding understood, remark next that there is some
evidence (see \cite{T2}) that pseudoholomorphic curves for certain
almost complex structures, compatible with this new symplectic form,
code information about the smooth structure on the underlying 4
dimensional manifold.  And, if such is the case, then a program to
decipher this code will almost surely need knowledge of the
multi-punctured sphere pseudoholomorphic curve moduli spaces on the
whole of $\mathbb{R}\times(S^{1}\times S^{2})$.  For example, these
moduli spaces will arise in a definition of a smooth 4--manifold
invariant that uses any sort of refined version of the
Eliashberg--Givental--Hofer symplectic field theory \cite{EGH}. (Some
refinement would have to be made since the symplectic form that
arises on $\mathbb{R}\times(S^{1}\times S^{2})$ comes from an
overtwisted contact structure on $S^{1}\times S^{2}$.)

This article provides an introduction to the multi-punctured sphere
moduli spaces, a description of some of their local properties, an
introduction to techniques used in the sequel article, and an
existence proof for the various components.  The afore-mentioned
sequel describes the components of the multi-punctured sphere moduli
spaces in great detail with the help of an explicit parametrization.
The reader is also referred to a sort of prequel to this series,
this the article \cite{T3} that describes the pseudoholomorphic
disks, cylinders and certain of the 3--holed spheres in
$\mathbb{R}\times(S^{1}\times S^{2})$.

\medskip
{\bf Acknowledgements}\qua
Before turning to the details, there is a debt to acknowledge: In
hindsight, the approach in these articles most probably owes a great
deal to the author's subconscious remembering of old conversations
with both Helmut Hofer and Michael Hutchings.

The author is supported in part by the National Science Foundation.

\subsection{The symplectic and contact geometry of $\mathbb{R}\times(S^{1}\times S^{2})$}\label{sec:1a}
An introduction to the relevant geometry is in order.  To start, introduce standard
coordinates $(s, t, \theta ,\varphi)$ for $\mathbb{R}\times(S^{1}\times S^{2})$ where $s$ is
the Euclidean coordinate for the $\mathbb{R}$ factor, $t\in\mathbb{R}/(2\pi\mathbb{Z})$ is the
coordinate for the $S^1$ factor and $(\theta,\varphi)\in[0,\pi]\times\mathbb{R}/(2\pi\mathbb{Z})$
are standard spherical coordinates for the 2--sphere factor.  The symplectic form that is used
here on $\mathbb{R}\times(S^{1}\times S^{2})$ comes as the `symplectification' of a contact
1--form on $S^{1}\times S^{2}$, this the 1--form
\begin{equation}\label{eq:1.1}
\alpha\equiv-\big(1-3\cos^{2}\theta\big)dt-\surd 6 \cos \theta\sin^{2}\theta d\varphi.
\end{equation}
To be explicit, here is the symplectic form:
\begin{equation}\label{eq:1.2}
\omega=d\big(e^{-\surd6s}\alpha\big).
\end{equation}
Note that the convention is that the $s\to\infty$ end of $\mathbb{R}\times(S^{1}\times S^{2})$
is the concave side end and the $s\to-\infty$ is the convex side
end. (The concave side end is the half that appears in the
4--manifold context.)  It proves convenient at times to write the form $\omega$ as
\begin{equation}\label{eq:1.3}
\omega=dt\wedge df+d\varphi\wedge dh,
\end{equation}
where
\begin{equation}\label{eq:1.4}
f\equiv e^{-\surd6s}\big(1-3\cos^{2}\theta\big) \qquad\text{and}\qquad h\equiv\surd6 e^{-\surd6s}
\cos\theta\sin^{2}\theta.
\end{equation}
The almost complex structure that defines here the notion of a pseudoholomorphic subvariety
is specified by the relations
\begin{equation}\label{eq:1.5}
J\cdot\partial_{t}=g\partial_{f}\text{ and }J\cdot\partial_{\varphi}=\sin^{2}\theta g \partial_{h},
\end{equation}
where $g =\surd6 e^{-\surd6s}(1 + 3\cos^{4}\theta)$.  This almost complex structure is
$\omega$--compatible.  In fact, the form $g^{-1}\omega(\cdot,J(\cdot))$ on the tangent bundle of
$\mathbb{R}\times(S^{1}\times S^{2})$ is the standard product
metric, $ds^{2}+dt^{2}+d\theta^{2}+\sin^{2}\theta d\varphi^{2}$. As remarked earlier, this
almost complex structure is not integrable.

Note that $J$ is invariant under an $\mathbb{R}\times(S^{1}\times S^{1})$ subgroup of the product
metric's group of isometries, $\mathbb{R}\times S^{1}\times SO(3)$.  Here, the $\mathbb{R}$
factor in this subgroup acts as the constant translations along the $\mathbb{R}$ factor in
$\mathbb{R}\times(S^{1}\times S^{2})$, the first $S^{1}$ factor in the subgroup
acts to rotate the $S^{1}$ factor of $\mathbb{R}\times(S^{1}\times S^{2})$, while the second
$S^{1}$ factor rotates the 2--sphere about the axis where $\theta\in\{0,\pi\}$.  Thus, the
$\mathbb{R}$ action is generated by the vector field $\partial_{s}$ and the two $S^{1}$ actions
are respectfully generated by the vector fields $\partial_{t}$ and $\partial_{\varphi}$
This particular $S^{1}\times S^{1}$ subgroup of the metric isometry group is denoted below as $T$.

\subsection{The pseudoholomorphic subvarieties}\label{sec:1b}

Following the lead of Hofer \cite{H1,H2,H3} and Hofer--Wysocki--Zehnder
\cite{HWZ1,HWZ2,HWZ3}, a pseudoholomorphic subvariety in
$\mathbb{R}\times(S^{1}\times S^{2})$ is defined here as follows:

\begin{deff}\label{deff:1.1}
A pseudoholomorphic subvariety $C\subset\mathbb{R}\times(S^{1}\times S^{2})$ is a non-empty,
closed subset with the following properties:
\begin{itemize}

\item
The complement in $C$ of a countable, nowhere accumulating subset is
a 2--dimensional submanifold whose tangent space is $J$--invariant.

\item
$\int_{C\cap K}\omega<\infty$ when $K\subset\mathbb{R}\times(S^{1}\times S^{2})$
is an open set with compact closure.

\item
$\int_{C}d\alpha<\infty$.
\end{itemize}
A pseudoholomorphic subvariety is said to be `reducible' if the removal of a finite set
of points makes a set with more than one connected component.
\end{deff}
Note that \cite{T3} uses the term `HWZ variety' for what is defined
here to be a pseudoholomorphic subvariety.

If $C\subset\mathbb{R}\times(S^{1}\times S^{2})$ is an irreducible pseudoholomorphic subvariety,
then $C$ defines a canonical `model curve', this a complex curve $C_0$ that comes with
a proper, pseudoholomorphic map to $\mathbb{R}\times(S^{1}\times S^{2})$ that is almost everywhere
1--1 and has image $C$.  A multi-punctured, pseudoholomorphic sphere is, by definition, an
irreducible, pseudoholomorphic subvariety whose model curve is a multiply punctured
sphere.

\subsection{The ends of pseudoholomorphic subvarieties}\label{sec:1c}

The set of pseudoholomorphic subvarieties comes with a natural topology whose description
occupies the next subsection.  This subsection constitutes a digression of sorts
to introduce various facts about the large $|s|$ portions of pseudoholomorphic subvarieties
that are used both to define this topology and to characterize the resulting space of
subvarieties.  This digression has three parts.

\step{Part 1}
Pseudoholomorphic subvarieties are
quite well behaved at large $|s|$.  In particular, as demonstrated
in  \cite[Section 2]{T3}, any given irreducible pseudoholomorphic
subvariety $C$ has the following property:
\qtaubes{1.6}
\textsl{There exists $R > 1$ such that the $|s|\geq R$ portion of $C$ is a finite disjoint union
of embedded cylinders to which the function $s$ restricts as an unbounded function without
critical points. Moreover, the constant $|s|$ slices of any such cylinder converge
in $S^1 \times S^2$ as $|s|\to\infty$ to a closed orbit of the Reeb vector field}
\begin{equation*}
\hat{\alpha}\equiv\big(1-3\cos^{2}\theta\big)\partial_{t}+\surd 6
\cos\theta\partial_{\varphi}.
\end{equation*}
\textsl{In addition, this convergence is such that any constant $|s|$ slice defines a closed
braid in a tubular neighborhood of the limit closed
orbit.}
\endqtaubes
The notion of convergence used here can be characterized as follows:  The diameter of a
tubular neighborhood of the limit closed orbit that contains any given $|s|\geq\mathbb{R}$ slice
can be taken to be a function of $|s|$ that decreases to zero at an exponential rate as $|s|$
diverges.

The closed orbits of $\hat{\alpha}$ are called Reeb orbits.  As
noted in \cite{T3}, they can all be listed; and here is the full
list:
\begin{itemize}

\item \textsl{The $\theta= 0$ and  $\theta=\pi$ loci.}
\item \textsl{The others are labeled by data $((p,p'),\iota)$ where
$\iota\in \mathbb{R}/(2\pi\mathbb{Z})$ and where $p$ and $p'$ are
relatively prime integers that are subject to
the following constraints:
\begin{enumerate}
\item[(a)] At least one is non-zero.
\item[(b)] $|\frac{p'}{p}|>\frac{\surd 3}{\surd 2}$ when $p < 0$.
\end{enumerate}
The Reeb orbit that is labeled by this data is the locus where  $p't - p\varphi=\iota$ and
where $\theta$ is the unique point in $(0,\pi)$ for which}
\setcounter{equation}{6}
\begin{equation}\label{eq:1.7}
p'\big(1-3\cos^{2}\theta\big)-p\surd6\cos\theta=0 \text{ and }p'\cos\theta\geq 0.
\end{equation}
\end{itemize}
The convention in \eqref{eq:1.7} takes a pair of integers to be relatively prime when they
have no common, positive integer divisor save for the number 1.  For example,
$(0,-1)$ and $(0, 1)$ are deemed to be relatively prime, as is $(-1, -1)$.  In this regard,
any pair $P = (p, p')$ that obeys the constraints in the second point above defines, via
\eqref{eq:1.7}, a unique angle between 0 and $\pi$.

The pair $p$ and $p'$ can be alternatively defined as the respective integrals over the Reeb
orbit of the 1--forms $\frac{1}{2\pi}dt$ and $\frac{1}{2\pi}d\varphi$ using the orientation
from $\hat{\alpha}$.

Note that all $\theta\in(0,\pi)$ Reeb orbits come in smooth, 1--parameter families.
In this regard, a $(p, p')$ Reeb orbit is fixed by the subgroup of $T$ generated by
$p\partial_t+ p'\partial_{\varphi}$ while its corresponding family is obtained from its
translates by the action of $T$. Meanwhile, the Reeb orbits where $\theta= 0$ and $\theta=\pi$
are $T$--invariant.

\step{Part 2}
Granted the preceding, it then follows that any end of $C$ whose associated limit Reeb orbit
lies where $\theta\in(0,\pi)$ determines a triple $(\varepsilon,(p,p'))$, where $\varepsilon\in
\{+,-\}$ and where $(p, p')$ are a pair of integers.  To elaborate, $\varepsilon$ is $+$ for
a concave side end and $-$ for a convex side one.  Meanwhile, the pair $(p, p')$ is a positive,
integer multiple of the relatively prime pair of integers that classifies the end's limiting
Reeb orbit. In particular, they are the respective integrals of the 1--forms
$\frac{1}{2\pi}dt$ and $\frac{1}{2\pi}d\varphi$ over any constant $|s|$ slice of the end with
the latter oriented so its homology class in a tubular neighborhood of the limit Reeb orbit
is a positive multiple of the class of the Reeb orbit.  For example, if $\gamma\subset
S^{1}\times S^{2}$ is a $(p, p')$ Reeb orbit, then $\mathbb{R}\times\gamma$ is
a pseudoholomorphic cylinder and in this case the integer pair that is associated to
either end is $(p, p')$.

Of course, an end $E\subset C$ of the sort just described also
determines an element, $\iota_{E}\in\mathbb{R}/(2\pi\mathbb{Z})$,
this the angular parameter in \eqref{eq:1.7} that helps to specify
its associated limit Reeb orbit. A convex side end also determines a
real number, $c_E$. This comes about as follows: Let $\theta_E$
denote the $|s|\to\infty$ limit of $\theta$ on $E$. The arguments
from  \cite[Section 2]{T3} can be used to prove that the function
$\theta$ on any end of $C$ can be written as
\begin{equation}\label{eq:1.8}
\theta=\theta_{E}+c_{E}e^{-\zeta|s|}+o\big(e^{-(\zeta+\varepsilon)|s|}\big),
\end{equation}
where $c_E$ is constant while $\zeta=\surd6\sin^{2}\theta_{E}
(1+3\cos^{2}\theta_{E})/(1+3\cos^{4}\theta_{E})$, and $\varepsilon$
is positive and also determined a priori by $\theta_E$. In this
regard, if $c_{E}= 0$, then the leading order term in \eqref{eq:1.8}
is both above and below $\theta_E$ over the large and constant $|s|$
slices of $E$ unless $\theta$ is constant on $E$. In the latter
case, $E$ is part of some $\mathbb{R}\times\gamma$ with $\gamma$ a
Reeb orbit.  Note that $c_E$ is always zero for a concave side end
(see the proof of the second point in  \cite[(4.21)]{T3}).

\step{Part 3}
As will now be explained, an end of $C$
whose limit Reeb orbit has $\theta= 0$ or $\pi$ but does not
coincide with the corresponding $\theta= 0$ or $\theta=\pi$ cylinder
can also be assigned a discrete triple $(\varepsilon, (p, p'))$ as
well as an angular parameter and a real number.  To begin the
explanation, note first that $\varepsilon\in\{+,-\}$ has the same
meaning as before, $+$ when the end is on the convex side and $-$
otherwise. Meanwhile $p$ and $p'$ are the respective integrals over
any sufficiently large and constant $|s|$ slice of the 1--forms
$\frac{1}{2\pi}dt$ and $\frac{1}{2\pi}d\varphi$ using the pull-back
of $-dt$ to define the orientation.  In this regard, the following
fact from  \cite[Section 2]{T3} is used:
\qtaubes{1.9}
\textsl{%
Let $C$ denote an irreducible pseudoholomorphic subvariety.  If $\gamma\subset S^{1}\times S^{2}$
is a Reeb orbit, and if $C\neq \mathbb{R}\times\gamma$,  then $C$  has at most a finite number of
 intersections with $R\times\gamma$.}
\endqtaubes
As can be proved using results from  \cite[Sections 2 and 3]{T3},
any pair $(p, p')$ that arises in the manner just described from an
end of $C$ where the $|s|\to\infty$ limit of $\theta$ is 0 or $\pi$
is constrained by the following rules:
\itaubes{1.10}
\textsl{$p < 0$  in all cases.}
\item \textsl{$\frac{p'}{p}< -\frac{\surd 3}{\surd 2}$  for $\theta= 0$
concave side ends and  $\frac{p'}{p}> -\frac{\surd 3}{\surd 2}$  for
$\theta = 0$  convex side ends.}

\item \textsl{$\frac{p'}{p}> \frac{\surd 3}{\surd 2}$  for $\theta=\pi$
 concave side ends and $\frac{p'}{p}< \frac{\surd 3}{\surd 2}$
  for $\theta=\pi$
 convex side ends.}
\end{itemize}

To explain the angle and real number to assign an end $E\subset C$ where the $|s|\to\infty$
limit of $\theta$ is 0, introduce the functions
\begin{equation}\label{eq:1.11}
a=|h|^{1/2}\cos\varphi\qquad\text{and}\qquad b\equiv|h|^{1/2}\sin\varphi.
\end{equation}
The analysis from  \cite[Section 3]{T3} can be used to prove that
an end of the sort under consideration can be parametrized in an
orientation preserving fashion at large $|s|$ by a complex parameter
$z\in\mathbb{C}-0$ via a map that sets
\begin{equation}\label{eq:1.12}
s-i\,t=p\,\ln(z)\qquad\text{and}\qquad a-i\,b=\hat{c}_{E}z^{\pm p}(1+o(1)),
\end{equation}
where the $+$ sign is used when the $|s|\to\infty$ limit of $\theta$ on the end is 0 and the $-$
sign when this limit is $\pi$.  Note that \eqref{eq:1.12} is valid for a concave side end only
where $|z|\ll 1$ and only where $|z|\gg 1$ for a convex side one.  In any case, the constant
$\hat{c}_{E}\in\mathbb{C}-0$ while the term indicated by $o(1)$ and its derivatives limits to
zero as $|\ln|z||\to\infty$. Granted the preceding, the real number, $c_E$, and the angle,
$\iota_{E}$, assigned to $E$ are the respective real and imaginary parts of $\ln(\hat{c}_{E})$.

An end $E\subset C$ where the $|s|\to\infty$ limit of $\theta$ is $\pi$ has an analogous
orientation preserving parametrization by $z\in\mathbb{C}-0$ that is obtained from the
preceding by using the fact that the almost complex geometry is invariant under the
involution of $S^{1}\times S^{2}$ that acts to send $(t,(\theta,\varphi))\to(t+\pi,(\pi-\theta,
-\varphi))$.

\subsection{The moduli spaces}\label{sec:1d}
Fix a finite set whose elements are of the following sort.  Each element
is a 4--tuple of
the form $(\delta,\varepsilon ,(p, p'))$ with $\delta\in\{-1,0,1\}$, $\varepsilon\in\{+, -\}$
and $(p, p')\in\mathbb{Z}\times\mathbb{Z}$. A given 4--tuple is allowed to appear more than
once in this set.  Let $\hat{A}$ denote the set that is obtained by augmenting this chosen
set of 4--tuples with a single pair, $(\mc_{+},\mc_{-})$, of non-negative integers.  Such data
sets are used in what follows to label subsets of the set of pseudoholomorphic subvarieties.

Given a non-negative integer, $\zeta$, and a set $\hat{A}$ as just described, let
$\mathcal{M}_{\hat{A},\zeta}$,  denote the set of irreducible pseudoholomorphic subvarieties
in $\mathbb{R}(S^{1}\times S^{1})$ with the following
three properties:  First, if $C\in\mathcal{M}_{\hat{A},\zeta}$, then $C$'s model curve has genus
$\zeta$.  Second, there is a $1-1$ correspondence between the 4--tuples in $\hat{A}$ and the set
of ends of $C$ so that when $E$ is any given end of $C$, then its
corresponding 4--tuple
in $\hat{A}$ is as follows:  The component $\delta$ is $1$, $0$ or $-1$ in the respective
cases that the $|s|\to\infty$ limit of $\theta$ on $E$ is 0, neither 0 nor $\pi$, or $\pi$.
Meanwhile, the component $\varepsilon$ is $+$ when $E$ is a concave side end
and $-$ otherwise.  Finally, the pair $(p, p')$ from the 4--tuple is the integral over any
sufficiently large $|s|$ slice of $E$ of the pair of 1--forms
$(\frac{1}{2\pi}dt,\frac{1}{2\pi}d\varphi)$. Said succinctly,
the 4--tuples from the set $\hat{A}$ describe the discrete asymptotic data of the ends of any
subvariety in $\mathcal{M}_{\hat{A},\zeta}$.  Finally, $C$ has intersection number $\mc_+$ with
the $\theta = 0$ cylinder and $\mc_-$ with the $\theta=\pi$ cylinder.

The genus 0 subvarieties are the multiply punctured spheres.  The $\zeta = 0$ version of
$\mathcal{M}_{\hat{A},\zeta}$, is denoted below as $\mathcal{M}_{\hat{A}}$.

Give $\mathcal{M}_{\hat{A},\zeta}$, the topology where a basis for the neighborhoods of a given
$C\in\mathcal{M}_{\hat{A},\zeta}$ are the subsets that consist of those
$C'\in\mathcal{M}_{\hat{A},\zeta}$ with
\begin{equation}\label{eq:1.13}
\sup_{z\in C}\dist(z,C')+\sup_{z\in C'}\dist(C,z)<k.
\end{equation}
Here, $k$ is some fixed positive real number.  The topological space
that results is a `moduli space'.  In this regard, note that the
definition of a moduli space given here differs from that in
\cite{T3} in the case that $\hat{A}$ has 4--tuples with first
component equal to $\pm 1$ since the definition in \cite{T3} does
not constrain the $p'$ component of the 4--tuple.

In any event, with the data set $\hat{A}$ fixed, the structure of the corresponding
moduli space $\mathcal{M}_{\hat{A}}$ of multi-punctured spheres is of prime interest in
this article.  In this regard, the first significant result in this
article is summarized by the following theorem:

\begin{theorem}\label{thm:1.2}
The multi-punctured sphere moduli space $\mathcal{M}_{\hat{A}}$ has the structure of a finite
dimensional, smooth manifold.
\end{theorem}
This theorem is proved in \fullref{sec:2}.  This same section also describes various local coordinate
charts for $\mathcal{M}_{\hat{A}}$.  In particular some are obtained using the
$\mathbb{R}/(2\pi\mathbb{Z})$ and real valued parameters that are defined from the $|s|\to\infty$
limits on various ends.

As explained in \fullref{sec:2}, the dimension of $\mathcal{M}_{\hat{A}}$ is determined by the set
$\hat{A}$.  A formula is given in \fullref{prop:2.5}.  This proposition provides a
formula for the `formal' dimension for any given $\zeta> 0$ version of
$\mathcal{M}_{\hat{A},\zeta}$,  in terms of $\hat{A}$ and $\zeta$. To elaborate, note first
that \fullref{sec:2} proves that any $\zeta>0$ version of $\mathcal{M}_{\hat{A},\zeta}$,  is a finite
dimensional variety in the sense that any given point has a neighborhood that
is homeomorphic to the zero locus near the origin of a smooth map between two
Euclidean spaces.  The difference between the dimensions of the domain and range Euclidean spaces
is independent of the chosen point in $\mathcal{M}_{\hat{A},\zeta}$.  This difference is taken
to be the formal dimension of $\mathcal{M}_{\hat{A},\zeta}$.

\subsection{When $\mathcal{M}_{\hat{A}}$ is non-empty}\label{sec:1e}
This subsection provides necessary and sufficient conditions on $\hat{A}$ so as to
guarantee a non-empty version of $\mathcal{M}_{\hat{A}}$. In this regard, it follows from
what has been said already that $\mathcal{M}_{\hat{A}}=\phi$ unless the constraints listed
next are obeyed.  A set, $\hat{A}$, of the sort under consideration that obeys these
constraints is said here to be an \textit{asymptotic data set}.

Here is the first asymptotic data set constraint:  Each $(\delta,\varepsilon,(p, p'))\in\hat{A}$
must obey:
\itaubes{1.14}
\textsl{If $\delta= 0$ and $p < 0$,  then  $|\frac{p'}{p}|>\frac{\surd
3}{\surd 2}$.}

\item \textsl{If $\delta=1$, then $p < 0$.  In addition,
$\frac{p'}{p}<-\frac{\surd 3}{\surd 2}$ when $\varepsilon= +$,  and
$\frac{p'}{p}>-\frac{\surd 3}{\surd 2}$  when $\varepsilon >-$.}

\item \textsl{If $\delta=-1$,  then $p < 0$.   In addition,
$\frac{p'}{p}>\frac{\surd 3}{\surd 2}$  when $\varepsilon= +$,
and $\frac{p'}{p}<\frac{\surd 3}{\surd 2}$  when $\varepsilon >-$.}
\end{itemize}
Two more constraints come via Stokes' theorem as applied to line integrals of $dt$ and $d\varphi$:
\begin{equation}\label{eq:1.15}
\sum_{(\delta,\varepsilon,(p,p'))\in\hat{A}}\varepsilon\,p=0 \qquad\text{and}\qquad
\sum_{(\delta,\varepsilon,(p,p'))\in\hat{A}}\varepsilon\,p'=-\big(\mc_{+}-\mc_{-}\big).\end{equation}
Here is the next asymptotic data set constraint:
\qtaubes{1.16}
\textsl{%
If $\hat{A}$  has two 4--tuples and $\mc_{+}=\mc_{-}= 0$, then the 4--tuples have
relatively prime integer pairs.}
\endqtaubes
To explain, note that when $\hat{A}$ has $\mc_{+}=\mc_{-}= 0$ and two
4--tuples, then $\mathcal{M}_{\hat{A}}$ has only cylinders.  All such
spaces are described in  \cite[Section 4]{T3}, and all obey
\eqreft1{16}.

The final two asymptotic data set constraints involve the set $\Lambda_{\hat{A}}\subset[0,\pi]$
that consists of the angle 0 when $\mc_{+}> 0$ or $\hat{A}$ has a $(1,\ldots)$ element,
the angle $\pi$ when $\mc_{-}> 0$ or $\hat{A}$ has a $(-1,\ldots)$ element, and the angles that
are defined via \eqref{eq:1.7} from the $(0,\ldots)$ elements in $\hat{A}$.  Granted this
definition, here are the last two constraints:
\itaubes{1.17}
\textsl{If $\Lambda_{\hat{A}}$  has one angle, then $\hat{A}$ has  $\mc_{+}=\mc_{-}= 0$ and two 4--tuples,
$(0,+,P)$  and $(0,-,P)$  with $P$  relatively prime.
}

\item \textsl{If $\Lambda_{\hat{A}}$  has more than one angle, then neither extremal
angle arises via \eqref{eq:1.7} from an integer pair of any $(0,+,\ldots)$ element in $\hat{A}$.
}
\end{itemize}
With regards to the first point here, note that any moduli space where the corresponding
$\Lambda_{\hat{A}}$ has one angle contains only $\mathbb{R}$ invariant cylinders.  The
final point arises by virtue of the fact noted previously that the constant $c_E$ in
\eqref{eq:1.8} is zero when $E$ is a concave side end of a subvariety where the $s\to\infty$
limit of $\theta$ is in $(0,\pi)$.

In all that follows, $\hat{A}$ refers to an asymptotic data set where $\Lambda_{\hat{A}}$ has
more than one angle. As it turns out, $\mathcal{M}_{\hat{A}}$ is nonempty if and only if
the data from $\hat{A}$ can be used to construct a certain linear graph with labeled edges.
The latter graph is denoted in what follows by $L_{\hat{A}}$.  The following three
part digression describes what is involved.

\step{Part 1}
To set the stage, introduce $\Lambda_{\hat{A}}$ to denote the set in $[0,\pi]$ that
consists of the distinct angles that come via \eqref{eq:1.7} from the integer pairs of
the $(0,\ldots)$ elements in $\hat{A}$ together with the angle 0 when $\mc_{+} > 0$ or when
$\hat{A}$ has a $(1,\ldots)$ element, and the angle $\pi$ when $\mc_{-}> 0$ or when $\hat{A}$ has
a $(-1,\ldots)$ element.

\step{Part 2}
In the present context, a linear graph is viewed as a finite set of distinct points in
$[0,\pi]$ with two or more elements.  In particular, each graph has at least one edge.
The points in the set are the vertices of the graph, and the intervals that connect adjacent
points are the edges. Thus, such a graph has two monovalent vertices and some number of
bivalent ones.  As a point in $[0,\pi]$, each vertex has a canonical angle assignment.
These angles should coincide with the angles in $\Lambda_{\hat{A}}$.

\step{Part 3}
As noted at the outset, the edges of the graph $L_{\hat{A}}$ are labeled.  In particular, each
edge is labeled by an ordered pair of integers subject to the set of six constraints that
appear in the upcoming list \eqreft1{18}.

The notation used is as follows:  When e designates an edge, then $Q_e$ or $(q_{e},{q_e}')$
is used to denote its corresponding ordered pair of integers. An edge is said to
`start a graph' when its smallest angle is the smallest angle on its graph.
By the same token, and edge is said to `end a graph' when its largest angle is the largest angle
on its graph.

Here are the constraints:
\itaubes{1.18}
\textsl{If  $e$  ends the graph at an angle in $(0,\pi)$,  then $-Q_{e}$
is the sum of the pairs  from each of the $(0,-,\ldots)$  elements in $\hat{A}$
 that define this maximal angle  via \eqref{eq:1.7}.
}

\item \textsl{If $\pi$  is the largest angle on $e$,  then  $Q_{e}$
 is obtained using the following rule: First, subtract the sum of the integer pairs from the $(-1,-,\ldots)$
 elements  in $\hat{A}$  from the sum of those from the $(-1,+,\ldots)$  elements, and then
 subtract $(0, \mc_{-})$  from the result.
}

\item \textsl{If $e$  starts the graph at an angle in $(0,\pi)$, then
$Q_{e}$  is the sum of the pairs  from each of the $(0,-,\ldots)$  elements in
$\hat{A}$ that define this minimal angle  via \eqref{eq:1.7}.
}

\item \textsl{If $0$  is the smallest angle on $e$,  then $Q_{e}$  is obtained using the following rule:
First, subtract the sum of the integer pairs from the $(1,+,\ldots)$
elements in $\hat{A}$  from the sum of those from the $(1,-,\ldots)$
elements and then subtract
$(0, \mc_{+})$  from the result.
}

\item \textsl{Let $o$ denote a bivalent vertex, let $\theta_0$ denote its angle, and let $e$  and $e'$
denote its incident edges with the
convention that $\theta_0$  is the largest angle on $e$.   Then
$Q_{e} - Q_{e'}$  is obtained by subtracting the sum of the integer pairs
from the $(0,-,\ldots)$  elements in $\hat{A}$  that define $\theta_0$  via \eqref{eq:1.7} from
the sum of the integer pairs from the $(0,+,\ldots)$  elements in $\hat{A}$
that defined $\theta_0$  via \eqref{eq:1.7}.
}

\item \textsl{
Let $\hat{e}$  denote an edge.  Then  $p{q_{\hat{e}}}'-p'q_{\hat{e}}>0$
in the case that $(p, p')$  is an  integer pair that defines the angle of a bivalent vertex on $\hat{e}$.
Moreover, if $ {q_{\hat{e}}}'<0$  and if neither vertex on $\hat{e}$
has angle $0$  or $\pi$,  and if the version of \eqref{eq:1.7}'s integer $p'$ for one of the vertex angles is
positive, then both  versions of $p'$  are positive.
}
\end{itemize}

A graph of the sort just described is deemed to be an `positive line graph'
for $\hat{A}$. The next theorem explains its significance.

\begin{theorem}\label{thm:1.3}
Suppose that $\hat{A}$ is an asymptotic data set.  Then $\mathcal{M}_{\hat{A}}$ is non-empty
if and only if $\hat{A}$ has a a positive line graph.
\end{theorem}
Note that \fullref{thm:3.1} provides equivalent necessary
and sufficient criteria for a non-empty $\mathcal{M}_{\hat{A}}$.

To explain something of the nature of \fullref{thm:1.3}, note that a graph much like $L_A$
can be constructed from any subvariety in $\mathcal{M}_{\hat{A}}$. As explained in the next
section, the edges of the latter graph are in $1-1$ correspondence with the components of
the complement in the subvariety of the singular and/or non-compact constant $\theta$ loci.
Meanwhile, the integer pair that is assigned to any given edge is obtained by integrating
the pair $\frac{1}{2\pi}dt$ and $\frac{1}{2\pi}d\varphi$ about any constant $\theta$ slice
of the corresponding component. However, a graph from a subvariety can differ from $T$ in
one fundamental aspect. Although the graph as defined by the subvariety is contractible and
connected, it might not be linear.

However, as is proved in the subsequent sections, if $L_{\hat{A}}$ obeys the constraints
in \eqreft1{18}, then there exists either a subvariety in $\mathcal{M}_{\hat{A}}$ that
supplies precisely this graph $L_{\hat{A}}$, or there is a sequence of subvarieties in
$\mathcal{M}_{\hat{A}}$ whose graphs converge to $L_{\hat{A}}$ in a suitable sense.

The conditions in \eqreft1{18} are more or less direct
consequences of the manner in which the graph $L_{\hat{A}}$ is
designed to encodes aspect of the topology of the constant $\theta$
slices of a pseudoholomorphic subvariety. Indeed, the only subtle
constraint is the final one. The latter can be seen as necessary by
considering intersections between the given pseudoholomorphic
subvariety and versions of $\mathbb{R}\times\gamma$ where
$\gamma\subset S^{1}\times S^{2}$ is an open subset with compact
closure in the integral curve of the Reeb vector field. Such
submanifolds are pseudoholomorphic, and the final constraint in
\eqreft1{18} follows from the fact that the local intersection
numbers with such submanifolds are necessarily positive. In this
regard, keep in mind that local intersection numbers between any two
pseudoholomorphic subvarieties are positive (see, for example, McDuff \cite{M}.)

As indicated, subvarieties in $\mathcal{M}_{\hat{A}}$ can be constructed
granted only the constraint in \eqreft1{18}.  Observations by Michael
Hutchings (see also Hutchings--Sullivan \cite{HS}) led the author
to think that the constraints in \eqreft1{18} are sufficient as well
as necessary.

\subsection{Outline of the remaining sections}\label{sec:1f}
The remainder of this article is organized as follows:
\fullref{sec:2} contains a proof of \fullref{thm:1.2}. This section also describes some
useful coordinate systems on $\mathcal{M}_{\hat{A}}$, in particular, some that are obtained
using the topology of the constant $\theta$ loci on the pseudoholomorphic subvarieties.
The latter play a prominent role in the later sections and in the sequel to this article.
The final subsection of this article explains how the $\theta$--level sets are used to assign
a contractible, connected graph to each subvariety in $\mathcal{M}_{\hat{A}}$.

\fullref{sec:3} starts the proof of \fullref{thm:1.3} with a result of the following sort:
A reasonably general class of immersed subvarieties with the large $|s|$ asymptotics of a
subvariety from $\mathcal{M}_{\hat{A}}$ can be deformed to give a subvariety from $\mathcal{M}_{\hat{A}}$.
The latter result with the constructions in \fullref{sec:4} provide necessary and sufficient
conditions for the existence of subvarieties in $\mathcal{M}_{\hat{A}}$ that differ from
those in \fullref{thm:1.3}. These alternative conditions are stated as \fullref{thm:3.1}.
Arguments, strategies and constructions from \fullref{sec:3} also play prominent roles in the sequel
to this article, as does \fullref{thm:3.1}.

\fullref{sec:4} completes the proof of \fullref{thm:3.1} by exhibiting subvarieties to start
the deformations that are described in \fullref{sec:3}.  The final section explains how \fullref{thm:1.3} follows from \fullref{thm:3.1}.

%
%

\setcounter{theorem}{0}

\section{Dimensions and regularity}\label{sec:2}

As remarked in the opening section, the interest here is in the
moduli space of pseudoholomorphic, multiply punctured spheres in
$\mathbb{R} \times (S^1 \times  S^2)$. In this regard, a given
component is labeled by an asymptotic data set, $\hat{A}$, subject
to the constraints in \eqreft1{14}--\eqreft1{17}. The added
constraints on $\hat{A}$ from \fullref{thm:1.3} are not assumed
in this section.

As in the introduction, $\mathcal{M}_{\hat{A}}$ denotes the part of the moduli
space that is labeled by $\hat{A}$. Among other things, this subsection
establishes that $\mathcal{M}_{\hat{A}}$ is a smooth manifold and derives a
formula in terms of the data from $\hat{A}$ for its dimension. The final three
parts of the subsection describe various useful constructions that are
subsequently used to study $\mathcal{M}_{\hat{A}}$. In particular, these
include various local coordinate charts for $\mathcal{M}_{\hat{A}}$ that can be
constructed from the $\mathbb{R}/(2\pi \mathbb{Z})$ and real valued parameters
that are associated to the ends of the subvarieties in $\mathcal{M}_{\hat{A}}$.

The subsection starts by introducing a somewhat more general context for the
subsequent discussions.

\subsection{Admissible almost complex structures}\label{sec:2a}

Arguments in the next section require facts about the moduli spaces of
pseudoholomorphic subvarieties in $\mathbb{R}  \times  (S^1  \times
S^2)$ as defined by almost complex structures that differ from $J$. The
almost complex structures that arise are deemed here to be `admissible', and
they are distinguished by three salient features: First, the almost complex
structure is `tamed' by all sufficiently small, constant and positive r
versions of the symplectic form $d(e^{-rs} \alpha )$. In this
regard, the form $d(e^{-rs}  \alpha )  \equiv \mu $
tames an almost complex structure, $J'$, when the quadratic function on
$T(\mathbb{R} \times  (S^1  \times  S^2))$ that sends any given
vector $v$ to $\mu (v, J'(v))$ is positive on the complement of the zero
section. Second, the almost complex structure sends $\partial_{s}$ to
$(1+3\cos^4 \theta )^{-1 / 2} \;\hat {\alpha }$, where $\hat
{\alpha }$ is the Reeb vector field in~\eqreft16. Third, the given almost
complex structure agrees with $J$ on the complement of some compact subset of
$\mathbb{R}  \times  (S^1  \times  S^2)$. The almost complex
structure $J$ is, of course, admissible as it is compatible with all $r > 0$
versions of $d(e^{ - rs}\alpha )$ and thus tamed by all of them.

As defined, the set of admissible almost complex structures should be viewed
as a Frech\^{e}t space with the topology defined so that a given sequence,
$\{J_{\alpha }\}$, of such structures converges to a given admissible $J'$
when there is $C^{\infty }$ convergence on compact sets and when there
exists some fixed compact subset such that each $J_{\alpha }= J$ on its
complement.

As it turns out, this space of admissible almost complex structures is
contractible. To see why, note first that if $J'$ is admissible, and $w  \in
\kernel(\alpha )$ is tangent to a constant $s$ slice of $\mathbb{R}  \times
(S^1  \times  S^2)$, then $J'w$ must be of the form $w' + a \partial
_{s} + b \hat {\alpha }$ where $a$ and $b$ can be any pair of real numbers
and where $w'$ is also annihilated by $\alpha $ and tangent to $S^1  \times
 S^2$. In this regard, $d\alpha (w, w')$ must be positive when $w  \ne
0$. Written in this way, the space of admissible almost complex structures
manifestly deformation retracts onto the subspace of those that map the
tangents to $S^1  \times  S^2$ in the kernel of $\alpha $ to
themselves. Meanwhile, the latter subspace is contractible since $SL(2; \mathbb{R})/SO(2)$ is contractible.

Consider now the following generalization of \fullref{deff:1.1}:

\begin{deff}
\label{deff:2.1}

Let $J'$ denote an admissible almost complex structure. A non-empty subset, $C$, in $\mathbb{R}\times(S^1\times S^2)$
is a $J'$--pseudoholomorphic subvariety if it is closed and has the following properties:
\begin{itemize}
\item The complement in $C$ of a countable, nowhere accumulating subset is
a 2--dimensional
submanifold whose tangent space is $J'$--invariant.

\item
$\smallint_{C \cap K}  \omega  < \infty $ when $K  \subset   \mathbb{R}\times  (S^1  \times  S^2)$
is an open set with compact closure.

\item
$\smallint_{C} d\alpha  < \infty $.
\end{itemize}
\end{deff}

A subvariety is called `pseudoholomorphic' below without reference to a
particular almost complex structure only in the case that $J$ is the unnamed
almost complex structure. Unless stated to the contrary, a subvariety should
be assumed irreducible.

Here is the simplest, yet very important example: Let $\gamma    \subset
S^1  \times  S^2$ denote a closed orbit of the vector field $\hat
{\alpha }$ from~\eqreft16. Let $J'$ denote any admissible almost complex
structure. Then $\mathbb{R}  \times   \gamma    \subset   \mathbb{R}  \times
 (S^1  \times  S^2 )$ is a $J'$--pseudoholomorphic subvariety.

With \fullref{deff:2.1} understood, the next proposition summarizes
results from \cite[Propositions~2.2 and~2.3]{T3} that are germane to
the situation at hand.

\begin{prop}\label{prop:2.2}

Suppose that $J'$ is an admissible almost complex structure and that $C$
is a $J'$--pseudoholomorphic subvariety.
Then, there exists $R > 1$ such that the $|s|    \ge  R$ portion of $C$ is a finite disjoint
union of embedded cylinders to which the function $s$ restricts as an unbounded function without
critical points. Moreover,
\begin{itemize}
\item
The constant  $|s|$ slices of any such cylinder converge in $S^1   \times  S^2 $ as $|s|   \to
\infty $ to some closed orbit in $S^1   \times S^2$ of the Reeb vector field $\hat {\alpha }$ from~\eqreft16.
In this regard, there exists $\kappa > 0$ such that the function of $|s|$ that assigns the maximum
distance from the large $|s|$ slices of $E$ to the limit closed orbit of $\hat {\alpha }$ is bounded by a
constant multiple of $e^{-\kappa| s| }$.

\item
This convergence is such that any sufficiently large and constant $|s|$ slice defines a closed braid in
any given tubular neighborhood of the limit closed orbit. In this regard, all sufficiently large $|s|$
slices are disjoint from the limit Reeb orbit unless the subvariety is the product of $\mathbb{R}$ with the Reeb orbit.

\item
The subvariety $C$ is the image of a complex curve via a proper,
$J'$--pseu\-do\-ho\-lo\-mor\-phic map into
$\mathbb{R}  \times  (S^1   \times  S^2 )$ that is 1--1 on the complement of a finite set of points.
\end{itemize}
\end{prop}

With \fullref{prop:2.2} understood, define an `end' of a $J'$--pseudoholomorphic
subvariety to be any of the cylinders that appear in \fullref{prop:2.2}. The
ends of any given $J'$--pseudoholomorphic subvariety comprise a set with a
natural 1--1 correspondence to some `asymptotic data set' as defined in the
previous section. In this regard, the correspondence is defined for the
$J'$--pseudoholomorphic subvariety in the same manner as with a
$J$--pseudoholomorphic one. Note as well that a $J'$--pseudoholomorphic
subvariety, if irreducible, can be assigned a `genus', the genus of its
model curve.

Granted the preceding, suppose now that $\hat{A}$ is an asymptotic data set and
that $\varsigma$ a non-negative integer. Define the $J'$ version of the
moduli spaces $\mathcal{M}_{\hat{A},\varsigma }$ as done for the $J$ version in
\fullref{sec:1d}; this the set of irreducible, $J'$--pseudoholomorphic
subvarieties with genus $\varsigma$ whose set of ends are in 1--1
correspondence with the data set $\hat{A}$. Use $\mathcal{M}_{\hat{A},\varsigma
,J'}$ to denote this set. Give this set the topology where a basis for the
open neighborhoods of a given subvariety have the form in~\eqref{eq:1.13}.

The following subsumes \fullref{thm:1.2}:

\begin{theorem}
\label{thm:2.3}
Let $J'$ denote an admissible almost complex structure and let $\hat{A}$ denote an asymptotic data set.
Then the multipunctured sphere moduli space $\mathcal{M}_{\hat{A},0,J'}$ is a smooth manifold.
\end{theorem}
The proof of this theorem occupies the next three parts of the subsection.

\subsection{The local structure of the moduli spaces}\label{sec:2b}

This section contains what is essentially a review of material from
  \cite[Sections 2, 3 and 4]{T3}. Before starting the review, agree
to fix an admissible almost complex structure $J'$, an asymptotic
data set $\hat{A}$ subject to the constraints in \eqreft1{14}--\eqreft1{17},
and fix a non-negative integer $\varsigma$. Set $\mathcal{M}
\equiv   \mathcal{M}_{\hat{A},\varsigma ,J'}$. This subsection
describes the local structure around points in $\mathcal{M}$.

The following three propositions summarize the story on the local
structure of $\mathcal{M}$ about any given subvariety. The proofs
are straightforward and mostly cosmetic modifications of arguments
from \cite[Sections 3 and 4]{T3}. An outline is given at the end
of this subsection but the details are left to the reader.

\begin{prop}\label{prop:2.4}
There exists a positive integer $\hat{I}$ that depends only on $\hat{A}$ and $\varsigma$; and given
$C  \in   \mathcal{M}$, there is a positive integer, $n$, a smooth map, $f$,
from an origin centered ball in $\mathbb{R}^{\hat{I} + n}$ to $\mathbb{R}^{n}$ that maps 0 to 0,
and a homeomorphism from $f^{-1}(0)$ onto a neighborhood of $C$ in $\mathcal{M}$ that maps the origin to $C$.
\end{prop}

The integer $\hat{I}$ that appears here is the `formal dimension' of $\mathcal{M}$.
The formula for $\hat{I}$ given in the next proposition uses $\varsigma$ and
some of the data from $\hat{A}$. The data of particular use in this regard are
listed next. First on the list is the integer
\begin{equation}\label{eq:2.1}
c_{\hat{A}}   \equiv  c_{ + } + c_{ - }.
\end{equation}
In this regard, note that $c_{\hat{A}}   \ge  0$ as it counts the number of
intersections, each weighted with its multiplicity, between any given $C \in   \mathcal{M}$
with the $\theta =0$ and $\theta =\pi$ cylinders. To
elaborate, let $C_{0}$ denote the model curve for $C$ and let $\phi \co  C_0
 \to   \mathbb{R}  \times  (S^1   \times  S^2 )$ denote its attending
$J'$--pseudoholomorphic map onto $C$. As the points in $C_{0}$ where the
pull-back of $\theta$ is either 0 or $\pi$ are isolated, the closure of a
small disk about each such point will intersect the $\theta =0$ or $\theta
=\pi$ cylinder only at its origin. Thus, the $\phi$--image of each such
disk has a well defined intersection number with the $\theta    \in \{
0, \pi \}$ locus. This number is positive because any two
$J'$--pseudoholomorphic subvarieties have only positive local intersection
numbers at their intersection points. The sum of these local intersection
numbers is the integer $c_{\hat{A}}$ in~\eqref{eq:2.1}. 

Second on the list are the integers $N_{ - }$ and $N_{ + }$, these the
respective number of elements of the form $(0, -, \ldots)$
and $(0, +, \ldots )$ in $\hat{A}$. Finally, use $\hat {N}$ to
denote the number of elements in $\hat{A}$ whose first entry is $1$ or $-1$. With
this notation set, consider:

\begin{prop}\label{prop:2.5}

The integer $\hat{I}$ from \fullref{prop:2.4} is given by the formula
\begin{equation}\label{eq:2.2}
\hat{I} = N_{ + } + 2(N_{ - }+\hat {N} + C_{\hat{A}}+\varsigma  - 1).
\end{equation}
\end{prop}
This proposition is also proved below

A point $C  \in   \mathcal{M}$ is called a regular point when the $n = 0$ case of
\fullref{prop:2.4} applies. The next proposition concerns the subset of regular
points in $\mathcal{M}$.

\begin{prop}\label{prop:2.6}

The set of regular points in $\mathcal{M}$ has the structure of a smooth manifold of dimension
$\hat{I}$. Moreover, if $C  \in   \mathcal{M}$ is a regular point, then \fullref{prop:2.4}'s
homeomorphism between the $\hat{I}$--dimensional ball and a neighborhood of $C$ in $\mathcal{M}$
defines a smooth coordinate chart.
\end{prop}

The remainder of this subsection sketches the argument for the preceding
propositions. To begin, fix a smooth Riemannian metric, $g'$, on $\mathbb{R}
\times  (S^1   \times  S^2 )$ for which $J'$ is orthogonal. In this
regard, such a metric can and should be chosen so that $\partial_{s}$ and
$\hat {\alpha }$ have norm 1, and such that $g'$ agrees with $ds^2  +
dt^2 + d\theta ^2 + \sin^2\theta  d\varphi ^2$ on the
complement of a compact subset of $\mathbb{R}  \times  (S^1   \times S^2 )$.

Now fix $C  \in   \mathcal{M}$ and let $C_{0}$ again denote the
model curve for $C$. For simplicity, assume that the
$J'$--pseudoholomorphic map $\phi$ from $C_{0}$ is an immersion. The
story when $C$ is not immersed is similar in most respects to that
given below and is summarized briefly in \fullref{sec:2d}. The reader is
referred to   \cite[Section 3]{T3} for a more detailed account.

Granted that $C$ is immersed, there exists a pull-back normal bundle, $N  \to
 C_{0}$; its fiber at any given point is the $g'$--orthogonal complement in
$T(\mathbb{R}  \times  (S^1   \times  S^2 ))|_{C}$ to $TC_0$
at the image point in $C$. This bundle inherits a complex, hermitian line
bundle structure from $J'$ and the metric $g'$. The latter structure endows $N$
with the structure of a holomorphic bundle over $C_{0}$. In addition, there
is a disk bundle $N_1   \subset  N$ of \color[rgb]{0,0,0} some fixed radius, $r_1$,
together with an immersion, $e\co  N_1   \to   \mathbb{R}\times
(S^1   \times  S^2 )$, onto a regular neighborhood of $C$ that restricts
to the zero section as the map from $C_{0}$. In this regard, $e$ is chosen so
as to embed any given fiber of $N_1$ as a pseudoholomorphic disk. Also,
the differential of $e$ is uniformly bounded, and it defines along the zero
section a $g'$--isometric map from $TC_{0}   \otimes  N$ to the pull-back of
$TX$.

Let $\eta $ denote a smooth section of $N_1$. Then $e(\eta )$ is a
pseudoholomorphic subvariety if and only if $\eta $ satisfies a certain
differential equation, one with the schematic form
\begin{equation}\label{eq:2.3}
\bar {\partial }\eta  + \nu \eta  + \mu
\bar {\eta }  + \mathcal{R} _{0}(\eta )+\mathcal{R}_{1}(\eta
)\cdot \partial \eta =0 ,
\end{equation}
where the notation is as follows: First, $\nu $ and $\mu $ are bounded
sections of $T^{0,1} C_0$ and $N^2 \otimes  T^{0,1} C_0$
respectively. Second, $\mathcal{R}_0$ is a smooth (but not complex
analytic), fiber preserving map from $N_1 \otimes  N_1$ to $N
\otimes  T^{0,1} C_0$. Meanwhile, $\mathcal{R}_1$ is a smooth (but not
complex analytic), fiber preserving map from $N_1$ to
$\Hom(T^{1,0} C_0$, $T^{0,1} C_0)$. These two maps obey
\begin{equation}\label{eq:2.4}
| \mathcal{R}_0 (\eta )|    \le  c | \eta | ^2 \qquad \text{and}\qquad
| \mathcal{R}_1 (\eta )|    \le  c | \eta |.
\end{equation}
where $c$ is a constant.

The linear part of~\eqref{eq:2.3} defines the first order, operator $D_C$,
an $\mathbb{R}$--linear map
from the space of sections of $N$ to those of $N  \otimes
T^{0,1} C_0$. Thus,
\begin{equation}\label{eq:2.5}
D_C \eta    \equiv   \bar {\partial }\eta +\nu   \eta +\mu
  \bar {\eta }.
\end{equation}
The operator $D_C$ induces a bounded, Fredholm operator between various
weighted, Sobolev space completions of certain subspaces of sections of $N$
and $N  \otimes  T^{0,1} C_{0}$. In particular, the completions of
interest are defined as follows: Fix a positive, but very small real number,
$\kappa $; an upper bound can be deduced from the data in $\hat{A}$. Now, fix a
smooth non-negative function, $r$, on $C_{0}$ with the following properties:
\begin{itemize}
\item \textsl{$r = -\kappa |s| $ on any end in $C_{0}$  that provides an element in $\hat{A}$  with first component 0.
}

\item \textsl{If $E  \subset  C_0$  is an end that contributes an element of the form $(\pm 1, \pm ,(p, p'))$, then
\begin{equation}\label{eq:2.6}
r = \bigg(-\kappa +| \frac{p' }{ p}| -\sqrt {\frac{3}{2}} \bigg) |s| \text{ on }E.
\end{equation}
}
\end{itemize}

With such a function chosen, define respective domain and range Hilbert
spaces for $D_C$ to be the completions of the spaces of smooth sections of
$N$ and $N  \otimes  T^{0,1} C_0$ for which the quadratic functionals
\begin{equation}\label{eq:2.7}
\eta \to \int_{C_0 } e^r (| \nabla \eta | ^2+| \eta
| ^2)\qquad \text{and}\qquad \eta    \to   \int_{C_0 } e^r | \eta
| ^2,
\end{equation}
are finite. The respective functionals in~\eqref{eq:2.7} define the Hilbert space
norms on the domain and range. These are both denoted here as $\| \cdot   \| $.

\cite[Lemma~3.3 in Section 3]{T3} asserts that the
operator $D_C$ defines a Fredholm operator from the domain Hilbert
space to the range. The vector spaces kernel $(D_C)$ and cokernel
$(D_C)$ refer to the respective kernel and cokernel of this Fredholm
operator. Arguments from  \cite[Sections 3d and 4b,d]{T3} prove
that the index of this Fredholm operator is the integer $\hat{I}$.

Meanwhile, almost verbatim copies of the constructions from  \cite[Section~3c]{T3} allow
\cite[Proposition~3.2]{T3}  to
construct a ball, $B  \subset \kernel(D_C)$ about the origin, a
smooth map, $f\co  B  \to \cokernel(D_C)$, and a smooth map $F\co  B  \to
C^{\infty }(N_1)$ with the following properties: First, $\| f(\eta
)\|
   \le  c \| \eta \| ^2$ and $|F(\eta )-\eta | +  | \nabla (F(\eta )-\eta |
 \le  c \| \eta \| ^2$. Second, $F$ composes
with the exponential map $e\co  N_1   \to   \mathbb{R}  \times  (S^1
\times  S^2 )$ so as to map $f^{-1}(0)$ homeomorphically onto a
neighborhood of $C$ in $\mathcal{M}$.

The conclusions of these last two paragraphs restate those of Propositions~\ref{prop:2.4}
and~\ref{prop:2.5}  The conclusions of \fullref{prop:2.6} follow as a formal
consequence of the role played by the implicit function theorem in the
construction of $F$. In this regard, keep in mind that when $C  \in   \mathcal{M}$ is a regular point,
then the use of the map $F$ provides the identification
\begin{equation}\label{eq:2.8}
T\mathcal{M}|_C = \kernel(D_C).
\end{equation}
It proves useful in subsequent arguments to write $D_C$ explicitly
on the ends of $C$. For this purpose, let $E$ denote a given end.
Constructions from  \cite[Section 2]{T3} parametrize $E$ by
coordinates $(\rho , \tau )   \in [0, \infty )  \times
\mathbb{R}/(2\pi \mathbb{Z})$ and trivialize $N$ over $E$ as $E
\times   \mathbb{R}^2$ so that $D_C$ becomes an operator of the form
\begin{equation}\label{eq:2.9}
\partial_{\rho }+
\begin{pmatrix}
 - \varsigma '  &  - \partial_\tau  \\
\partial_\tau  &  - \varsigma
\end{pmatrix}
 + \hat{o}_1 + \hat{o}_2 \cdot \partial
_{\rho } + \hat{o}_3\cdot \partial_{\tau },
\end{equation}
acting on 2--component column vectors. To elaborate, $\varsigma$ is a
positive constant for concave side ends and negative for convex side ones.
In all cases, the value of $\varsigma$ is determined by the element from
$\hat{A}$ that labels $E$. Meanwhile, $\varsigma ' = \varsigma$ when the $|
s|    \to   \infty $ limit of $\theta$ on $E$ is 0 or $\pi$.
Otherwise, $\varsigma ' = 0$. Finally, $\hat{o}_{1 - 3}$ are smooth, $2 \times
 2$ matrix valued functions whose components with their derivatives decay to
zero as $\rho    \to   \infty $ faster than $e^{-\kappa '\rho }$
with $\kappa' $ a positive constant. Note that the coordinates $(\rho ,
\tau )$ are such that the function $|s|$ restricts to $E$ as a
multiple of $\rho $. Also, $d\rho    \wedge  d\tau $ orients $E$.
Meanwhile, the trivialization chosen for $N$ is such that when $\gamma
\subset  S^1   \times  S^2 $ is a Reeb orbit where $\theta    \notin
 \{0, \pi \}$ and $E  \subset   \mathbb{R}  \times   \gamma$, then
the column vector with 0 in its top entry and 1 in its lower entry is a
positive multiple of the projection to $\gamma$'s normal bundle of the
vector field $-\partial_{\theta }$. The column vector with 1 in its top
entry and 0 in its lower entry defines a deformation of $\gamma$ along its
orbit under the action of the group, $T$, of isometries of $\mathbb{R}  \times
(S^1   \times  S^2 )$ generated by the vector fields $\partial_{t}$
and $\partial_{\varphi }$.

\subsection{The moduli space for multi-punctured spheres near immersed varieties}\label{sec:2c}

Restrict attention now to the genus zero case. Thus, $\mathcal{M}=\mathcal{M}_{\hat{A},0}$
contains only multi-punctured spheres. Using the notation in
\fullref{prop:2.5}, the model curve of each $C  \in   \mathcal{M}$ has $N_{ +}+N_{ - }+\hat {N}$
punctures. Here is the first key observation about $\mathcal{M}$:

\begin{prop}\label{prop:2.7}
The space $\mathcal{M}$ is a smooth manifold of dimension $N_{ + }+2(N_{ - }+\hat
{N}+c_{\hat{A}}-1)$ on some neighborhood of any given subvariety with only immersion singularities.
In this regard, the operator $D_C$ has trivial cokernel for each immersed subvariety in
$\mathcal{M}$ and thus each such subvariety is a regular point of $\mathcal{M}$.
\end{prop}

To explain the terminology, a subvariety is said to have only immersion
singularities when the tautological map from its model curve is an
immersion.

An analogous assertion for the non-immersed submanifolds is given in \fullref{sec:2d}.

The proof of \fullref{prop:2.7} requires a preliminary digression to introduce
new pairings between the fundamental class of a pseudoholomorphic subvariety
and certain classes from $H^2(\mathbb{R}  \times  (S^1   \times
S^2 ); \mathbb{Z})$. In this regard, keep in mind that the fundamental
class of a non-compact subvariety does not canonically define a linear
functional on this second cohomology.

In this digression, $C$ denotes any given $J'$--pseudoholomorphic
subvariety without restriction on its genus or its singularities. To
start, note that Subsection 3.a in \cite{T3} defines an integer valued
pairing between the fundamental class of an irreducible,
pseudoholomorphic subvariety and its Poincare' dual. It also defines
an integer valued pairing between the fundamental class of such a
subvariety and the first Chern class for the given almost complex
structure on $\mathbb{R}  \times  (S^1   \times  S^2 )$. With $C$
denoting the subvariety in question, these integer pairings are
respectively denoted by $\langle e, [C]\rangle$ and $\langle c_1,
[C]\rangle $; they enter both in \cite[Proposition 3.1]{T3} to
compute the Euler characteristic of $C$, and in \cite[Proposition 3.6]{T3}.

To elaborate, $\langle e, [C]\rangle $ is a `self-intersection' number
that is defined by pushing $C$ off of itself using a fiducial push-off on its
ends. In this regard, the fiducial push-off used to define $\langle e,
[C]\rangle $ pushes along any section of $N$ over the large $|s|$
part of $C$ that is homotopic there through nowhere zero sections to a
particular standard section. To obtain this standard section, note first
that $\partial_{\theta }$ is not tangent to any Reeb orbit that is
obtained as the $|s|    \to   \infty $ limit of the constant $s$
slices of any end of $C_0$ where $\lim_{| s| \to \infty }  \theta
   \notin \{0, \pi \}$. Thus, the projection of $\partial_{\theta
}$ along the large $|s|$ part of $C$ to its local normal bundle
defines a non-vanishing section, $\eta_{0}$, of $N$ over any such end. This
$\eta_{0}$ is used for the standard section over any end of $C$ where
$\lim_{| s| \to \infty }  \theta    \notin \{0, \pi \}$. A
different section is used on those ends where $\lim_{| s| \to \infty
}  \theta    \in \{0, \pi \}$.

Meanwhile, $\langle c_1, [C]\rangle $ is defined using a certain
standard section for the restriction to the large $|s|$ part of
$C$ of the $J$--version of the canonical bundle for $\mathbb{R}  \times  (S^1
\times  S^2 )$. The standard section over an end $E  \subset  C_0$
is $(dt + \frac{i}{g} df)  \wedge  (\sin^2\theta  d\varphi
+\frac{i}{g} dh)$ in the case that the $| $s$|    \to
  \infty $ limit of $\theta$ on E is neither 0 nor $\pi$. A different
section is used on the other ends.

If $C$ is not a $\theta =0$ or $\theta =\pi$ cylinder, then $\eta
_{0}$ is nowhere zero at large $|s|$ on $C$ and so can be used
on all ends of $C$ to define a pairing between $C$'s fundamental class and its
Poincare' dual. This new pairing is denoted here by $\langle e, [C]\rangle
_*$. Of course, the new pairing is identical to the old when $\hat{A}$
has no elements with first component $\pm 1$.

If $C$ is not a $\theta =0$ or a $\theta =\pi$ cylinder, then $(dt +
\frac{i}{g}df)  \wedge  (\sin^2\theta  d\varphi  +
\frac{i}{g} dh)$ is also non-zero at large $|s|$ on $C$,
and so this section can be used on all of $C$'s ends to define a new pairing
of $C$ with $c_{1}$. This new pairing is denoted by $\langle c_1,[C]\rangle_*$.
This new pairing has the virtue that it is equal
to minus the number of intersections (counted with multiplicity) between $C$
and the locus where $\theta$ is 0 or $\pi$.

To end the digression, let $m_C$ denote the number of double points
of any compactly supported perturbation of $C$ that gives an
immersed, symplectic curve with purely double point immersion
singularities with positive local intersection numbers. (Such a
perturbation always exists.) It then follows from  \cite[Proposition 3.1 and Proposition 4.1]{T3}
using observations from  \cite[Section 4c]{T3} that
\begin{align}\label{eq:2.10}
\begin{split}
-\chi (C) &= \langle e, [C]\rangle_*+\langle c_1,
[C]\rangle_* - 2m_C.\\
\hat{I} &= \langle e, [C]\rangle_*-\langle c_1,
[C]\rangle_* - 2m_C + N_{ - }+\hat {N}.
\end{split}
\end{align}
These last formulas end the digression.

\begin{proof}[Proof of \fullref{prop:2.7}]
The proof has five steps.

\step{Step 1}
By definition, no $C$ in any $\mathcal{M}_{\hat{A},\varsigma }$ is a $\theta  =
0$ or $\theta =\pi$ cylinder. Thus, the section $\eta_{0}$ as defined
above is nowhere zero at large $|s|$ on $C$. Define $\Deg(N)$ to be
the usual algebraic count of the zero's of any section of $N$ that has no
zeros where $|s|$ is large and is homotopic on the large $|s| $ slices of $C_0$
through non-vanishing sections to $\eta_0$. The following lemma identifies $\Deg(N)$:

\begin{lemma}\label{lem:2.8}
Let $\hat{A}$ be an asymptotic data set and $\varsigma$ a non-negative real number. If
$C  \in \mathcal{M}_{\hat{A},\varsigma }$ is an immersed subvariety, then
\begin{equation}\label{eq:2.11}
\Deg(N) = N_{ + } + N_{ - }+\hat {N} + c_{\hat{A}} + 2\varsigma
-2.
\end{equation}
\end{lemma}

\begin{proof}[Proof of \fullref{lem:2.8}]
By definition, the integer $\Deg(N)$ and the
pairing $\langle e, [C]\rangle_*$ are related by the formula
$\Deg(N) = \langle e, [C]\rangle_* - 2m_C$. This being the
case, add the two lines in~\eqref{eq:2.10} to obtain
\begin{equation}\label{eq:2.12}
-\chi (C) + \hat{I} = 2\Deg(N) + N_{ - }+\hat {N}.
\end{equation}
Upon rearrangement and division by 2, this last equation gives~\eqref{eq:2.11}.
\end{proof}

\step{Step 2}
Now suppose that $C  \in   \mathcal{M}_{\hat{A},0}   \equiv   \mathcal{M}_{\hat{A}}$. Let $\eta $ be
an element in $\kernel (D_C)$. It
follows from the large $|s|$ picture of $D_C$ in~\eqref{eq:2.9} that
each such $\eta $ is bounded on $C$ and has a well defined $| s|
\to   \infty $ limit on each end of $C$. Moreover, the form of~\eqref{eq:2.9} also
implies that $\eta $ has finitely many zeros on $C_0$ and so has a well
defined degree at large $|s|$ on each end of $C_0$ as measured
with respect to the degree zero standard of $\eta_{0}$ and with the
orientation of the constant $\rho $ circles reversed. (Imagine gluing a disk
to such a circle and then the degree is the degree as viewed from the origin
of the glued disk.) Moreover, the form for $D_C$ given in~\eqref{eq:2.9} implies
that this large $|s|$ degree on each end of $C_0$ is
non-negative. Indeed, a column vector with negative degree at all large
$\rho $ that is annihilated by the operator in~\eqref{eq:2.9} will grow in size as
$\rho    \to   \infty $ faster than allowed as a member of
$\kernel(D_C)$. (The growth is faster than the exponential of a positive,
constant multiple of $\rho $ whose value is determined by the data from
$\hat{A}$. Meanwhile, the function $r$ that appears in~\eqref{eq:2.7}  is defined so as to
rule out elements with such growth.)

When $E  \in  \End(C)$, use $\delta_{E}(\eta )$ to denote the degree of
$\eta $ at large $|s|$ on $E$.

Meanwhile, if $z  \in  C_0$ is a point where $\eta $ vanishes, use
$\deg_{z}(\eta )$ to denote the local degree of $\eta $. Note that
$\deg_{z}(\eta ) > 0$ at each zero of $\eta $ as can be proved using the
fact that $D_C\eta =0$.

The following identity is now a consequence of the various definitions:
\begin{equation}\label{eq:2.13}
\sum_{E}  \delta_{E}(\eta )+\sum_{\{z: \eta (z) = 0\}}
\deg_{z}(\eta ) = \Deg(N)
\end{equation}
For reference later, note that right hand equality in~\eqref{eq:2.13} is valid even
if $C$ is not everywhere $J'$--pseudoholomorphic and $\eta $ is not a solution to
a particular differential equation. Indeed,~\eqref{eq:2.13} holds provided only that
$C$ is an immersed subvariety with the large $|s|$ asymptotics of
a $J'$--pseudoholomorphic subvariety, and that the section $\eta $ of $C$'s
pull-back normal bundle has no large $|s|$ zeros. Of course, in
this general context, there need not exist sign constraints on $\delta
_{(\cdot )}(\eta )$ and $\deg_{(\cdot )}(\eta )$.

\step{Step 3}
The next point to make is that there exists a subspace, $K_{0}$, of
codimension no greater than $N_{ + }$ in $\kernel(D_C)$ with the property
that $| \eta |    \to  0$ as $| s|  \to   \infty $
on each concave side end of $C_0$ where $\lim_{| s| \to \infty }
\theta    \notin  \{0, \pi \}$. Indeed, the upper bound on the
codimension of $K_{0}$ follows from the assertion that the requirement of a
non-zero limit on a concave side end of $C_0$ defines a codimension 1
condition on $\kernel(D_C)$. And, the latter claim follows from the form of
$D_C$ given by~\eqref{eq:2.9} since any 2--component column vector that is bounded,
has non-zero limit as $\rho    \to   \infty $ and is annihilated by~\eqref{eq:2.9}
must limit to zero or to the column vector with zero in the lower entry and
1 in the upper. In this regard, remember that the constant $\varsigma$ in~\eqref{eq:2.9}
is positive for concave side ends.

To summarize: The subspace $K_0$ has $\dim(K_0)   \ge  2(N_{ -
}+\hat {N} +c_{A} -1)$ with a strict inequality when $\cokernel(D_C)$ is
non-trivial.

\step{Step 4}
The subspace $K_0$ also has the following key property: If $\eta    \in
 K_0$ is non-trivial, then the degree $\delta_{(\cdot )}(\eta )
\ge 1$  on each concave side end. Indeed, this is another consequence of the
form for $D_C$ given in~\eqref{eq:2.9} because a 2--component column vector with
zero degree at all sufficiently large $\rho $ that is annihilated by~\eqref{eq:2.9}
will converge to a non-zero multiple of the column vector with zero in the
lower entry and 1 in the upper.

With this positivity noted, then~\eqref{eq:2.11} and~\eqref{eq:2.13} imply that
\begin{equation}\label{eq:2.14}
\sum_{\{z: \eta (z) = 0\}} \deg_{z}(\eta )   \le  N_{ - }+\hat
{N}+c_{\hat{A}}-2
\end{equation}
if $\eta $ is a non-trivial element of $K_0$. In this regard, note that~\eqref{eq:2.14}
is an equality if and only if $\eta $ has degree $\delta_{(\cdot
)}(\eta ) = 1$ on each concave side end of $C_0$ where $\lim_{|
s| \to \infty }  \theta    \notin  \{0, \pi\}$ and $\delta
_{(\cdot )}(\eta )=0$ on all other ends. (Remember that $\delta
_{(\cdot )}(\eta )   \ge 0$  on all ends.)

\step{Step 5}
Now take a set $\Omega    \subset  C_0$ of $N_{ - }+\hat {N} +
c_{\hat{A}} - 1$ distinct points. These points define a subset $K_{1}
\subset  K_{0}$ of codimension no greater than $2(N_{ - }+\hat
{N}+ c_{\hat{A}}-1)$ whose elements vanish at each point in $\Omega $. By
virtue of the dimension count in \refstep{Step 3}, above, this subspace $K_1$ is
non-trivial if $\cokernel(D_C)$ is non-trivial. However, by virtue
of~\eqref{eq:2.14}, and by virtue of the fact that $\deg_{(\cdot )}(\eta ) > 0$ at each
of its zeros, the subspace $K_0$ has no elements that vanish at more than
$N_{ - }+\hat {N}+c_{\hat{A}}-1$ points of $C_0$. Thus $K_1$ must be
trivial. This being the case, $\cokernel(D_C) = \{0\}$, and it follows
that $\mathcal{M}$ is smooth near $C_0$ of the asserted dimension.
\end{proof}

\subsection{The story on $\mathcal{M}$ near non-immersed subvarieties}\label{sec:2d}

This subsection explains why $\mathcal{M}$ is a smooth manifold on neighborhoods
of its non-immersed subvarieties. There are five parts to the story.

\step{Part 1}
To start, suppose that an admissible $J'$ has been
fixed and that $C$ is a subvariety in the $J'$ version of
$\mathcal{M}$. Let $\phi \co  C_0   \to \mathbb{R}  \times  (S^1 \times
S^2 )$ again denote the tautological map from $C$'s model curve onto
$C$. Since $\phi$ is $J'$--pseudoholomorphic, it fails as an
immersion only where its differential is zero. Use $\Xi \subset
C_{0}$ to denote the set of points where this occurs. By virtue of
\fullref{prop:2.2}, the set $\Xi $ is compact, and standard
arguments about the local structure of pseudoholomorphic
subvarieties (as can be found in McDuff \cite{M}, and McDuff and Salamon
\cite{MS}) prove that $\Xi $ is
a finite set. To be more explicit, results from \cite{M} (see also
Ye \cite{Ye}) can be used to prove that there is a holomorphic
coordinate u on disk in $C_0$ centered at a given point in $\Xi $
and complex coordinates $(x, y)$ on a ball in $\mathbb{R}  \times
(S^1   \times S^2 )$ centered at the image of the given point that
give $\phi$ the form
\begin{equation}\label{eq:2.15}
\phi (u) = (a u^{q + 1}, 0) + o(| u| ^{q + 2})
\end{equation}
where $a$ is a non-zero complex number and $q  \ge  1$ is an integer. As a
consequence, the pull-back by $\phi$ of the $(1,0)$ part of the complexified
tangent space to $\mathbb{R}  \times  (S^1   \times  S^2 )$ canonically
splits as $\phi^*T_{1,0}(\mathbb{R}  \times  (S^1   \times  S^2 ))
= W  \oplus  N$ where $W$ and $N$ are complex line bundles that are
characterized as follows: The differential of $\phi$ provides a complex
linear map from $T_{1,0}C$ into $W$ and $N$ restricts to $C_0 -\Xi $
as its pull-back normal bundle.

\step{Part 2}
This step constitutes a digression to elaborate on the assertion in
\fullref{prop:2.4} in the present case. To begin the digression, note that a
deformation of the map $\phi$ that moves image points only slightly can
always be written as the image via an exponential map of a section of
$\phi^*T_{1,0}(\mathbb{R}  \times  (S^1   \times  S^2 ))$. In this
regard, an exponential map restricts to the zero section as $\phi$ and it
embeds a ball in each of the fibers that is centered at the origin and has a
base-point independent radius. Note that an exponential map can be chosen to
embed a disk about the origin in each fiber of each of the $W$ and $N$
subbundles as $J'$--pseudoholomorphic submanifolds in $\mathbb{R}  \times
(S^1   \times  S^2 )$. These disks can be assumed to have base-point
independent radii.

A section of $\phi^*T_{1,0}(\mathbb{R}  \times  (S^1   \times
S^2 )$ defines $J'$--pseudoholomorphic map from $C_0$ into $\mathbb{R}
\times  (S^1   \times  S^2 )$ if and only if it obeys a certain
non-linear differential equation whose linearization along the zero section
has the form $\underline{d}\eta =0$ where $\hat{D}$ is a first
order, $\mathbb{R}$--linear operator with the symbol of the Cauchy--Riemann
operator $\bar {\partial }$. Here, $\hat{D}$ maps the space of sections
of $\phi^*T_{1,0}(R \times  (S^1   \times  S^2 ))$ to those of
$\phi^*T_{1,0}(\mathbb{R}  \times  (S^1   \times  S^2 )) \otimes
 T^{0,1} C_0$. The operator $\hat{D}$ is described in \cite[Part 2 of Section 3b]{T3}.
 Note in particular that $\hat{D}$
maps the sections of the $W$ summand of $\phi_* T_{1,0}(\mathbb{R}
\times (S^1   \times  S^2 ))$ to sections of $W  \otimes  T^{0,1}
C_0$. In particular, if $v$ is a section of $T_{0,1} C_0$, then
$\hat{D}(\phi_* v) = \phi_*(\bar {\partial }v)$. Meanwhile,
$\hat{D}$ followed by the orthogonal projection onto $N
\otimes  T^{0,1} C_0$ acts as the operator in~\eqref{eq:2.5} on
sections over $C_0 -\Xi $ of the $N$ summand of $\phi_*
T_{0,1}(\mathbb{R}  \times  (S^1   \times  S^2 ))$.

This operator $\hat{D}$ is Fredholm when mapping a certain
$L^2_1$ Hilbert space completion of a particular subspace of sections
of $\phi^*T_{0,1}(\mathbb{R}  \times  (S^1   \times  S^2 ))$ to a
corresponding $L^{2}$ Hilbert space completion of one of $\phi^*T_{1,0}(\mathbb{R}  \times  (S^1   \times  S^2 ))  \otimes
T^{0,1} C_0$. These completions are defined by norms on the $W$ and $N$
summands of these bundles that are straightforward analogs of those that are
depicted in~\eqref{eq:2.7}. In this regard, the norms on the $W$ summand force the
sections to have limit zero as $| s|    \to   \infty $, while
those on the $N$ summands are weighted exactly as depicted in~\eqref{eq:2.7}.

Now deformations of $\phi$ that preserve the $J'$--pseudoholomorphic condition
are not of primary interest. Rather, the interest is in deformations of
$\phi$ that are $J'$--pseu\-do\-ho\-lo\-mor\-phic for an appropriately deformed complex
structure on $C_0$. The description of the latter requires the
introduction of the vector space of first order deformations of the complex
structure on $C_0$. In particular, when there are three or more ends to
$C_0$, this last vector space has dimension $N_{ + }+N_{ - }+\hat
{N}-3$, and it is the quotient of a suitably constrained (at large $s$) space
of sections of $T_{1,0} C_{0}   \otimes  T^{0,1} C_0$ by the image of
$\bar {\partial }$. This the case, fix a vector space $V$ of smooth sections
of $T_{1,0} C_0   \otimes  T^{0,1} C_0$ that projects
isomorphically to said quotient.

All this understood, let $D_C$ now denote $(1 - \prod )\cdot \underline
{D}$ where $\prod $ denotes the orthogonal projection in $D$'s range space onto
$\phi_*V$. Here, $\phi_*V$ denotes the image of $V$
under the tautological map that is defined by the differential of $\phi$,
thus a vector subspace of the summand $W  \otimes  T^{0,1} C_0$ of
$\phi ^*T_{1,0}(\mathbb{R}  \times  (S^1   \times
S^2 )) \otimes  T^{0,1} C_0$. The operator $D_C$ is viewed here
as mapping the domain of $\hat{D}$ to the image of $(1 - \prod )$ in
the range Hilbert space for $\hat{D}$.

The operator $D_C$ as just described plays the role for the
non-immersed subvarieties that is played by its namesake in the
story in the immersed case following \fullref{prop:2.6} and
in the previous subsection. In particular, arguments from
\cite[Section 3]{T3} prove the following: A neighborhood of $C$ in
$\mathcal{M}$ is homeomorphic to $f^{-1}(0)$ where $f$ is a certain
smooth map from a ball in the kernel of $D_C$ to the cokernel of
$D_C$ that vanishes with its first derivatives at the origin.
Moreover, $C$ is a regular point of $\mathcal{M}$ when
$\cokernel(D_C) = \{0\}$ in which case $\mathcal{M}$ is a smooth
manifold near $C$ and~\eqref{eq:2.8} holds. Arguments from
\cite[Section 3]{T3} also prove that the index of $D_C$ is equal to
$\hat{I}$ from~\eqref{eq:2.2}.

The identification between $f^{-1}(0)$ and $\mathcal{M}$ uses the
exponential map from a uniform radius ball subbundle of $\phi^*T(\mathbb{R}
\times  (S^1   \times  S^2 )$ into $\mathbb{R}  \times  (S^1
\times  S^2 )$. The identification also involves a certain smooth map
from the domain of $f$ to the domain of $D_C$. The latter map, $F$, is
smooth in the $C^{\infty }$ topology, it maps 0 to 0 and it is the identity
to first order. The embedding of $f^{-1}(0)$ into $\mathcal{M}$
identifies any $\lambda    \in   f^{-1}(0)$ with the image in
$\mathbb{R}  \times  (S^1   \times  S^2 )$ of the composition of
exponential map with $F(\lambda )$.

Granted all of this, here is the promised analog of \fullref{prop:2.7}:

\begin{prop}\label{prop:2.9}
The space $\mathcal{M}$ is a smooth manifold of dimension $N_{ + }+(N_{ - }+\hat{N}+c_{\hat{A}}-2)$
on some neighborhood of any given non-immersed subvariety. In particular the operator
$D_C$ as just described has trivial cokernel at each such subvariety and each is a regular point of $\mathcal{M}$.
\end{prop}

\step{Part 3}
This part of the subsection contains the following proof.

\begin{proof}[Proof of \fullref{prop:2.9}]
It follows from the discussion in the
preceding part of the subsection that it is sufficient to prove that the
dimension of the kernel of $D_C$ is equal to its index.

To start this task, recall that the restriction of the operator $D_C$ to
the elements in its domain that are sections of $W$ maps this subspace to the
subspace of its domain whose elements are sections of $W  \otimes
T^{0,1} C_0$. Denote this restricted operator as $D^W$. Meanwhile, the
composition of $D_C$ and then pointwise orthogonal projection to $N
\otimes  T^{0,1} C_0$ restricts to the subspace of sections of $N$ in
$D_C$'s domain to give a differential operator that is denoted in what
follows as $D^N$. Since $D_C$ is Fredholm, so are $D^W$ and $D^N$.
Furthermore, the sum of their indices is the index of $D_C$. In this
regard, note that $D_C$ followed by pointwise orthogonal projection to the
summand $W  \otimes  T^{0,1} C_0$ defines a zeroth order operator from
$D^N$'s domain to the range space of $D^{W}$.

To give formulas for the indices of $D^W$ and $D^N$, associate to each $z
 \in   \Xi $ the integer $q  \equiv  q_{z}$ that appears in the relevant
version of~\eqref{eq:2.15}, then set $\wp    \equiv   \sum_{z \in
\Xi } q_{z}$. An analysis much like that used in
\cite[Section 3]{T3} for $D_C$'s namesake in~\eqref{eq:2.7} finds that $D^W$
has index $2\wp $ and $D^N$ has index $\hat{I} - 2\wp $. This is all
relevant to $D_C$ by virtue of the following observation: If the
cokernel of $D^W$ is trivial, then $\kernel(D_C)$ is isomorphic to
the direct sum of the kernels of $D^W$ and $D^N$. Thus, it is enough
to prove that both $D^W$ and $D^N$ have trivial cokernel.

Consider first the case of $D^W$. Were the kernel to have $2\wp +1$
linearly independent vectors, then by virtue of~\eqref{eq:2.15}, this kernel would
have a non-trivial vector from the image via $\phi_*$ of
$TC_0$. This would provide a vector field on $C_0$. Let $v$ denote the
latter. Then $\bar {\partial }v  \in  V$ and this implies that $\bar
{\partial }v$ must be zero since $V$ is defined so as to project
isomorphically onto the cokernel of $\bar {\partial }$. But then $v$ must be
zero because its $| s|    \to   \infty $ limit is zero on all
ends of $C_0$. Thus, $D^W$ has index $2\wp $ and trivial cokernel.

Consider next the case of $D^N$. To this end, note that the bundle $N$ has
an associated degree that is defined with respect to the large $|s| $ section
that was used to give~\eqref{eq:2.8} in the immersed case. As a
consequence of~\eqref{eq:2.15}, the degree of $N$ so defined is $N_{ + }+(N_{ -
}+\hat {N}+c_{\hat{A}}-2) - \wp $. Now, \fullref{cor:2.11} to come asserts
that $\wp    \le  N_{ - }+\hat {N}+c_{\hat{A}}-2$ in all cases. Grant
this bound. As $D^N$ has index $N_{ + }+2(N_{ - }+\hat
{N}+c_{\hat{A}}-1) - 2\wp $, its index is at least $N_{ + }+2$. Since this
index is positive, the argument given in the previous subsection can be
taken in an essentially verbatim fashion to prove that $D^N$ has trivial
cokernel and thus its kernel dimension is $N_{ + }+2(N_{ - }+\hat
{N}+c_{\hat{A}}-1) - 2\wp $.
\end{proof}

\subsection{The critical points of $\theta$}\label{sec:2e}

This subsection serves as a digression of sorts to describe various key
properties of the pull-back of $\theta$ to any given subvariety in $\mathcal{M}$.
The discussion here has four parts. In this regard, note that the last
part has the promised \fullref{cor:2.11}.

\step{Part 1}
Let $J'$ denote an admissible almost complex structure and let $C$ be a
$J'$--pseudoholomorphic subvariety. Assume that $C$ is not an
$\mathbb{R}$--invariant
cylinder, but there is no need to assume that $C$ is a multiply punctured
sphere. This part explains why the only local maxima and minima of $\theta
$'s pull-back to $C$'s model curve occur where $\theta =0$ or $\theta  =\pi$.

To see why such is the case, consider a point $z  \in  C_0$ where
$\theta    \in  (0, \pi )$ and $d\theta =0$. The point $\phi (z)$ sits
in some pseudoholomorphic disk $D'$ whose tangent space is everywhere spanned
by $\partial_{s}$ and the vector field $\hat {\alpha }$. Note that
$\theta$ is constant on $D'$. Since $C$ is not $\mathbb{R}$--invariant, the closure
of $D'$ intersects $C$ only at $\phi (z)$ if its radius is small, so assume that
such is the case. Then $D'$ has a well defined, intersection number with the
$\phi$--image of any sufficiently small radius disk in $C_0$ centered on
$z$, and this is positive because both $D'$ and the image of the disk in $C_0$
are pseudoholomorphic. Were $\theta (z)$ a local maximum or minimum of
$\theta$ on $C$, then a sufficiently small radius version of $D'$ could be
pushed in the respective $\partial_{\theta }$ or $-\partial_{\theta }$
directions so that the resulting isotopy has the following two properties:
First, it avoids the $\phi$--image of the boundary of any sufficiently small
radius disk in $C_0$ centered on $z$. Second, it results in a disk that is
entirely disjoint from the $\phi$--image of the whole any sufficiently small
radius disk centered at $z$. Such an isotopy is precluded by the positive
intersection number between $D'$ and the $\phi$--image of the disks in
$C_0$ centered on $z$.

\step{Part 2}
Continuing the story from \refstep{Part 1}, define the degree of vanishing of
$d\theta$ at $z$ to be one less than the intersection number between $D'$ and
the $\phi$--image of any sufficiently small radius disk in $C_0$ centered
at $z$. Denote this number by $\deg(d\theta |_{z})$. What follows is
an equivalent definition of this number.

To start, note that by virtue of the fact that $\theta$ is constant on
the integral curves of the vector field $\hat {\alpha }$ but not so on
$C$, a neighborhood in $\mathbb{R} \times (S^1 \times S^2 )$ of $\phi
(z)$ has complex coordinates, $(x, y)$, with the following four
properties: First, $(0, 0)$ is the point $z$. Second, $dx$ and $dy
\spann T^{1,0}(\mathbb{R} \times (S^1 \times S^2 ))$ at $\phi
(z)$. Third, the $y = 0$ disk is $J'$--pseudoholomorphic. Finally, the
constant $x$ disks centered where $y = 0$ are $J'$--pseudoholomorphic
disks whose tangent planes are everywhere spanned by the vector fields
$\partial_{s}$ and $\hat {\alpha }$. Thus, $\theta$ is constant on the
$x = \constant$ disks and so a function of $x$ only. This understood,
then $\theta$ can be written as $\theta = \Ree(\sigma x) + o(| x| ^2)$
with $\sigma$ a non-zero constant

Meanwhile, by virtue of the fact that $\phi$ is $J'$--pseudoholomorphic, there
is a complex coordinate, $u$, for a small radius disk in $C_0$ centered at $z$
such that $\phi$ pulls the coordinate $x$ back as $\phi^*x = a u^{p
+ 1} + o(| u| ^{p + 2})$ with $a  \in   \mathbb{C}$ and with $p$
an integer of size at least 1. The integer $p$ is the degree of vanishing of
$d\theta$ at $z$ because $d\theta$ pulls back near $z$ as
\begin{equation}\label{eq:2.16}
d\theta  =p \Ree(\sigma  a u^{p}du) + o(| u| ^{p + 1}).
\end{equation}

\step{Part 3}
Note that $C_0$ has at most a finite number of critical points in any
given compact subset. This is a consequence of~\eqref{eq:2.16} and the fact that
$\theta$'s extremal critical points occur where $C$ intersects the $\theta
= 0$ or $\theta =\pi$ cylinders.

The next point is that $d\theta$ is non-vanishing at all
sufficiently large values of $|s|$ on $C_0$. To see that such is the
case, note first that if $E$ is an end of $C$ where the $| s|    \to
\infty $ limit of $\theta$ is either 0 or $\pi$,
then~\eqref{eq:1.12} implies that there are no large $|s|$ critical
points of $\theta$ on $E$. The analysis used in
\cite[Sections 2 and 3]{T3} also serves to prove this in the case that the $| s| \to
\infty $ limit of $\theta$ on $E$ is in $(0, \pi )$. To be more
precise with regard to the latter case, these techniques in
\cite{T3} find coordinates $(\rho , \tau )$ for $E$ such that $\rho
$ is a positive multiple of $s$, $\tau    \in \mathbb{R}/(2\pi
\mathbb{Z})$, and $d\rho    \wedge d\tau $ is positive. Moreover,
when written as a function of $\rho $ and $\tau $, the function
$\theta$ has the form
\begin{equation}\label{eq:2.17}
\theta (\rho , \tau )=\theta_{E} + e^{-r\rho } \big(b
\cos(n(\tau +\sigma )) + \hat{o}\big),
\end{equation}
where the notation is as follows: First, $\theta_{E}$ is the $|
s|  \to   \infty $ limit of $\theta$ on $E$. Second, $r > 0$ when $E$ is
on the concave side and $r < 0$ when $E$ is on the convex side of $C$. Third, $b$ is
a non-zero real number, $n$ is a non-negative integer, but strictly positive
if $E$ is on the concave side of $C_0$, and $\sigma    \in   \mathbb{R}/(2\pi \mathbb{Z})$.
Finally, $\hat{o}$ and its first derivatives limit to
zero as $| \rho |    \to   \infty $.

In what follows, the integer $n$ that appears in~\eqref{eq:2.17} is denoted as
$\deg_{E}(d\theta )$. In case that $E$ is an end of $C$ where $\lim_{|
s| \to \infty }  \theta$ is 0 or $\pi$, define $\deg_{E}(d\theta )$
to be zero.

\step{Part 4}
This part starts with the following key proposition:

\begin{prop}\label{prop:2.10}
Let $C$  be a $J'$--pseudoholomorphic subvariety that is not
$\mathbb{R}$--in\-var\-i\-ant and introduce $k_C$
to denote the number of points in $C_0$ where $\theta$ is either zero or $\pi$. Then
\begin{equation}\label{eq:2.18}
\sum_{E} \deg_{E}(d\theta )_+\sum_{z} \deg(d\theta |
_{z}) = N_{ + }+N_{ - }+\hat {N}+k_{C}+2\varsigma -2
\end{equation}
where the first sum on the left hand side is indexed by the ends of $C$, and the second sum
on the left hand side is indexed by the set of non-extremal critical points of $\theta$ on $C_0$.
\end{prop}

\begin{proof}[Proof of \fullref{prop:2.10}]

This is a standard Euler class
calculation given that all of the extremal points of $\theta$'s pull-back
to $C_0$ occur where $\theta =0$ or $\theta =\pi$. The $k_C$
summand on the right hand side of~\eqref{eq:2.18} accounts for the singular behavior
of $d\theta$ at the points in $C_0$ where $\theta$ is either 0 or $\pi$.

This proposition has three immediate consequences. These are stated in the
upcoming corollary. This corollary refers to the integer $\wp $ that is
defined from the singular points of $\phi_*$ as in the proof of
\fullref{prop:2.9}. In particular, $\wp $ can be defined for any
$J'$--pseudoholomorphic subvariety that is not $\mathbb{R}$--invariant. To be
precise, each zero of $\phi_*$ on such a subvariety has a version
of~\eqref{eq:2.15} with an attending integer $q$. Then $\wp $ is obtained by adding the
resulting set of integers. The corollary that follows also refers to the
integer $\wp_*$ that is obtained by restricting the sum for $\wp
$ to the integers that are associated to the zeros of $\phi_*$ that
lie where $\theta =0$ or $\theta =\pi$.
\end{proof}

\begin{corollary}\label{cor:2.11}
Suppose that $C$  is a $J'$--pseudoholomorphic subvariety that is not $\mathbb{R}$--invariant. Then
\begin{itemize}
\item
The number of non-extremal critical points of $\theta$'s pull-back to the model curve $C_0$
is no greater than $N_{ - }+\hat {N}+k_C+2\varsigma -2$.

\item
$\wp_* \le c_{\hat{A}}-k_C$.

\item
$\wp \le N_{ - }+\hat {N}+c_{\hat{A}}+2\varsigma -2$.
\end{itemize}
\end{corollary}

\begin{proof}[Proof of \fullref{cor:2.11}]
The asserted bound on the critical points
follows directly from~\eqref{eq:2.18} by virtue of the fact that $\deg_{E}(d\theta
)$ is in all cases non-negative and at least 1 on the $(0,+,\ldots )$ elements in $\hat{A}$.

As for the bounds on $\wp_*$ and $\wp $, remark first that
$d(\cos\theta )$ pulls back as zero at the singular points of $\phi_*$.
Now, to argue for the bound on $\wp_*$, focus attention on
a point $z  \in C_0$ where $\phi_*$ is zero and $\theta$ is
either zero or $\pi$. Let $q_z$ denote the integer that appears in the
corresponding version of~\eqref{eq:2.15}. Then it follows from~\eqref{eq:2.15} that $z$
contributes a factor of at least $q_{z}+1$ to the count for $c_A$. This
observation implies the bound for $\wp_*$.

Granted the $\wp_*$ bound, then the asserted bound on $\wp $ then
follows from the fact that $\wp -\wp_*$ is no greater
than the second sum on the left hand side of~\eqref{eq:2.18}.
\end{proof}

\subsection{Local parametrizations for points in $\mathcal{M}$}\label{sec:2f}

A close reading of the proofs of Propositions~\ref{prop:2.7} and~\ref{prop:2.9} indicate that
certain natural functions on the subvarieties in $\mathcal{M}_{\hat{A}}$ can
serve as local coordinates. For example, the proof suggest that such is the
case for the $\mathbb{R}/(2\pi   \mathbb{Z})$ parameters that characterize
concave side ends where $\lim_{| s| \to \infty }  \theta    \notin
 \{0, \pi \}$. The purpose of this subsection is to prove that such is
the case, and also to provide an expanded list of local coordinates.

To begin, partition $\hat{A}$ into the disjoint subsets whereby any two elements
in the same subset are identical and any two elements from distinct subsets
are distinct. Let $\Lambda $ denote this list of subsets. Label the elements
in each subset in the partition $\Lambda $ by consecutive integers starting
at 1.

Now, let $\mathcal{M}^{\Lambda }$ denote the set of elements of the form $(C, L
 \equiv \{L_{\lambda} \}_{\lambda \in \Lambda})$ where $C \in
\mathcal{M}_{\hat{A}}$ and where any given $L_{\lambda }$ is a 1--1
correspondence between the subset $\lambda $ and the set of ends of $C$ that
contribute elements to $\lambda $. This space $\mathcal{M}^{\Lambda }$ has a
natural topology whereby the evident projection to $\mathcal{M}_{\hat{A}}$ is a
covering map. It is fair to view $\mathcal{M}^{\Lambda }$ as a moduli
space of subvarieties with labeled ends.

The point of this is that a pair consisting of a subset $\lambda $ from the
partition $\Lambda $ and an element $r  \in   \lambda $, define a function,
\begin{equation}\label{eq:2.19}
\varpi_{\lambda ,r}\co  \mathcal{M}^{\Lambda }   \to   \mathbb{C}^{*}
\end{equation}
that maps a given $(C, L)$ to $\exp(c_{E(r)} + i \iota_{E(r)})$.
Here, $E(r)  \subset C$ is the end that $L_{\lambda}$  assigns to $r$, while $c_{(\cdot
)}$ and $\iota_{(\cdot )}$ are the continuous parameters that are
associated to the given end. In this regard, note that when $\lambda $
consists of elements of the form $(0,+,\ldots )$, then any $r
 \in   \lambda $ version of $\varpi_{\lambda ,r}$ maps to the unit
circle in $\mathbb{C}$. Let
\begin{equation}\label{eq:2.20}
\varpi_{ + }\co \mathcal{M}^{\Lambda } \to  \times
_{N_ + } S^1
\end{equation}
denote the product of the $N_{ + }$ versions of $\varpi_{\lambda ,r}$
where $\lambda   \subset \hat{A}$ is a such a $(0,+,\ldots )$
element.

Now consider the following:

\begin{prop}\label{prop:2.12}
Fix integers $b  \in \{ 0,\ldots , N_{ - }\}$ and $c\le
N_{ - }+\hat {N}+c_{\hat{A}}-b-2$, then fix a size $b$ subset of $(0,-,\ldots )$
elements from $\hat{A}$ and a  $c$--element subset in $(0, \pi )$. Use $B$ for the former and
$\theta$ for the latter, and use $\mathcal{M}^{B}[\Theta ] \subset   \mathcal{M}^{\Lambda }$
to denote the subset of subvarieties with the following three properties: First,
$\theta$'s pull-back to the associated model curve has precisely $c$ non-extremal critical points.
Second, $\theta$ is the corresponding $c$--element set of critical values. Finally, the versions of
$c_{(,)}$ from~\eqref{eq:1.8} vanish for the ends from $B$ but for no other convex side ends.
If non-empty, this set $\mathcal{M}^{B}[\Theta ]$ is a smooth submanifold of $\mathcal{M}^{\Lambda }$
of dimension $N_{ + }+b+c+2$. Furthermore, choose any map $\varpi_{\lambda ,r}$ for which $r$ is a $(\pm 1, \ldots)$
from $\hat{A}$ or a $(0, -, \ldots )$ element from $\hat{A}-B$. Then, the map $\varpi_{ + }  \times \,(\times
_{(\lambda ', r'):r' \in B}  \varpi_{\lambda ',r'})\times   \varpi_{\lambda ,r}$ restricts over
$\mathcal{M}[\Theta ]$ as a smooth submersion to $(\times_{N_ + }
S^1 )  \times  (\times_{b} S^1 )  \times  \mathbb{C}$*.
\end{prop}

Note that in the case $B = \emptyset$, then $\mathcal{M}^{B}[\Theta]$ is denoted
by $\mathcal{M}[\Theta ]$ and viewed as a submanifold of dimension $N_{
+ }+c+2$ in $\mathcal{M}$.

Local coordinates near any given subvariety in $\mathcal{M}^{B}[\Theta ]$
can be obtained in the following manner: Let $(C, L)  \in   \mathcal{M}^{B}[\Theta ]$ and let $\{z_1, \ldots , z_c\}
 \subset C_0$ denote a labeling of the non-extremal critical points of $\theta$'s
pull-back to $C_0$. This set denoted by $\Crit(C)$ in what follows. Let
$\phi \co  C_{0}   \to   \mathbb{R}  \times
(S^1   \times  S^2 )$ again denote the tautological map. There is an
open neighborhood of $C$ in $\mathcal{M}$ whose subvarieties enjoy the following
property: The model curve of each such subvariety can be viewed as the image
in the normal bundle $\phi^*T_{1,0}(\mathbb{R}  \times  (S^1   \times
 S^2 ))$ of a section with everywhere very small norm. This is in accord
with the story from \fullref{prop:2.4} and the discussion prior to \fullref{prop:2.9}.
If $C'$ is in such a neighborhood and comes from a point near $C$ in $\mathcal{M}^{B}[\Theta ]$,
then the non-extremal critical points of $\theta$'s
pull-back to the model curve of $C'$ can be put in 1--1 correspondence with the
points in the set $\{z_1, \ldots , z_c\}$ by associating any given
critical point of $\theta$ on $C_0'$ with the closest critical point
in $C_0$ as measured by distance in $\phi^*T_{1,0}(\mathbb{R}  \times
(S^1   \times  S^2 ))$. This understood, use $\{z_1', \ldots ,z_c'\}$ to denote the
corresponding labeled set of non-extremal
critical points of $\theta$'s pull-back to the model curve of $C'$. Note that
the degree of vanishing of $d\theta$ at any given $z_{k}'$ on the model
curve for $C'$ is identical to the degree that $d\theta$ vanishes on $C'$s
model curve at $z_{k}$.

To continue, suppose $z  \in  \{z_1, \ldots , z_n\}$.  Fix
respective $\mathbb{R}$--valued lifts, $\hat {t}$ and $\hat {\varphi }$, of the
$\mathbb{R}/(2\pi \mathbb{Z})$ valued functions $t$ and $\varphi$ that are
defined on a neighborhood in $\mathbb{R}  \times  (S^1   \times  S^2 )$
of the image of $z$. Then set
\begin{equation}\label{eq:2.21}
v  \equiv (1 - 3\cos^2 \theta )  \hat {\varphi }-\surd 6
\cos\theta   \hat {t}.
\end{equation}
The $c$ versions of~\eqref{eq:2.21} define $c$ functions, $\{v_{1}(\cdot ), \ldots ,
v_{c}(\cdot )\}$, on a neighborhood of $C$ in $\mathcal{M}^{B}[\Theta ]$
as follows: The value of $v_{k}$ on a given subvariety $C'$ is the value of
the $z_{k}$ version of~\eqref{eq:2.21} at the image in $\mathbb{R}  \times  (S^1
\times  S^2 )$ of the $C'$ version of the $\theta$--critical point
$z_{k}'$.

To obtain $N_{ + }+b+2$ additional functions, the map $\varpi_{ + }$ from~\eqref{eq:2.20}
pulls back an affine coordinate from each of the $S^1 $ factors from
its range space. Let $\{\varpi_{ + 1}, \ldots \}$ denote an ordering
of the resulting $N_{ + }$ element set of such functions. Use $\{\varpi
_{ - 1},\ldots $, $\varpi_{ - b}\}$ to denote an ordering of
the $b$ affine function that are pulled back by viewing $\times_{(\lambda
',e'):r' \in B} \varpi_{\lambda ',r'}$ as a map from $\mathcal{M}^{B}[\Theta ]$ to $\times_{b} S^1 $.
Finally, either choose one
of the following: A pair $(\lambda , r)$ with $\lambda    \in   \Lambda $
and $r$ some $(\pm 1,\ldots )$ element from $\hat{A}$. Or, a pair
$(\lambda , r)$ with $\lambda    \in   \Lambda $ and $r  \in  \hat{A}-B$
some $(0,-,\ldots  )$ element. Or, a point  $z \in C_0$
where $\theta$ is 0 or $\pi$. In the first two cases, use $\varpi
_{\lambda ,r}$ from~\eqref{eq:2.19} to pull-back the standard complex coordinate on
$\mathbb{C}$ and call this coordinate $\varpi'$. For the third case, note that
by virtue of~\eqref{eq:2.18}, each $C'$ from $\mathcal{M}^{B}[\Theta ]$ near $C$ has an
unambiguous point in its model curve that corresponds to $z$ and also maps
very near to $z$ in the $\theta    \in \{0, \pi \}$ locus. Let $z'$
denote the latter. Because $C'$ comes from $\mathcal{M}^{B}[\Theta ]$, the
contribution from $z'$ to the intersection number between $C'$ and the $\{0,
\pi \}$ locus is the same as that from $z$ to $C'$s intersection number with
the $\{0, \pi \}$ locus. Use $\varpi' $ in this case to denote the
complex valued function on $C'$s neighborhood that assigns to any given $C'$ the
value of the complex coordinate on the $\theta    \in \{0, \pi \}$
locus at the image of $z'$.

\begin{prop}\label{prop:2.13}
The functions $\{v_{j}: 1  \le  j  \le  c\}$, $\{\varpi_{ + \alpha }\}$, $\{\varpi_{ - \alpha }\}$
and $\varpi' $ together define local coordinates on a neighborhood of $C$  in $\mathcal{M}^{B}[\Theta ]$.
\end{prop}

\begin{proof}[Proof of Propositions~\ref{prop:2.12} and~\ref{prop:2.13}]

The proof is given in five parts.

\step{Part 1}
This part provides some comments on the dimension count for $\mathcal{M}^{B}[\Theta ]$. For this purpose, let $C  \in
\mathcal{M}^{B}[\Theta ]$ and let $\Crit(C)$ denote the set of non-extremal critical
points of $\theta$'s pull-back to $C'$s model curve. As noted earlier, if $C'
 \in   \mathcal{M}^{B}[\Theta ]$ is near $C$, then corresponding critical
points of $\theta$'s pull-back to the model curves for $C$ and $C'$ must have
the same values for $\theta$ and also for $\deg(d\theta |_{(\cdot)})$.
This places $2\sum_{z \in \Crit(C)} \deg(d\theta |_{z})- c$
constraints on the subvarieties that are near $C$ in $\mathcal{M}^{B}[\Theta ]$.
Requiring that they also have the same values for the
functions $v_{j}$ adds c more constraints.

There are more constraints that come from $C'$s ends. To see these, let  $E \subset C$
denote an end that corresponds to some $(0,\ldots)$ element in $\hat{A}$. A given $C'$ near to $C$ in $\mathcal{M}$
has a corresponding
end $E'$. It follows from~\eqref{eq:2.17} and~\eqref{eq:2.18} that $\theta$'s pull-back to the
model curve of $C'$ has critical points at large values of $|s|$
on $E'$ in the case that $\deg_{E'}(d\theta ) < \deg_{E}(d\theta )$.
Taking this into account finds an additional set of $2(\sum_{E}\deg_{E}(d\theta )-N_{ + })-b$
constraints on the elements in $\mathcal{M}^{B}[\Theta ]$.

Even more constraints arise from the local intersection numbers of the
subvarieties with the $\theta    \in \{0, \pi \}$ locus. To
elaborate, let $z$ denote a point in $C'$s model curve that maps to a point in
this locus, and let $p$ denote $z'$s contribution to the intersection number
$c_{\hat{A}}$. If $C'  \in   \mathcal{M}^{B}[\Theta ]$ is near $C$, then~\eqref{eq:2.18}
requires $C'$ to have a corresponding point very near $z$ in its model curve
which contributes $p$ to the intersection number between $C'$ and the $\theta
 \in  \{0, \pi \}$ locus. All such intersection points thus account
for an additional $2(c_{A}- k_C)$ constraints on the subvarieties in
$\mathcal{M}^{B}[\Theta ]$.

When totalled, the number of constraints that must be satisfied for placement
in $\mathcal{M}^{B}[\Theta ]$ is
\begin{equation}\label{eq:2.22}
2\Bigl(\sum_{E} \deg_{E}(d\theta )+\sum_{z \in \Crit(C)}
\deg(d\theta |_{z}) + c_{\hat{A}} - k_C\Bigr) - 2N_{ + } - b -c.
\end{equation}
What with~\eqref{eq:2.18}, this number is $N_{ + }+b+c$ less than the dimension of
$\mathcal{M}$.

\step{Part 2}
Granted this count, the assertion that $\mathcal{M}^{B}[\Theta ]$ is a
manifold of the asserted dimension with the given local coordinates is
proved using an application of the implicit function theorem. The
application requires the introduction of the linearized constraints on the
tangent space to $C$ in $\mathcal{M}$. For this purpose, identify $T\mathcal{M}|
_{C }= \kernel(D_{C})$ and let $K^*   \subset  \kernel(D_C)$
denote the subspace of vectors that satisfy all of the linear constraints and
also annihilate all of the functions that are listed in \fullref{prop:2.13}.
Both Propositions~\ref{prop:2.12}  and~\ref{prop:2.13} follow from the implicit function theorem if
$K^*$ is trivial.

The identification given below of $K^*$ requires a preliminary
digression to introduce some notation that concerns a given $\lambda    \in
 \kernel(D_C)$. First, $\lambda ^{N}$ is used below to denote $\lambda$'s image in $N$ via the
 projection from $\phi^*T_{1,0}(\mathbb{R}  \times
(S^1   \times  S^2 ))$. Second, $\langle d\theta , \lambda \rangle $
is used to denote the pairing on the complement of the $\theta = 0$
and $\pi$ points between $\lambda $ and $d\theta$ when the latter
form is viewed in $\phi^*T^{1,0}(\mathbb{R}  \times  (S^1   \times
S^2 ))$. Finally, when $z$ is a point in $C'$s model curve where
$\theta$ is 0 or $\pi$, then $r(\lambda )$ denotes the projection
of $\lambda $ to the holomorphic tangent bundle of the $t =
\constant$ pseudoholomorphic subvariety through the given point.
These subvarieties are described in   \cite[Section 4a]{T3}, and for
the present purposes, it is enough to know that these subvarieties
are embedded, they foliate $\mathbb{R}  \times  (S^1   \times
 S^2 )$, and the tangent bundle to any such subvariety on the $\theta  =0$ or $\pi$ locus is normal to the locus.

What follows are the conditions for membership in $K^*$:
\itaubes{2.23}
\textsl{If $\varpi$' is defined by an element $r  \in  \hat{A}$, then $\deg_{E}(\lambda ^N) > 0$ on the end $E  \subset
C$ that corresponds to $r$.
}

\item \textsl{If $E  \subset C$ is an end where $\deg_{E}(d\theta ) > 0$, then $\delta_{E}(\lambda
^{N})   \ge  \deg_{E}(d\theta )$.
}

\item \textsl{If $z  \in C_0$ is a non-extremal critical point of the pull-back of $\theta$ and $u$ is a local
holomorphic coordinate for a disk in $C_0$ centered at $z$, then $\langle d\theta , \lambda
\rangle  = o(| u| ^{k})$ near $z$ with $k  \ge  \deg(d\theta |_{z})$.
}

\item \textsl{If $z  \in C_0$ is a point where $\theta =0$ or $\pi$, let $p$ denote $z'$s contribution to $c_{\hat{A}}$. Let $u$
denote a complex coordinate for a disk in $C_0$ centered at $z$. If $p  \ge  2$, then $r(\lambda
) = o(| u| ^{k})$ near $z$ with $k  \ge  p-1$.
}

\item \textsl{Suppose that $z  \in C_0$ is a point where $\theta =0$ and that $z$ is used to define $\varpi$'. If
$\phi_*|_{z}$ is zero, then $\lambda |_{z}$ must also vanish.  If $\phi_*|_{z}$ is non-zero, then $\eta
   \sim   | u| ^{k}$ near $z$ where $k  \ge  p$.
}
\end{itemize}

Note that the final three constraints only involve $\eta $ where $\phi_*$ is non-zero.

\step{Part 3}
As there are $\dim(\kernel(D_C))$ conditions in~\eqreft2{23}, a proof that they
are linearly independent proves that $K^* = \{0\}$. The argument
for linear independence invokes two observations that concern a section,
$\xi $, of the $W$ summand in $\phi^*T_{1,0}(\mathbb{R}  \times  (S^1
\times  S^2 ))$. To set the stage, let $z$ denote a given point in $C_0$
and let $q_{z}$ denote the integer that appears in $z'$s version of~\eqref{eq:2.15}.
Thus, $q_{z}=0$ if $\phi_*|_{z}   \ne 0$  and $q_{z}> 0$ otherwise.

Here is the first observation: If $z  \in  \Crit(C)$ and $u$ is a complex
coordinate for a disk centered in $C_0$ with center $z$, then $\langle
d\theta , \xi \rangle =o(| u| ^{k})$ near $z$
with $k  \ge  \deg(d\theta |_{z})-q_{z}$. Moreover, if $\langle
d\theta , \xi \rangle  = o(| u| ^{k})$ with $k  \ge
\deg(d\theta |_{z})$, then $\xi $ near $z$ is the image via $\phi
_*$ of a section of $T_{1,0} C_0$. This observation follows
from~\eqref{eq:2.15} and~\eqref{eq:2.16}.

The second observation concerns a point $z  \in  C_0$ where $\theta$ is
0 or $\pi$, and it involves the integer, $p_{z}$, that $z$ contributes to
$c_{\hat{A}}$. Here is the observation: Let $u$ denote a complex coordinate for
a disk in $C_0$ centered at $z$. Then $r(\xi )=o(| u|
^{k})$ where $k  \ge  p_{z}-q_{z}-1$ near $z$, and if $\xi |
_{z}=0$, then $r(\xi ) = o(| u| ^{k})$ with $k  \ge
p_{z}-q_{z}$. Moreover, if $r(\xi ) = o(| u| ^{k})$ with $k
 \ge  p_{z}-1$, then $\xi $ near $z$ is the image via $\phi_*$
of a section of $T_{1,0} C_0$. This observation follows from~\eqref{eq:2.15}.

\step{Part 4}
The analysis of the conditions in~\eqreft2{23} begins by considering those that
involve only the projection $\lambda ^{N}$ of $\lambda $. In particular,
such is the case for the first two. As is explained next, some of the others
also involve only $\lambda ^{N}$. In particular, suppose first that $z
\in  \Crit(C)$. As noted previously, $\deg(d\theta |_{z})   \ge
q_{z}$ and so, by virtue of \refstep{Part 3}'s first observation, the vanishing to
order $\deg(d\theta |_{z})$ of $\langle d\theta , \lambda
\rangle $ at $z$ puts $2(\deg(d\theta |_{z})-q_{z})$ constraints
on $\lambda ^{N}$ as it requires that $\lambda ^{N} = o(| u|
^{k})$ near $z$ with $k  \ge  (\deg(d\theta |_{z})-q_{z})$.

Consider next a point $z  \in C_0$ where $\theta$ is 0 or $\pi$. As
noted previously, $p_{z}   \ge  q_z+1$, and so by virtue of \refstep{Part 3}'s
second observation, the vanishing to order $p_{z}-1$ of $r(\lambda )$ at $z$
forces $\lambda ^{N}$ to be $o(| u| ^{k})$ near $z$ where $k
\ge p_z -q_z-1$. Of course, this is a constraint only in the case
that $p_{z} > q_z+1$. In the case that $z$ is used to define $\varpi' $,
the final point in~\eqreft2{23} forces $\lambda ^{N}$ to be $o(| u|
^{k})$ near $z$ with $k  \ge p_{z} -q_z$.

If $\lambda ^{N}$ is to satisfy all of these constraints, then
\begin{multline}\label{eq:2.24}
\sum_{E}  \delta_{E}(\lambda ^{N})+\sum_{z}
\deg_z(\lambda ^{N})   \\
\qquad \ge   \sum_{E} \deg_E (d\theta ) +
\hspace{-5pt}\sum_{z \in \Crit(C)} (\deg(d\theta |_{z})-q_{z})
+\hspace{-5pt}\sum_{z:\theta = 0 or \pi } (p_{z}-q_{z}-1) + 1,
\end{multline}
and this, according to~\eqref{eq:2.18}, is greater than $\deg(N) = N_{ + }+N_{ -
}+\hat {N}+c_{\hat{A}}-2-\wp $. As a consequence, $\lambda ^ N = 0$ if $\lambda    \in K^*$.

\step{Part 5}
As just noted, any $\lambda \in K^*$ is a section of $W$.
Granted the first observation from \refstep{Part 3}, the third condition in~\eqreft2{23}
implies that such $\lambda $ is the image via $\phi_*$ near each
$z  \in  \Crit(C)$ of a section of $T_{1,0} C_0$. What with second
observation from \refstep{Part 3}, the fourth condition in~\eqreft2{23} implies that
$\lambda $ is also in the image of $\phi_*$ near each point in
$C_0$ where $\theta$ is 0 or $\pi$. Thus, $\lambda$ is in the
image of $\phi_*$ on the whole of $C_0$ since any $\theta
\in  (0, \pi )$ zero of $\phi_*$ is a zero of $\phi^*d\theta$.
However, as noted in the proof of \fullref{prop:2.9}, the kernel of $D^W$
has only 0 from the image of $\phi_*$. Thus, $K^* = \{0\}$ as required.
\end{proof}

\subsection{Slicing curves by $\theta$ level sets}\label{sec:2g}

This subsection constitutes a digression of sorts to discuss some algebraic
and geometric issues that arise in conjunction with the use of the critical
points of $\theta$'s pull-back to construct coordinates on $\mathcal{M}_{\hat{A}}$.
Some of these issues appear both implicitly and explicitly in
the subsequent sections of this article, and they play a central role in the
sequel to this article. In any event, the subsection starts by examining the
nature of the $\theta$--level sets in any given subvariety from $\mathcal{M}$.
These level sets are then used to associate to each such variety a certain
connected, contractible graph with labeled vertices and labeled edges. The
discussion of the constant $\theta$ loci is contained in \refstep{Part 1}--\refstep{Part 3} of this
subsection, while \refstep{Part 4} contains the definition of the associated graph. In
all of what follows, it is assumed that the subvariety $C$ in question is not
an $\mathbb{R}$--invariant cylinder, thus not of the form $\mathbb{R}  \times
\gamma$ where $\gamma    \subset  S^1   \times  S^2 $ is a Reeb
orbit.

\step{Part 1}
To begin the story here, suppose that
$\hat{A}$ is an asymptotic data set, $J'$ is an admissible almost complex
structure, and $C$ is a subvariety from the $J'$--version of $\mathcal{M}_{\hat{A}}$.
Let $C_0$ again denote the model curve for $C$. Introduce now the locus,
$\Gamma    \subset C_0$, which is defined as follows: The components
of this set consist of the level sets of $\theta$ on $C_0$ that are
either zero dimensional, singular or non-compact. In particular, $\Gamma$
contains all of the critical points of $\theta$ on $C_0$.

To continue, note that any given component of $\Gamma$ can be viewed as the
embedded image in $C_0$ of an oriented graph with labeled edges and
vertices. To elaborate, the zero dimensional components of $\Gamma$ are the
points in $C_0$ where $\theta$ is 0 or $\pi$. In particular each zero
dimensional component is a graph with a single vertex, the latter labeled by
a non-zero integer whose absolute value is the contribution of the given
$\theta =0$ or $\pi$ point to $c_{\hat{A}}$ . The sign of the integer is
positive when the $\theta$ value is 0 and the integer is negative when the
$\theta$ value is $\pi$.

Each singular point in a non-point like component of $\Gamma$ is a critical
point of $\theta$. These points constitute the vertices of the
corresponding graph. The components of the complement of these singular
points constitute the edges in the graph. In this regard, these edges are
henceforth referred to as `arcs' so as not to confuse them with the edges in
the graph that is defined subsequently in \refstep{Part 4} from $C$. These arcs are
oriented by the pull-back of the 1--form $x  \equiv  (1-3\cos^2 \theta )
d\varphi -\surd 6\cos \theta  dt$. Note that this 1--form is nowhere
zero on the smooth portion of any given $\theta$ level set in $C_0$ for
the following reason: The differential of the contact form in~\eqref{eq:1.1} is
$\surd 6 \sin \theta d\theta    \wedge  x$ and because $J'$ is admissible
and $C$ is $J'$--pseudoholomorphic, this form is positive on $TC_0$ save at the
critical points of $\theta$ where it vanishes.

Because a given vertex in a non-point like component of $\Gamma$ is a
critical point of $\theta$, it has an even number, at least 4, of incident
arcs. This follows from the form of $d\theta$ in~\eqref{eq:2.16}. Moreover, half of
the incident arcs are oriented to point towards the vertex and half are
oriented to point away. Indeed, a circumnavigation of a small radius circle
about the critical point will alternately meet inward pointing and outward
pointing arcs. For example, if $\Gamma_*   \subset   \Gamma$
is a compact, singular component with a single, non-degenerate, non-extremal
critical point, then the associated graph has a single vertex and looks like
the figure `8'.

Meanwhile, the complement of the $\theta$ critical points in a given
non-compact component of $\Gamma$ has an even number of unbounded arcs.
Indeed, this follows from~\eqref{eq:2.17}. In particular, any given end of $C$ where
the $| s|    \to   \infty $ limit of $\theta$ is neither 0 nor
$\pi$ has the following property: Let n denote the integer that appears in~\eqref{eq:2.17}
 for the given end. Then any sufficiently large and constant $|s| $
 slice of the end intersects precisely 2n components of $\Gamma$
and this intersection is transverse. Moreover, a circumnavigation of the
constant $|s|$ slice meets components whose orientations
alternate towards increasing $|s|$ and towards decreasing $|s|$.

By way of an example, suppose that $E$ is a concave side end of $C_0$ where
$\lim_{| s| \to \infty }  \theta    \in  (0, \pi )$, and suppose
that this limit is distinct from all other $|s|    \to   \infty$
limits of $\theta$ on $C$. Suppose as well that this limit is distinct from
all of the critical values of $\theta$ on $C_0$. Then the large $|s|$
portion of $E$ will intersect precisely one component of $\Gamma$,
the latter a smooth, properly embedded copy of $\mathbb{R}$ whose large $|s|$
portions are properly embedded in the large $|s|$
part of $E$.

\step{Part 2}
By virtue of the definition of $\Gamma$,
any given component $K  \subset  C_{0}-\Gamma$ is a cylinder to
which $d\theta$ and $x$ pullback without zeros. In fact, $\theta$ and
the restriction of $x$ to the constant $\theta$ level sets of $K$ can be used
to give coordinates to such a cylinder. To elaborate, let $(\theta_{o},
\theta_{1})$ denote the range of $\theta$ on $K$. Next, let $q$ and $q'$
denote the respective integrals around the constant $\theta$ slices of $K$
(as oriented by $x$) of the closed forms $\frac{1}{2\pi}dt$ and
$\frac{1}{2\pi}d\varphi$. Then $K$ can be parametrized by the
open cylinder $(\theta_{o}, \theta_{1})  \times   \mathbb{R}/(2\pi
\mathbb{Z})$ so that the restriction to $K$ of the tautological immersion of
$C_0$ into $\mathbb{R}  \times  (S^1   \times  S^2 )$ has a rather
prescribed form. To be more specific, let $\sigma    \in  (\theta
_{0}, \theta_{1})$ and $v  \in \mathbb{R}/(2\pi \mathbb{Z})$
denote the coordinates for the cylinder. Written using these coordinates,
the tautological immersion involves two smooth functions on $(\theta
_{0}, \theta_{1})  \times   \mathbb{R}/(2\pi \mathbb{Z})$, these
denoted by $a$ and $w$; and it sends any given point $(\sigma , v)$ to the point
where
\begin{multline}\label{eq:2.25}
\bigl(s = a, \ t = q v + (1-3\cos^2 \theta ) w \mod (2\pi \mathbb{Z}),  \\ \theta  =
\sigma , \ \varphi  = q'v + \surd 6\cos \theta  w \mod(2\pi
\mathbb{Z})\bigr).
\end{multline}
Note that the $J'$--pseudoholomorphic nature of the immersion of $K$ requires
that the pair $(a, w)$ obey a certain non-linear differential equation. For
example, in the case where $J' = J$, this equation reads
\begin{equation}\label{eq:2.26}
\begin{aligned}
&\alpha_{Q} a_{\sigma }{-}\surd 6 \sin\sigma  (1{+}3 \cos^2 \sigma
) w a_{v} = -\frac{1{+}3\cos ^4\sigma }{\sin \sigma}
\biggl(w_{v}{-}\frac{1}{1{+}3\cos ^4\sigma }
\beta \biggr) \\
&(\alpha_{Q}w)_{\sigma }-\surd 6 \sin\sigma  (1 + 3
\cos^2 \sigma ) w w_{v}=\frac{1}{\sin \sigma }
a_{v},
\end{aligned}
\end{equation}
Here, $\alpha_{Q}=\alpha_{Q}(\sigma )$ is the function
\begin{equation}\label{eq:2.27}
\alpha_{Q} = (1 - 3 \cos^2 \sigma ) q' - \surd 6 \cos \sigma  q ,
\end{equation}
and $\beta  = q (1 - 3\cos^2 \sigma ) + q' \surd 6 \cos\sigma
\sin^2\sigma$. In these equations and below, $Q$ denotes the pair $(q, q')$.

By the way, $\alpha_{Q}$ is necessarily positive on $(\theta_{o},
\theta_{1})$ by virtue of the fact that the parametrization in~\eqref{eq:2.25} of
$K$ pulls back the exterior derivative of the contact form $\alpha $ as
\begin{equation}\label{eq:2.28}
\surd 6\sin\sigma   \alpha_{Q}(\sigma )d\sigma    \wedge  dv .
\end{equation}
In this regard, keep in mind that the form $d\alpha $ is non-negative on
$J'$--pseu\-do\-ho\-lo\-mor\-phic 2--planes in $\mathbb{R}\times(S^1\times
S^2)$. Moreover, $d\alpha $ is zero on such a plane only if the latter is
spanned by $\partial_{s}$ and the Reeb vector field $\hat {\alpha }$.

This last conclusion has the following converse: Suppose that $(a, w)$ are any
given pair of functions on $(\theta_{o}, \theta_{1})  \times
\mathbb{R}/(2\pi \mathbb{Z})$. Then, the resulting version of~\eqref{eq:2.25} immerses
the points in its domain where $\alpha_{Q}$ is positive.

Here is one final remark about any map having the form given in~\eqref{eq:2.25}:
Suppose that $\alpha_{Q}$ is positive on $(\theta_{o}, \theta_{1})$.
Now, let $(a, w)$ denote any given pair of functions on the cylinder
$(\theta_{o}, \theta_{1})  \times   \mathbb{R}/(2\pi \mathbb{Z})$.
By virtue of the fact that the coordinates $t$ and $\varphi$ are defined only
modulo $2\pi \mathbb{Z}$, the image cylinder in $\mathbb{R}  \times  (S^1
\times  S^2 )$ via the map in~\eqref{eq:2.25} is unchanged under the action of
$\mathbb{Z}  \times   \mathbb{Z}$ on the space of function pairs $(a, w)$ whereby
a given integer pair $N = (n, n')$ acts to send $(a, w)$ to $(a^N, w^N)$
with the latter given by
\begin{equation}\label{eq:2.29}
\begin{aligned}
a^{N}(\sigma , v) &= a\bigg(\sigma , v - 2\pi
\frac{\alpha_N (\sigma )}{\alpha_{Q} (\sigma )}\bigg)
\\
\text{and}\qquad w^{N}(\sigma , v) &= w(\sigma , v - 2\pi   \frac{\alpha_N (\sigma
)}{\alpha_{Q} (\sigma )}) + 2\pi   \frac{qn' - q' n}{\alpha_{Q} (\sigma )}.
\end{aligned}
\end{equation}

\step{Part 3}
This part of the subsection discusses the
behavior of the parametrization in~\eqref{eq:2.25} at points near the boundary of the
closure of the parametrizing cylinder.

To start, remark that if a given $\theta_*   \in \{\theta
_{0}, \theta_{1}\}$ is not achieved by $\theta$ on the closure of
$K$, then there exists $\varepsilon >0$ such that the portion of $K$ where
$| \theta -\theta_*|    \le   \varepsilon $ is
properly embedded in an end of $C$. In particular, the constant $\theta$
slices of this portion of $K$ are isotopic to the constant $|s|$
slices when $\theta$ is very close to $\theta_*$. Moreover, if
$\theta_* \notin \{0, \pi \}$, then such an end
is on the convex side of $C$ and the associated integer $n$ that appears in~\eqref{eq:2.17} is zero.

On the other hand, if $\theta_*   \in \{\theta_{o},
\theta_{1}\}$ is neither 0 nor $\pi$ and if $\theta_*$ is
achieved on the closure of $K$, then the complement of the $\theta$ critical
points in the $\theta =\theta_*$ boundary of this closure is
the union of a set of disjoint, embedded, open arcs. The closures of each
such arc is also embedded. However, the closures of more than two arcs can
meet at any given $\theta$--critical point. Conversely, every arc in any
given component of $\Gamma$ is entirely contained in the boundary of the
closures of precisely two components of $C_{0} -\Gamma$.

This decomposition of the $\theta =\theta_*$ boundary of $K$
into arcs is reflected in the behavior of the parametrizations in~\eqref{eq:2.25}  as
$\sigma$ approaches $\theta_*$. To elaborate, each critical
point of $\theta$ on the $\theta =\theta_*$ boundary of the
closure of $K$ labels one or more distinct points on the $\sigma =\theta
_*$ circle in the cylinder $[\theta_{o}, \theta_{1}]
\times   \mathbb{R}/(2\pi \mathbb{Z})$. These points are called `singular
points'. Meanwhile, each end of $C$ that intersects the $\theta =\theta
_*$ boundary of the closure of $K$ in a set where $|s|$
is unbounded also labels one or more distinct points on this same circle.
The latter set of points are disjoint from the set of singular points. A
point from this last set is called a `missing point'.

The complement of the set of missing and singular points is a disjoint set
of open arcs. Each point on such an arc has a disk neighborhood in $(0, \pi
)  \times   \mathbb{R}/(2\pi \mathbb{Z})$ on which the parametrization in~\eqref{eq:2.25}
 has a smooth extension as an embedding into $\mathbb{R}  \times
(S^1   \times  S^2 )$ onto a disk in $C$. This last observation is
frequently used in subsequent arguments from this article and from the
sequel.

As might be expected, the set of arcs that comprise the complement of the
singular and missing points are in 1--1 correspondence with the set of arcs
that comprise the $\theta =\theta_*$ boundary of the closure
K. In particular, the extension to~\eqref{eq:2.25} along any given arc in the $\sigma
=\theta_*$ boundary of $[\theta_{o}, \theta_{1}]
\times   \mathbb{R}/(2\pi \mathbb{Z})$ provides a smooth parametrization of
the interior of its partner in the $\theta =\theta_*$
boundary of the closure of $K$.

By way of an example, consider the case that $\theta_*$ is a
critical value of $\theta$ on $C_0$ that is realized by a single critical
point with the latter non-degenerate. Assume further that $\theta_* $ is not an $| s| \to \infty $
limit of $\theta$
on $C$. Thus, the critical locus is a `figure 8'. In this case, there are
three components of $C_{0} -\Gamma$ with boundary on this locus,
one whose boundary maps to the top circle in the figure 8, another whose
boundary maps to the lower circle, and a third whose boundary traverses the
whole figure 8. The first two have but one singular point on the $\sigma  =
\theta_*$ boundary of any parametrizing domain, while the
third has two singular points.

For a second example, suppose that $\theta_*$ is neither 0 nor
$\pi$ and is the $| s|    \to   \infty $ limit of $\theta$ on
a convex side end, $E$, for which $\deg_{E}(d\theta ) = 1$. Suppose, in
addition, that $\theta_*$ is not the $| s|
\to   \infty $ limit of $\theta$ on any other end of $C$. In this case, the
corresponding $\theta =\theta_*$ component of $\Gamma$ is a
properly embedded copy of $\mathbb{R}$. Furthermore, there are two components
of $C_{0} -\Gamma$ whose closures lie in this $\theta =\theta
_*$ component of $\Gamma$, and both have just a single missing
point on the $\sigma =\theta_*$ boundaries of any of their
parametrizing domains.

In the case that $\theta_*   \in \{\theta_{o}, \theta
_{1}\}$ is either 0 or $\pi$ and $\theta$ takes value $\theta_*$ on the closure of $K$,
then the map in~\eqref{eq:2.25} extends to the $\sigma  =
\theta_*$ boundary of the cylinder as a smooth map that sends
this boundary to a single point. This extended map factors through a
pseudoholomorphic map of a disk into $\mathbb{R}  \times  (S^1   \times
S^2 )$ with the $\sigma =\theta_*$ circle being sent to the
disk's origin.

\step{Part 4}
The graph assigned to a given $C$ from $\mathcal{M}_{\hat{A}}$ is denoted here by $T_C$.
As remarked at the outset, this is
a connected, contractible graph with labeled edges and vertices. In this
regard, the edges of $T_C$ are in 1--1 correspondence with the components
of $C_{0}-\Gamma$. If $e$ denotes an edge, then $e$ is labeled by an
integer pair, $Q_{e}   \equiv  (q_{e}, {q_e}')$, these being the
respective integrals of $\frac{1}{2\pi}dt$ and $\frac{1}{2\pi}d\varphi$
about the constant $\theta$ slices of the
corresponding component of $C_{0}-\Gamma$. Here, as in \refstep{Part 1},
these slices are oriented using the pull-back of the 1--form
$(1-3\cos^2 \theta )d\varphi -\surd 6\cos \theta  dt$.

The multivalent vertices of $T_C$ are in 1--1 correspondence with the
subsets of a certain partition of the components of $\Gamma$. To define
this partition, first introduce a new graph, $G$, as follows: The vertices of
$G$ are in 1--1 correspondence with the components of $\Gamma$. Meanwhile, an
edge connects two vertices of $G$ when there is an end of $C$ with the following
property: Every sufficiently large and constant $|s|$ slice of
the end intersects at least one arc from each of the corresponding
components of $\Gamma$. With $G$ understood, then the components of $G$
naturally partition the set of components of $\Gamma$. For example, every
compact component of $\Gamma$ defines its own, single point set in this
partition.

The set of multivalent vertices in $T_C$ are in 1--1 correspondence with
this partition of $\Gamma$. If $o$ is a multivalent vertex of $T_C$, then
the incident edges to $o$ label the components of $C_{0}-\Gamma$
whose closure intersects that part of $\Gamma$ that is assigned to $o$.

The monovalent vertices in $T_C$ are in 1--1 correspondence with the
elements in the union of three distinct sets. The first set consists of
points where $C$ intersects the $\theta =0$ and $\theta =\pi$ loci. In
this regard, a given $\theta =0$ or $\theta =\pi$ point can label
more than one monovalent vertex of $T_C$. To elaborate, suppose that a
small ball about such a point intersects $C$ in some $k  \ge 0$ irreducible
components that all meet at the given point. Then this point labels $k$
monovalent vertices of $T_C$. The second set consists of the ends of $C$
where the $| s|    \to   \infty $ limit of $\theta$ is either 0
or $\pi$. The third set consists of the convex side ends of $C$ where the
$| s|    \to   \infty $ limit of $\theta$ is neither 0 nor $\pi
$ and where the integer n in~\eqref{eq:2.17} is zero. Said differently, the second
and third sets consist of those ends of $C$ where the $| s|    \to
  \infty $ limit of $\theta$ is not achieved at any sufficiently large
value of $|s|$. A given monovalent vertex lies on the most
obvious edge.

Note that various versions of $T_C$ will be defined here and in the sequel
to this article that differ in the complexity of the labels that are
assigned to the vertices. In all versions, the vertex label contains an
angle in $[0, \pi ]$, this the obvious one available.
Elements from $\hat{A}$
are part of labels that are used in the next subsequent sections of this
article. The most sophisticated labeling occurs in the sequel to this
article where any given multivalent vertex label is a certain sort of graph,
this defined from the components of $\Gamma$ that are contained in the
corresponding partition subset.

%
%

\setcounter{theorem}{0}

\section{Existence}\label{sec:3}

This and the remaining sections derive necessary and sufficient conditions
that insure that any given $\mathcal{M}_{\hat{A}}$ is non-empty. The result is a
proof of \fullref{thm:1.3}. The strategy used here is to construct proper
immersions of multi-punctured spheres with the correct $|s| \to \infty $ asymptotics and
then deform them so that the result is pseudoholomorphic.

It is assumed here that an asymptotic data set $\hat{A}$  has been specified that
obeys a certain set of auxiliary constraints that differ from those stated
in \fullref{thm:1.3}. \fullref{sec:5} explains why the constraints listed here are
satisfied if and only if $\hat{A}$  satisfies the conditions in  \fullref{thm:1.3}.

\subsection{An associated graph}\label{sec:3a}

Granted that $\hat{A}$  has been specified, fix a partition of the set of
$(0,+,\ldots )$ elements in $\hat{A}$  subject to the following
constraint: The integer pairs from any two elements in the same partition
subset define the same angle via~\eqref{eq:1.7}. Let $\wp $ denote the given
partition. What follows is a description of a contractible graph, $T$, with
labeled vertices and labeled edges that is defined using $\hat{A}$  and $\wp $.
This graph is used in the subsequent construction as a blueprint of sorts
for constructing the initial subvariety in $\mathbb{R}    \times (S^1
\times S^{2})$.

The story on $T$ begins with the remark that $T$ has $N_{ - }+\hat
N+\mc_{\hat{A} }$ monovalent vertices, $N_{\wp }$ bivalent vertices and
$N_{-}+\hat N+\mc_{\hat{A} }-2$ trivalent vertices. Here $N_{\wp }$ denotes
the number of sets in the partition $\wp $. A subset of monovalent vertices
with $N_{ - }+\hat {N}$ elements are labeled by assigning a 1--1
correspondence between the vertices in the subset and the subset of $\hat{A}$
whose elements are either of the form $(\pm 1,\ldots)$ or
$(0,-,\ldots)$. Of those that remain, $c_+ $ are labeled
(1) and $c_- $ by $(-1)$.

The bivalent vertices are labeled by assigning a 1--1 correspondence between
the set of such vertices and elements in $\wp $. Thus, each bivalent vertex
is labeled by a partition subset.

The labeling just described associates an angle in $[0, \pi ]$ to each
monovalent and each bivalent vertex; this is the angle 0 when the vertex has
label either $(1,\ldots )     \in \hat{A}$  or (1), the angle
$\pi$ when the label is either $(-1,\ldots)     \in \hat{A}$
or $(-1)$. Meanwhile, a monovalent vertex with label $(0,-,\ldots)$ from $\hat{A}$  is assigned the angle that is defined
 via~\eqref{eq:1.7} by the
integer pair from this element. Finally, a bivalent vertex is assigned the
angle that is defined via~\eqref{eq:1.7} by the integer pair from any element in its
corresponding partition subset.

As for the trivalent vertices, each is labeled by an angle in $(0, \pi )$ so
that no two vertices are assigned the same angle, and none are assigned an
angle that is associated to any monovalent or bivalent vertex.

There are three further constraints on the angle assignments to the vertices
of $T$. Here is the first:

\begin{constraint}\setobjecttype{Con}\label{cons:1}\hfill
\begin{enumerate}
\item[(a)] The vertices that share an edge have distinct angle assignments.

\item[(b)]
The angle assigned to any given multivalent vertex is neither a largest
nor a smallest angle in the set of angles that are assigned those vertices on its incident edges.
\end{enumerate}
\end{constraint}

Each edge of $T$ is labeled by a non-trivial, ordered pair of integers. If $e$
denotes an edge, then its assigned pair is denoted here by $Q_e$ or
$(q_e, {q_e}')$. What follows is the second constraint.

\begin{constraint}\setobjecttype{Con}\label{cons:2}

These integer pair assignments to the edges are constrained as follows:
\itaubes{3.1}
If $o$ is a monovalent vertex on $e$ with a 4--tuple label from $\hat{A}$ , then $Q_e=\pm P_{o}$ where $P_o$
is the integer pair from the label. Here, the $+$ sign appears if and only if one of the following hold:

\begin{enumerate}\leftskip 25pt

\item[(a)] $o$'s angle is in $(0, \pi )$ and it is the smaller of $e'$s vertex angles.

\item[(b)] $o$ is labeled by either a $(1,-,\ldots)$ or a $(-1,+,\ldots)$ element in $\hat{A}$.
\end{enumerate}

\item
If $o$ is a monovalent vertex with label $\delta      \in  (\pm 1)$, then its incident edge, $e$,
has $q_e=0$ and ${q_e}' = -1$.

\item
If $o$ is a bivalent vertex with incident edges $e$ and $e'$ with the convention that $e$ connects $o$
to a vertex with smaller angle label, then $Q_e - Q_{e'} = P_{o}$ where the integer pair
$P_{o}$ is the sum of the pairs from the elements that comprise $o'$s partition subset.

\item
If $o$ is a trivalent vertex with incident edges $e$, $e'$ and $e''$, then $Q_e - Q_{e'} - Q_{e''}=0$ given that
the angle labels of the vertices opposite $o$ on $e'$ and $e''$ lie on the same side of the angle
that labels $o$ in $(0, \pi )$.
\end{itemize}
\end{constraint}

With the collection $\{Q_e\}$ now defined, here is the third
constraint on the angle assignments to the vertices in $T$:

\begin{constraint}\setobjecttype{Con}\label{cons:3}
 Let $e$ denote any given edge of $T$ and let $\theta_{o} < \theta
_{1}$ denote the angles that are assigned the vertices on $e$. Then
\begin{equation*}
q_e'(1-3\cos^2 \theta )- q_e\surd 6\cos (\theta )     \ge  0
\end{equation*}
at all $\theta      \in  [\theta_{o}, \theta_{1}]$ with equality if and only if $\theta$ is either $\theta
_{0}$ or $\theta_{1}$ and the corresponding vertex is monovalent with label $(0,-,\ldots)$ from $\hat{A}$.
\end{constraint}

A graph $T$ that obeys all of the preceding constraints is said here to be a
`moduli space graph' for $\hat{A}$ .

Two moduli space graphs are said here to be isomorphic when there is an
isomorphism of the unlabeled graphs that preserves the labelings of the
vertices and edges.

To explain the relevance of such a graph to $\mathcal{M}_{\hat{A}}$, remember
that any $C  \in     \mathcal{M}_{\hat{A} }$ defines a graph $T_C$ as described
in \refsteps{2.G}{Part 4} of \fullref{sec:2g}. In particular, if $\mathcal{M}_{\hat{A}}$ is non-empty,
then according to Propositions~\ref{prop:2.12} and~\ref{prop:2.13}, the version of $T_C$ that is
defined by any sufficiently generic choice of $C  \in     \mathcal{M}_{\hat{A}}$
defines just such a moduli space graph for a particular choice of $\wp $. To
elaborate, the aforementioned propositions guarantee an open and dense
subset in $\mathcal{M}_{\hat{A}}$ whose subvarieties have the following
property: All of the critical points of the function $\cos(\theta )$ on the
corresponding model curve are non-degenerate, and there are $N_{ - }+\hat
N+c_{\hat{A}}-2$ non-extremal critical values with no two identical and
none equal to an $|s|      \to     \infty $ limit of $\theta$. If
$C$ is such a generic subvariety, then the corresponding graph $T_C$ has
only monovalent, bivalent and trivalent vertices. The angles of the vertices
satisfy the requirements in \fullref{cons:1}. Meanwhile, the requirements of
\fullref{cons:2} are also met when the partition assigned to any given bivalent
vertex consists of the 4--tuples in $\hat{A}$  that label those ends of the
subvariety that contain the very large $|s| $ parts of the
corresponding components of the locus $\Gamma$. As explained in \fullref{sec:2g},
the third constraint is met automatically.

Henceforth $T_C$ denotes the just described moduli space version of the
graph from \refsteps{2.G}{Part 4} of \fullref{sec:2g} in the case that $C  \in \mathcal{M}_{\hat{A}}$
satisfies the stated genericity requirement.

As just indicated, if $\mathcal{M}_{\hat{A}}$ is non-empty, then it has a moduli
space graph. The theorem that follows states this fact and its converse:

\begin{theorem}\label{thm:3.1}
The space $\mathcal{M}_{\hat{A}}$ is non-empty if and only if $\hat{A}$  has a
moduli space graph. Moreover, if $T$
is a moduli space graph, then there is a subvariety in $\mathcal{M}_{\hat{A}}$ whose version of $T_{(\cdot)}$
is isomorphic to $T$.
\end{theorem}
The criteria here for $\mathcal{M}_{\hat{A}}$ were suggested by observations of
Michael Hutchings.

The proof starts by assuming the existence of a moduli space graph for $\hat{A}$
and ends with the conclusion that $\mathcal{M}_{\hat{A}}$ is non-empty. This
proof occupies the remainder of \fullref{sec:3} and all of \fullref{sec:4}.

\subsection{Parametrizations of cylinders in $\mathbb{R} \times (S^{1}\times S^{2})$}\label{sec:3b}

As remarked above, the graph $T$ is used as a blue-print for the construction
of a properly immersed, multi-punctured sphere in $\mathbb{R} \times (S^{1} \times S^{2})$.
All of the monovalent vertices with 4--tuple
labels from $\hat{A}$  correspond to ends of the surface; here, the label from
$\hat{A}$  on a given vertex is used to specify the asymptotic behavior of the
corresponding end. The monovalent vertices with either (1) or $(-1)$ labels
correspond to intersection points between the surface and the respective
$\theta =0$ or $\theta =\pi$ cylinders. Meanwhile, the trivalent
vertices label the index one critical points of the restriction to the
surface of the function $\cos(\theta )$. The angle label of a vertex gives
the value of $\theta$ at the corresponding critical point. The only local
maxima or minima of $\cos(\theta )$ on the surface are relegated to the
intersection with the respective $\theta =0$ and $\theta =\pi$
cylinders.

Meanwhile, any given edge of $T$ labels an open, cylindrical component of the
surface where $\theta$ ranges between the values given by the edge's end
vertices. The ordered integer pair that is associated to the edge specifies
the respective integrals of $\frac{1}{2\pi }dt$ and
$\frac{1}{2\pi }d\varphi$ around any constant $\theta$ slice
as oriented by $(1-3\cos^2\theta ) d\varphi -\surd 6\cos\theta
dt$. The incidence relations at the vertices direct the manner in which the
edge labeled cylinders attach to form a closed surface. In this regard, the
closure of any component cylinder whose edge label ends in a (1) or $(-1)$
labeled monovalent vertex is a disk that intersects the respective $\theta =0$ or $\theta =\pi$ locus.

The basic building blocks for the surface are thus the edge labeled
cylinders. When $e  \subset T$ denotes an edge, its corresponding cylinder
is denoted by $K_e$. Let $o$ and $o'$ denote the angles that label the end
vertices of $e$ with the convention that the $\theta$--label of $o$ is less than
that of $o'$. These labels are respectfully denoted here by $\theta_{o}$
and $\theta_{o'}$.

The parametrization of $K_e$ is via a map from its `parametrizing
cylinder', this being the interior of $[\theta_{o}, \theta_{o'}]
\times     \mathbb{R}/(2\pi \mathbb{Z})$. With $\sigma$ denoting the
coordinate on $[\theta_{o}, \theta_{o'}]$ and $v$ an affine coordinate
on $\mathbb{R}/(2\pi \mathbb{Z})$, the parametrizing map can be written using
two functions on the open cylinder, $(a_e, w_e)$. To be precise, the
parametrizing map sends any given point $(\sigma , v)$ to the $Q  \equiv
Q_e$ and $(a, w)  \equiv (a_e, w_e)$ version of~\eqref{eq:2.25}. Thus,
\begin{multline}\label{eq:3.2}
\bigl(s = a_{e}, \
t = q_e v + (1 - 3\cos^{2}\sigma ) w_{e} \mod(2\pi ),\\
\theta =\sigma,\ \varphi  = q_e'v + \surd 6\cos\sigma
w_{e} \mod(2\pi )\bigr)
\end{multline}
Unless specified to the contrary, a `parametrization' is one given as in~\eqref{eq:3.2}.

For future reference, note that the 2--form $d\sigma      \wedge  dv$ orients
$K_e$. Also, note that the map into $\mathbb{R}    \times  (S^1    \times
S^2)$ as defined by~\eqref{eq:3.2} defines an immersion. As is explained next,
this is a consequence of the positivity of the $Q = Q_{e}     \equiv
(q_e, q_e')$ version in~\eqref{eq:2.27} of the function $\alpha_{Q}$. To
see why, first let $\phi$ denote the parametrizing map to $\mathbb{R}    \times  (S^1    \times  S^2)$
and reintroduce the contact 1--form $\alpha $
from~\eqref{eq:1.1} and~\eqref{eq:1.2}. Then $\phi^*d\alpha $ is the 2--form that appears in~\eqref{eq:2.28}.
Granted that such is the case, the rank of $\phi^*$ is two on the
parametrizing cylinder.

The next section provides a specific version of each $(a_e, w_e)$.
These versions are chosen to meet the following five criteria:
\itaubes{3.3}
\textsl{The collection $\{(a_e, w_e)\}_{e \subset T}$ are constrained near the boundaries of their
corresponding parametrizing cylinders so as to insure that the closure of $ \cup_{e \subset T}
K_e$ is the image in $\mathbb{R}    \times  (S^1    \times S^2)$ via a proper immersion of an oriented,
multiply punctured sphere.
}

\item \textsl{The singularities of this immersion are transversal double points with positive local intersection number.
}

\item \textsl{The critical points on the multi-punctured sphere of the pull-back of $\theta$ are non-degenerate,
the 1--form $J\cdot d\theta$ pulls back as zero at these critical points, and the symplectic form
$\omega $ pulls back as a positive form at these critical points. Moreover,
no critical point maps to an immersion point of $C$.
}

\item \textsl{The subvariety has the asymptotics of a subvariety from $\mathcal{M}_{\hat{A}}$.
}

\item \textsl{The level sets of the pull-back of $\theta$ to the multiply punctured sphere defines a
moduli-space graph that is isomorphic to the given graph $T$.
}
\end{itemize}
The precise meaning of the fourth point is given in the definition that
follows.

\begin{deff}
\label{deff:3.2}

A subvariety $C  \subset     \mathbb{R}    \times  (S^1\times  S^2)$ is said to have the
asymptotics of a subvariety from $\mathcal{M}_{\hat{A}}$ when the requirements listed below are met.

\item{{\rm Requirement 1:}}
There is a compact subset in $C$ whose complement is a disjoint union
of embedded cylindrical submanifolds in $\mathbb{R} \times(S^1 \times S^2)$ that are
in 1--1 correspondence with the elements in $\hat{A}$. Such a cylinder is called an `end' of $C$.

\item{{\rm Requirement 2:}}
Let $E  \subset  C$ denote any given end and let $(\delta , \varepsilon , (p, p'))$ denote its
label from $\hat{A}$. Then the following conditions are satisfied:

\begin{enumerate}
\item[(a)] The function $s$ restricts to $E$ as a smooth function without critical points with
$\varepsilon s$ bounded from below on the closure of $E$. Moreover, the 1--form
$\alpha $ in~\eqref{eq:1.1} has nowhere zero pull-back on each constant $|s| $ slice of $E$.

\item[(b)] The restriction to $E$ of the function $\theta$ has a unique $|s|      \to     \infty $
limit; and the latter equals $0$ when $\delta  = 1$, it equals $\pi$ when $\delta $ is $-1$,
and it is given by $P$ via~\eqref{eq:1.7} when $\delta =0$.
Moreover, when $\delta =\pm 1$, this convergence makes the
$\kappa      \equiv   | \delta \sqrt{
\frac{3}{2}} +\frac{p' }{p}|$ version of $e^{\kappa | s| } \sin\theta$ converge to a
unique limit as $|s|      \to     \infty $.

\item[(c)] The integers $p$ and $p'$ are the respective integrals of $\frac1 {2\pi }dt$
and $\frac1 {2\pi
}d\varphi$ about any
given constant $|s| $ slice of $E$ when the latter is oriented by the 1--form $\alpha $.

\item[(d)] Any given anti-derivative on $E$ for the restriction of the 1--form $p'dt - pd\varphi$
has a unique $|s|      \to     \infty $ limit.

\item[(e)] Let $N_{E}     \to  E$ denote the normal bundle to $E$. Define
$\prod_{J}\co  TE  \to N_{E}$ to be the
composition of the map $J\co  TE  \to  T(\mathbb{R}    \times  (S^1    \times S^{2}))|
_{E}$ with the projection to $N_{E}$. Define the norm of $\prod_{J}$, the covariant derivative of
$\prod_{J}$, and the latter's norm using the metrics and connections on $TE$ and $N_{E}$ that
are induced by the metric on $\mathbb{R}    \times  (S^1    \times S^{2})$. Then the norm of
$\prod_{J}$ and that of its covariant derivative limit to zero as $|s|      \to \infty $ on $E$.
\end{enumerate}
\end{deff}

As is explained subsequently, a collection of pairs $\{(a_e,w_e)\}_{e \subset T}$
that meet the criteria in~\eqreft33 will serve as a
starting point for the deformation to a pseudoholomorphic subvariety. The
remainder of this section assumes that such a collection has been specified.

\subsection{A preliminary deformation}\label{sec:3c}

Let $C$ denote the closure of $ \cup_{e} K_{e}$ as defined using the
given collection of pairs, $\{(a_{e}, w_{e})\}$. This
subsection begins the construction a family of deformations of $C$ whose end
member is a subvariety in $\mathcal{M}_{\hat{A}}$. In particular, a preliminary
deformation of $C$ is constructed here so that the result is pseudoholomorphic
with respect to an admissible almost complex structure.

To start the task, return to the observation that $d\alpha $ pulls back as a
non-zero 2--form to any given version of $K_{e}$. This, being the case, it
follows that $d\alpha $ is non-negative on $TC$. With the third point in~\eqreft33,
the last observation has the following consequence: There exists $r >
0$ such that the symplectic form $d(e^{-rs}    \alpha )$ is uniformly
positive on $TC$. This is to say that its pull-back to the multi-punctured
sphere is a multiple of the induced area form that is positive and uniformly
bounded away from zero. Indeed, by virtue of the third point, any positive $r$
version of this form is positive on $TC$ near the images of the critical
points of the pull-back of $\theta$. Meanwhile, the form is positive for
small $r$ on any given compact subset in the complement of these same critical
points. At large $|s| $ on any given end of $C$, both $-ds  \wedge \alpha $ and
$d\alpha $ are positive.

Now, specify a positive real number, $\varepsilon $. Granted the second and
fourth points of~\eqreft33, there are standard constructions that provide
$\mathbb{R}    \times  (S^1    \times S^{2})$ with an almost complex structure,
$J_{0}$, with the following properties:
\itaubes{3.4}
The subvariety $C$ is $J_{0}$--pseudoholomorphic.

\item
$J_{0} = J$ where the distance to $C$ is greater than $\varepsilon $.

\item
$J_{0}\partial_{s}=\frac{1}{(1 + 3\cos ^4\theta )^{1 /2}}\hat {\alpha }$ where the distance
to any singular point of $C$ or critical point of the restriction of $\theta$ is greater than
$\varepsilon$.

\item
Both $J_{0}- J$ and its covariant derivative converge uniformly with limit zero as
$|s| \to \infty $ on $\mathbb{R} \times  (S^1    \times S^{2})$.

\item
All sufficiently small but positive $r$ versions of $d(e^{-rs}    \alpha )$ tame $J_{0}$.
\end{itemize}
Having specified $J_{0}$, fix a Riemannian metric on $\mathbb{R} \times(S^1 \times S^2)$
to be called $g_0$, one with the following
three properties: First $J_{0}$ acts as a $g_0$--isometry. Second, $g_0$
converges uniformly as $|s|      \to     \infty $ to the metric
$ds^2 + dt^2 + d\theta ^2 + \sin^2\theta  d\varphi^2$,
and its covariant derivative (as defined by the latter metric)
converges uniformly to zero as $|s|      \to     \infty $. Finally,
$g_0$ agrees with the latter metric where $J_0 = J$.

The almost complex structure $J_{0}$ would be admissible in the sense given
prior to \fullref{deff:2.1} were the third point to hold on the whole of
$\mathbb{R} \times  (S^1    \times S^{2})$, and were $J_{0}$ and $J$ to agree
on the nose outside of some compact subset of $\mathbb{R}    \times (S^1
\times S^{2})$. This part of the subsection describes how to move $C$
slightly so that the result is an immersed subvariety that is
pseudoholomorphic for an admissible almost complex structure.
The
construction of such a deformation is done in three steps. A fourth step
explains how this can be done so that the version of \fullref{sec:2g}'s graph
$T_{(\cdot )}$ for the resulting subvariety gives the starting moduli
space graph $T$.

\step{Step 1}
This first step specifies $J_{0}$ in a more precise manner near the images
of the critical points of the pullback of $\theta$. To start, note that if
$z  \in  C_{0}$ is a critical point of the pullback of $\theta$, then
there is a small radius, embedded disk, $D  \subset     \mathbb{R}    \times
(S^{1}    \times  S^{2})$, that is contained in $C$ and centered at the
image of $z$. By virtue of the third point in~\eqreft33, the vectors $\partial
_{s}$ and $\hat {\alpha }$ span $TD$ at the image of $z$. Thus, $J_{0}$ must
map one to a multiple of the other at this particular point. This being the
case, there is a constant, $c$, with the following significance: For all
sufficiently small yet positive $\varepsilon $, an almost complex structure
$J_{0}$ can be found that obeys~\eqreft34 and is such that
\begin{equation}\label{eq:3.5}
| J-J_{0}|  < C \varepsilon  \qquad\text{and} \qquad | \nabla
(J-J_{0})|  < C
\end{equation}
at points with distance $\varepsilon $ or less from the image in $\mathbb{R}
\times  (S^{1}    \times S^{2})$ of any critical point of the pullback
of $\theta$.

\step{Step 2}
This step modifies both $C$ and $J_{0}$ near the singular points of $C$ so
that~\eqref{eq:3.5} is also obeyed at each of the latter points. To explain how this
is done, let $D  \subset C$ for the moment denote an embedded disk whose
closure is disjoint from the image of any critical point of the pullback of
$\theta$. Let $z$ denote the center point of $D$. Fix some $J$--pseudoholomorphic
disk, $D'  \subset     \mathbb{R}    \times  (S^{1}    \times S^{2})$, with
center $z$ whose tangent space at $z$ is spanned by $\partial_{\theta }$ and
$J\cdot \partial_{\theta }$. Having chosen such a disk, there exists
$\rho  > 0$ and complex coordinates $(x, y)$ centered at $z$, defined for $|
x|  < \rho $ and $| y|  < \rho $, such that $D'$ is the $y
= 0$ disk, and such that $\partial_{s}$ and $\hat {\alpha }$ are tangent
to each $x = \constant$ disk. Thus $\theta$ is constant on each of the latter.
In these coordinates, the disk $D$ can be viewed as the image of a
neighborhood of the origin in $\mathbb{C}$ to $\mathbb{C}^{2}$ that maps
the complex coordinate $u$ on $\mathbb{C}$ as
\begin{equation}\label{eq:3.6}
u  \to  \big(x = u, y = a u + b \bar {u} + o(| u| ^{2})\big),
\end{equation}
where $a$ and $b$ are complex numbers. Note that the $x$--coordinate of the map can
be defined in this way by virtue of the fact that $\theta$ is a function
only of $x$.

Consider now deforming $D$ in a manner that will now be described. To start,
pick some small, positive $\delta $ with the property that the disk of
radius 4$\delta $ in $\mathbb{C}$ is mapped via~\eqref{eq:3.6} some distance from the
boundary of $D$. Let $\beta $ denote a favorite, smooth function on $[0,
\infty )$ that is identically $1$ on $[0, 1]$, vanishes on $[2, \infty )$ and
is non-increasing. With $\beta $ chosen, consider the deformed disk,
$D(\delta )$, that is defined by the image of the map
\begin{equation}\label{eq:3.7}
u  \to  \bigr(x = u, y = a u + \beta \bigl(\tfrac1 {\delta }|
u| \bigr) b\bar {u} + o(| u| ^{2})\bigl).
\end{equation}
The image of this new map agrees with the old where $|u|  >2\delta $.
The new subvariety will be immersed and pseudoholomorphic for an
almost complex structure that also obeys the constraints on $J_{0}$ in~\eqreft34.
However, such an almost complex structure exists that agrees with $J$
near the point $z$.

Now suppose that $z$ is a singular point of $C$. Thus, there are two versions of
$D$ with center at $z$, these now denoted by $D_{1}$ and $D_{2}$. No generality
is lost by assuming here the respective closures of $D_{1}$ and $D_{2}$ are
disjoint save for the shared point $z$. Each such disk is described by a map
as in~\eqref{eq:3.6} using respective $(a_1, b_1)$ and $(a_2, b_2)$
versions of the pair $(a, b)$ of complex numbers. For sufficiently small
$\delta $, each of $D_{1}$ and $D_{2}$ has their corresponding deformation
as given in~\eqref{eq:3.7}. The claim here is that the resulting disks, $D_{1*}$ and $D_{2*}$,
intersect only at z, transversely, and with
positive intersection number. To explain, remark that any point in
$D_{1*}     \cap D_{2*}$ is the image of a point $u \in  \mathbb{C}$
where
\begin{equation}\label{eq:3.8}
(a_{1} - a_{2}) u = (b_{1} - b_{2})    \beta \bar {u} + o(|
u| ^{2}) .
\end{equation}
Since $0  \le     \beta      \le 1$, this can happen at non-zero $u$ when
$\delta $ is small only if $| b_1- b_2|      \ge     | a_1- a_2| $.
However, the latter inequality is forbidden by
the fact that the $D_{1}$ and $D_{2}$ have transversal intersection at $z$
with positive intersection number.

Thus, the new disks, $D_{1*}$ and $D_{2*}$, intersect
transversely only at $z$ with positive intersection number. Moreover, each is
$J$--pseudoholomorphic at $z$. This understood, if such a deformation is made for
each singular point of $C$, then the result is the image (henceforth named $C$)
of $C_{0}$ via an immersion that is pseudoholomorphic for a new version of
the almost complex structure $J_{0}$, one that obeys~\eqreft34 and also obeys~\eqref{eq:3.5}
at each of the singular points of the immersion and at the image of
each of the critical points of the pull-back of $\theta$.

\step{Step 3}
At this point, the stage is set to deform the newest version of $C$ so that
the result is pseudoholomorphic for an admissible almost complex structure.

To begin describing the latter deformation, let $\phi \co  C_{0}     \to
\mathbb{R}    \times  (S^1    \times  S^{2})$ now denote the tautological
immersion with image $C$. Introduce the bundle $N  \to C_{0}$ to
denote the pull-back normal bundle; this defined so that the fiber over any
given $z  \in C_{0}$ is the $g_0$--normal bundle at $\phi (z)$ to the
$\phi$ image of any given sufficiently small radius disk in $C_{0}$ with $z$
as center. As in the case for $J$--pseudoholomorphic subvarieties, the almost
complex structure $J_{0}$ and the metric $g_0$ together endow $N$ with the
structure of a complex line bundle with a Hermitian and thus holomorphic
structure. In addition, there exists $\delta  > 0$, a disk subbundle,
$N_{1}     \subset  N$ of $g_0$--radius $\delta $, and an `exponential' map
$e\co  N_{1}     \to     \mathbb{R}    \times (S^{1}    \times S^{2})$
with the following properties: First, $e$ is an immersion that restricts to
the zero section as the map $\phi$. Second, $e$ is a $g_0$--isometry along
the zero section. Third, the differential of $e$ is uniformly bounded.
Finally, $e$ embeds each fiber disk in $N_{1}$ as a $J_{0}$--pseudoholomorphic
disk.

With $e$ chosen, there exists some $\delta_{0}     \in (0, \delta
)$ with the following significance: If $\eta$ is a section of
$N_{1}$ with suitable decay at large $|s| $, and if both $|
\eta | $ and $| \nabla \eta | $ are both everywhere
less than $\delta_0$, then $\phi '  \equiv  e \circ \eta $
will immerse $C_{0}$ as a pseudoholomorphic subvariety for a complex
structure, $J_{0}'$, that also obeys the constraints in~\eqref{eq:3.5}. Moreover, if
the norms of $| \eta | $ and $| \nabla \eta | $
are small, then the critical values and critical points of the pullback of
$\theta$ via the new immersion will hardly differ from those of the
original. This is an important point in subsequent arguments, so keep it in
mind.

In any event, the plan is to find such a section $\eta $ with a
corresponding $J_{0}'$ that is admissible. For this purpose, remark that
there exists a constant, $c$, and, given some very small, but positive
$\constant \varepsilon $, there exists an admissible complex
structure, $J'$, that has the following properties:
\itaubes{3.9}
\textsl{$J' = J_{0}$ except where $|s|  > \frac{1}{\varepsilon }$ and where the distance to any
singular point of $C$ or image of a critical point of the pullback of $\theta
$ is less than $\varepsilon$.
}

\item \textsl{$| J_{0}-J'|  < c | J-J_{0}| $ and $| \nabla (J-J')|
\le c (| \nabla (J-J_{0})| +|J-J_{0}| )$ where
$|s|  > \frac1 {\varepsilon }$.
}

\item \textsl{$| J_{0}-J'|      \le  c\cdot     \varepsilon $ and $| \nabla (J_{0}-J')|      \le  c$
where the distance is less than $\varepsilon $ to any singular point of $c$ or to the image
of any critical point $\theta$'s pullback.
}
\end{itemize}

With $\varepsilon $ now chosen very small (an upper bound appears below),
fix an admissible $J'$ that obeys the constraints in~\eqreft39. This done, the
plan here is to search for an immersion, $\phi '\co  C_{0}     \to     \mathbb{R}
\times  (S^{1}    \times S^{2})$, whose image is a $J'$--pseudoholomorphic
subvariety. Thus, $\prod J'\cdot d\phi ' = 0$, where $\prod $ is the
projection to the normal bundle of the immersion. The sought for deformation
of $C$ is obtained by composing e with a suitable section of the bundle
$N_{1}$. In particular, if $\eta $ is a section of $N_{1}$, then the
condition on $\eta $ can be written schematically as
\begin{equation}\label{eq:3.10}
D_{C}\eta +\mathcal{R}_{0}(\eta )+\mathcal{R}_{1}(\eta )\cdot
\partial \eta +\gamma (\eta ) + \hat{\imath} = 0,
\end{equation}
where the notation is as follows: First, $D_{C}$ is the $(C, J_{0})$ version
of the operator that is depicted in~\eqref{eq:2.5}  while $\mathcal{R}_{0}$ and $\mathcal{R}_{1}$
are the $(C, J_{0})$ versions of their namesakes from~\eqref{eq:2.3} and~\eqref{eq:2.5}.
Meanwhile, $\gamma$ is a smooth, fiber preserving map from $N_{1}$
to $N_{1}     \otimes  T^{0,1} C_{0}$ that obeys
\begin{equation}\label{eq:3.11}
| \gamma (\eta )|      \le  c'| J'-J_{0}|
(1+| \nabla \eta | )+| \nabla (J'-J_{0})|
| \eta |  ,
\end{equation}
where $c'$ is a constant that can be taken to be independent of the choice of
$\varepsilon $ and $J'$. Note that $\mathcal{R}_{0}$, $\mathcal{R}_{1}$ and
$\gamma$ are defined on some small, positive and constant radius disk
subbundle of $N$, and the latter can be taken equal to $N_{1}$ with no loss of
generality. Finally, $\hat{\imath}$ is a linear map from a certain finite
dimensional vector subspace of $C^{\infty }(N  \otimes  T^{0,1}C)$ back
into the latter space whose image has compact support where the distance to
any singular point of $C$ or image of a critical point of $\theta$'s
pull-back is large. The form of $\hat{\imath}$ is described momentarily.

The operator $D_{C}$ in~\eqref{eq:3.10} has the same sort of Fredholm extension as a
bounded linear operator between the $C$ versions of the range and domain
spaces that appear in~\eqref{eq:2.7}. Note that the index of this Fredholm version of
$D_{C}$ is the integer $\hat{I}$ in~\eqref{eq:2.2}. In the present context, it may well
be the case that $D_{C}$ has a non-trivial cokernel.

If $\cokernel(D_{C})=0$, then $\hat{\imath}$ in~\eqref{eq:3.10} can be discarded. If
$\cokernel(D_{C})$ has positive dimension, then $\hat{\imath}$ is necessary. To
elaborate, $\hat{\imath}$ can be any linear map from $\cokernel(D_{C})$ into
$C^{\infty }(N  \otimes  T^{0,1}C_{0})$ with the following properties:
First, the support of the image of $\hat{\imath}$ is compact and with all points in
the support mapped to points with distance at least one from singular point
of $C$ or image of a critical point of $\theta$'s pullback. Second, the
orthogonal projection in $D_{C}$'s range Hilbert space composes with $\hat{\imath}$
to give the identity map on $\cokernel(D_{C})$.

With the preceding understood, let $H_{R}$ denote the orthogonal complement
in the range space of $D_{C}$ to the $D_{C}$'s cokernel and let $\prod $
denote the orthogonal projection in this range Hilbert space onto $H_R$.
Meanwhile, use H$_{D}$ to denote the orthogonal complement in the domain
space of $D_{C}$ to its kernel. Note that $D_{C}$ restricts to $H_{D}$ to
define a bounded, invertible map onto $H_R$. The inverse of the latter map
is denoted below as $(D_{C})^{-1}$.

To continue, let $H'  \subset  H_{D}$ denote the subset of smooth elements
with pointwise norm no greater than half the radius of the disk
bundle $N_{1}$. Finally, define the smooth map $Y\co  H' \times
\cokernel(D_{C})     \to  H_{D}$ by the rule
\begin{equation}\label{eq:3.12}
Y(\eta , \lambda ) = -(D_{C})^{-1}\prod \big[\mathcal{R}_{0}(\eta )
+\mathcal{R}_{1}(\eta )\cdot \partial \eta  +
\gamma (\eta ) + \hat{\imath}(\lambda )\big].
\end{equation}
By design, if $\eta  = Y(\eta , \lambda )$, then $\eta $ solves~\eqref{eq:3.10}
provided that
\begin{equation}\label{eq:3.13}
\lambda    =    - \Big(1-\prod \Big)\big[\mathcal{R}_{0}(\eta )+\mathcal{R}_{1}(\eta )\cdot \partial \eta
+\gamma (\eta )\big].
\end{equation}
The existence of such a pair $(\eta , \lambda )$ is guaranteed when
$\varepsilon $ is very small. Indeed, the analysis used in
\cite[Section 3c and the proof of its Proposition~3.2]{T3} can be used here to construct a
version of the contraction mapping theorem to prove the following:

\begin{lemma}\label{lem:3.3}
Given small $\varepsilon ' > 0$, then all sufficiently small $\varepsilon $
versions of the fixed point equation $\eta  = Y(\eta , \lambda )$ have a unique solution with
$\lambda $ given by~\eqref{eq:3.13} and with $\varepsilon' $ bounding both the Hilbert space norm
and pointwise $C^{1}$--norm of $\eta $.
\end{lemma}

\step{Step 4}
This last step explains how to make the preceding construction result in a
subvariety whose version of \fullref{sec:2g}'s graph $T_{(\cdot )}$ is the
given moduli space graph $T$.  To start, take note that given $T$,  there exists a
positive constant, $\delta_{T}$, with the following significance: Suppose
that $T'$  is a labeled graph, isomorphic to $T$ save for the fact that its
trivalent vertex angle assignments differ. Even so, suppose that there is a
`quasi' isomorphism that identifies the underlying graphs so as to pair like
labeled monovalent and bivalent vertices, match edge labels and pair
trivalent vertices only if their respective angle assignments differ by less
than $\delta_{T}$. This graph $T'$  is also a moduli space graph for $\hat{A}$.
If $\delta      \in $ (0, $\delta_{T})$, say that a graph $T'$  is ``$\delta $
close to $T$''  when such a quasi-isomorphism pairs trivalent angles so that all
such pairs differ by less than $\delta $.

If $\delta $ is small, then graphs that are $\delta $ close to $T$ can be
parametrized by a cube of side length $2\delta $ in the product of $(N_{ -
}+\hat N+c_{\hat{A}}-2)$ copies of $(0, \pi )$; this the cube centered
on the angle assignments for the trivalent vertices of $T$.  Let $B_{\delta }$
denote this cube.

There are now three remarks to make: First, with both $\varepsilon $ from~\eqreft39
and $\varepsilon'$ from \fullref{lem:3.3} taken to be very small, the
constructions just given in \refstep{Step 1}--\refstep{Step 3} can be made for any graph that comes
from a point in $B_{\delta }$. In this way, each point in $B_{\delta }$
produces a subvariety, and thus a version of \fullref{sec:2g}'s graph $T_{(\cdot
)}$. Second, with $\varepsilon_{1} > 0$ fixed and then both $\varepsilon
$ from~\eqreft39 and $\varepsilon' $ from \fullref{lem:3.3} even smaller, any chosen
point from $B_{\delta / 2}$ provides a version of $T_{(\cdot )}$ that
is isomorphic to $T$ via an isomorphism with the following additional
property: It pairs trivalent vertices so that the resulting angle
assignments define a point in $B_{\delta }$ with distance $\varepsilon
_{1}$ or less from the initially chosen point. Finally, the constructions
in \refstep{Step 1}--\refstep{Step 3} can be made so that the result of all this is a continuous map
from $B_{\delta / 2}$ to $B_{\delta }$.

Now take $\varepsilon_{1}$ very small. Granted the preceding three
observations, some starting point in $B_{\delta / 2}$ gives $T_{(\cdot )}= T$
for the simple reason that $B_{\delta/2}$ doesn't
retract onto its boundary.

\subsection{The deformation to a $J$--pseudoholomorphic subvariety}\label{sec:3d}

Let $J'$ be an admissible almost complex structure, let $\vartheta$ denote an unordered set of
$N_{ + }$ points in $S^{1}$, and let $\mathcal{M}_{\hat{A}}[\Theta , \vartheta ]$ denote the
subset in the $J'$--version
of $\mathcal{M}_{\hat{A}}[\Theta ]$ that consists of subvarieties whose
inverse images in $\mathcal{M}^{\Lambda }$ are sent to the points in
$\vartheta$ by~\eqref{eq:2.21}'s map $\varpi_{ + }$. The respective sets $\Theta
$ and $\vartheta$ are deemed `generic' when the following conditions apply:
\itaubes{3.14}
\textsl{The set $\theta$ has $N_{ - }+\hat N+c_{\hat{A} }-2$ elements, and these elements are
pairwise disjoint and none arises via~\eqref{eq:1.7}  from an integer pair of any $(0,\ldots)$
element in $\hat{A}$.
}

\item \textsl{The set $\vartheta$ contains $N_{ + }$ distinct elements.
}
\end{itemize}

According to \fullref{prop:2.12}, any generic $\theta$ and $\vartheta$
version of $\mathcal{M}[\Theta , \vartheta ]$ is a submanifold of $\mathcal{M}_{\hat{A}}$.

Now, if $T$ is given by the graph from \fullref{sec:2g} of a subvariety from the $J'$
version of $\mathcal{M}_{\hat{A}}$, then the subvariety is in the version of
$\mathcal{M}[\Theta ]$ where $\theta$'s angles are those assigned to the
trivalent vertices in $T$.  Thus, $\theta$ is generic if $T$ is generic. More to
the point, \fullref{prop:2.12} finds a subvariety in a generic $\vartheta$
version of $\mathcal{M}[\Theta , \vartheta ]$ whose graph is also isomorphic
to $T$.

With the preceding understood, and granted what has been said in the
previous subsections, there exists an admissible almost complex structure
$J'$, a generic pair $(\Theta , \vartheta )$ and a subvariety $C$ in the $J'$
version of $\mathcal{M}[\Theta , \vartheta ]$ whose graph from \fullref{sec:2g} is
isomorphic to the graph $T$.

The remainder of this subsection explains how $C$ is used to construct a
$J$--pseudoholo\-morphic subvariety in the $J$--version of $\mathcal{M}[\Theta , \vartheta ]$
whose version of $T_{(\cdot )}$ is isomorphic to the
given moduli space graph $T$.  The description of such a deformation is broken
into six steps.

\step{Step 1}
This step explains the strategy for obtaining the desired subvariety. To
begin, choose a continuously parametrized family, $\{J^{a}\}_{a \in
[0,1]}$, in the space of admissible almost complex structures whose initial
element, $J^{0}$, is $J'$, and whose final element, $J^{1}$, is $J$. Having made
such a choice, an attempt is made to construct a corresponding family,
$\{C^{a}\}_{a \in [0,1]}$, of subvarieties in $\mathbb{R}\times  (S^{1}    \times S^{2})$
that has $C^{0} = C$ and is such
that any given $C^{a}$ is a $J^{a}$--pseudoholomorphic subvariety in the
$J^{a}$ version of $\mathcal{M}_{\hat{A}}$. In particular, the goal is to
construct such a family where each  $a \ge  0$ version of $C^{a}$ is in the
submanifold $\mathcal{M}[\Theta , \vartheta ]$ from the $J^{a}$ version of
$\mathcal{M}_{\hat{A}}$ and whose corresponding version of $T_{(\cdot )}$
is the given moduli space graph $T$.

To proceed, introduce the set, $\mf$, of points $r  \in  [0, 1]$ for which
$C^{a}$ exists for every value of  $a \le r$. This $\mf$ is non-empty since it
contains 0. The next step explains why $\mf$ is open. Modulo a technical
proposition, an argument is given in the third step that proves the
following: Either $\mf$ is closed, or else the submanifold $\mathcal{M}[\Theta , \vartheta ]$ in
the $J$--version of $\mathcal{M}_{\hat{A}}$ contains the desired
subvariety. If $\mf$ is closed, then $1  \in  \mf$ and the submanifold $\mathcal{M}[\Theta,\vartheta]$
in the $J$--version of $\mathcal{M}_{\hat{A}}$
contains the desired subvariety. Thus, the desired conclusion follows in
either case.

A bit more work will establish that $\mf$ is, in fact, closed. Moreover, the
resulting parametrized family $\{C^{a}\}_{a \in [0,1]}$, can be
constructed so that the parametrization varies continuously with the
parametrization, or smoothly in the case that the parametrization $a  \to
J^{a}$ is smooth. However, this extra work is left to the reader.

(By the way, the terms `continuous' and `smooth' for the parametrization
that sends $a  \to C^{a}$ are defined as follows: The parametrization is
continuous if there exists a multi-punctured sphere, $C_{0}$, with a
continuous map $\Phi \co  C_{0}    \times  [0, 1]  \to     \mathbb{R}\times(S^{1}\times S^{2})$
such that each $\Phi (\cdot , a)$ is a smooth,
proper immersion with image $C^{a}$ that is 1--1 to its image on the
complement of a finite set. The parametrization is smooth when there is such
a map $\phi$ that is smooth.)

\step{Step 2}
The proof that $\mf$  is open makes fundamental use of the generalization of
\fullref{prop:2.6} that follows. The proof of this proposition, like that of
\fullref{prop:2.6}, is much like that of  \cite[Proposition~3.2]{T3} and thus is
omitted.

\begin{prop}\label{prop:3.4}
Let $J'$  be an admissible almost complex structure, let $\hat{A}$  be any given
asymptotic data set, and let $C'$  be a subvariety in the $J'$  version of
$\mathcal{M}_{\hat{A}}$.  Let $C_{0}$
denote the model curve for $C'$  and let $\phi \co  \mathbb{R}    \times  (S^1    \times S^{2})$
denote its attending $J'$--pseudoholomorphic map. Then, there exists a constant $\kappa
\ge 1$, a ball $B  \subset \kernel(D_{C'})$, an open neighborhood, $\mathcal{U} $, of $J'$
in the space of admissible almost complex structures, and a smooth map $F$, from $\mathcal{U}
\times  B$ to $C^{\infty }(\phi ^{*}T_{1,0}(\mathbb{R}    \times
 (S^{1}    \times  S^{2}))$ with the following properties:
\begin{itemize}

\item
$| F(J', \eta )-\eta | +| \nabla (F(J', \eta ) -\eta |      \le     \kappa     | | \eta | |^{2}$.

\item
The exponential map on the $C'$ version of $\phi ^{*}T_{1,0}(\mathbb{R}
\times  (S^{1}    \times
S^{2}))$ composes with $F$ to give a smooth map, $\Phi \co  \mathcal{U}\times B\times  C_{0}     \to
\mathbb{R}    \times  (S^1    \times S^{2})$.

\item
With $(J'', \eta )     \in     \mathcal{U}    \times  \kernel(D_{C'})$ fixed, then $\Phi (J'',
\eta , \cdot  )$ maps $C_{0}$ onto a $J''$--pseudoholomorphic subvariety.

\item
As $\eta $ varies in $B$ with $J''$  fixed, the resulting family of subvarieties
defines an embedding,
$\psi_{J''}$, from $B$  onto an open set in the $J''$  version of $\mathcal{M}_{\hat{A}}$.
In particular, if $C''$ is in the $J''$ version of $\mathcal{M}_{\hat{A}}$ and if
\begin{equation}\label{eq:3.15}
\sup_{z \in C'} \dist(z, C'') + \sup_{z \in C''} \dist(C', z) < \frac1 {\kappa },
\end{equation}
then $C''$ is in the image of $\psi_{J''}$.
\end{itemize}
\end{prop}

With \fullref{prop:3.4}  in hand, what follows explains why $\mf$  is open. To begin,
suppose that $\tau  \in \mf$  and that $\tau  < 1$. Let $e\co  \phi ^{*}T_{1,0}(\mathbb{R}
\times  (S^{1}
\times S^{2})) \to     \mathbb{R}    \times  (S^{1}    \times S^{2})$
denote an exponential map of the
sort that is described in \fullref{sec:2d}. For values of $a$ that are somewhat
greater than $\tau $, the subvariety $C^{a}$ is the image of the composition
of $e$  with a suitably chosen, smooth section, $\eta ^{a}$, of $\phi
^{*}T_{1,0}(\mathbb{R}    \times  (S^{1}    \times  S^{2}))$. The
section $\eta ^{a}$ has very small norm when  $a \sim \tau $
and, in any event, is an element in the domain Hilbert space for the
$C^{\tau }$ version of the operator $D_{C}$ as described in \fullref{sec:2d}.
The properties of $\eta ^{a}$ are summarized by the next lemma. An
immediate corollary is that $f$  is open.

To prepare for the lemma, first note that it refers to the functions
$$\{v_{j}: 1  \le  j  \le  c\}, \{\varpi_{ + \alpha }: 1
\le     \alpha      \le  N_{ + }\}
\text{ and } \varpi_{\lambda ,r}$$
that
appear in the $C^{\tau }$ version of \fullref{prop:2.13}. In this regard, be
aware that the domain of definition of these functions extends in a
straightforward manner to include any subvariety in
$\mathbb{R}    \times  (S^1    \times S^{2})$ that has the asymptotics of a subvariety from
$\mathcal{M}_{\hat{A}}$ and is the image via the exponential map of a pointwise
small section of $\phi ^{*}T_{1,0}(\mathbb{R}    \times  (S^{1} \times S^{2}))$
with pointwise small covariant derivative. For example,
these functions are defined for the $a < \tau $ versions of $C^{a}$ when $a$
is sufficiently close to $\tau $.

\begin{lemma}\label{lem:3.5}
Given  $c \in  \{0, \ldots , N_{ - }+\hat N+c_{\hat{A} }-2\}$ and a $c$ element subset
$\Theta \subset  (0, \pi )$, suppose that $\tau \in
[0, 1)$ and that $C^{\tau }$ is any element of the $J^{\tau }$ version of
$\mathcal{M}[\Theta ]$. Then, there exists some $\delta  > 0$ and a continuous map from
$(-\delta, \delta )$ into the intersection of the domain of $D_{C}$ with the
$C^{\tau }$ version of
$C^{\infty }(\phi ^{*}T_{1,0}(\mathbb{R} \times  (S^{1}    \times S^{2})))$
such that the image of any given $\tau ' \in  (-\delta , \delta )$ is a section,
$\eta ^{\tau+\tau' }$,
whose composition with the exponential map sends $C_{0}$ onto a
$J^{\tau + \tau'}$--pseudoholomorphic
subvariety from the $J^{\tau + \tau '}$ version of the space $\mathcal{M}[\Theta ]$.
Moreover, the following is also true: Suppose that the 1--parameter family $\{C^{a}\}$
is continuously defined along the interval $[0, \tau ]$.

\begin{itemize}

\item
The map $\tau '  \to     \eta ^{\tau + \tau '}$ can then be constructed as a
continuous map with domain $(-\delta , \delta )$ such that each $a  \in
(\tau -\delta , \tau ]$ version of $C^{a}$ is the image of the composition of the exponential
map with the corresponding $\eta ^{a}$.

\item
If the family $\{J^{a}\}_{a \in [0,\tau + \delta ]}$ is smoothly parametrized,
and if the original family $\{C^{a}\}_{a \in [0,\tau
]}$ is smoothly parametrized on $[0, \tau ]$, then the map $\tau '  \to \eta ^{\tau + \tau '}$
can be constructed to be smooth on the whole of $(-\delta ,\delta )$.

\item
If a given subset of the functions $\{v_{j}: 1  \le  j  \le  c\}, \{\varpi_{ + \alpha }: 1
\le     \alpha      \le N_{ + }\}$ and $\varpi_{\lambda ,r}$ are constant on $C^{a}$ for values of
$a$ near to but less than $\tau $. Then the map $\tau '  \to \eta ^{\tau + \tau '}$
can be constructed so that the same subset of these functions are constant on the
$a > \tau $ subvarieties
as well.
\end{itemize}

In any event, the given 1--parameter family, $\{C^{a}\}_{a \in [0,\tau ]}$,
has an extension that is parametrized by the points in the interval $[0, \tau +\delta ]$.

\end{lemma}

\begin{proof}[Proof of \fullref{lem:3.5}]
Granted Propositions~\ref{prop:2.7},~\ref{prop:2.9},~\ref{prop:2.13}
and~\ref{prop:3.4}, all of the assertions are proved by using various straightforward
applications of the implicit function theorem.
\end{proof}

\step{Step 3}
This step explains why the submanifold $\mathcal{M}[\Theta , \vartheta ]$
in the $J$--version of $\mathcal{M}_{\hat{A}}$ contains the desired subvariety when
 $\mf$  is not closed. To begin, suppose that $\tau      \le 1$ and that the
family  $a \to  C^{a}$ has been defined for $a < \tau $ with each $C^{a}$ in
the appropriate version of the submanifold $\mathcal{M}[\Theta , \vartheta ]$.
The issue here is whether the domain of definition for the 1--parameter
family extends to the parameter value $a = \tau $ as well. The focus here is
thus on the convergence or lack there of for the sequence
$\{C^{a}\}_{a < \tau }$ as  $a \to     \tau $.

The next proposition asserts some facts about sequences of subvarieties of
the sort that is under consideration. The following is a direct corollary:
Either $\mf$  is closed or else the subset $\mathcal{M}[\Theta , \vartheta ]$ in
the $J$--version of $\mathcal{M}_{\hat{A}}$ is non-empty and contains a subvariety
whose corresponding graph from \fullref{sec:2g} gives the graph $T$.

\begin{prop}\label{prop:3.6}
Assume here that $\theta$ and $\vartheta$ are generic in the sense of~\eqreft3{14}.
Let $J'$ be an admissible almost complex structure, and suppose that $\{(J_{j},
C_{j})\}_{j = 1\ldots }$ is a sequence of pairs of the following sort: First, $\{J_{j}\}$
is a sequence of admissible almost complex structures that converges to $J'$.
Meanwhile, each $C_{j}$ is in the subset $\mathcal{M}[\Theta , \vartheta ]$
from the $J'$  version of
$\mathcal{M}_{\hat{A}}$ and each has graph $T_{(\cdot )}$ giving $T$.
Then one of the following two assertions hold:

\begin{enumerate}
\item[\textbf{A}]
There exists a subvariety $C'$  in the subset $\mathcal{M}[\Theta , \vartheta ]$ of the $J'$
version of $\mathcal{M}_{\hat{A}}$
with graph $T_{C'} = T$,  a subsequence of $\{C_{j}\}$ and a corresponding sequence,
$\{\eta_{j}\}$,
of sections of a fixed radius ball subbundle in the $C'$ version of
$\phi ^{*}T_{1,0}(\mathbb{R} \times
(S^{1}    \times S^{2}))$ such that composition of the exponential map with any given $\eta_{j}$
sends the model curve of $C'$ onto $C_{j}$. Moreover, the sequence of supremum norms over $C'$
of the elements in $\{\eta_{j}\}$ limits to zero as $j  \to     \infty $,
as do the analogous sequences of norms of the higher derivatives.

\item[\textbf{B}] There exists a subvariety $C$ in the subset $\mathcal{M}[\Theta , \vartheta ]$
of the $J$ version of $\mathcal{M}_{\hat{A}}$ with graph
$T_C = T$ a subsequence of $\{C_{j}\}$ and a corresponding sequence, $\{\eta_{j}\}$,
of sections of a fixed radius ball subbundle in the $C$ version of
$\phi ^{*}T_{1,0}(\mathbb{R} \times
(S^{1}    \times S^{2}))$ such that composition of the exponential map with any given $\eta_{j}$
sends the model curve of $C$ onto a translate of $C_{j}$ along the $\mathbb{R}$ factor of
$\mathbb{R}    \times  (S^1 \times S^{2})$. As before, the sequence of supremum norms over $C$
of the elements in $\{\eta_{j}\}$ limits to zero as $j  \to \infty $, as do the analogous
sequences of norms of the higher derivatives.
\end{enumerate}

\end{prop}

The proof of \fullref{prop:3.6} exploits convergence theorems that are modified
versions of assertions from Hofer, Wysocki and Zehnder \cite{HWZ1} about the behavior of limits of
pseudoholomorphic curves. (See also Bourgeois, Eliashberg, Hofer, Wysocki
and Zehnder \cite{BEHWZ}, which appeared during the
preparation of this article.) The next proposition summarizes the needed
results. Note that it makes no assumptions about $\theta$ and $\vartheta$
or any given moduli space graph such as $T$.

\begin{prop}\label{prop:3.7}
 Let $\{J_{j}\}$ denote a sequence of admissible almost complex structures with the following
 two properties: First, the derivatives of each such endomorphism to any fixed, non-negative order
 are bounded over the whole of $\mathbb{R}    \times (S^1 \times S^{2})$ by a
 $j$--independent constant.
 Second, there is an admissible almost complex structure, $J'$, such that the
 restriction of $\{J_{j}\}$
 to any given compact set in $\mathbb{R}    \times (S^1 \times S^{2})$
 converges in the $C^{\infty }$
 topology to the corresponding restriction of $J'$.  Next, let $\{C_{j}\}$
 denote a sequence where each
 $C_{j}$ is in the corresponding $J_{j}$
version of $\mathcal{M}_{\hat{A}}$. Then, there exists a subsequence of $\{C_{j}\}$
(hence renumbered by consecutive integers starting from 1) and a finite set, $\Xi $,
of pairs of the form
$(S, n)$ where $n$ is a positive integer and $S$ is an irreducible,
$J'$--pseudoholomorphic multi-punctured sphere; and these have the following properties:

\begin{itemize}

\item
$\lim_{j \to \infty }    \int_{C_j } \varpi =\sum_{(S,n) \in
\Xi } n \int_{S }\varpi$ for each compactly supported 2--form $\varpi$.

\item
If $K  \subset     \mathbb{R}    \times  (S^{1}    \times S^{2})$ is compact,
then the following limit exists and is zero:
\begin{equation}\label{eq:3.16}
\lim_{j \to \infty } \bigg(\sup_{z \in C_j \cap K} \dist(z,  \cup_{\Xi
} S) + \sup_{z \in ( \cup_\Xi S) \cap K} \dist(z, C_{j})\bigg)
\end{equation}

\end{itemize}

\end{prop}

The proof is given momentarily.

\step{Step 4}
A portion of the proof of the previous proposition, as well as subsequent
arguments in this article and in the sequel, require the lifting of certain
submanifolds of $\Sigma      \equiv      \cup_{(S,n) \in \Xi }S$  to large
$j$ versions of $C_{j}$. This step explains how these liftings are done.

To start, it is necessary to first pass to a subsequence of $\{C_{j}\}$
where the corresponding sequence of sets of critical points of $\theta$ and
sequence of sets of singular points converge in $\mathbb{R}    \times (S^1
\times S^{2})$. Let $Y_{j}     \subset C_{j}$ denote the set of
critical points of $\theta$ and singular points of $C_{j}$, and let $Y_{o}
 \subset     \sigma$ denote the limit of $\{Y_{j}\}$. Next, let $\Sigma
_{*}     \subset     \sigma$ denote the union of the irreducible
components that are not of the form $\mathbb{R}    \times     \gamma$ where
$\gamma      \subset S^{1}    \times S^{2}$ is a Reeb orbit. Now, let
$Y_{*}     \subset     \Sigma_{*}$ denote the union of the set
$Y_{o}$, the critical points of $\theta$ on the subvarieties that comprise
$\Sigma_{*}$ and the singular points of $\sigma$ that lie in
$\Sigma_{*}$. Note that the latter set may contain points that are
not points of $Y_{o}$. In any event, $Y_{*}$ is a finite set.

Suppose next that $K  \subset     \Sigma_{*}-Y_{*}$
is a given compact set. Such a set $K$ has a tubular neighborhood, $U_{K}
\subset     \mathbb{R}    \times  (S^{1}    \times S^{2})$ with projection
$\pi \co  U_{K}     \to  K$ whose fibers are disks on which $\theta$ is
constant and that are pseudoholomorphic for any admissible almost complex
structure. Indeed, the fiber of the projection to any given  $p \in  K$ is a
disk centered on $p$ inside $\mathbb{R}    \times     \gamma ^{p}$, where $\gamma
^{p}     \subset S^{1}    \times S^{2}$ is a small segment of the
integral curve of the Reeb vector field through the image of $p$. In this
regard, note that $\mathbb{R}    \times     \gamma ^{p}$ intersects $K$
transversely at $p$ by virtue of the fact that $p$ is not a critical point of
$\theta$ on $\sigma$. With $K$ fixed, then each sufficiently large $j$ version
of $C_{j}$ will have proper intersection with $U_{K}$ and intersect each
fiber precisely $n$ times, with each a transversal intersection and local
intersection number $+1$. Such is the case precisely because the large $j$
version of $C_{j}$ has no $\theta$ critical points in $U_{K}$. In any
event, here is a fundamental consequence: The projection, $\pi \co  C_{j}
\cap  U_{K}     \to  K$ defines a smooth, proper, degree $n$ covering map. In
particular, any compact, embedded arc in $\Sigma_{*} -Y_{*}$ has lifts to $C_{j}$
under the projection $\pi$.

The lifts just described can be extended as lifts of arcs in a somewhat
larger set in $\Sigma_{*}$. To define this set, let $Y  \subset
Y_{*}$ denote the subset of points that are either singular points of
$\sigma$, critical points of $\theta$ on $\sigma$, or limits of
convergent sequences of the form $\{p_{j}\}$ where $p_{j}$ is either a
critical point of $\theta$ on $C_{j}$ or a non-immersion singular point of
$C_{j}$. Then a smooth arc in a compact subset of $\Sigma -Y$ has a
well defined lift to all large $j$ versions of $C_{j}$ that extends the lift
defined in the preceding paragraph. In the discussion that follows, such a
lift is deemed a `$\theta$--preserving preimage' in $C_{j}$.

\step{Step 5}
Here is the proof of \fullref{prop:3.7}:

\begin{proof}[Proof of \fullref{prop:3.7}]

This result is essentially from Hofer \cite{H1} or Hofer--Wysocki--Zehnder
\cite{HWZ1} (see also the article by Bourgeois, Eliashberg,
Hofer, Wysocki and Zehnder \cite{BEHWZ} that appeared during the writing
of this article). Here is the basic idea: There is a bound on the area
of the intersection of $C_{j}$ with any given $[s-1, s+1] \times  (S^{1}
\times S^{2})$ subcylinder that is independent of both $s$ and $j$. The
existence of such a bound can be deduced from the fact that the integral
of $d\alpha $ on $C_{j}$ is finite. The existence of this local area
bound is the key observation. The existence of the asserted limit data
set $\Xi $ is deduced from the latter using arguments that are very
similar to those used for the convergence theorems about sequences of
pseudoholomorphic curves on compact symplectic manifold. Given the local
area bound, a somewhat different proof of the convergence assertion can
be obtained using  \cite[Proposition~3.3]{T4}, but without the observation
that each subvariety from $\Xi $ is a multipunctured sphere.

What follows explains why each subvariety from $\Xi $ is a sphere with
punctures. In this regard, it is enough to consider only the non--$\mathbb{R}$
invariant subvarieties. To start this chore, take any given irreducible
component $S$ from $\Sigma_{*}$, and note that it is enough to prove
that some non-zero multiple of any class in the first homology of the model
curve for $S$ is generated by loops on the ends of $S$, thus `end-homologous'.
For this purpose, fix a generator of the first homology of the model curve
for $S$, and let $\tau \subset S$  be the image of an embedded representative
that is disjoint from $Y  \cap  S$. In addition, fix $R \gg 1$ so that the
$|s|  > R$ part of $S$ lies out on the ends of $S$ and is disjoint
from $\tau $. In this regard, remark that given $R$, there exists $\varepsilon
 > 0$ such that the $|s|      \le R$  portion of the ends of $S$
have pairwise disjoint, radius $\varepsilon $ tubular neighborhoods.

Take the index $i$ to be very large, and let $\tau_{i}     \subset C_{i}$
denote a connected, $\theta$--preserving preimage of $\tau $. Then $\tau
_{i}$ is homologous in a regular neighborhood of $S$ to some non-zero
multiple of $\tau $. However, $\tau_{i}$ is also homologous in $C_{i}$ to
a union of curves on the ends of $C_{i}$. Intersect this homology with the
$|s|      \le R$  part of $C_{i}$ and then deform the latter back
to $S$ in a small radius, regular neighborhood. The result is a homology
between a non-zero multiple of $\tau $ and a union of curves on the ends of
 $S$ as well as a union of circles that are very close to points in $S \cap
Y$. As the inverse image of the latter circles are null-homologous in the
model curve, so the chosen generator of the model curve's first homology is
end-homologous.
\end{proof}

\step{Step 6}
With the proof of \fullref{prop:3.7}  now complete, remark that it may well be
the case that each subvariety from $\Xi $ is an $\mathbb{R}$--invariant
cylinder, thus of the form $\mathbb{R}    \times     \gamma$ where $\gamma
\subset S^{1}    \times S^{2}$ is an orbit of the Reeb vector field.
Item B  of \fullref{prop:3.6} holds when all subvarieties from $\Xi $
are $\mathbb{R}$--invariant cylinders, and Item A of \fullref{prop:3.6} holds when such is
not the case. To explain how this dichotomy comes about, suppose first that
there exists some subvariety from $\Xi $ that is not $\mathbb{R}$--invariant. In
this case, the proof of \fullref{prop:3.6}  proceeds to establish that $\Xi $
consists of a single pair, and that the latter has the form $(S, 1)$ with $S$ in
the submanifold $\mathcal{M}[\Theta , \vartheta ]$ from the $J'$  version of
$\mathcal{M}_{\hat{A}}$. This implies that the graph $T_S$ from \fullref{sec:2g} can
be labeled as a moduli space graph. The fact that the $T_S$ is isomorphic
to $T$ is seen as an automatic consequence of the strengthened versions of~\eqref{eq:3.16}
that appear in the next three subsections.

In the case that all subvarieties from $\Xi $ are $\mathbb{R}$--invariant, the
argument for Item $B$  of \fullref{prop:3.6}  proceeds as follows: The original
sequence $\{C_{i}\}$ is now replaced by a new sequence,
$\{C'_{i}\}$, where each $C'_{i}$ is obtained by translating the
corresponding $C_{i}$ along the $\mathbb{R}$ factor in $\mathbb{R} \times (S^1 \times S^{2})$.
It should be evident from the description given
below that the resulting sequence of translations (as elements in $\mathbb{R})$
does not have convergent subsequence. In any event, each $C'_{i}$ is
pseudoholomorphic for the translated almost complex structure, this denoted
by $J'_{i}$. The latter sequence has a subsequence that converges on compact
sets in $\mathbb{R}    \times  (S^1 \times S^{2})$ to the almost
complex structure $J$. This understood, the sequence of translations is chosen
so as to insure that $\{C'_{i}\}$ converges as described in \fullref{prop:3.7},
but with a limit data set of pairs that contains one whose
subvariety component is not $\mathbb{R}$--invariant. The argument proceeds from
here as in the previous case: It demonstrates that \fullref{prop:3.7}'s limit
data set of pairs is a single pair, this of the form $(S, 1)$, where $S$ is in
the submanifold $\mathcal{M}[\Theta , \vartheta ]$ of the $J$--version of
$\mathcal{M}_{\hat{A}}$ whose corresponding graph is isomorphic to $T$.

The translation, $s_{j}     \in     \mathbb{R}$, that takes $C_{j}$ to
$C'_{j}$ is defined as follows: Chose a fixed angle, $\theta_{*}
\in  (0, \pi )$, with the following properties: First, $\theta_{*}$ is strictly
between the minimal and maximal angles that are defined by
the elements of $\hat{A}$. Thus the $\theta =\theta_{*}$ locus in
$C_{j}$ is non-empty. Second, no pair $(p, p')$ makes the $\theta  =
\theta_{*}$ version of~\eqref{eq:1.7}  hold. Now, define $s_{j}$ so that the
translate  $s \to  s + s_{j}$ moves $C_{j}$ so that the result, $C'_{j}$,
has a point on its $\theta =\theta_{*}$ locus where $s = 0$.

Granted this definition, it then follows from the $\{C'_{j}\}$ version
of~\eqref{eq:3.16} that the resulting limit data set, $\{(S, n)\}$, contains some $S$
on which $\theta$ takes value $\theta_{*}$. Such a subvariety can
not be an $\mathbb{R}$--invariant cylinder.

\subsection{Convergence}\label{sec:3e}

With \fullref{prop:3.7}  and with what has been said so far, \fullref{prop:3.6}
follows directly as a corollary to

\begin{prop}\label{prop:3.8}

Assume that $\theta$ and $\vartheta$ are generic. Let $\{J_{i}\}$
denote a sequence of admissible, almost complex structures with the following three properties:
First, the derivatives of each $J_{j}$ to any fixed, non-negative order are bounded over
$\mathbb{R}    \times  (S^1 \times S^{2})$ by a $j$--independent constant. Second,
there is a constant,
$L$, such that $J_{j} = J$ on the complement of some length $L$ subcylinder in
$\mathbb{R}    \times  (S^1 \times S^{2})$. Finally, there is an admissible, almost complex
structure, $J'$, such that the restriction of $\{J_{j}\}$ to any compact set in
$\mathbb{R}    \times  (S^1 \times S^{2})$ converges in the $C^{\infty }$ topology to the
corresponding
restriction of $J'$. Let $\{C_{j}\}$ denote a sequence where each $C_{j}$ is in the submanifold
$\mathcal{M}[\Theta , \vartheta ]$ of the $J_{j}$--version of $\mathcal{M}_{\hat{A}}$,
and where each has its graph from \fullref{sec:2g} giving $T$. Now suppose that $\{C_{j}\}$ converges
as described in \fullref{prop:3.7}
with limit data set $\Xi $. In this regard, assume that $\Xi $ contains at least one
subvariety that is not of the form $\mathbb{R}    \times     \gamma$ where $\gamma$ is a
Reeb orbit in $S^{1}    \times S^{2}$. Then $\Xi$ consists of a single pair, this pair has the
form $(S, 1)$, $S$ is in the submanifold $\mathcal{M}[\Theta , \vartheta ]$ from the
$J'$--version of
$\mathcal{M}_{\hat{A}}$, and the graph of $S$ is isomorphic to $T$.

\end{prop}
The remainder of this subsection and the next two subsections are occupied
with the proof of this proposition. In this regard, note that there
are various ways to prove this proposition, in particular, some using mostly
differential equation techniques such as can be found in Hofer
\cite{H1,H2,H3} and Hofer--Wysocki--Zehnder \cite{HWZ1,HWZ2,HWZ3}
and the very recent Bourgeois--Eliashberg--Hofer--Wysocki--Zehnder \cite{BEHWZ}. The proof offered below relies almost entirely
on arguments that are of a topological nature. In any event, the arguments
used below are exploited in various modified forms in the sequel to this
article.

The proof starts with a proof that the convergence assertion in~\eqref{eq:3.16} holds
even in the case that $K = \mathbb{R}    \times  (S^1 \times S^{2})$.
This first part of the proof occupies the remainder of this subsection.

\begin{proof}[Part 1 of the Proof of \fullref{prop:3.8}]

The proof that the $K = \mathbb{R}    \times  (S^1 \times S^{2})$ version of~\eqref{eq:3.16} holds
starts here by making the assumption that the $K = \mathbb{R}    \times  (S^1 \times S^{2})$
version of~\eqref{eq:3.16} is false. The proof proceeds
to derive a patently nonsensical conclusion.

To start this derivation, let $\Sigma      \equiv      \cup_{(S,n) \in \Xi
}S$, let $E$ denote any end from $\sigma$, and let $\mathbb{R}    \times
\gamma_{*}$ denote the translationally invariant cylinder with the
same large $|s| $ asymptotics of as $E$. This is to say that the
$|s|      \to     \infty $ limits of $E$ converge in $S^{1}    \times
S^{2}$ to the Reeb orbit $\gamma_{*}$. Now, fix some very small
but positive number $\varepsilon $ and there exists a value, $s_{0}$, of $s$
on $E$ and a pair of sequences, $\{s_{j + }\} \subset  [s_{0},
\infty )$ and $\{s_{j - }\} \in  (-\infty , s_{0}]$ with the
following properties:
\itaubes{3.17}
\textsl{If $E$ is on the concave side of $\sigma$, then $\{s_{j + }\}$ has no convergent subsequences;
and if $E$ is on the convex side of $\sigma$, then $\{s_{j - }\}$ has no convergent subsequences.
}

\item \textsl{For each index $j$, the intersection of $C_{j}$ with the  $s \in  [s_{j - }, s_{j + }]$
portion of the radius $\varepsilon $ tubular neighborhood of $\mathbb{R}\times     \gamma_{*}$
has an irreducible component, $C_{j*}$, where $|s| $ takes both the values $s_{j - }$
and $s_{j + }$ and whose points have distance $\frac{1}{4}\varepsilon $ or less from $\mathbb{R}
\times  \gamma_{*}$.
}

\item \textsl{For each index $j$, there exists a subinterval, $I_{j} \subset  [s_{j - }, s_{j + }]$ such that
}
\begin{enumerate}

\item[(a)] \textsl{The sequence whose $j'$th element is the length of $I_{j}$ diverges as $j\to\infty$.}

\item[(b)] \textsl{The sequence whose $j'$th element is the maximum distance from the
$s \in  I_{j}$ portion
of $C_{j*}$ to $\mathbb{R}    \times     \gamma_{*}$ limits to zero as $j  \to     \infty $.}
\end{enumerate}

\end{itemize}
The fact that all of this can be arranged is a straightforward consequence
of the manner of convergence that is dictated by \fullref{prop:3.7}.

Now comes a key point: Because each $C_{j}$ is irreducible, when
$\varepsilon $ is small, there must exist an infinite sequence of positive
integers $j$ (hence renumbered consecutively from 1) and at least one end $E
\subset     \sigma$ with the following properties:
\qtaubes{3.18}
\textsl{%
Values for $s_{j + }$ and $s_{j - }$ can be chosen for use in $E'$s
version of~\eqreft3{17} so that both
the $s = s_{j + }$ and $s = s_{j - }$ loci in $C_{j*}$ contain some point with distance
$\frac{1}{4}\varepsilon $ from $\mathbb{R}    \times     \gamma_{*}$.}
\endqtaubes
With~\eqreft3{18} understood, there are two cases to consider.
\end{proof}

\step {Case 1}
There is an end $E  \subset     \sigma$
where~\eqreft3{18} holds and where the value, $\theta_{*}$, of $\theta$
on $\gamma_{*}$ is neither 0 nor $\pi$. The derivation of nonsense
in this case is a four step affair.

\substep{Step 1}
This step starts with a crucial lemma.

\begin{lemma}\label{lem:3.9}

There exists $\delta  > 0$ such that for all large $j$, the angle $\theta$ takes both the values
$\theta_{*}+\delta $ and $\theta_{*}-\delta $ on the $s = s_{j - }$ locus in $C_{j*}$.
Meanwhile, $\theta$ takes at least one of these values on the $s = s_{j + }$ locus in $C_{j*}$.
\end{lemma}

The proof of this lemma is given below.

Granted \fullref{lem:3.9}, the `mountain pass' lemma with the third point in~\eqreft3{17}
implies that there is a critical point of $\theta$ on each large $j$ version
of $C_{j*}$ with critical value equal to $\theta_{*}$.

As is explained next, the relatively prime integer pair $(p, p')$ that is
defined by $\theta_{*}$ via~\eqref{eq:1.7}  must be proportional to either
$Q_{e'}$ or $Q_{e''}$ where $e'$, $e''$ with $e$  label the three edges in $T$ that
are incident to the vertex that labels the critical point with critical
value $\theta_{*}$. Here, the convention for distinguishing $e$  from
$e'$ and $e''$ is as follows: The respective vertices on $e'$ and $e''$ that lie
opposite that with angle label $\theta_{*}$ have angle labels on
the same side of $\theta_{*}$ in $(0, \pi )$. This last conclusion
exhibits the required nonsense since it requires the vanishing at $\theta
= \theta_{*}$ of either the $Q = Q_{e'}$ or $Q = Q_{e''}$ version
of $\alpha_{Q}$.

To see why $(p, p')$ are proportional to one of $Q_{e}$ or $Q_{e'}$, let $j$ be
very large and let $K_{e}$, $K_{e'}$ and $K_{e''}$ denote the components in
the complement of the $C_{j}$ version of $\Gamma$ in $C_{j}$'s model curve.
Slice $C_{j*}$ into two pieces near the $s = \frac{1}{2}s_{j}$ locus.
This then slices $C_{j}$ into two parts, where one part,
$C_{j + }$, contains the larger $s$ portion of the sliced component of
$C_{j*}$. Let $C_{j - }$ denote the other part. By virtue of the fact
that $\theta$ spreads uniformly to both sides of $\theta_{*}$ on $C_{j - }$,
the latter must contains most of both $K_{e}$ and one of
$K_{e'}$ or $K_{e''}$. Agree to distinguish the latter as $K_{e'}$.
Meanwhile, $C_{j + }$ contains most of $K_{e''}$. In this regard, the
portions that are missing in either of the three cases are portions where
$\theta$ is everywhere very close $\theta_{*}$.

Now, recall from \fullref{sec:2g} that the $\theta =\theta_{*}$ part
of the $\Gamma$--locus in $C_{j}$'s model curve has the form of a `figure
8', where one of the circles is the $\theta =\theta_{*}$
boundary of the closure of $K_{e'}$ and the other that of the closure of
$K_{e''}$. This implies that any given constant $\theta$ circle in
$K_{e''}$ is homologous to the union of a constant $\theta$ circle in
$K_{e'}$ and a constant $\theta$ circle in $K_{e}$. Take these circles to
have $\theta$ value that differ by order one from $\theta_{*}$.
This the case, the obvious `pair of pants' in $C_{j}$'s model curve with
these three constant $\theta$ circles as boundary provides a homology. Now,
this pair of pants is sliced in two pieces by the  $s \sim \frac{1}{2}s_{j}$ cut.
In particular, the $C_{j + }$ piece
contains one of the boundary circle in $K_{e''}$ and the $C_{j - }$ piece
contains the other two boundary circles. This understood, it then follows
from the definitions of $Q_{e''}$ and $(p, p')$ as integrals of $\frac1
{2\pi }dt$ and $\frac1 {2\pi }d\varphi$ that these
integer pairs are proportional.

\substep{Step 2}
This step, \refstep{Step 3} and \refstep{Step 4} contain the following proof.

\begin{proof}[Proof of \fullref{lem:3.9}]

Translate each $C_{j}$ in $\mathbb{R}    \times  (S^1 \times S^{2})$ by sending
$s$ to $s + s_{j - }$ in the $\mathbb{R}$
factor, and let ${C_{j}}'$ denote the corresponding subvariety. The sequence
$\{{C_{j}}'\}$ converges in the manner dictated by \fullref{prop:3.7} with
some limit data set $\Xi' $. It follows from~\eqreft3{18} that $\Xi' $ contains an
irreducible subvariety with a concave side end, $E'$, with the following
property: Given some very large $R$, a value of $s$ on $E'$, there is an infinite
subsequence from $\{C_{j}\}$ (hence renumbered from 1) such that the  $s\to  s + s_{j - }$
translates of the $s_{j - }+R$ slices of $C_{j*}$
converge pointwise to the $s = R$ slice of $E'$. Let $\gamma$' denote the Reeb
orbit that is the limit of the the $|s|      \to     \infty $
slices of $E'$. \refstep{Step 3} proves that $\gamma ' = \gamma_{*}$.

Granted that $\gamma ' = \gamma_{*}$, there are two
possibilities: Either $E'$ sits as a sub-cylinder in $\mathbb{R}    \times
\gamma_{*}$ or not. If not, then $E'$ is a concave side end of some
subvariety from $\Xi '$ that is not $\mathbb{R}$--in\-var\-i\-ant. This understood, it
follows from~\eqref{eq:2.17} that $\theta$ takes values both above and below $\theta
_{*}$ on any constant $|s| $ slice of $E'$. (Remember that
the $n = 0$ case of~\eqref{eq:2.17}
is reserved solely for convex side ends.) Thus
$\theta$ must take values on the $s = s_{j - }$ slice of each large $j$
version of $C_{j*}$ that differ from $\theta_{*}$ in both
directions by some $j$--independent, non-zero amount.

Suppose, on the other hand that $E'$ is contained in $\mathbb{R}    \times
\gamma_{*}$. As there are points on $C_{j*}$ where s $ \sim
 s_{j - }$ with distances at least $\frac{1}{4}\varepsilon $
from $\mathbb{R}    \times     \gamma_{*}$, the convergence described by
\fullref{prop:3.7}  requires a subvariety from $\Xi'$ that is not
$\mathbb{R}$--invariant and contains a disk with the following property: The center
point is on $\mathbb{R}    \times     \gamma_{*}$ and all other points
are limit points of sequences whose $j'$th element is in the  $s \to  s +
s_{j - }$ translate of $C_{j*}$. Since $\theta$ has no local maximum
or minimum on such a disk, it thus follows that $\theta$ must take values
on the $s = s_{j - }$ slice of each large $j$ version of $C_{j*}$ that
differ from $\theta_{*}$ in both directions by some $j$--independent,
non-zero amount.

To establish the asserted behavior of $\theta$ where $s$ is near $s_{j
+ }$ on $C_{j}$, translate each $C_{j}$ in $\mathbb{R} \times (S^1
\times S^{2})$ by sending $s \to s + s_{j + }$ in the
$\mathbb{R}$--factor. Let ${C_{j}}'$ now denote the result of this new
translation. Invoke \fullref{prop:3.7} once again to describe the
convergence of this new version of $\{{C_{j}}'\}$, using $\Xi'$ to
denote the new limit data set. In this case, there is a convex side
end, $E'$, with the following property: Given some very large $R$, a
value of $s$ on $E'$, there is an infinite subsequence from
$\{C_{j}\}$ (hence renumbered from 1) such that the $s \to s + s_{j +
}$ translates of the $s_{j + }- R$ slices of $C_{j*}$ converge
pointwise to the $s = -R$ slice of $E'$. Let $\gamma'$ denote the Reeb
orbit that is the limit of the the $|s| \to \infty $ slices of
$E'$. \refstep{Step 3} and \refstep{Step 4} prove that $\gamma ' =
\gamma_{*}$.  Granted this, then the argument for the desired
conclusion that $\theta$ takes some value on the $s = s_{j + }$ slice
of $C_{j*}$ that differs from $\theta_{*}$ by a $j$--independent
amount is much the same as that given in the preceding paragraph. In
fact, here is the only substantive difference: The argument now only
finds values of $\theta$ on the $s \sim s_{j + }$ locus in $C_{j*}$
that differ in at least one direction from $\theta_{*}$ by a non-zero,
$j$--independent amount. This is because the integer $n$ that appears
in a convex side end version of~\eqref{eq:2.17} can vanish.

\substep{Step 3}
But for one claim, this step proves that with $\varepsilon $ small, the
Reeb orbits $\gamma' $ and $\gamma_{*}$ agree in all of their Step
2 incarnations. To start, note first that with $\varepsilon $ small, the
Reeb orbit $\gamma '$ sits in a tubular neighborhood of $\gamma_{*}$, and so
$\gamma '$ must be a translate of $\gamma_{*}$ by some
element in the group $T = S^{1}    \times S^{1}$ whose distance is
$o(\varepsilon )$ from the identity.

To make further progress, let $(p, p')$ again denote the relatively prime pair
of integers that $\theta_{*}$ defines via~\eqref{eq:1.7}. As can be readily
verified, the 1--form $p d\varphi  - p' dt$ is exact on a tubular
neighborhood of $\gamma_{*}$, and this form pulls back as
zero on any translate of $\gamma_{*}$ by the group $T$.  Moreover, on
such a tubular neighborhood, the values of any chosen anti-derivative of
this 1--form distinguish the $T$--translates of $\gamma_{*}$. In this
regard, let $f$ be an anti-derivative with value zero on $\gamma_{*}$. Then
\begin{equation}\label{eq:3.19}
\int_{\gamma }f\big(pdt + p'\sin^{2}\theta_{*} d\varphi \big) =
f|_{\gamma } 2\pi  \big(p^{2} + p'^{2} \sin^{2}\theta
_{*}\big),
\end{equation}
when $\gamma$ is any translate of $\gamma_{*}$ by an element from
a small radius ball about the identity in $T$.  Thus, the value of the integral
on the left side here will distinguish $\gamma' $ from $\gamma_{*}$
if these two orbits are not one and the same.

With the preceding understood, take $j$ large, and let $s_{j0}$ and $s_{j1}$
denote any two regular values of $s$ on $C_{j*}$, chosen so that $J_{j}
= J$ on the cylinder where  $s \in  [s_{j0}, s_{j1}]$. As is explained in
\refstep{Step 4}, the respective integrals of $f (pdt + p'\sin^{2}\theta_{*} d\varphi )$ over the
$s = s_{j0}$
and $s = s_{j1}$ slices of $C_{j*}$ agree by virtue of the fact that $C_{j*}$
is pseudoholomorphic.
The use of this last fact is simplest in the case that there exists some
$j$--independent, positive number, $R$, such that $J_{j} = J$ where  $s \in
[s_{j - }+R, s_{j + }-R]$. If such is the case, then take any fixed $r > R$
such that each large $j$ version of $s_{j - } + r$ is a regular value of $s$ on
$C_{j*}$. Granted this, take $s_{j0}$ to be $s_{j - } + r$. For $r$
large and then $j$ very large, the convergence as described in \fullref{prop:3.7}
guarantees that the integral of $f (pdt + p'\sin^{2}\theta_{*}d\varphi )$ over the $s = s_{j0}$
slice of $C_{j*}$ is very close to
the $\gamma =\gamma '$ version of the right hand side in~\eqref{eq:3.19}.
Meanwhile, take $s_{j1}$ to be some generic value of s from the interval
$I_{j}$ from the third point of~\eqreft3{17}. As the latter slice of
$C_{j*}$ is very close to $\mathbb{R}    \times     \gamma_{*}$, the
integral of the form $f (pdt + p'\sin^{2}\theta_{*} d\varphi)$
about such a slice will be very close to the $\gamma =\gamma_{*}$ version of the right
hand side of~\eqref{eq:3.19}. As $r$ can be as large as
desired, and likewise $j$, it thus follows that $\gamma ' = \gamma_{*}$.

Of course, it may well be the case that there is no $j$--independent choice of
$R$ that excludes points with $J_{j } \ne  J$ from where  $s \in  [s_{j -
}+R, s_{j + }-R]$. There exists in this case, some fixed $R > 0$ such that
$J_{j} = J$ save where $s  \in  [-R, R]$. If the lower bound for $s$ on all
large $j$ versions of the interval $I_{j}$ contain points where $s < -R$, then
the argument given in the previous paragraph works just fine. The situation
is different if the lower bound of $s$ on all large $j$ versions of $I_{j}$ is
greater than $-R$. Assuming that such is the case, first fix some large value
of $r$ so that for every sufficiently large $j$, each of $s_{ - j} + r$, $-(R+r)$ and $R+r$
are regular values of $s$ on $C_{j*}$. This done, first
take $s_{0j} = s_{ - j} + r$ and $s_{1j}=- (R+r)$ to establish
that the integral of $f (pdt + p'\sin^{2}\theta_{*} d\varphi)$
about the $s = - (R+r)$ slice of $C_{j*}$ is very close to the
$\gamma =\gamma '$ version of the right hand side of~\eqref{eq:3.19}. Next,
take $s_{0j}$ to equal $R+r$ and take $s_{1j}$ to be a generic value in
$I_{j}$ so as to establish that the integral of $f (pdt +
p'\sin^{2}\theta_{*} d\varphi )$ about the $s = (R+r)$ slice of
$C_{j*}$ is nearly the $\gamma =\gamma_{*}$ version of
the right hand side of~\eqref{eq:3.19}.

It remains now to establish that the integrals of $f (pdt +
p'\sin^{2}\theta_{*} d\varphi )$ along the respective
$s = - (R+r)$ and $s = (R+r)$ slices of $C_{j*}$ are, for very large $j$, very
close to each other. To argue for this, note that the sequence
$\{C_{j}\}$ converges according to \fullref{prop:3.7}  to a limiting
$J'$--pseudoholomorphic subvariety. The sequence $\{ C_{j*}\}$ thus
has a subsequence that converges on every subcylinder of the form $[-R+r,
R+r] \times  (S^{1}    \times S^{2})$ to a component, $C'_{*}$,
of this subvariety. In this regard, all points of $C'_{*}$, must have
distance less than $\frac{1}{4}\varepsilon $ from $\mathbb{R}
\times     \gamma_{*}$. Furthermore, the constant $|s| $
slices of the convex side ends of $C'_{*}$ must converge as $|s| \to \infty $ to $\gamma_{*}$.
Meanwhile, those of its concave side ends must converge as $|s| \to \infty $
to $\gamma' $. But, this then implies that $\gamma ' = \gamma_{*}$ since the angle
$\theta$ is the same on $\gamma'$ as on $\gamma_{*}$ and so has to be constant on $C'_{*}$.

\substep{Step 4}
To tie up the final loose end, suppose that $\gamma_{*}$ is a Reeb
orbit where the value, $\theta_{*}$, of $\theta$ is neither 0 nor
$\pi$. Let $(p, p')$ denote the relatively prime pair of integers that
$\theta_{*}$ defines via~\eqref{eq:1.7}. Let  $U \subset S^1
\times S^2$ denote a tubular neighborhood of $\gamma_{*}$ and
let $f$ denote an anti-derivative of $p d\varphi  - p'dt$ on $U$. Now, suppose
that $s_{ + } > s_{ - }$, and that $C_{*}$ is closed,
$J$--pseudoholomorphic subvariety inside $[s_{ - }-1, s_{ + }+1] \times
U$. Thus, $C_{*}$ has compact intersection with the subcylinder $[s_{ -
}, s_{ + }] \times U$. Suppose that both $s_{ - }$ and $s_{ + }$ are
regular values of the restriction of $s$ to $C_{*}$. Proved here is the
assertion that the respective integrals of the 1--form $f (pdt +
p'\sin^{2}\theta_{*}d\varphi )$ over the $s = s_{ - }$ and $s =
s_{ + }$ slices of $C_{*}$ agree.

For this purpose, use Stokes' theorem to write the difference of the two
integrals as the integral over $C_{*}$'s intersection with $[s_{ - },
s_{ + }] \times  U$ of the form $df     \wedge  (pdt + p'\sin^2\theta
_{*}d\varphi )$. Written out, the latter is
$-(p^{2}+p'^{2}\sin^{2}\theta_{*}) dt  \wedge  d\varphi$.
Now, as $C_{*}$ is $J$--pseudoholomorphic, the latter integral is
identical to that of the 2--form $-(p^{2}+p'^{2}\sin^{2}\theta_{*})
ds \wedge \frac{1}{\sin \theta }d\theta$.
A second application of Stokes' theorem establishes the assertion.
\end{proof}

\step{Case 2}
This case assumes that the $|s| \to \infty $ limits of $\theta$ is 0 or $\pi$ on every end in
$\sigma$ where~\eqreft3{18} holds. Note that the arguments below consider only
the case where the aforementioned limit of $\theta$ is 0. The argument for
the case where the limit is $\pi$ is identical but for some minor
notational modifications.

To start the derivation of nonsense in this case, remark that the Reeb orbit
$\gamma_{*}$ in~\eqreft3{17} is the $\theta =0$ locus. Let $\theta
_{j - }$ denote the maximum value of $\theta$ on the $s = s_{j - }$ slice
of $C_{j*}$, and let $\theta_{j + }$ denote the maximum of $\theta
$ on the $s = s_{j + }$ slice. Since both of these slices have points with
distance $\frac{1}{4}\varepsilon $ from the $\theta =0$ locus,
it follows that for fixed, small $\varepsilon $ and large $j$, both $\theta
_{j - }$ and $\theta_{j + }$ are greater than $\frac{1}{100}\varepsilon $.
This understood, let $\theta_{j*}$ denote
the minimum value of the function on the interval $[s_{j - }, s_{j + }]$
that assigns to any given $a  \in [s_{j - }, s_{j + }]$ the maximum
value of $\theta$ on the $s = a$ slice of $C_{j*}$. In this regard,
$\theta_{j*} > 0$, but by virtue of the third point in~\eqreft3{17},
$\lim_{j \to \infty }    \theta_{j*}$ is zero.

Granted the preceding, the mountain pass lemma dictates that all large $j$
versions of $C_{j*}$ have a critical point of $\theta$ where $\theta
=\theta_{j*}$. Of course, the latter conclusion is nonsense as
all critical values of $\theta$ that are neither 0 nor $\pi$ lie in the
fixed, $j$--independent set $\theta$.

\subsection{Part 2 of the proof of Proposition \ref{prop:3.8}}\label{sec:3f}

This second part of the proof establishes that the set $\Xi $ contains just
one element. The argument here is presented in four steps.

\substep{Step 1}
The first point to make is that the set of $|s|      \to \infty$ limits of $\theta$ on $\sigma$
are identical to the set of such
limits on any given $C_{j}$. Of course, this follows from the fact that the
$K = \mathbb{R}    \times  (S^1 \times S^{2})$ version of~\eqref{eq:3.16} is
valid here. In the case that no angle from any $(0,-,\ldots)$ element in $\hat{A}$  is the same as
that from a $(0,+,\ldots )$ element, this last conclusion rules out an
$\mathbb{R}$--invariant component
of $\sigma$ where $\theta$ differs from either 0 or $\pi$.

To rule out such components in any case, suppose for the sake of argument
that one were present. Denote the latter as $\mathbb{R}    \times     \gamma
_{*}$ where $\gamma_{*}$ is a Reeb orbit, and use $\theta
_{*}$ to denote the value of $\theta$ on $\gamma_{*}$.
Keep in mind that $\theta_{*}$ comes both from a $(0,+,\ldots )$ and a $(0,-,\ldots)$
element in $\hat{A}$.
Let $\varepsilon  > 0$ be very small and let $U  \subset S^{1}    \times S^{2}$ denote
the radius $\varepsilon $ tubular neighborhood of $\gamma
_{*}$. The argument that follows proves that $\theta$ on all large
 $j$ versions of $C_{j}$ has a critical value that differs by at most a uniform
multiple of $\varepsilon $ from $\theta_{*}$. Of course, this is
conclusion is nonsense for small $\varepsilon $ since \fullref{prop:3.8}
assumes that no angle from $\theta$ coincides with an angle from a
$(0,\ldots)$ element in $\hat{A}$.

To start the argument, note that \fullref{prop:3.7} dictates that there is an
irreducible component of each sufficiently large $j$ version of $C_{j}     \cap
 (\mathbb{R}    \times  U)$ with one rather specific property. To describe
this property, let $C_{j*}$ denote the component in question. Here is
the property: This $C_{j*}$ contains two ends of $C_{j}$ where the
$|s| \to \infty $ limit of $\theta$ is $\theta_{*}$, one on the convex side end and the
other on the concave side.

To see where these observations lead, remark that by virtue of the
assumptions in \fullref{prop:3.8}, the integer $\deg_{E}(d\theta )$
from~\eqref{eq:2.16} is zero on each convex side end of each $C_{j}$
where the $|s|\to \infty $ limit of $\theta$ is neither 0 nor $\pi$.
This is to say that the invariant $c_{E}$ that appears
in~\eqref{eq:1.8} is non-zero on any such end. This understood, it
follows that there exists some $s_{j0} > 0$ such that either $\theta >
\theta_{*}$ or else $\theta < \theta_{*}$ where $|s| > s_{j0}$ on any
convex side end in any large $j$--version of $C_{j*}$. Meanwhile, by
virtue of the fact that the integer $n$ that appears
in~\eqref{eq:2.17} is non-zero on any concave side end of $C_{j}$
where $\lim_{| s| \to \infty } \theta = \theta_{*}$, so $\theta$ takes
on the value $\theta_{*}$ at points in $C_{j*}$.

Now, suppose that $\nu \co  [0, \infty )     \to  C_{j*}$ is a smooth,
proper map with $\nu (0)$ on the $\theta =\theta_{*}$ locus and
with the image of $\nu $ on a convex side end in $C_{j*}$ at all
sufficiently large values of its domain. Associate to $\nu $ the value of
$\theta$ where $| \theta -\theta_{*}| $ is maximal.
Next, minimize the latter over all such $\nu $. By virtue of what was said
in the previous paragraph, the resulting min-max angle is not equal to
$\theta_{*}$. Moreover, the mountain pass lemma guarantees that the
latter is a critical value of $\theta$ on $C_{j}$ and so an angle in
$\theta$.

However, since $C_{j*}$ is connected, and $\theta$ takes value
$\theta_{*}$ on $C_{j*}$, the large $j$ versions of this
critical value can differ by at most some $j$--independent multiple of
$\varepsilon $ from $\theta_{*}$ since the paths in $C_{j*}$
stay in the tubular neighborhood $U$.

\substep{Step 2}
This step rules out the existence of either a $\theta =0$ or a $\theta
= \pi$ cylinder as an irreducible component of $\sigma$ by again
producing nonsense, a positive critical value of $\theta$ that is either
too small or too large to be in $\theta$. Suppose, for the sake of
argument, that the $\theta =0$ cylinder is an irreducible component of
$\sigma$. Only this case is discussed, as the $\theta =\pi$ argument
is identical save for some notation.

To see why no $\theta= 0$ cylinder can appear, fix some small,
positive $\varepsilon $ such that $100\varepsilon $ is less than the
smallest angle in $\theta$, and let $U \subset S^{1}    \times S^{2}$
denote the radius $\varepsilon $ tubular neighborhood of the $\theta
=0$ Reeb orbit.

In this case, \fullref{prop:3.7} provides an irreducible component of each
large $j$ version of $C_{j}     \cap  (\mathbb{R}    \times  U)$ with one very
particular property. To state the latter, let $C_{j*}$ again denote
the component. Here is the salient property: This $C_{j*}$ contains
two ends of $C_{j}$ where the $|s|      \to     \infty $ limit of
$\theta$ is 0, one on the convex side of $C_{j}$ and the other on the
concave side.

Granted the preceding, take $j$ very large. If $C_{j*}$ intersects the
$\theta =0$ cylinder, let $s_{j}$ denote the largest value that is taken
on by $s$ at any $\theta =0$ points in $C_{j*}$. If there are no such
points, set $s_{j}$ to equal $-\infty $. Set ${C_{j*}}'  \subset
C_{j*}$ to denote that portion where $s > s_{j}$.

Now, let $\tau \co  \mathbb{R}     \to {C_{j*}}'$ be any smooth map with
the following two properties: First, $\lim_{r \to \infty }    \tau^*s|
_{r}=\infty $ and $\lim_{r \to - \infty }    \tau^*s|_{r}
= s_{j}$. Second, $\lim_{r \to - \infty }    \tau^*\theta =0$.
Associate to $\tau $ the maximum value of $\tau^*\theta$, and then let
$\theta_{*j}$ denote the infimum of these maxima over the set of
all such maps from $\mathbb{R}$ to ${C_{j*}}'$. The mountain pass lemma now
guarantees that $\theta_{*j}$ is a non-extremal critical value of
$\theta$. As $\theta_{*j} > 0$ it is a point in $\theta$.
However, it is also the case that $\theta_{*j}$ is bounded by
$100\varepsilon $ if $\varepsilon $ is small, and this is nonsense as
$100\varepsilon $ is smaller than the minimal angle in $\theta$.

\substep{Step 3}
This step and the next complete the proof that $\Xi $ contains but a single
element. The argument begins by assuming, to the contrary, that $\Xi $ has
more than one component so as to derive some nonsense. In this case, the
nonsensical conclusion finds distinct critical values of $\theta$ on each
large $j$ version of $C_{j}$ that are closer than the minimal separation
between the points in $\theta$.

To start this derivation, reintroduce the set  $Y \subset     \sigma$ as
defined in \refsteps{3.D}{Step 4} of \fullref{sec:3d}. As remarked earlier, $Y$ is a finite set. Fix
a pair of points in $\Sigma - Y$ that lie on distinct irreducible
components and choose a $\theta$--preserving preimage of each point and so
obtain, for each large $j$, a pair, $z_{j}$ and ${z_{j}}'$, of points in
$C_{j}$. Associate to each path in $C_{j}$ between $z_{j}$ and ${z_{j}}'$ the
supremum of $|s| $ along the path, and let $r_{j}$ denote the
infimum of the resulting subset of $[0, \infty )$. As is explained next, the
sequence $\{r_{j}\}$ is bounded.

To see that there must be such a bound, suppose to the contrary that this
sequence diverges. Since neither $\{z_{j}\}$ nor $\{{z_{j}}'\}$
diverges, there is a path in each large $j$ version of $C_{j}$ between $z_{j}$
and ${z_{j}}'$ that avoids all convex side ends of $C_{j}$ where the
$|s| \to \infty $ limit of $\theta$ is neither 0 nor $\pi$.
Indeed, if each such end is defined so that the values of $|s| $
are everywhere greater than its value on either $z_{j}$ or ${z_{j}}'$, then
any path between these points must exit any such end that it enters.

Meanwhile, let $E  \subset C_{j}$ denote a concave side end where the
$|s|      \to     \infty $ limit of $\theta$ is neither 0 nor $\pi
$, and let $\theta_{*}$ denote said limit. A path between $z_{j}$
and ${z_{j}}'$ can also be chosen to avoid $E$ even if it must cross the $\theta
=\theta_{*}$ locus. This is a consequence of the following
observations: First, as is described in \fullref{sec:2g}, the intersection of the
$\theta =\theta_{*}$ locus with $E$ contains just two connected
components. These are either the two ends of a single component of the
$\theta =\theta_{*}$ locus, a properly embedded copy of
$\mathbb{R}$, or ends of distinct components. In the latter case, the other
ends of the corresponding copies of $\mathbb{R}$ are in other ends of $C_{j}$.
In any event, if there is no path from $z_{j}$ to ${z_{j}}'$ that crosses the
$\theta =\theta_{*}$ locus at reasonable values of $s$, it must
be the case that the infimum of $s$ on some component of the large $j$ version
of this locus is very large. Were such to occur, then the whole of this
component would lie very close to some end of $\sigma$, thus in $\mathbb{R}
\times U$ where $U$ is a small radius tubular neighborhood in $S^{1}
\times S^{2}$ of a $\theta =\theta_{*}$ Reeb orbit.
Since $\vartheta$ is generic, this means that the two components of the
intersection of the $\theta =\theta_{*}$ locus with $E$ lie on
the same component of this locus. This understood, let $(p, p')$ denote the
relatively prime pair of integers that $\theta_{*}$ defines via~\eqref{eq:1.7}.
As remarked previously, the 1--form $p d\varphi  - p'dt$ is exact on
 $U$. In particular, the integral of $p d\varphi  - p'dt$ from one end to the
other of all large $j$ versions of the $\theta =\theta_{*}$
component in question must then be zero. However, the latter integral can
not be zero because the pointwise restriction of $p d\varphi  - p'dt$ to
such a component is nowhere zero.

With the preceding understood, the only way that $\{r_{j}\}$ can diverge
is if all paths between the large $j$ versions of $z_{j}$ and ${z_{j}}'$ have
their large values of $|s| $ where $\theta$ is nearly 0 or
nearly $\pi$. However, such an event is ruled out using the mountain pass
lemma. Indeed, under the circumstances just described, this lemma would
provide a non-extremal critical point of $\theta$ on every large $j$ version
of $C_{j}$, one whose critical value was either too close to 0 or too close
to $\pi$ to come from $\theta$.

\substep{Step 4}
This part of the proof argues that $\Xi $ has but one element, and makes
use of the following auxiliary lemma:

\begin{lemma}\label{lem:3.10}

Let $Q  \equiv  (q, q')$ denote a pair of integers and let $\theta_{o} < \theta_{1}$
denote a pair of angles such that the function $\alpha_{Q}(\cdot )$ 
from~\eqref{eq:2.27}
is positive on
$[\theta_0, \theta_1]$.  Given $\varepsilon  > 0$, there exists
$\varepsilon '  \in (0, \varepsilon )$
such that the following is true: Let $J_{*}$ denote an admissible almost complex structure,
and let
$\phi$ denote an immersion of $(\theta_{o}, \theta_{1})    \times     \mathbb{R}/(2\pi
\mathbb{Z})$ into $\mathbb{R}    \times  (S^1 \times S^{2})$ that is $J_{*}$--pseudoholomorphic
and defined using a pair of functions, $(a_{*}, w_{*})$, by the rule
\begin{multline*}
\bigl(s = a_{*},\ t = q v + (1 - 3\cos^{2}\sigma ) w_{*}
\mod(2\pi ), \\
\theta =\sigma , \varphi  = q'v + \surd 6\cos\sigma
w_{*}  \mod(2\pi )\bigr)
\end{multline*}
Then, any two points in $(\theta_{o}, \theta_{1})    \times     \mathbb{R}/(2\pi \mathbb{Z})$
with $\phi$--image in the complement of any given, radius $\varepsilon $ ball in
$\mathbb{R}    \times  (S^1 \times S^{2})$ are the endpoints of a continuous path whose
$\phi$ image lies in the complement of the concentric, radius $\varepsilon'$ ball.
Moreover, if the two points have the same $\sigma$--coordinate, then such a path exists on which
$\sigma$ is constant, and if they have the same $v$--coordinate, such a path exists on which $v$
is constant.
\end{lemma}

\begin{proof}[Proof of \fullref{lem:3.10}]

This is the case simply because the variation
of $\theta$ in a radius $\varepsilon $ ball and the integral of $(1 -
3\cos^2\theta ) d\varphi -\surd 6 \cos\theta  dt$ on any constant
$\theta$ slice of such a ball are $O(\varepsilon )$ for small $\varepsilon$.

To resume the proof that $\Xi$ has a single element, the next point
to make is that if such isn't the case, then $\sigma$ contains a pair of
points that have the same $\theta$ value but lie in distinct, irreducible
components. Indeed, were there no such pair, then the subvarieties that
comprise $\sigma$ would be pairwise disjoint, and this last conclusion is
incompatible with the results of the previous steps. This understood, fix
some very small, but positive $\varepsilon $ and then fix points $p_{0}$ and
$p_{1}$ in $\sigma$ that lie on distinct irreducible components, have the
same $\theta$ value and have distance at least $2\varepsilon $ from any
point in the set $Y$. In particular, choose the $\theta$ value to be
different than any $|s|\to \infty $ limit of $\theta$ on
$\sigma$. When the index $j$ is very large, these two points are very close
to respective $\theta$--preserving preimages, $p_{0j}$ and $p_{1j}$ in
$C_{j}$. As the model curve for $C_{j}$ is connected, there exists a smooth
path, $\gamma_{j}     \subset C_{j}$, that connects $p_{0j}$ to
$p_{1j}$.
\end{proof}
Next, introduce, as in \fullref{sec:2g}, the $C_{j}$ version of the model curve,
$C_{0}$, and the corresponding locus $\Gamma      \subset C_{0}$. Suppose
now that $p_{0j}$ and $p_{1j}$ lie in the same component of the $C_{j}$
version of $C_{0}- \Gamma$. According to \fullref{lem:3.10}, the path
$\gamma_{j}$ can be chosen to be an arc on the constant $\theta$ locus
that avoids all points of $Y$ by some fixed, $j$--independent amount. Moreover,
as this $\theta$ value is not one of the $|s|      \to     \infty$
limits of $\theta$ on $\sigma$, the large $j$ versions of such a path must
lie everywhere very close to $\sigma$. Thus, any large $j$ version of $\gamma
_{j}$ has a well defined projection to give a path in $\sigma$ that
avoids all points in $Y$ and runs from $p_{0}$ to $p_{1}$. As this is
impossible, it must therefore be the case that $p_{0j}$ and $p_{1j}$ lie on
distinct components of the $C_{j}$ version of $C_{0}- \Gamma$.

To see that the latter case is also impossible, fix some very small but
positive $\varepsilon $ and use \fullref{lem:3.10}  to construct, for each large
index $j$, a path $\gamma_{j} \subset C_{j}$ that runs from $p_{0j}$
to $p_{1j}$ and is an end to end concatenation of two kinds of paths. Paths
of the first kind avoid the radius $\varepsilon $ balls about the points of
$Y$. Meanwhile, a path of the second kind is an arc on some constant $\theta$
locus that is contained in a radius $\varepsilon $ ball about a point in $Y$
and passes through a critical point of $\theta$ on $C_{j}$.

Remember now that each critical point on $C_{i}$ is non-degenerate and the
critical values of distinct critical points are distinct. Thus, if an arc
portion of $\gamma_{j}$ contains a critical point of $\theta$, there is
a unique component of the $C_{j}$ version of $C_{0}- \Gamma$ whose
image in $C_{j}$ contains this arc in its closure. This understood, it
follows from \fullref{lem:3.10}  that there exists an $j$--independent choice for its
constant $\varepsilon '$ such that the endpoints of this same arc are
connected by a path that also avoids the radius $\varepsilon '$ balls about
the points of $Y$. Thus, each large $j$ version of $C_{j}$ contains a second
path, $\upsilon_{j}$, that connects $p_{0j}$ to $p_{1j}$ and avoids the
radius $\varepsilon '$ balls about the points of $Y$. Of course, no such path
exists if $p_{0}$ and $p_{1}$ are in distinct, irreducible components of
$\sigma$.

\subsection{Part 3 of the proof of Proposition \ref{prop:3.8}}\label{sec:3g}

This part of the proof establishes that the one element in $\Xi $ has the
form $(S, 1)$ with $S \in \mathcal{M}[\Theta , \vartheta ]$. A such, the
graph $T_S$ from \fullref{sec:2g} can be labeled as a moduli space graph and it
is explained here why the latter is isomorphic to $T$.  The argument for all of
this is given below in nine steps.

\step{Step 1}
Let $S$ denote the subvariety from $\Xi'$s one pair. This step establishes
that every non-extremal critical value of $\theta$ on the model curve of $S$
is in the set $\theta$. In fact, the argument below proves that every
non-extremal critical value of $\theta$ on the model curve of $S$ maps to a
point in $S$ that is a limit point of a sequence whose $j'$th element is the
image of a critical point on the model curve of $C_{j}$.

To start the argument, let $S_{0}$ denote the model curve for $S$ and let $z
\in S_{0}$ be a non-extremal critical point of $\theta$ and let $\theta
_{*}$ denote the associated critical value. Also, set $k  \equiv
\deg(d\theta |_{z}) + 1$. Now, if $D  \subset S_{0}$ is a small
radius disk that is centered at $z$, then the $\theta =\theta_{*}$ locus in $S_{0}$
will intersect the boundary of $D$ transversely in $2k$
points. This can be seen, for example, using the local coordinate on $D$ that
appears in~\eqref{eq:2.16}. If the radius of $D$ is sufficiently small, then the
tautological map from $S_{0}$ to $\mathbb{R}    \times  (S^1 \times S^{2})$ will
embed the closure of $D$. Furthermore, the image of $z$ will be
the only point from $Y$ in the image of this closure. Take any such small
radius for $D$.

Let $B\subset     \mathbb{R}    \times  (S^{1}    \times S^{2})$ be a ball
that contains the closure of the image of $D$ with center at the image of $z$.
Introduce $\hat {t}$ and $\hat {\varphi }$ to denote the respective
anti-derivatives on $B$  of $dt$ and $d\varphi$ that vanish at $z$. Then, let
$\hat {v} \equiv  (1-3\cos^2\theta )    \hat {\varphi }-\surd 6
\cos\theta \hat {t}$. As can be seen using the parametrizations provided
by~\eqref{eq:2.25}, the pair $(\theta , \hat {v})$ pulls-back to $D$ as bona fide
coordinates on the complement of $z$. Likewise, they pull back to the model
curve of any large $j$ version of $C_{j}$ as coordinates on the complement of
the $\theta$--critical points in the inverse image of $B$.

Note for use below that $\hat {v}$ is annihilated by $\partial_{s}$ and
the Reeb vector field $\hat {\alpha }$. As a consequence, the values of both
$\theta$ and $\hat {v}$ are constant on the set of $\theta$--preserving
preimages of any given point in $S \cap  B$.

As can be seen using the parametrization from~\eqref{eq:2.16}, there exists an
embedded circle, $\gamma      \subset  D- z$, with the following
properties: First, $\gamma$ intersects the $\theta =\theta_{*}$
locus transversely in $2k$ points.
Second, $\theta$ has $k$ local maxima
and $k$ local minima on $\gamma$, and the values of $\theta$ at the local
maxima are identical, as are the values at the local minima.

When the index $j$ is very large, then $\gamma$ has a $\theta$--preserving
preimages in $C_{j}$. Let $\gamma_{j}$ denote the embedded loop in
${C_{j}}'$s model curve that maps to one of these preimages. Suppose, for the
sake of argument that there is no $\theta$--critical point in the component
in ${C_{j}}'$s model curve of the preimage of $B$  that contains $\gamma_{j}$.
Thus, $(\theta , \hat {v})$ provide local coordinates on an open set in
${C_{j}}'$s model curve that contains $\gamma_{j}$.

To continue, start at the local maximum of $\theta$ on $\gamma_{j}$ with
smallest $\hat {v}$ value and traverse $\gamma_{j}$ in the direction of
increasing $\hat {v}$. Since $\gamma_{j}$ is embedded, successive local
maxima must occur at successively larger values of $\hat {v}$. Such must also
be the case for the successive minima. This is impossible if $\gamma_{j}$
is embedded, for when the largest value of $\hat {v}$ is attained, the
traverse must then return to its start without crossing itself even as it
crosses values of $\hat {v}$ that are achieved at the local maxima and local
minima.

Thus, ${C_{j}}'$s model curve has a critical point that maps to $B$. Since the
radius of $D$, thus that of $B$, can be as small as desired,
so $\theta_{*}$ must come from $\theta$.

\step{Step 2}
As will now be explained, an argument much like the one just given proves
that each critical point of $\theta$ on $S_{0}$ is non-degenerate. A
modification is also used here to prove that $\Xi$ is $(S, 1)$ as
opposed to $(S, n)$ with $n > 1$ in the case that $\theta$ has a non-extremal
critical point on $S_{0}$. Finally, a slightly different modification proves
that every element of $\theta$ is a critical value of $\theta$ on $S_{0}$
in the case that $n = 1$.

Here is the proof that $n = 1$: Suppose first that $z \in S_{0}$ is a
critical point of $\theta$, and let $D  \subset S_{0}$ be as before, a
very small radius disk that contains $z$. Let $\gamma$ be as before. Suppose
that $j$ is large and that $\gamma$ has more than one $\theta$--preserving
preimage in the model curve of $C_{j}$. As is explained next, this
assumption leads to a contradiction. To start, remark that the $\theta$--preserving
preimage of $\gamma$ in $C_{j}$ is contractible. Indeed, this
can be seen using the parametrizations provided in \fullref{sec:2g} with the fact
that each such preimage maps to a small radius ball in $\mathbb{R}\times (S^1 \times S^{2})$.
Being contractible, each preimage of $\gamma$
is the boundary of an embedded disk in the model curve of $C_{j}$. In this
regard, the argument given in \refstep{Step 1} implies that each such disk contains
the $\theta =\theta_{*}$ critical point. This implies that the
preimages are nested. In particular, one such preimage, call it $\gamma
_{0}$, bounds a disk that contains all of the others. Let $D_{0}$ denote
the latter disk. Note that the function $\theta$ must take its maxima and
minima on the boundary of $D_{0}$ since its only critical point is in the
interior. Now, let $\gamma_{1} \ne \gamma_{o}$ denote a
hypothetical second $\theta$--preserving preimage of $\gamma$. Since
$\gamma_{1}     \subset  D_{0}$ and since the maximum value of $\theta$
on $\gamma_{1}$ is also the maximum value of $\theta$ on $\gamma
_{0}$, it follows that $\gamma_{1}$ must intersect $\gamma_{o}$.

Up now is the proof that the non-extremal critical points of $\theta$ on
$S_{0}$ are non-degenerate. To start, let $z$ denote the critical point in
question and let $\gamma$ again be as before. Reintroduce the integer $k$
from \refstep{Step 1}. Thus, $k = 1$ if and only if $z$ is a non-degenerate critical
point. In any event, a circumnavigation of $\gamma$ meets $k$ local maxima of
$\theta$ and $k$ local minima with all local maxima having the same
$\theta$--value,
and likewise all local minima. For large values of $j$,
let $\gamma
_{j}$ denote the $\theta$--preserving preimage of $\gamma$ in $C_{j}$. As
argued in the preceding paragraph, $\gamma_{j}$ bounds an embedded disk
in the model curve of $C_{j}$ that contains the $\theta =\theta_{*}$ critical point.
Let $D_{0}$ denote the latter

To continue, use the descriptions from \fullref{sec:2g} to find an embedded disk,
$U$, in the model curve of $C_{j}$ with the following properties: The disk
contains the $\theta =\theta_{*}$ critical point, it contains
$\gamma_{j}$, and its intersection with the $\theta =\theta_{*}$ locus consists of
four properly embedded, half open arcs that meet only
at their common endpoint, this being the $\theta =\theta_{*}$
critical point. These arcs are called `legs' of the $\theta =\theta
_{*}$ locus. In any event, here is one more requirement for $U$: The
complement of the $\theta =\theta_{*}$ locus in this disk
consists of four open sets, each an embedded disk.

Now, let $U'  \subset U$ denote any one of the four components of the
complement of the $\theta =\theta_{*}$ locus. The closure of $U'$
in $U$ intersects $\gamma_{j}$ in some number of properly embedded,
disjoint, closed arcs. In this regard, there are $2k + 2$ such arcs amongst
the four components, with at least one in each. Now, any given such arc in
$U'$ together with the stretch of the boundary of $U'$ between its two endpoints
defines a piecewise smooth circle in $U$ which is the boundary of the closure
of an embedded disk. If the interior of a second arc lies in this disk, then
the second arc is said to be nested with respect to the first. Of course, no
arc in $U'$ can be nested with respect to another by virtue of the fact that
all local maxima of $\theta$ on $\gamma_{j}$ have the same $\theta$--value,
as do all local minima.

There is one more point here to keep in mind: One and only one component arc
in $\gamma_{j}     \cap  U'$ `encircles' the $\theta =\theta_{*}$ critical point in
the following sense: This critical point is contained
in the segment of the boundary of the closure of $U'$ that lies between the
arc's two endpoints. Indeed, this is because the disk $D_{0}$ contains the
$\theta =\theta_{*}$ critical point. Note that the endpoints of
the latter arc lie on distinct legs of the $\theta =\theta_{*}$
locus. Such an endpoint on a given leg of the $\theta =\theta_{*}$ locus is nearer to the
$\theta =\theta_{*}$ critical point
then any other arc endpoint on the given leg. This is a consequence of the
fact that the arcs that comprise $\gamma_{j}     \cap  U'$ are not
nested.

Granted all of the above, each component of the complement of $\theta  =
\theta_{*}$ locus in $U$ has its one arc that encircles the $\theta
=\theta_{*}$ critical point, so there are four such
`encircling' arcs in all. Moreover, by virtue of what is said about
endpoints in the preceding paragraph,
these four arcs concatenate to define
a closed loop in $U$. This loop must thus be $\gamma$, and so $k = 1$ as
claimed.

What follows is the proof that every angle from $\theta$ is a critical
value of $\theta$ on $S_{0}$ in the case that the integer $n$ that is paired
with $S$ in $\theta$ is equal to 1. To start, let $\theta_{*}     \in
\theta$ and assume that $\theta_{*}$ is not a critical value of
$\theta$ on $S_{0}$ so as to derive some nonsense.

To start the derivation, choose a very small, but positive constant $\delta
$, subject to the following constraints: First, neither the $\theta  =
\theta_{*}+\delta $ nor $\theta_{*}-\delta $ loci in
$S_{0}$ should contain points of $Y$. Second, no critical values of $\theta$
on $S_{0}$ lie in the interval $[\theta_{*}-\delta , \theta
_{*}+\delta ]$.

The parametrizations described in \fullref{sec:2g} can now be used to first find a
positive constant, $\varepsilon $, and then construct for any sufficiently
large $j$, an embedded circle in ${C_{j}}'$s model curve with four special
properties: First, the circle bounds a disk in the model curve that contains
the $\theta =\theta_{*}$ critical point. Second, $\theta$ has
two local maxima on the circle, both where $\theta =\theta_{*}+\delta $; and $\theta$ has two
local minima on the circle both where
$\theta =\theta_{*}-\delta $. Finally, the tautological map
to $\mathbb{R}    \times  (S^1 \times S^{2})$ embeds the circle.
Finally, all points in the circle's image have $|s|  <
\frac{1}{\varepsilon }$ and distance $\varepsilon $ or more from
all points of the set $Y$.

Any given large $j$ version of such a circle has its $\theta$--preserving
projection as an embedded circle in $S$. (This is where the $n = 1$ assumption
is used. If $n > 1$, then this circle will not be embedded.) The latter has
its inverse image circle, $\gamma$, in $S_{0}$. Now, $\gamma$ is
null-homotopic since the integrals of $dt$ and $d\varphi$ over the original
circle in the model curve of $C_{j}$ are zero. Meanwhile, $\gamma$ is
embedded, it lies where $| \theta -\theta_{*}|  <
2\delta $, and $\theta$'s restriction to $\gamma$ has two local maxima,
both with the same $\theta$--value, and two local minima, also with the same
$\theta$--value. A repetition of one of \refstep{Step 1}'s arguments now proves that
there is no such loop. This nonsense thus proves that $\theta_{*}$
must be a critical value of $\theta$ on $S_{0}$.

\step{Step 3}
This step investigates the concave side ends of $S$ where the $|s| \to \infty $
limit of $\theta$ is neither 0 nor $\pi$. In
particular, this step establishes $\deg_{(\cdot )}(d\theta ) = 1$ for all
such ends. A variation of the latter argument is then used to prove that the
integer $n$ that $\Xi $ associates to $S$ is equal to 1 if $S$ has a concave side
end where $\lim_{| s| \to \infty }    \theta      \notin  \{0, \pi\}$.
A second variation of the argument proves that there is precisely one
such end of $S$ for every $(0,+,\ldots )$ element in $\hat{A}$.
With regards to this last point, keep in mind that the ends of $S$ define an
unordered set of points in $S^{1}$ via~\eqref{eq:2.19}'s map $\varpi_{ + }$, and
the latter is the same as $\vartheta$ up to multiplicity. This follows
directly from the $K = \mathbb{R}    \times  (S^1 \times S^{2})$
version of~\eqref{eq:3.16}.

To start the analysis, let $E  \subset S$ denote a concave side end where
the $|s|      \to     \infty $ limit of $\theta$ is neither 0 nor
$\pi$. As will now be explained, $\deg_{E}(d\theta ) = 1$. To prove that
such is the case, suppose $\deg_{E}(d\theta ) > 1$ so as to derive
nonsense. Thus, let $k > 1$ denote $\deg_{E}(d\theta )$. Let $\theta_{*}$ now denote the
$|s| \to \infty $ limit of $\theta$
on $E$. As can be seen using~\eqref{eq:2.17}, there exists $s_{0}     \ge 0$  such that
the $\theta =\theta_{*}$ locus intersects the $s \ge s_{0}$
portion of $E$ as a disjoint union of $2k$ properly embedded copies of $[s_{0},
\infty )$, with the diffeomorphism given by the function $s$ itself.
Moreover, it follows from~\eqref{eq:2.17} that if $\delta> 0$  and is
sufficiently small, there exists an embedded circle, $\nu      \subset  E$,
with the following properties: First, this loop $\nu $ has transversal
intersections with the $\theta =\theta_{*}$ locus. Second,
$| \theta -\theta_{*}|  < \delta $ on $\nu $.
Third, $\theta$'s restriction to $\nu $ has precisely $k$ local maxima and $k$
local minima. Finally, all local maxima have the same $\theta$ value, this
greater than $\theta_{*}$, and all local minima have the same
$\theta$ value, this less than $\theta_{*}$.

The circle $\nu $ has $\theta$--preserving preimages in every large $j$
version of $C_{j}$. Let $\nu_{j}$ denote one of the latter. Because the
variation of $\theta$ on $\nu_{j}$ is small, the $\theta  < \theta
_{*}$ portion of $\nu_{j}$ is contained in a single component of
the $C_{j}$ version of $C_{0}-\Gamma$. Call this component $K$. For the
same reason, the $\theta  > \theta_{*}$ part of $v_{j}$ is
entirely in a single component also. Use $K'$ for the latter. Note that the
closure of the $| \theta -\theta_{*}|  < 2\delta $
portion of $K  \cup  K'$ is diffeomorphic to a closed cylinder with some
number of punctures, all on the $\theta =\theta_{*}$ circle.

To proceed, now view the loop $\nu_{j}$ sitting in this abstract
cylinder. Here, it sits as an embedded, null-homotopic loop. To explain,
introduce the relatively prime pair of integers, $(p, p')$, that $\theta
_{*}$ defines via~\eqref{eq:1.7}. The 1--form $pd\varphi  - p'dt$ is exact
near $E$ and so integrates to zero around $\nu_{j}$. Meanwhile, this form
has non-zero integral around any essential loop in the unpunctured cylinder
since this form restricts as a positive form on the $\theta =\theta
_{*}$ locus in $C_{j}$. To summarize: As a loop in the abstract
cylinder, $\nu_{j}$ is embedded, it is null homotopic, it intersects the
$\theta =\theta_{*}$ locus in $2k$ points, it has $k$ local maxima
all with the same value of $\theta$, this greater than $\theta
_{*}$, and $k$ local minima, all with the same value of $\theta$,
this less than $\theta_{*}$. Granted all of this, the argument from
\refstep{Step 1}'s second to last paragraph can be borrowed with only minor cosmetic
changes to obtain a contradiction unless $k = 1$.

Given that there exists a concave side end $E  \subset S$ where $\lim_{|
s| \to \infty }    \theta      \notin  \{0, \pi \}$, what follows
proves that the integer $n$ that $\Xi $ associates to $S$ is equal to 1. For
this purpose, construct a loop, $\nu $, as just described. The loop $\nu $
again has $\theta$--preserving preimages in every large $j$ version of
$C_{j}$. If $n > 1$ and if there are less than $n$ such preimages, then one of
them has the following properties: The restriction of $\theta$ to the loop
has more than one local maximum, and more than one local minimum. Moreover,
all local maxima have the same $\theta$ value, and all local minima have
the same $\theta$ value. Finally, the form $pd\varphi  - p'dt$ integrates
to zero over this loop. The argument given in the preceding paragraph shows
that this is impossible. Thus, there are $n$ disjoint, $\theta$--preserving
preimages of $\nu $ in every large $j$ version of $C_{j}$, each mapping via
the $\theta$--preserving projection to $\nu $ as a diffeomorphism.

To see that $n = 1$, first note that all $\theta$--preserving preimages of
$\nu $ must lie in the closure of the union of the same two components of
the $C_{j}$ version of $C_{0}- \Gamma$. Indeed, such is the case
because the $K = \mathbb{R}    \times  (S^1 \times S^{2})$ version of~\eqref{eq:3.16}
 holds and because the set $\vartheta$ is both $j$--independent and has
$N_{ + }$ elements.

To continue, let $K$ and $K'$ again denote the two relevant components of the
$C_{j}$ version of $C_{0}- \Gamma$ and once again view the $\theta
     \in  [\theta_{*}-2\delta , \theta_{*}+2\delta
]$ part of the union of the closures of $K$ and $K'$ as a closed,
multi-punctured cylinder. Viewed in this cylinder, any $\theta$--preserving
preimage of $\nu $ must encircle one or more of the punctures. Were this
otherwise, then the preimage would be null-homotopic in the $\theta      \in
 (\theta_{*}-2\delta , \theta_{*}+2\delta )$ part
of $C_{j}$ and thus the same part of $\mathbb{R} \times (S^1 \times S^{2})$.
But such a loop represents the same homotopy class as $\nu $, a
non-zero class in the $\theta \in  (0, \pi )$ portion of $\mathbb{R}
\times  (S^{1}    \times S^{2})$.

With the preceding understood, pick a point on the $\theta =\theta
_{*}-2\delta $ boundary circle of the punctured cylinder, and then
draw a smooth path from this chosen point to any given puncture. A loop in
the interior of the punctured cylinder that encircles the given puncture has
non-zero intersection number with this path. Draw a specific path as
follows: Let $s_{0}$ denote the maximum of the function $s$ on $\nu $. Start
the path at the puncture and draw it to decrease $s$ until the latter is equal
to $2s_{0}$. Call this path $\gamma_{o}$. Note that when $j$ is very large,
the whole of this path is in a very small radius tubular neighborhood of a
single end of S, and thus it is far from any other end of $C_{j}$. This
follows because the $K = \mathbb{R}    \times  (S^1 \times S^{2})$
version of~\eqref{eq:3.16} holds and because the set $\vartheta$ is both
 $j$--independent and has $N_{ + }$ elements. When $j$ is very large, the $s =
2s_{0}$ endpoint of $\gamma_{o}$ will be very close to one
particular end, $E'  \subset  S$, and so its $\theta$--preserving projection
to $E'$ is well defined. Draw an $s$--decreasing path in $E'$ from the latter point
to a point where both $s$ and $\theta$ are less than their minimal values on
$\nu $. If $E = E'$, then have this path intersect $\nu $ transversely at a
single point. In any event, call this path $\gamma_{E'}$. When $j$ is
large, one of the $\theta$--preserving preimages of $\gamma_{E'}$
attaches to the $s = 2s_{0}$ endpoint of $\gamma_{o}$ This understood,
continue the concatenation of the latter preimage with $\gamma_{o}$ as a
path from where $\theta$ is less than its minimum on $\nu $ to the chosen
point on the $\theta =\theta_{*}-2\delta $ boundary circle of
the cylinder. Make $\theta$ decrease monotonically on this continuation.

Only one of the paths just described will intersect any $\theta$--preserving
preimage of $\nu $. This is the path that at large $s$ is very close to $E$.
Moreover, only one $\theta$--preserving preimage of $\gamma_{E}$ can
intersect any given $\theta$ preserving preimage of $\nu $. Thus, only one
$\theta$--preserving preimage of $\nu $ encircles a puncture, and so $n = 1$.

Here is why $S$ does not have a pair of ends whose constant $s$ slices limit as
 $s \to     \infty $ to the same $\theta =\theta_{*}$ closed
Reeb orbit in $S^{1}    \times S^{2}$: Were this otherwise, let $E$ and $E'$
denote the two ends involved, and define loops $\nu \subset E$ and $\nu '
 \subset  E'$ as just described. Each has a $\theta$--preserving
preimage in the same multi-punctured cylinder in $C_{j}$. However, only one
will intersect a path as described above from a puncture to the $\theta  =
\theta_{*}-2\delta $ boundary circle of the multi-punctured
cylinder.

\step{Step 4}
This step proves that $\deg_{(\cdot )}(d\theta )=0$ for each convex
side end of $S$ where the $|s|      \to     \infty $ limit of $\theta
$ is neither 0 nor $\pi$. A variation of the latter argument also proves
that there is only one such end for each $(0,-,\ldots)$
element in $\hat{A}$.

To start, suppose that $E  \subset S$ is a convex side end of the sort in
question, and suppose that $\deg_{E}(d\theta )$ is non-zero. Let $\theta
_{*}$ denote the $|s|      \to     \infty $ limit of
$\theta$ on $E$. Now~\eqref{eq:2.17} guarantees that given any $\delta  > 0$, there
exists $R$ with the following significance: First, the function $|
\theta -\theta_{*}| $ is less than $\delta $ on the
$|s|  > R$ part of $E$. Second, $\theta$ takes values both greater
than $\theta_{*}$ and less than $\theta_{*}$ on any
constant $|s|      \ge  R$ slice of $E$.

To see that such an event is nonsensical, take $\delta $ very
small. Granted this, take $j$ very large, and there is but one
component of the $C_{j}$ version of $C_{0}- \Gamma$ that maps very
close to the $s \ge R$ portion of $E$. Indeed, such is the case
because small $\delta $ guarantees that there are no components of the
$C_{j}$ version of $\Gamma$ where $\theta \in (\theta_{*},
\theta_{*}+2\delta )$.  Use $K$ to denote the component of the $C_{j}$
version of $C_{0}- \Gamma$ in question.  This component must contain a
convex side end of $C_{j}$ where the $|s| \to \infty $ limit of
$\theta$ is $\theta_{*}$. Since $\deg_{(\cdot )}(d\theta )=0$ on such
an end, the description of $K$ offered in \fullref{sec:2g} finds that
the function $\theta$ must be either strictly less than $\theta_{*}$
or strictly greater than $\theta_{*}$ on $K$. Of course, this is
impossible when $j$ is large. Indeed, when $j$ is large, then $K$ maps
very close to the $s \ge R$ portion of $E$ and so to points where
$\theta$ is greater than and to points where $\theta$ is less than
$\theta_{*}$.

To see that there is but one end of $S$ for each $(0,-,\ldots)$ element in $\hat{A}$,
suppose for the sake of argument that there were two, $E$
and $E'$. Let $\theta_{*}$ denote the common $|s|      \to
    \infty $ limit of $\theta$ on $E$ and $E'$. Fix some $\delta  > 0$ and very
small, and then let $\nu      \subset E$ and $\nu '  \subset E'$ denote
the respective loci where $| \theta -\theta_{*}|  =
\delta $. By virtue of~\eqref{eq:2.17},
these loci are embedded circles. Each has
its $\theta$--preserving preimage in any sufficiently large $j$ version of
$C_{j}$. An argument from the preceding paragraph can be readily modified to
prove that these preimages must lie in the same component of the $C_{j}$
version of $C_{0}- \Gamma$. As such, they must coincide.

Note that the same argument proves the following: Suppose that the integer $n$
that $\Xi $ pairs with $S$ is greater than 1, and suppose that $E  \subset S$
is a convex side end on which the $|s|      \to     \infty $ limit
of $\theta$ is neither 0 nor $\pi$. Let $\nu      \subset E$ denote an
embedded circle in $E$ that is homologically non-trivial. Then there is but
one $\theta$--preserving preimage of $\nu $ in every sufficiently large $j$
version of $C_{j}$; and the latter maps back to $\nu $ as an $n$ to 1 covering
map.

\step{Step 5}
This step investigates the nature of the ends of $S$ where the $|s|\to\infty $
limit of $\theta$ is either 0 or $\pi$. In
particular, it is proved here that such ends are naturally in 1--1
correspondence with the set of $\{\pm 1,\ldots\}$
elements in $\hat{A}$  unless $S$ is either a disk or a cylinder and the integer
that $\Xi $ pairs with $S$ is greater than 1.

The discussion starts with the following claim:
\qtaubes{3.20}
\textsl{%
There exists $R  \ge 0$  such that when $j$ is large, the intersections of $C_{j}$
with the $\theta      \in  \{0, \pi \}$ cylinders occur where $|s|      \le  R$.}
\endqtaubes
This claim is proved momentarily. Note first that it has the following
corollary: In the case that the integer that $\Xi $ pairs with $S$ is 1, the
respective intersection numbers between $S$ and the $\theta =0$ and $\theta
=\pi$ cylinders are those prescribed by $c_+ $ and $c_- $.
Indeed, this corollary follows using the parametrizations from~\eqref{eq:2.25} given
\fullref{prop:3.7} and given that $\theta$ contains all non-extremal critical
values of $\theta$ on the model curve of $S$.

To see why~\eqreft3{20} holds, fix a very small but positive number, $\delta $,
chosen so that there are no elements of $\theta$ that lie where
$\theta  < 2\delta $ and none where $\pi -\theta  < 2\delta $. Also,
choose $\delta $ so that no angle as defined via~\eqref{eq:1.7} from the integer pair
of any $(0,\ldots)$ element in $\hat{A}$  lies either between
$2\delta $ and 0 or between $\pi -2\delta $ and $\pi$.

Having chosen $\delta $, then choose $R$ so that the $|s|\ge\frac{1}{2}R$ part of
$S$ is contained in the ends of $S$. Moreover,
choose $R$ so that the variation of $\theta$ on the $|s|\ge \frac{1}{2}R$ part of any
end of $S$ is very much smaller than
$\delta $. Granted this, suppose that $j_{0}$ is such that the $|s| =\frac{1}{2}R$
locus in any end of $S$ has its full set
of $\theta$--preserving preimages in all $j \ge j_{0}$ versions of
$C_{j}$.

Now suppose, for the sake of argument, that some $j \ge j_{0}$ version
of $C_{j}$ intersects the $\theta =0$ locus at a point where $|s|  > R$.
This point is the image of a point in the closure of a
particular component of the $C_{j}$ version of $C_{0}- \Gamma$.
Let $K$ denote the latter. The 1--form $dt$ must pull back to $K$ as an exact form.
However, as indicated in the preceding paragraph, there is some end of $S$
where $\lim_{| s| \to \infty }    \theta =0$ whose $|s| = \frac{1}{2}R$ slice has a
$\theta$--preserving preimage in $K$.
Since the 1--form $dt$ is not exact on such a slice, so its pull-back to $K$ can
not be exact. Thus, there is no such $K$.

A very minor modification of the argument just given also proves the
following: Any given integer pair that appears in some $(1,+,\ldots)$ element in $\hat{A}$
is $n$ times that of an integer pair that is
defined by a concave side end of $S$ where the $|s|      \to \infty $ limit of $\theta$ is 0,
and vice-versa. Here, $n$ is the integer
that $\Xi $ pairs with $S$. Of course, the analogous assertion holds for
$(1,-,\ldots)$ elements and convex side ends where the
$\lim_{| s| \to \infty }    \theta =0$. Likewise, a similar
assertion holds for $(-1,\ldots)$ elements in $\hat{A}$  and ends
of $S$ where the $|s|      \to     \infty $ limit of $\theta$ is
$\pi$.

Note that this correspondence assigns precisely one end of $S$ to each end of
every large $j$ version of $C_{j}$. Indeed, if not then there exists some very
small $\varepsilon  > 0$ and two disjoint $\theta =\varepsilon $ circles
in $S$, or two disjoint $\theta =\pi -\varepsilon $ circles in $S$ whose
$\theta$--preserving preimages in all sufficiently large $j$ versions of
$C_{j}$ lie in the same component of the $C_{j}$ version of $C_{0}- \Gamma$.
This means that the preimages coincide, an impossibility when $j$ is large.

The next point to make is that this correspondence is a 1--1 correspondence
unless $S$ is either a disk or a cylinder of a certain sort. The assertion
that the correspondence is 1--1 follows from the following claim: Two ends of
any large $j$--version of $C_{j}$ can not both lie very close to the same end
of $S$. To see why the latter claim holds, remark first that the occurrence of
two ends very close to the same end of $S$ can occur only in the case that the
integer $n$ is greater than 1. This is because distinct ends of $C_{j}$ that
are convex or have $|s|\to\infty $ limit of $\theta$
either 0 or $\pi$ lie in distinct components of the $C_{j}$ version of
$C_{0}- \Gamma$. Now, if $n  \ne 1$, then it follows from~\eqref{eq:2.18}
and from what has been said in previous steps that $S$ has at most two ends,
and neither is a concave side end unless the corresponding $|s|  \to     \infty $
limit of $\theta$ is 0 or $\pi$. In particular, $S$ is
either a cylinder or a disk.

By the way, if $S$ is a disk, then~\eqref{eq:2.18} requires that $S$ have a single,
transversal intersection with one but not both of the $\theta =0$ or
$\theta =\pi$ cylinders. In this case, the integral of $dt$ over any
constant $\theta$ circle in $S$ must be zero, and so the large $|s| $ slices of $S$
converge in $S^{1}    \times S^{2}$ to one of the
two $\cos^2\theta =\frac{1}{3}$ Reeb orbits. In particular,
the sign of $\cos\theta$ on this orbit is the same as its sign at the zero
of $\sin\theta$. In any event, any large and constant $|s| $
slice of $S$ is isotopic to the $|s|      \to     \infty $ limit Reeb
orbit.

On the other hand, if $S$ is a cylinder, then~\eqref{eq:2.18} requires that it be
disjoint from both the $\theta =0$ and $\theta =\pi$ cylinders. In
this case, it must have at least one convex side end where the $|s|\to\infty $
limit of $\theta$ avoids 0 and $\pi$. Indeed,
if not, then the fact that the restriction of $\theta$ to $S$ has no extremal
critical values in $(0, \pi )$ would require the $|s|      \to $
$\infty $ limit of $\theta$ to be 0 on one end and $\pi$ on the other.
Were this the case, the whole of $S$ could be parametrized as in~\eqref{eq:2.25} by
$(0, \pi )    \times \mathbb{R}/(2\pi \mathbb{Z})$. However, this is
impossible because the corresponding function $\alpha_{Q}$ as defined in~\eqref{eq:2.27}
 would then vanish at some value of $\sigma$ that is realized
on the parametrizing cylinder.

\step{Step 6}
This step proves that the integer $n$ that $\Xi $ pairs with $S$ is equal to 1
in the cases that $S$ is a disk as described in the preceding step. The cases
where $S$ is a cylinder are discussed in \refstep{Step 7} and \refstep{Step 8}.

The argument has five parts, with the first four constituting a digression
to set the stage. In what follows, keep in mind that $S$ is a
$J'$--pseudoholomorphic disc, an element in the $J'$--version of the moduli space
$\mathcal{M}_{\hat{A} '}$ where $\hat{A}'$  has only the 4--tuple $(0,-,(0, 1))$. In
particular, $S$ intersects the $\theta =0$ cylinder transversely in a single
point, there are no non-extremal critical points of $\theta$ on $S$, and the
$|s|      \to     \infty $ limit of the constant $|s| $
slices on $S$ converge to a Reeb orbit where $\cos\theta =\sqrt{\unfrac{1}{3}} $.

\substep{Part 1}
For $j$ large, \fullref{prop:3.4} provides
$C^{\infty }$--small deformations of $S$ that results in a
$J_{j}$--pseudoholomorphic subvariety in the $J_{j}$--version of $\mathcal{M}_{\hat{A} '}$.
Any such subvariety intersects the $\theta =0$ cylinder
transversely, also at a single point and there are no non-extremal critical
points of $\theta$ on any such $S_{j}$.

\fullref{prop:2.12}  and \fullref{prop:2.13}  in
conjunction with \fullref{prop:3.4}
provide two different parametrizations of the subsets of the respect $J'$  and,
for large $j$, $J_{j}$ versions of $\mathcal{M}_{\hat{A}}$ whose subvarieties are
everywhere close to $S$ in $\mathbb{R}    \times  (S^1 \times S^{2})$.
The first parametrizes the constituent subvarieties by the point where they
intersect the $\theta =0$ cylinder, and the other by their large $|s| $ asymptotics on
their one end.

To be more explicit about the parametrization by points in the $\theta =0$
cylinder, let $z_{0}$ denote the point where $S$ intersects this cylinder.
Then respective neighborhoods that contain the subvarieties that are
pointwise near $S$ in the $J'$  version and, for large $j$, in the $J_{j}$ version
of $\mathcal{M}_{\hat{A} '}$ can be parametrized as follows: The coordinates for
the parametrization are the points that lie in a $j$--independent disk centered
at $z_{0}$ in the $\theta =0$ cylinder. The parametrization provides a 1--1
correspondence that assigns a subvariety in the relevant moduli space to the
point where it intersects the $\theta =0$ cylinder.

\substep{Part 2}
The second parametrization uses the large
$|s| $ asymptotics on the subvariety. To be more precise, first
note that the $\cos\theta =\sqrt {\unfrac{1}{3}} $ Reeb orbits
are parametrized by the constant value on the Reeb orbit of the coordinate,
 $t$,  on the $S^{1}$ factor in $S^{1}    \times S^{2}$. Let $\tau_{0}$
denote the value for the orbit that is obtained as the $|s|
\to     \infty $ limit of the constant $|s| $ slices of $S$.
Meanwhile, let $c_{0}$ denote the constant that appears in the version of~\eqref{eq:1.8}
 that is relevant for the one end of $S$. In this regard, note that
$c_{0} < 0$.

According to the aforementioned propositions, there exists some $\delta  > 0$
such that respective neighborhoods of subvarieties that are pointwise near
 $S$ in the $J'$  version and, for large $j$, in the $J_{j}$ version of
 $\mathcal{M}_{\hat{A} '}$ can be parametrized as follows: The parametrization uses 
those $(c, \tau )     \in  (-\infty , 0) \times     \mathbb{R}/(2\pi \mathbb{Z})$
where $| c - c_{0}| ^{2}+| \tau -\tau
_{0}| ^{2} < \delta ^{2}$. In particular, the parametrization
provides a 1--1 correspondence that assigns a subvariety to a pair $(c, \tau)$
when the subvariety's one end provides $c_{E} = c$ in~\eqref{eq:1.8}, while the
large $|s| $ slices of the subvariety converge as $|s|
     \to     \infty $ to the Reeb orbit where $\cos\theta =\sqrt
{\unfrac{1}{3}} $ and $t = \tau $. In this regard, note that the
assigned value for $\tau $ is simply the value on the end in question of
\fullref{sec:1}'s parameter ${\iota}_{(\cdot )}$.

\substep{Part 3}
Both of the preceding parametrizations are
compatible with the notion of convergence as given in \fullref{prop:3.7}.
Indeed, suppose that either a point in the $\theta =0$ cylinder is fixed
in the parametrizing disk about $z_{0}$, or else a point $(c, \tau )$ is
fixed with distance less than $\delta $ from $(c_{0}, \tau_{0})$. Now,
in either case, let $S_{j}$ denote the subvariety in the $J_{j}$ version of
$\mathcal{M}_{\hat{A}}$ that is parametrized by the given point. Meanwhile, let
$S'$ denote the corresponding $J'$--version. Then the sequence $\{S_{j}\}$
converges in the sense of \fullref{prop:3.7} with $(S', 1)$ in the role of $\Xi$,
and with~\eqref{eq:3.16} valid for $K = \mathbb{R}    \times  (S^1 \times S^{2})$.

\substep{Part 4}
The part of the argument is summarized by the following lemma:

\begin{lemma} \label{lem:3.11}

There exists $\varepsilon  > 0$ with the following significance: If $j$ is large and if
 $(c, \tau )$ has distance less than
$\delta $ from $(c_{0}, \tau_{0})$ but $| \tau -\tau_{0}|
> \frac{1}{4}\delta $, then all points in the $(c, \tau )$ subvariety from the
$J_{j}$--version of $\mathcal{M}_{\hat{A} '}$ have distance at least $\varepsilon $ from $S$.
\end{lemma}

The digression ends with the following proof.

\begin{proof}[Proof of \fullref{lem:3.11}]
Granted the contents in \refstep{Part 2} of this
digression, it is sufficient to prove the following: If $(c, \tau )$ is
$\delta$ close to $(c_{0}, \tau_{0})$ and if $\tau      \ne
\tau_{0}$, then the $(c, \tau )$ subvariety in the $J'$--version of
$\mathcal{M}_{\hat{A} '}$ is disjoint from $S$. To establish this claim, let $S'$ denote
the $(c, \tau )$ subvariety. If $S'$ intersects $S$, then it has positive
intersection number with $S$. Thus, any deformation of $S'$ has positive
intersection number with $S$ if the intersections with $S$ along the way remain
in a fixed, compact subset of $\mathbb{R}    \times  (S^1 \times S^{2})$.
This understood, push $S'$ along the vector field $\partial_{s}$.
In this regard, there is a compact subset of $\mathbb{R}    \times (S^1
\times S^{2})$ that contains all the putative intersections between this
deformation and $S$. Moreover, the same subset suffices no matter how far $S'$
is pushed along $\partial_{s}$. Indeed, such is the case because $\tau
 \ne     \tau_{0}$ and $s$ is bounded from above on $S$.

If pushed far enough, the resulting subvariety at values of $s$ that are
achieved on $S$ is very close to the pseudoholomorphic cylinder that is
defined by the $((0, 1), \tau )$ Reeb orbit. As $\theta$ has no local
maximum on $S$, it follows that $S$ must be disjoint from this Reeb orbit, and
so disjoint from a large push of $S'$ along $\partial_{s}$.

\substep{Part 5}
With the preliminaries over, what follows
are the final arguments to prove that the integer paired to $S$ by $\Xi $ is
1. To start, let $\varepsilon $ denote the constant from \fullref{lem:3.11}. When $j$
is large, then every point in $C_{j}$ has distance less than $\frac{1}{100}\varepsilon $ from $S$.
Also, when $j$ is large, the
parametrizations from \refstep{Part 1} and \refstep{Part 2} above provide some pair $(c, \tau )$
that is very close to $(c_{0}, \tau_{0})$ and have the following
significance: This pair parametrizes a disk in the $J_{j}$ version of $\mathcal{M}_{\hat{A} '}$
that intersects $C_{j}$ at a point on the $\theta =0$
cylinder. Let $S_{j}$ denote the latter disk. Now, the $|s|
\to     \infty $ limit of the constant $|s| $ slices of $C_{j}$
converge as a multiple cover to some $\cos\theta =\frac{1}{{\surd 3}}$
Reeb orbit with parameter $\tau_{j}$ very close to $\tau
_{0}$. If $\tau_{j }=\tau $ and if the integer paired to $S$ in $\Xi $
is greater than 1, then use Propositions~\ref{prop:2.12}  and~\ref{prop:2.13}  to find a new
subvariety, ${C_{j}}'$, from the $J_{j}$--version of $\mathcal{M}_{\hat{A}}$ that
intersects $S_{j}$ on the $\theta =0$ cylinder, has asymptotic parameter
${\tau_{j}}'  \ne     \tau $, and lies entirely in the radius
$\frac{1}{100}\varepsilon $ tubular neighborhood of $S$. In
particular, $S_{j}$ and ${C_{j}}'$ have positive intersection number. Now,
move $S_{j}$ in its moduli space to some ${S_{j}}'$ whose corresponding
parameters $(c', \tau ')$ obey $| \tau ' - \tau_{0}|  >
\frac{1}{2}\delta $. Do so by a path $r \to  (c, \tau (r))$
where $\tau (0) = \tau $ and where $| \tau (r) - {\tau_{j}}'| $ is strictly increasing. The intersection number between
${S_{j}}'$ and ${C_{j}}'$ is thus the same as that between $S_{j}$ and ${C_{j}}'$.
However, according to \fullref{lem:3.11}, the subvarieties ${S_{j}}'$ and ${C_{j}}'$ are
disjoint. This is a contradiction and so $\Xi $ must pair $S$ with 1.
\end{proof}

\step{Step 7}
This step and \refstep{Step 8} prove that the integer paired with $S$ by $\Xi $ is 1 in
the case that $S$ is a cylinder with no concave side ends where $\lim_{|
s| \to \infty }    \theta$ is not in $(0, \pi )$ and at least one convex
side end where the analogous limit is in $(0, \pi )$. In this regard, the
arguments are, with minor modifications, a reprise of those given previously
in \refstep{Step 6}.

This step considers the case that neither 0 nor $\pi$ is the $|s|\to\infty $
limits of $\theta$ on $S$. In this case, $S$ is a
subvariety in the $J'$  version of the moduli space $\mathcal{M}_{\hat{A} '}$ where
$\hat{A}'$ consists of two elements, $(0,-, (p, p'))$ and $(0,-, (-p, -p'))$. Here,
$p$ and $p'$ are relatively prime integers and such that
$\frac{p' }{p}\ge\sqrt{\unfrac{3}{2}} $.
No generality is lost by
taking $p$ and $p'$ to be positive. The argument here has three parts.

\substepp{Part 1}
As in the case considered by \refstep{Step 6}, there
are two parametrizations for neighborhoods that contain the subvarieties that
are pointwise near $S$ from the $J'$  version and, for large $j$, from the $J_{j}$
version of $\mathcal{M}_{\hat{A} '}$. In this case, a description of these
parametrizations requires the use of the respective pairs $(c_{ + 0}, \tau
_{ + 0})     \in  (0, \infty )    \times     \mathbb{R}/(2\pi \mathbb{Z})$
and $(c_{ - 0}, \tau_{ - 0})     \in  (-\infty , 0) \times     \mathbb{R}/(2\pi \mathbb{Z})$
to parametrize the asymptotics of the $(p, p')$ and
$(-p, -p')$ ends of $S$. Here, the parameters $\tau_{\pm 0}$ parametrize the
respective $(p, p')$ and $(-p, -p')$ Reeb orbits that are obtained as $|s| \to \infty $
limits of the constant $|s| $ slices
of $S$. Meanwhile, $c_{\pm 0}$ are the respective versions of the parameter
$c_{E}$ from~\eqref{eq:1.8}.

It then follows from Propositions~\ref{prop:2.13} and~\ref{prop:3.4} that there exists
$\delta  > 0$ such that one of the parametrizations in question is by the subset in $(0,
\infty )    \times     \mathbb{R}/(2\pi \mathbb{Z})$ of points $(c, \tau )$
where $| c - c_{ + 0}| ^{2}+| \tau -\tau_{ +
0}| ^{2} < \delta $, and the other is by the subset in
$(-\infty, 0) \times\mathbb{R}/(2\pi\mathbb{Z})$ that consists of the points
$(c, \tau )$ where $| c -c_{ - 0}| ^{2}+| \tau  -\tau_{ - 0}| ^{2} < \delta $.
The first parametrization
assigns a subvariety to the pair $(c, \tau )$ that describe the asymptotics
of its $(p, p')$ end, and the second those of its $(-p, -p')$ end. Thus, $\tau $
is again the value of \fullref{sec:1}'s parameter ${\iota}_{(\cdot )}$ on the end in
question. Call the first parametrization the `plus' parametrization and call
the second one the `minus' parametrization.

\substepp{Part 2}
The following observation is the analog of
that made in \fullref{lem:3.11}: There exists $\varepsilon  > 0$ with the following
significance: If $(c, \tau )$ has distance less than $\delta $ from $(c_{ +
0}, \tau_{ + 0})$ and if $| \tau -\tau_{ + 0}| $ is
greater than $\frac{1}{4}\delta $, then the large $j$ version of
the subvariety that is parametrized by $(c, \tau )$ via the plus
parametrization lies outside the radius $\varepsilon $ tubular neighborhood
of $S$. Of course, the analogous statement holds for the minus parametrization
when $(c, \tau )$ is $\delta $ close to $(c_{ - 0}, \tau_{ - 0})$ and
$| \tau -\tau_{ - 0}|  > \frac{1}{4}\delta$.

To prove the assertion in this case, it is enough to consider, as in the
proof of \fullref{lem:3.11}, the intersections between $S$ and the $J'$  subvariety, $S'$,
that is parametrized by $(c, \tau )$ via the appropriate parametrization.
For this purpose, note that if $(c, \tau )$ gives $S'$ by the plus
parametrization and if $(c', \tau ')$ gives $S'$ via the minus one, then $\tau
'  \ne     \tau_{ - 0}$ if and only if $\tau      \ne     \tau_{ + 0}$.

Having said this, consider deforming $S'$ by pushing it along the vector field
$\partial_{s}$. Such a deformation keeps the intersections with $S$ in a
compact set of $\mathbb{R}    \times  (S^1 \times S^{2})$ because $\tau
     \ne     \tau_{ + 0}$ and $\tau '  \ne     \tau_{ - 0}$. Thus, the
intersection number between the resulting subvariety and $S$ is that between
 $S'$ and $S$. Of course, the latter is zero if and only if $S$ is disjoint from
$S'$. Now, by virtue of the fact that the function $s$ is bounded from above on
$S$, if $S'$ is pushed far enough along the vector field $\partial_{s}$, then
the portion of the resulting subvariety where $s$ has values that are also
achieved on $S$ has two components, each very close to an $\mathbb{R}$--invariant
cylinder. Of course, one of these cylinders is the product of $\mathbb{R}$ with
the Reeb orbit parametrized by $((p, p'), \tau )$, and the other the product
of $\mathbb{R}$ with the Reeb orbit parametrized by $((-p, -p'), \tau ')$. Now,
as $\theta$ has neither maxima nor minima on $S$, it follows that $S$ stays a
uniform distance from both of these constant $\theta$ cylinders. Thus, the
deformation of $S'$ is disjoint from $S$ and so $S'$ is also disjoint from $S$.

\substepp{Part 3}
Granted all of the proceeding, take $j$ very
large, and in particular, large enough so that $C_{j}$ is contained in the
radius $\frac{1}{100}\varepsilon $ tubular neighborhood of $S$.
The point now is that there exists some $S_{j}$ in the $J_{j}$ version of
$\mathcal{M}_{\hat{A} '}$ that intersects $C_{j}$. This follows
using \fullref{prop:3.4},
the lead observation in \refstepp{Part 2}, and the observation made towards the
end of \refstepp{Part 2} that distinct subvarieties in any given version of $\mathcal{M}_{\hat{A} '}$
are disjoint.

If the integer paired with $S$ by $\Xi $ were greater than 1, then any
subvariety $S_{j}$ from the $J_{j}$ version of $\mathcal{M}_{\hat{A} '}$ that
intersects $C_{j}$ must do so in a finite set of points. This understood,
pick such an $S_{j}$ that is parametrized by some $(c, \tau )$ via the plus
parametrization, and some $(c', \tau ')$ via the minus parametrization. Now
use \fullref{prop:2.13} to find a subvariety ${C_{j}}'  \in\mathcal{M}_{\hat{A}}$
with the following properties: First, it lies in the radius
$\frac{1}{100}\varepsilon $ tubular neighborhood of $S$ and it
intersects $S_{j}$. Second, the $|s|      \to     \infty $ limit of
the constant $|s| $ slices converge to Reeb orbits that are
distinct from both the $((p, p'), \tau )$ and $((-p, -p'), \tau )$ Reeb
orbits. Note that by virtue of $S_{j}$ and ${C_{j}}'$ intersecting in a finite
set of points, the latter have positive intersection number between them.

Having chosen ${C_{j}}'$, now deform $S_{j}$ by pushing it along the vector
field $\partial_{s}$. The argument given at the end of the previous part
works as well here to establish that result of a large push has the same
intersection number with ${C_{j}}'$ as does $S_{j}$, but is also disjoint from
${C_{j}}'$. As these two constraints are mutually exclusive, it follows that
$\Xi $ assigns 1 to $S$.

\step{Step 8}
This step proves that the integer paired with $S$ by $\Xi $ is 1 in the case
where $S$ is a cylinder with one end where the $|s|      \to
\infty $ limit of $\theta$ is in $\{0, \pi \}$ and where the other end
is a convex side end where the $|s| \to \infty $ limit of
$\theta$ is neither 0 nor $\pi$. Granted the discussion in the previous
step, the simplest case to consider is that where both ends of $S$ are convex
side. In this case, the argument from the previous step translates with
almost no essential changes to handle this case. In fact, the only slight
substantive difference arises in from, the different meanings of \fullref{sec:1}'s
parameters $(c_{E}$, ${\iota}_{E})$ in the cases that $E$ is an end where
$\lim_{| s| \to \infty }    \theta$ is or is not one of 0 or $\pi$.
In any event, the details for this case are left to the reader.

Turn instead to the case where $S$ has a concave side end where the
$|s|\to\infty $ limit of $\theta$ is either 0 or $\pi$.
Again, save for notation, no generality is lost by taking this limit to be
0. The argument for this case differs somewhat from that in the preceding
case and in \refstep{Step 7} because the function $s$ on $S$ ranges over the whole of
$\mathbb{R}$. In particular, a somewhat different argument must be used to
establish that distinct subvarieties in any given version of $\mathcal{M}_{\hat{A}'}$
are disjoint. In particular, where in \refstep{Step 7} (and in the
proof of \fullref{lem:3.11}), the subvariety $S'$ was pushed along the vector field
$\partial_{s}$, the argument now pushes $S'$ along the vector field
-$\partial_{\theta }$ to values of $\theta$ very near zero. Make this
change and then the rest of the argument amounts to little more than a
notationally changed version of that given previously.

\step{Step 9}
Here is a summary of what has been established by the preceding steps:
First, the one element in $\Xi $ has been shown, in all cases, to have the
form $(S, 1)$. As for $S$, its ends are known to be canonically in 1--1
correspondence with the elements in $\hat{A}$. In addition, the value of
$\deg_{(\cdot )}(d\theta )$ on all $(0,+,\ldots )$ ends of $S$
has been shown to be 1, and its value on all $(0,-,\ldots)$
ends has been shown to be zero. The set $\vartheta$ arises from the
asymptotic data for the ends of $S$ that correspond to $(0,+,\ldots )$ elements of $\hat{A}$.
Meanwhile, the angles in $\theta$ are now
known to be in 1--1 correspondence with the critical values of $\theta$ on
$S$, and each non-extremal critical point of $\theta$ has been proved
non-degenerate. The arguments given above also prove that there are $N_{ -
}+\hat N+c_{\hat{A} }-2$ such critical points in all. Finally, the
respective numbers, counting multiplicity, of the intersections between $S$
and the $\theta =0$ and $\theta =\pi$ cylinders are $c_+ $ and
$c_- $.

Granted all of this, it follows directly that $S\in\mathcal{M}_{\hat{A}}[\Theta,\vartheta ]$.
Moreover, since the $K = \mathbb{R}\times  (S^{1}    \times S^{2})$
version of~\eqref{eq:3.16} holds, it follows
directly that the graph $T_S$ from \fullref{sec:2g} when labeled as a moduli
space graph is isomorphic to the graph $T$.

%
%

\setcounter{theorem}{0}

\section{Constrained punctured spheres}\label{sec:4}

This section completes some unfinished business from \fullref{sec:3} by finishing
the proof of \fullref{thm:3.1}. This is done with the specification of a
collection $\{(a_{e}, w_{e})\}_{e\subset{T}}$ that meets the
criteria that are laid out in \fullref{sec:3b} and~\eqreft33.

What follows is a brief outline of the manner in which $\{(a_{e},w_{e})\}$ are specified.
\fullref{sec:4a} starts the story with a
description of any given pair $(a_{e}, w_{e})$ at points in the
parametrizing cylinder that are comparatively far from the boundary circles.
This description involves a set of two positive but very small numbers,
$\{\rho_{e0}, \rho_{e1}\}$, that are constrained in the
subsequent subsections plus a function, $\varepsilon_{e}$, of the
coordinate $\sigma$ on the closed parametrizing cylinder. In this regard,
$\varepsilon_{e}$ is strictly positive. Keep in mind throughout that all
of the subsequent constraints involve only upper bounds on $\varepsilon_{e}$, $\rho_{e0}$ and $\rho_{e1}$.
No positive lower bounds arise.

The mid-cylinder definition of $(a_{e}, w_{e})$ also involves three
additional functions of the coordinate $\sigma$ on the closed parametrizing
cylinder, these denoted by $a^{0}_{e}$, $w^{0}_{e}$ and $v_{e}$. These
three have no essential role until \fullref{sec:4e}, and until then,
they are unconstrained save for their boundary values. However, substantive
constraints do arise on $a^{0}_{e}$ and $w^{0}_{e}$ in the final section
so as to insure that distinct versions of $K_{(\cdot)}$ intersect
transversely with $+1$ local intersection numbers.

Sunsections~\ref{sec:4b}, \ref{sec:4c} and~\ref{sec:4d} 
specify $\{(a_{e}, w_{e})\}$ near the
boundaries of the parametrizing cylinders. In this regard, \fullref{sec:4b}
specifies these pairs near boundaries that correspond to the monovalent
vertices in $T$. \fullref{sec:4c} does this same task near the boundaries that
correspond to the bivalent vertices in $T$; and \fullref{sec:4d} gives the
specifications for boundaries that correspond to the trivalent vertices in
$T$. The criteria in \fullref{deff:3.2} are addressed in 
Subsections~\ref{sec:4b} and~\ref{sec:4c}. In
this regard, note that the definitions in \fullref{sec:4c} are relevant only to
the case where partition for the graph $T$ has only single element subsets.

These subsections also provide constraints on the relevant versions of
$(\varepsilon_{e}, \rho_{e0}, \rho_{e1})$, but all are of the
following sort: An upper bound appears for the values of $\varepsilon_{e}$ near each boundary
circle of the parametrizing domain. A particular choice for $\varepsilon_{e}$ then determines upper
bounds for $\rho_{e0}$ and $\rho_{e1}$. As remarked above, no positive
lower bounds arise. Mild constraints on $\{(a^{0}_{e}, w^{0}_{e}, v^{0}_{e})\}$ occur in these subsections.

Sections~\ref{sec:4b}, \ref{sec:4c} and~\ref{sec:4d} also address the nature of the singular points in the
resulting versions of $K_{(\cdot )}$. In particular, they prove that any
singular point in the closure of any given version of $K_{(\cdot )}$ arises
as the transversal intersections of two disks with $+1$ local intersection
number. Note that these subsections do not address the nature of the
intersections between versions of $K_{(\cdot )}$ with distinct edge labels.

\fullref{sec:4e}, addresses this last issue by explaining how to modify the
original choices for $\{(\varepsilon_{e}, \rho_{e0}, \rho_{e1}, a^{0}_{e}, w^{0}_{e}, v^{0}_{e})\}$ subject to all
previously noted constraints to guarantee that distinctly labeled versions
of $K_{(\cdot )}$ have transversal intersections with $+1$ local intersection number.

Granted the results from \fullref{sec:3}, the discussion in \fullref{sec:4e} completes the proof of
\fullref{thm:3.1} in the case that the partition for $T$ has only single element subsets. \fullref{sec:4f}
completes the proof of \fullref{thm:3.1} in the cases where the latter assumption does not hold.

\subsection{Parametrizations in the mid-cylinder}\label{sec:4a}

To start, fix a number, $\delta$ that is positive but less than
$\frac{1}{1000}$ times the difference between the maximal and
minimal angle labels of the vertices on every edge of $T$.

Now, let $e$ denote a given edge in $T$, and let $\theta_{o}$ and $\theta_{1}> \theta_{o}$ denote the
angles that are assigned to the vertices of $T$ that lie on $e$. Fix a positive numbers
$\rho_{0}\equiv\rho_{e0}$,
$\rho_{1}\equiv \rho_{e1}$ but constrained so that both are much smaller than $\delta$. In
addition, choose a similarly small, strictly positive function $\varepsilon\equiv\varepsilon_{e}$ on
$[\theta_{o}, \theta_{1}]$. The constructions that follow assume that $\rho_{0}$, $\rho_{1}$ and
$\varepsilon $ are all very small.

Let $\sigma$ denote the coordinate on $[\theta_{o}, \theta_{1}]$
and let $v$ denote the usual affine coordinate on $\mathbb{R}/(2\pi\mathbb{Z})$. At values of
$\sigma\in[\theta_{o}+2\rho_{0}, \theta_{1}-2\rho_{1}]$, the pair $(a_{e}, w_{e})$ are given by
\begin{equation}\label{eq:4.1}
\begin{aligned}
a_{e}(\sigma, v) &= a^{0}_{e}(\sigma)+\varepsilon (\sigma )
\cos\bigl(v + v^{0}_{e}(\sigma)\bigr)\\
\text{and}\quad w_{e}(\sigma, v) &=
w^{0}_{e}(\sigma )-\varepsilon (\sigma ) \sin\bigl(v +v^{0}_{e}(\sigma
)\bigr)
\end{aligned}
\end{equation}
where $a^{0}_{e}$, $w^{0}_{e}$ and $v^{0}_{e}$ functions on $[\theta_{o}, \theta_{1}]$. These functions are
as yet unconstrained.

One point to verify at the outset is whether the use of~\eqref{eq:4.1} leads via~\eqref{eq:3.2} to
an embedding in $\mathbb{R}\times(S^{1}\times S^{2})$ of
the $\sigma\in [\theta_{o}+2\rho_{0}, \theta_{1}-2\rho_{1}]$ portion of the parametrizing cylinder. That such is
the case when $\varepsilon $ is small is one consequence of the following
lemma.

\begin{lemma}\label{lem:4.1}
Suppose that $\theta_{o} < \theta_{1}$ are angles in $[0, \pi ]$ and that $Q=(q,q')$ is an integer pair
such that $\alpha_{Q}(\sigma)> 0$ when $\sigma\in(\theta_{o},\theta_{1})$. Now, suppose that $a^{0}$, $w^{0}$
and $v^{0}$ are smooth functions on $[\theta_{o}, \theta_{1}]$, and suppose that $\varepsilon$
is a strictly positive function of $\sigma$ and constrained so that $\varepsilon\alpha_{Q}< \frac{1}{2}$
at all points. Use this data to define the functions
\begin{equation*}
a \equiv a^{0}+\varepsilon  \cos(v+v^{0}) \qquad\text{and}\qquad  w \equiv w^{0}-\varepsilon \sin\big(v+v^{0}\big)
\end{equation*}
on the cylinder $(\theta_{o}, \theta_{1})\times \mathbb{R}/(2\pi\mathbb{Z})$. The pair $(a, w)$
then define an embedding of the cylinder $(\theta_{o}, \theta_{1})\times \mathbb{R}/(2\pi{\mathbb {Z}})$
into $\mathbb{R}\times(S^{1}\times S^{2})$ via the map in~\eqref{eq:3.2}.
\end{lemma}

\begin{proof}[Proof of \fullref{lem:4.1}]
Suppose for the moment that the functions $a$ and $w$ that are used in~\eqref{eq:2.25} are any given pair
of functions on $(\theta_{o},\theta_{1})$. As remarked in \fullref{sec:2g}, the resulting map then
defines an immersion of the parametrizing cylinder when $\alpha_{Q}>0$.
This understood, the issue is whether two distinct points in the domain are
mapped to the same point in the range. To analyze this last issue, note that
points $(\sigma, v)$ and $(\sigma',v')$ from the parametrizing cylinder
$(\theta_{o}, \theta_{1})\times \mathbb{R}/(2\pi\mathbb{Z})$
are mapped to the same point in $\mathbb{R}\times(S^{1}\times S^{2})$ if and only if both
$\sigma =\sigma'$ and there exists an integer pair $N=(n, n')$ such that
\begin{equation}\label{eq:4.2}
\begin{split}
& v'=v-2\pi\frac{\alpha_{N}(\sigma)}{\alpha_{Q_{e}}(\sigma)}\mod(2\pi\mathbb{Z}),\\
& a_{e}\bigg(\sigma,v-2\pi\frac{\alpha_{N}(\sigma)}{\alpha_{Q_{e}}(\sigma)}\bigg)= a_{e}(\sigma , v),\\
& w_{e}\bigg(\sigma, v-2\pi\frac{\alpha_{N}(\sigma)}{\alpha_{Q_{e}}(\sigma)}\bigg)=w_{e}(\sigma ,v) - 2\pi
\frac{nq_e'-n' q_e}{\alpha_{Q_e}(\sigma)}.
\end{split}
\end{equation}
Here, and below, $\alpha_{N}$ denotes the function $\sigma\to(1-3\cos^{2}\sigma)n'-\surd 6\cos\sigma n$.

To analyze the condition in~\eqref{eq:4.2}, fix a pair $Z\equiv(z, z')\in\mathbb{Z}\times\mathbb{Z}$
of integers such that $z\ {q_e}'- z'\ q_{e}= m$ with $m$ used here to denote the greatest common
divisor of the ordered pair of integers that comprise $Q_{e}\equiv(q_{e},{q_e}')$. With $Z$ fixed in
this way, then $N$ can be written as $N= \hat{x}(z, z')+\frac{\hat{y}}{m}(q_{e},{q_e}')$ with $\hat
{x}$ and $\hat {y}$ integers. This notation allows the third point in~\eqref{eq:4.2}
to be written as
\begin{equation}\label{eq:4.3}
\alpha_{Q_e}(\sigma) w_{e}\left(\sigma,v -2\pi\ \hat
{x}\frac{\alpha_Z(\sigma)}{\alpha_{Q_e}(\sigma)}-2\pi \frac{\hat{y}}{m}\right)=\alpha_{Q_e}(\sigma)
w_{e}(\sigma,v)- 2\pi\ m \hat{x}.
\end{equation}
In particular, this last condition implies that there are at most a finite
number of possible values for $\hat{x}$ that can appear at any given value
of $\sigma\in(\theta_{o}, \theta_{1})$. Meanwhile, the
second condition in~\eqref{eq:4.2} requires that
\begin{equation}\label{eq:4.4}
a_{e}\left(\sigma,v - 2\pi\ \hat{x}\frac{\alpha_{Z}(\sigma )}
{\alpha_{Q_e}(\sigma)}- 2\pi \frac{\hat{y}}{m}\right)
=a_{e}(\sigma,v).
\end{equation}
Together,~\eqref{eq:4.3} and~\eqref{eq:4.4} imply that only $\hat {y}$'s reduction modulo $m$ is
relevant. Thus, there is a finite set of possible values for $(\hat {x},\hat{y})$ that
need be considered for immersion points in any given
compact subset of $(\theta_{o}, \theta_{1})\times\mathbb{R}/(2\pi\mathbb{Z})$.

Now consider the additional ramifications of~\eqref{eq:4.3} and~\eqref{eq:4.4} in the case
that $a$ and $w$ are as described in the lemma. The first point to make is that
with $\varepsilon$ constrained as indicated, only $\hat{x}=0$ can appear
in~\eqref{eq:4.3}. This understood, it then follows that both~\eqref{eq:4.3} and~\eqref{eq:4.4} can
hold simultaneously only if $\hat{y}=0\mod(m\mathbb{Z})$. Thus $v$ and
$v'$ in~\eqref{eq:4.2} agree mod
$(2\pi\mathbb{Z})$ and~\eqref{eq:4.1} is an embedding.
\end{proof}

\subsection{Parametrizations near boundary circles with a monovalent vertex label}\label{sec:4b}

Suppose here that $e\subset T$ is an edge and that $\theta_{o}$ and
$\theta_{o}< \theta_{1}$ are the angles that are assigned to the
vertices on $e$. Suppose, in addition that $o\in e$ is a monovalent vertex
from $T$. For the sake of argument, suppose that the latter is assigned the
angle $\theta_{o}$. The story when the assigned angle is $\theta_{1}$
is identical save for notation and some sign changes and so the latter case
is not presented.

There are three separate cases to consider, these depending on the label
given the vertex $o$. These cases are considered in turn below. In what
follows, $\beta$ denotes a favorite smooth function on $[0, \infty)$ that
takes value 1 on $[0, 1]$, value 0 on $[2, \infty)$, and has negative
derivative on $(1, 2)$. Having chosen $\beta$, and granted that $\rho >0$
and $\theta_{*}\in [0, \pi]$, introduce the function
\begin{equation}\label{eq:4.5}
\beta' \equiv \beta \left(\frac{1}{\rho ^4}|
\sigma-\theta_{*}| \right)
\end{equation}

\begin{case}\label{cas:1}
\textsl{In this case, $o$ is assigned a $(1,\pm,\ldots)$ label in
$\hat{A}$. This is to say that in the respective $+$ and $-$ cases, the image in
$\mathbb{R}\times(S^{1}\times S^{2})$ of the $\sigma < 2\rho_{e0}$ part of the
parametrizing cylinder should have the asymptotics of a concave side or convex side end
of a $J$--pseudoholomorphic subvariety where the $| s|\to\infty$ limit of $\theta$ is 0.}
\end{case}

In the remainder of this Case~\ref{cas:1} discussion, $\rho$ denotes $\rho_{e0}$.
In this regard, $\rho\equiv\rho_{e0}$ along with $\theta_{*}=0$ are to be used for defining
the function $\beta '$ via~\eqref{eq:4.5}.

To extend the definition in~\eqref{eq:4.1} of $(a_{e}, w_{e})$ to the $\sigma <
2\rho$ portion of the parametrizing cylinder, first constrain the
functions $\varepsilon$, $a^{0}_{e}$, $w^{0}_{e}$ and $v^{0}_{e}$
that appear in~\eqref{eq:4.1} to be constant where $\sigma< 2\rho$. This
understood, extend the definition in~\eqref{eq:4.1} to the points where $\sigma<2\rho$ by setting
\begin{equation}\label{eq:4.6}
\begin{split}
&a_{e}=\frac{1}{\kappa}\beta' \ln\sigma +a^{0}_{e}+\big(\varepsilon (1-\beta') + \sigma
\beta')\cos(v + v^{0}_{e}).\\
&w_{e}=(1-\beta ') w^{0}_{e}-(\varepsilon (1-\beta') + \sigma\beta ') \sin(v + v^{0}_{e}).
\end{split}
\end{equation}
Here,
\begin{equation}\label{eq:4.7}
\kappa\equiv\frac{{q_e}'}{q_{e}}+\sqrt{\frac{3}{2}},
\end{equation}
According to \fullref{lem:4.1}, any small $\varepsilon$ version of~\eqref{eq:4.1} embeds the
$\sigma\in(0, \theta_{1}-2\rho ]$ portion of the
parametrizing cylinder in $\mathbb{R}\times (S^{1}\times S^{2})$.
Moreover, as $| a_{e}|\to\infty$ uniformly as $\sigma\to 0$, any such version of~\eqref{eq:4.1}
defines a proper embedding of this
same portion of the parametrizing cylinder. Thus, the only issue to consider
is whether the $\sigma\to 0$ asymptotics are correct. In particular,
the key point here is to verify~\eqref{eq:1.12}, and the latter task is
straightforward so left to the reader.

\begin{case}\label{cas:2}
\textsl{In this case, the vertex $0$ is assigned the element (1) from $\hat{A}$. To
start, once again set $\rho\equiv\rho_{e0}$ when referring to
the function $\beta '$ in~\eqref{eq:4.5}, also set $\theta_{*}=0$. It is
also to be understood here that $a^{0}_{e}$, $w^{0}_{e}$ and
$v^{0}_{e}$ from~\eqref{eq:4.1} are again constrained to be constant where $\sigma\leq 2\rho$.
Granted these conventions, the extension of the pair $(a_{e}, w_{e})$ to the points where
$\sigma< 2\rho$ is given by
\begin{equation}\label{eq:4.8}
\begin{aligned}
a_{e} &= a_{e}^{0}+\varepsilon (1-\beta')\cos(v +v^{0}_{e}) \\
\text{and}\quad
w_{e} &= w^{0}_{e}-\varepsilon (1-\beta')\sin(v + v^{0}_{e}).
\end{aligned}
\end{equation}
}
\end{case}

The reader is left to verify that the resulting extension of~\eqref{eq:4.1} to the
closed cylinder $[0,\theta_{1}-2\rho_{1}]\times\mathbb{R}/(2\pi\mathbb{Z})$ maps it onto an
embedded, closed disk in $\mathbb{R}\times(S^{1}\times S^{2})$ that intersects the $\theta = 0$
circle transversely and with intersection number $+1$ with respect to the
latter's symplectic orientation.

Granted that the extension given in~\eqref{eq:4.8} maps onto an embedded disk, note
that at the latter's intersection point with the $\theta= 0$ cylinder, the
restriction of the symplectic form on its tangent space is positive. Indeed,
this can be seen from the fact that the symplectic form pulls back along the
$\sigma= 0$ circle in the parametrizing cylinder to the form $\sigma d\sigma\wedge\,dv$.

\begin{case}\label{cas:3}
\textsl{In this case, the monovalent vertex is assigned some $(0,-,\ldots)$ element from $\hat{A}$. As
in the previous cases, set $\rho$ to equal $\rho_{e0}$. Use this value for $\rho$ and use
$\theta_{*}=\theta_{o}$ for defining the function $\beta'$.
}
\end{case}

The parametrization given below requires that $a^{0}_{e}$, $w^{0}_{e}$
and $v^{0}_{e}$ from~\eqref{eq:4.1} are constant where $\sigma\in [\theta_{o},\theta_{o}+2\rho]$.
Granted that such is the case, extend the definition of $(a_{e}$, $w_{e})$ to the $\sigma <\theta_{o}+2\rho$
portion of the parametrizing cylinder using the rule
\begin{equation}\label{eq:4.9}
\begin{split}
&a_{e}\equiv\frac{1}{\varsigma}\beta' \ln(\sigma-\theta_{o})+a^{0}_{e}+(\varepsilon (1-\beta') +
(\sigma -\theta_{o}) \beta') \cos(v + v^{0}_{e}),\\
&w_{e}=(1-\beta') w^{0}_{e}-(\varepsilon (1-\beta ')+(\sigma -\theta_{o})
\beta')\sin(v + v^{0}_{e}),
\end{split}
\end{equation}
where $\varsigma =\surd 6 \sin^{2}\theta_{o}(1+3\cos^{2}\theta
_{0})/(1+3\cos^{4}\theta_{o})$. It is left as another exercise for
the reader to verify that~\eqref{eq:4.9} and~\eqref{eq:3.2} together define a proper embedding
of the $\sigma\in (\theta_{o}, \theta_{1}-2\rho_{e1}]$
portion of $(\theta_{o}, \theta_{1})\times\mathbb{R}/(2\pi
\mathbb{Z})$ into $\mathbb{R}\times(S^{1}\times S^{2})$ as
submanifold whose large $| s| $ asymptotics meet the requirements
of \fullref{deff:3.2} to be those of a convex side end in some
$J$--pseudoholomorphic subvariety.

\subsection{Parametrizations near boundary circles with a bivalent vertex label}\label{sec:4c}

In this subsection, $o$ denotes a bivalent vertex in $T$ whose associated
partition subset has but a single element. In what follows, $e$ and $e'$ are the
two incident edges to $o$ with the convention that $o'$s angle label, $\theta_{o}$, is
the greater of the two angles that label the vertices on $e$, and
so the lesser of the two that label the vertices on $e'$.

The story starts with a preliminary digression to set the stage. To begin
the digression, note that $\rho_{e1}=\rho_{e'0}$ in what follows,
and $\rho$ denotes either. Take $\rho \ll\delta$. The function $\beta'$ now refers to the version in~\eqref{eq:4.5}
with this same value for $\rho $ and
with $\theta_{*}$ set to equal $\theta_{o}$.

Require that both the $e$ and $e'$ versions of $\varepsilon$ in~\eqref{eq:4.1} are
constant where $|\sigma-\theta_{o}| < \delta$ and
that these constants agree. Require that $a^{0}_{e}$ is constant where
$\sigma > \theta_{o}-2\rho$, that $a^0_{e'}$ is constant where
$\sigma <\theta_{o}+2\rho$, and that these two constants also have
the same value. Use $a_{0}$ for the latter. In addition, require that both
$w^{0}_{e}$ and $w^0_{e'}$ are zero where $\sigma$ is within $2\rho $
of $\theta_{o}$. Finally, require that both $v^{0}_{e}$ and
$v^{0}_{e'}$ are constant where $\sigma$ is respectively greater than
$\theta_{o}-2\rho$ and less than $\theta_{o}+2\rho$ with equal
value, and for notational convenience, take the constant to equal $0$. To
obtain the case where the constant value for $v^{0}_{e}$ and
$v^{0}_{e'}$ is non-zero, replace $v$ in what follows by $v-v^{0}_{e}$.

To proceed with the digression, introduce $P_{0}\equiv (p_{0},{p_0}')$ to denote the integer pair
from $o'$s element in $\hat{A}$, and set
\begin{equation}\label{eq:4.10}
x_{0}\equiv {q_e}' p_{0}- q_{e}{p_0}'.
\end{equation}

Note that $x_{0}$ is a positive integer. To explain, remember that $(p_{0},{p_0}')$ is a positive
multiple of the pair $(p, p')$ that is defined by
$\theta_{o}$ via~\eqref{eq:1.7}. This understood, positivity of $x_{0}$ is a
consequence of the positivity of the $Q = Q_{e}$ version of $\alpha_{Q}(\theta_{o})$. Note that
the formula for $x_{0}$ in~\eqref{eq:4.10} can be
written with the pair $Q_{e'}$ replacing $Q_{e}$; this a consequence of~\eqreft31.

The next task for this digression is to define certain `polar' coordinates
for respective neighborhoods of $(\theta_{o}, 0)$ in both the $e$ and $e'$
versions of the parametrizing cylinder. In both cases, the `radial'
coordinate is denoted as $r$; it takes values in $[0, 3\rho)$. Meanwhile, the
angular coordinate is denoted as $\tau$; it takes values in $[-\pi, 0]$ on
$e'$s version of the parametrizing cylinder, and it takes values in $[0, \pi]$ on the $e'$ version.
To describe the coordinate transformation from $(r,\tau)$ coordinates to the standard coordinates,
it is necessary to fix an $\mathbb{R}$--valued anti-derivative, $\hat{v}$, for $dv$ that is defined near 0
in $\mathbb{R}/(2\pi\mathbb{Z})$ and vanishes at 0. Thus, $v$ is the $mod(2\pi)$ reduction of $\hat{v}$.

With the preceding understood, here is the coordinate transformation between
the $(\sigma,\hat{v})$ and the variables $(r,\tau)$ for $e'$s
parametrizing cylinder:
\begin{equation}\label{eq:4.11}
\begin{split}
& \sigma=\theta_{*}+\varepsilon r\sin(\tau).\\
& \hat{v}= \left(1-\frac{\alpha_{Q_{e'}}(\sigma)}{\alpha_{Q_e}(\sigma )}\right)
\tau+\frac{1}{\alpha_{Q_e}(\sigma )}r \cos(\tau).
\end{split}
\end{equation}
In this regard, keep in mind that $\tau\in[-\pi,0]$. To verify
that $(r,\tau)$ are bona fide coordinates, use Taylor's theorem with
remainder while referring to~\eqreft31 and the first point in~\eqref{eq:4.11} to write
$\hat{v}=r(c_{0}\varepsilon\tau \sin(\tau)+\cos(\tau))+0(\rho r)$ with $c_{0}$ a positive constant
that is determined by $\theta_{o}$. In particular, the Jacobian of the map $(r, \tau)\to(\sigma, v)$
therefore has the form $-r (1- c_{0}\varepsilon\sin^{2}(\tau))+ 0(\rho r)$, and this is negative
if $\varepsilon$ is small and $\rho $ is very small.

Meanwhile, the coordinate transformation between the ($\sigma, \hat{v})$
and $(r,\tau)$ coordinates for the $e'$ version of the parametrizing
cylinder is given as follows:
\begin{equation}\label{eq:4.12}
\begin{split}
&\sigma=\theta_{*}+\varepsilon r\sin(\tau).\\
&\hat{v}=\left(\frac{\alpha_{Q_e}(\sigma)}{\alpha_{Q_{e'}}(\sigma )}- 1\right)
\tau +\frac{1}{\alpha_{Q_{e'}}(\sigma)}r \cos(\tau).
\end{split}
\end{equation}
In this case, the coordinate $\tau$ ranges in $[0,\pi]$.

The digression continues with the introduction of a certain function,
$v_{*}$, a function that is defined where $r>\frac{1}{8}\rho$ in the $|\sigma-\theta_{o}|< 3\rho$
portion of $(0, \pi )\times\mathbb{R}/(2\pi\mathbb{Z})$. To define
$v_{*}$, it is necessary to view $v$ as taking values in $[0, 2\pi]$.
This understood, set
\begin{equation}\label{eq:4.13}
v_{*}=\left((1-\beta ')+\frac{{\alpha_{Q_e}}(\sigma )}{\alpha_{Q_{e'}}(\sigma)}\beta '\right)v.
\end{equation}

As its final task, this digression introduces two versions of a function,
$\beta_{*}$, one on $e'$s version of the parametrizing cylinder and
the other on the $e'$ version. In both cases, $\beta_{*}$ is defined
to be zero on the complement of the set where $r$ is defined and less than
$3\rho$. Meanwhile, where $r\leq 3\rho$, this function is set equal to
$\beta(\frac{1}{\rho}r)$.

With the digression now over, what follows are the rules for extending the
definition of $(a_{e}$, $w_{e})$ to the $\sigma>\theta_{o}-2\rho$
portion of $e'$s version of the parametrizing cylinder. With $\tau$ viewed as
taking values in $[-\pi, 0]$, set
\begin{equation}\label{eq:4.14}
\begin{split}
&a_{e}=-\beta_{*}\ln(r)+a_{0}+\varepsilon\Big(\beta_{*}+(1-\beta_{*})\cos(v_{*})\Big).\\
&w_{e}= -\varepsilon (1-\beta_{*})\sin(v_{*})
\\ &\hspace{15mm}
+x_{0}\beta'\left(\frac{1}{\alpha_{Q_e}}
\beta_{*}\left(\tau-\frac{1}{2\alpha_{Q_{e'}}}r \cos(\tau)\right)-\frac{1}{2\alpha_{Q_e}}(1-\beta_{*})
v_{*}\right).
\end{split}
\end{equation}
Here, one must view $v$ and $\sigma$ as functions of $r$ and $\tau $ where $r<3\rho$.
As for $(a_{e'}, w_{e'})$, view $\tau $ as taking values in $[0,\pi]$ and set
\begin{equation}\label{eq:4.15}
\begin{split}
&a_{e'} = -\beta_{*}\ln(r)+a_{0}+\varepsilon\left(\beta_{*}+(1-\beta_{*})\cos(v)\right).\\
&w_{e'}= -\varepsilon(1-\beta_{*})\sin(v)\\
&\hspace{16mm}+ x_{0}\beta'\left(\frac{1}{\alpha_{Q_{e'}}}
\beta_{*}\left(\tau +\frac{1}{2\alpha_{Q_e}}r \cos(\tau)\right)+\frac{1}{2\alpha_{Q_e}}(1-\beta_{*})v\right).
\end{split}
\end{equation}
In this last equation, $v$ and $\sigma$ must again be viewed as functions of
$r$ and $\tau $ where $r < 3\rho$ with $v$ taking values in $[0, 2\pi]$.

As is proved below, these extensions have the following three special
properties: First, the union of the images in the $| \theta -\theta_{o}| < 3\rho$
portion of $\mathbb{R}\times (S^{1}\times S^{2})$ of the $e$ and $e'$ parametrizing cylinders
fit along the $\theta=\theta_{o}$ locus so as to define a smooth, properly immersed, thrice
punctured sphere. Second, the closure of this thrice punctured sphere in
${\mathbb{R}}\times(S^{1}\times S^{2})$ has the asymptotics as
dictated by \fullref{deff:3.2} of a $\lim_{| s|\to\infty}\theta=\theta_{o}$,
concave side end of a $J$--pseudoholomorphic subvariety.
Third, this thrice punctured sphere has a finite number of singular points,
all are transversal double points, and all have $+1$ local intersection
number.

The remainder of this subsection is divided into four parts, with the first two
containing the proofs of the first two of the preceding assertions. The
final two parts contain the proof of the third assertion.

\step{Part 1} This part addresses the manner in which the images of the two
parametrizing cylinders match along the $\theta=\theta_{o}$ locus in
$\mathbb{R}\times(S^{1}\times S^{2})$. In this regard, it
follows from~\eqref{eq:2.25} that these images define a smoothly immersed surface
near the $\theta =\theta_{o}$ locus provided that the following is
true: Let $v_{1}\in (0,2\pi)$. Then, there exists an integer pair,
$N = (n,n')$, and extensions of the definitions of $(a_{e},w_{e})$ and
$(a_{e'}, w_{e'})$ to some neighborhood in
$(0,\pi)\times\mathbb{R}/(2\pi\mathbb{Z})$ of $(\theta_{*}, v_{1})$ so that
\begin{equation}\label{eq:4.16}
\begin{split}
&a_{e}\left(\sigma,\frac{\alpha_{Q_{e'}}(\sigma)}{\alpha_{Q_e}(\sigma)}v
+ 2\pi\frac{\alpha_{N}(\sigma)}{\alpha_{Q_{e}}(\sigma)}\right)=a_{e'}(\sigma, v),\\
&
w_{e}\left(\sigma,\frac{\alpha_{Q_{e'}}(\sigma)}{\alpha_{Q_e}(\sigma)}v
+2\pi\frac{\alpha_{N}(\sigma)}{\alpha_{Q_{e}}(\sigma)}\right)\\
&\hspace{10mm}=w_{e'}(\sigma, v)+\frac{1}{\alpha_{Q_e}(\sigma)}({q_e}'q_{e'}-q_{e}{q_{e'}}')v
+\frac{2\pi}{\alpha_{Q_e}(\sigma)}({q_e}'n-q_{e}n').
\end{split}
\end{equation}
To verify that~\eqref{eq:4.16} holds, let $U\subset(0,\pi)\times\mathbb{R}/(2\pi\mathbb{Z})$ denote
the complement of the point $(\theta_{o}, 0)$. Now observe that the formulae in~\eqref{eq:4.14}
and~\eqref{eq:4.15} make perfectly good sense on some neighborhood in $U$ of $U'$s intersection with the
$\sigma =\theta_{o}$ circle. In particular, where $r\leq 3\rho$,
the formula in~\eqref{eq:4.14} makes good sense where $\tau\in(-\frac{3\pi}{2},\frac{\pi}{2})$ and that
in~\eqref{eq:4.15} makes good sense where $\tau\in(-\frac{\pi}{2},\frac{3\pi}{2})$. Meanwhile,
where $r$ is either undefined or greater than $2\rho$, both formula make good sense where $|\sigma
-\theta_{o}| <\rho^{4}$. Use these extensions of~\eqref{eq:4.14}
and~\eqref{eq:4.15} to provide extensions for use in~\eqref{eq:4.16} of the domains of
$(a_{e}, w_{e})$ and $(a_{e'}, w_{e'})$.

Consider now~\eqref{eq:4.16} at a point $(\theta_{o}, v_{1})$ where both
$|\sigma-\theta_{o}|<\rho^{4}$ and $r$ is either undefined
or greater than $2\rho$. In this case, the $N=(0,0)$ version of~\eqref{eq:4.16}
holds by virtue of two facts: First, $\beta' = 1$ where $|\sigma -\theta_{o}| < \rho ^{4}$ and thus
\begin{equation}\label{eq:4.17}
v_{*}\left(\frac{\alpha_{Q_{e'}}(\sigma)}{\alpha_{Q_e}(\sigma)}v\right) = v.
\end{equation}
Second,~\eqreft31 equates $-x_{0}$ with ${q_e}'q_{e'}-q_{e}{q_{e'}}'$.

Consider next the story for a point $(\theta_{o}, v_{1})$ where $r <3\rho$ that sits on
the $\cos(\tau ) = 1$ ray. To begin, note that~\eqref{eq:4.15}
and~\eqref{eq:4.16} require the evaluation of $(a_{e'}, w_{e'})$ at parameters $(r,
\tau)$ that give $\sigma$ in~\eqref{eq:4.12}, and that gives the
$\mathbb{R}$--valued
parameter $\hat{v}$ as it ranges over some interval of small length in $(0,
2\pi)$ that contains $v_{1}$. In particular, let $v'$ denote a point in such
an interval and take $\sigma$ to lie within $\rho^{4}$ of
$\theta_{o}$. Let $(r', \tau ')$ denote the values for $(r,\tau)$ that give
$(\sigma, v')$.

Now suppose that $N = (0, 0)$ again so that the pair $(a_{e}, w_{e})$
in~\eqref{eq:4.16} are to be evaluated at parameters $r = r_{*}$ and $\tau=\tau_{*}$ that
give the value of $\sigma$ used for $(a_{e'},w_{e'})$ and give $v'$ for
the $\mathbb{R}$--valued parameter $\hat{v}$. As
such,~\eqref{eq:4.11} and~\eqref{eq:4.12} require that $(r_{*}, \tau_{*})$ is
determined by $(r', \tau ')$ via the identities
\begin{equation}\label{eq:4.18}
\begin{split}
&r_{*} \sin(\tau_{*})= r'\sin(\tau '),\\
&\left(1-\frac{\alpha_{Q_{e'}}(\sigma)}{\alpha_{Q_e}(\sigma)}\right)
\tau_{*}+\frac{1}{\alpha_{Q_e}(\sigma)}r_{*}\cos(\tau_{*})\\&\qquad\qquad =\frac{\alpha_{Q_{e'}}(\sigma)}
{\alpha_{Q_e}(\sigma)}\left[\left(\frac{\alpha_{Q_e}(\sigma)}{\alpha_{Q_{e'}}(\sigma)}-1\right)
\tau ' + \frac{1}{\alpha_{Q_{e'}}(\sigma)}r'\cos(\tau')\right]
\end{split}
\end{equation}
In particular, if $\varepsilon$ is small and $\rho $ very small, then~\eqref{eq:4.19}
requires $r_{*}= r'$ and $\tau_{*}=\tau'$.

Granted this last conclusion, the top equality in~\eqref{eq:4.16} now follows from~\eqref{eq:4.14}
and~\eqref{eq:4.15}. Meanwhile, since $\beta'=1$ at the given value of
$\sigma$, the lower equality holds provided that
\begin{multline}\label{eq:4.19}
x_{0}\beta_{*}\left[\left(\frac{1}{\alpha_{Q_e}}-\frac{1}{\alpha_{Q_{e'}}}\right)\tau' -
\frac{1}{\alpha_{Q_e}\alpha_{Q_{e'}}}r \cos(\tau')\right]- x_{0}(1-\beta_{*})\frac{1}{2\alpha_{Q_e}}
(2v') \\
\qquad= \frac{1}{\alpha_{Q_e}}(-x_{0})v';
\end{multline}
and such is the case by virtue of~\eqref{eq:4.12}.

The last case to consider is that where the point $(\theta_{o}, v_{1})$
sits where $r < 3\rho$ on the $\cos(\tau )= -1$ ray. Supposing that $\sigma$ is within
$\rho^{4}$ of $\theta_{o}$ and $v'$ is very close to
$v_{1}$, then the pair $(a_{e'}, w_{e'})$ in~\eqref{eq:4.16} must be evaluated at
parameters values $r'$ for $r$ and $\tau'$ for $\tau$ that give the point
$(\sigma, v')$ via~\eqref{eq:4.12}. In this regard, note that $\tau '\sim \pi$.

Now suppose that $N = Q_{e'}$ in~\eqref{eq:4.16}. This being the case, then the pair
$(a_{e}, w_{e})$ are to be evaluated at parameters $r = r_{*}$ and
$\tau =\tau_{*}$ with $\tau_{*}\sim -\pi$ that give the value of
$\sigma$ used for $(a_{e'}, w_{e'})$, but
now give the $\mathbb{R}$--valued parameter $\hat {v}$ that obeys
\begin{equation}\label{eq:4.20}
\hat{v}=\frac{\alpha_{Q_{e'}}(\sigma)}{\alpha_{Q_e}(\sigma)}(v' + 2\pi )\mod(2\pi \mathbb{Z}).
\end{equation}
As such,~\eqref{eq:4.11} and~\eqref{eq:4.12} require that $(r_{*},\tau_{*})$
is determined by $(r', \tau')$ via the identities
\begin{equation}\label{eq:4.21}
\begin{aligned}
&r_{*}\sin(\tau_{*}) = r' \sin(\tau '),\\
&\bigg(1{-}\frac{{\alpha_{Q_{e' } } (\sigma )}} {{\alpha_{Q_e }
(\sigma )}}\bigg) \tau_{*}{+}
\frac{1}{\alpha_{Q_e }(\sigma )}
r_{*}\cos(\tau_{*}) = \\
&\hspace{50pt}\frac{\alpha_{Q_{e'}} (\sigma)}   {\alpha_{Q_e} (\sigma)}
\Bigg[\bigg(\frac{\alpha_{Q_e } (\sigma)}
  {\alpha_{Q_{e'}}(\sigma)}{-}1\bigg)\tau '
{+} \frac{1}   {\alpha_{Q_{e'}}(\sigma)}
r'\cos(\tau ')\Bigg] \\
&\hspace{200pt}{+}
\bigg(\frac{\alpha_{Q_{e'}} (\sigma)} {\alpha_{Q_e} (\sigma)}{-}1\bigg) 2\pi .
\end{aligned}
\end{equation}
When $\varepsilon $ is small and $\rho $ is very small, then~\eqref{eq:4.21} requires
$r_{*} = r'$ and $\tau_{*}=\tau '-2\pi$.

Granted the preceding, the top equality in~\eqref{eq:4.16} again follows from~\eqref{eq:4.14}
and~\eqref{eq:4.15} straight away. Meanwhile, the lower equality holds provided that
\begin{multline}\label{eq:4.22}
x_{0} \beta_{*}\Bigg[\bigg(\frac{1}   {\alpha_{Q_e }}-
\frac{1}   {\alpha_{Q_{e' } } }\bigg)\tau '-
\frac{1}   {\alpha_{Q_e } \alpha_{Q_{e' } } }r \cos(\tau
')\Bigg] \\
- 2\pi  x_{0} \beta_{*}  \frac{1}   {\alpha_{Q_e } } -
x_{0}(1-\beta_{*})  \frac{1}   {2\alpha_{Q_e } }(2v'+ 4\pi ) \\
\quad= \frac{1}   {\alpha_{Q_e } }(-x_{0}) (v'+2\pi).
\end{multline}
And this last equation does indeed hold; this is another consequence
of~\eqref{eq:4.12}.

\step{Part 2}
This part discusses the image of the points in the respective $e$ and $e'$
parametrizing cylinders that are close to the $(\sigma =\theta_{o},
v = 0)$ point. The result is a verification that the images of this portion
of the two parametrizing cylinders fit together so as to define a
submanifold that has the required large $|s|$ asymptotics as
dictated by \fullref{deff:3.2}.

To start the discussion, note that the respective images of the $\theta\ge \theta_{o}$
and $\theta  \leq \theta_{o}$ parts of the
complement of $(\theta_{o}, 0)$ in a closed, small radius disk about this
point define a properly immersed surface with boundary in $\mathbb{R} \times
 (S^{1}  \times  S^{2})$. In this regard, let $e$ denote the union of the
respective images in $\mathbb{R} \times  (S^{1} \times  S^{2})$ of the
portions of the $e$ and $e'$ versions of the parametrizing cylinder where $r$ is
defined and where $r  \leq   \rho ^{4}$.

Here is the first observation: With both $\rho $ and $\varepsilon $ chosen
to be very small, then the map from $(0, 3\rho ] \times   \mathbb{R}/(2\pi\mathbb{Z})$
to $\mathbb{R} \times  (S^{1} \times  S^{2})$ that uses the
appropriate pair of~\eqref{eq:4.11} and~\eqref{eq:4.14} or~\eqref{eq:4.12}
and~\eqref{eq:4.15} defines a smooth,
proper map. As is explained some paragraphs hence, this map embeds the
subset of the disk where $\beta_{*}=\beta ' = 1$. In
particular, $E$ is embedded.

To verify the second requirement, note that the 1--form $ds$ pulls back to the
$(r, \tau )$ disk via the parametrizing map as $-\frac{1}{r}dr$
where $\beta_{*} = 1$. Thus s has no critical points on $E$.
Moreover, as $r = e^{{-}s}$, so $\theta$ has a unique  $s \to \infty$ limit on $E$.
The condition on the pull-back of the contact 1--form in~\eqref{eq:1.1}
is considered momentarily. The fact that $p_{0}$ and ${p_0}'$ give the
respective integrals of the pull-backs of $\frac{1}{{2\pi }}dt$
and $\frac{1}{{2\pi }}d\varphi$ follow from the relationship
between $Q_{e}$ and $Q_{e'}$ given in~\eqreft31.

To see about the requirement for the 1--form $p_{0}d\varphi  - {p_0}'dt$,
first use~\eqref{eq:4.11} and~\eqref{eq:4.14} where $\beta_{*}=\beta ' = 1$ and
$\tau \in  [-\pi , 0]$ to write
\begin{multline}\label{eq:4.23}
p_{0}\varphi  - {p_0}'t =
x_{0}\Bigg[\bigg(1-\frac{\alpha_{Q_{e'}}(\sigma)} {\alpha_{Q_e}(\sigma )}\bigg) \tau
+\frac{r}{\alpha_{Q_e }(\sigma)} \cos(\tau )\Bigg] \\
\qquad-
\alpha_{P_o }
\frac{1}   {\alpha_{Q_e}}x_{0}\Bigg[\tau -\frac{1}
{2\alpha_{Q_{e'}}}r \cos(\tau )\Bigg] \mod(2\pi \mathbb{Z}).
\end{multline}
Now use~\eqreft31 to conclude that
\begin{equation}\label{eq:4.24}
p_{0}\varphi  - {p_0}'t = x_{0}  \frac{1}{2}\bigg(\frac{1}   {\alpha_{Q_e } (\sigma )}+\frac{1}
{\alpha_{Q_e ' } (\sigma )}\bigg) r \cos(\tau ) \mod(2\pi \mathbb{Z})
\end{equation}
where $\beta_{*}=\beta ' = 1$ and $\tau  \in  [-\pi , 0]$.
Meanwhile, the same expression for $p_{0}\varphi  - {p_0}'t$ appears
when~\eqref{eq:4.12} and~\eqref{eq:4.15} are used where $\beta_{*}=\beta ' = 1$
and $\tau  \in  [0, \pi ]$. Thus, the right hand side of~\eqref{eq:4.24} without
the `$\mod(2\pi \mathbb{Z})$' proviso provides an anti-derivative for
$p_{0}d\varphi  - {p_0}'dt$ with a unique  $s \to \infty $ limit on $E$.

Return next to the question of the contact form in~\eqref{eq:1.1}. The fact that it
has nowhere zero pull-back at large $|s|$ on $e$ can be readily
deduced from the following three facts: First, the vectors $({p_0}',
-p_{0})$ and $((1-3\cos^{2}\theta_{o}), \surd 6\cos\theta_{o}
\sin^{2}\theta_{o})$ are linearly independent in $\mathbb{R}^{2}$; this
a consequence of the $\theta_{o}$ version of~\eqref{eq:1.7}. Second,~\eqref{eq:4.24}
implies that the 1--form $p_{0}d\varphi  - {p_0}'dt$ is $o(e^{{-}s})$
on the large s slices of $E$. Finally, the $(q, q') = Q_{e}$ and $(q,
q') = Q_{e'}$ versions of $qd\varphi  - q'dt$ differ from $x_{0} d\tau $
by $o(e^{{-}s})$ on the respective portions of the large s slices in $E$
where $\tau \in  [-\pi , 0]$ and where $\tau \in  [0, \pi ]$.

Here is the promised explanation as to why the punctured $(r, \tau )$ disk
is embedded where $\beta_{*}=\beta ' = 1$. To see that such is
the case, note first that two points are mapped to the same point only if
the corresponding pull-backs of $\theta$ agree. Moreover, the corresponding
pull-backs of a chosen $\mathbb{R}$ lift of the right hand side of~\eqref{eq:4.24} must
also agree. These two requirements can be met only if the two points are one
and the same.

The final issue concerns the size of the projection $\prod_{J}$ as
defined in the last line of \fullref{deff:3.2}'s second requirement. The fact is
that this projection and its covariant derivative are both $o(e^{{-}s})$
at large $s$ on $E$, this by virtue of~\eqref{eq:4.24} and the formulae for $\theta$ in
the top lines in~\eqref{eq:4.11} and~\eqref{eq:4.12}.

\step{Part 3}
This and the remaining part of the discussion in this subsection analyze
the singular points of the $| \theta -\theta_{o}|  <3\rho $ portion of the closure of
$K_{e} \cup  K_{e'}$. In this
regard, it follows from what has been said so far that this closure is an
immersed surface with compact singular set. Indeed, this happens because
both the $e$ and $e'$ parametrizations extend across the $\sigma =\theta
_{0}$ circle save at the missing point where $\sigma =\theta_{o}$
and $v = 0$. With the preceding understood, the task at hand is to verify that
there are but a finite number of singular points, all regular double points
and all with $+1$ local intersection number.

To start the story, remark that points in $K_{e}$ and $K_{e'}$ are disjoint
as they have distinct $\theta$ values. Thus, a pair of points $(\sigma ,v)$ and
$(\sigma', v')$ in the extended domain of either the $e$ or $e'$
parametrizing maps are sent to the same point in $\mathbb{R}\times (S^{1} \times  S^{2})$ if
and only if $\sigma=\sigma '$ and
the respective $e$ or $e'$ versions of the conditions in~\eqref{eq:4.2} are satisfied.
This noted, the discussions that follows in this \refstep{Part 3} and in \refstep{Part 4} focus
exclusively on points in the closure of $K_{e}$; thus points in the image of
the complement of $(\theta_{o}, 0)$ in $e'$s version of the closed
parametrizing cylinder. The analysis for $K_{e'}$ is very much the same and
so omitted.

The rest of the story from this \refstep{Part 3} is summarized by the following lemma:

\begin{lemma}\label{lem:4.2}
There exists $\varepsilon_{0}$ and given $\varepsilon\in (0,\varepsilon_{0})$, there exist
positive constants $\delta $ and $\rho_{*}\ll \varepsilon$ such that when $\rho  < \rho_{*}$,
then the closure of $K_{e}$ is smooth near any point with an inverse image in the portion of
the closed parametrizing cylinder that lie where $\beta_{*}>0$, or where $\beta' < \delta$,
or where $\beta' > 1-\delta$.
\end{lemma}

\noindent The proof of this lemma occupies the remainder of \refstep{Part 3}.

\begin{proof}[Proof of \fullref{lem:4.2}]
The proof is facilitated by introducing a relatively prime pair of integers, $Z \equiv (z, z')$,
such that $z{q_e}'- z'q_{e} = m$ where $m$ denotes the greatest common divisor
of $(q_{e}, {q_e}')$. With $Z$ so specified, any given integer pair $N = (n,
n')$ can be written as $N = \hat{x}(z, z') + \frac{\hat{y}}{m}(q_{e}, {q_e}')$ with
$\hat{x}$ and $\hat{y}$ integers. Writing $N$ in
this way makes the third point in $N'$s version of~\eqref{eq:4.2} into the condition
in~\eqref{eq:4.3} and the second point into the condition in~\eqref{eq:4.4}. Note that $\hat{x}$
does not depend on the choice for $Z$ but $\hat{y}$ does. In any event, only
values for $\hat{y}$ that lie in $\{0,\ldots,m-1\}$ need be considered.

The rest of the proof is broken into four steps.

\substep{Step 1}
This step derives the lower bound for
$\beta' $ on the inverse image of a singularity. For this purpose, note that
with $\varepsilon $ small in~\eqref{eq:4.14} and $\beta '  \leq     \varepsilon
^{2}$, then $| w_{e}|  < 2\varepsilon $ and so only the case
that $\hat {x} = 0$ case can possibly arise in~\eqref{eq:4.3}. However, this case
then requires $\hat {y} = 0$ as well since the form of $a_{e}$ and $w_{e}$
in~\eqref{eq:4.14} precludes any other $\hat {x} = 0$ solutions to both~\eqref{eq:4.3}
and~\eqref{eq:4.4}. The argument here is essentially the same as one given in \refstep{Part 1},
above.

\substep{Step 2}
This step proves the following assertion:
If $\varepsilon $ is small and then $\rho $ is very small, the closure of
$K_{e}$ is smooth at the image of any point in the closed parametrizing
cylinder that lies where $r$ is defined and less than $\frac{3}{4}\rho $.

To start the proof, suppose that $\zeta $ and $\zeta '$ are two points in
the closed parametrizing cylinder that map to the same point in $\mathbb{R}
\times  (S^{1} \times  S^{2})$ and are such that one lies in the
indicated region. As the respective values for $a_{e}$ must agree at the two
points,~\eqref{eq:4.14} demands that both points lie where $r < \rho $ when
$\varepsilon $ is small. In particular, $\beta_{*} = 1$ at both
points, so they both lie where $r$ is defined, and their respective $r$
coordinates must agree. As the respective values of $\sigma$ also agree at
the two points,~\eqref{eq:4.11} demands that their respective $\tau $ coordinates
either agree or are interchanged by the involution of $[-\pi , 0]$ that
sends $\tau $ to $-\pi -\tau $. The latter must be the case if the two
points are distinct.

Meanwhile,~\eqref{eq:4.3} and~\eqref{eq:4.4} requires that the respective $\mathbb{R}/(2\pi
\mathbb{Z})$ coordinates of the two points are related by
\begin{equation}\label{eq:4.25}
v(\zeta ') = v(\zeta ) - 2\pi     \hat {x}\frac{{\alpha_Z
(\sigma )}}{   {\alpha_{Q_e } (\sigma )}} - 2\pi \frac{\hat {y}}{m} \mod(2\pi ).
\end{equation}
with $\hat {x}$ and $\hat {y}$ integers and with $Z = (z, z')$ as in~\eqref{eq:4.3}. To
see what this implies, note that as $\beta '  \ne $ 0, it follows that
$| \sigma -\theta_{o}|  < 2\rho ^{4}$ and so
\begin{equation}\label{eq:4.26}
\frac{\alpha_Z (\sigma )}{\alpha_{Q_e } (\sigma )} =
\frac{{z' p_o - zp_o ' }}{   {x_o }}    + o(\rho
^{4}),
\end{equation}
where the term that is indicated by $o(\rho ^{4})$ is bounded by $\kappa
\rho ^{4}$ where $\kappa $ depends only on $Q_{e}$ and $P_{0}$.  As the
ratio on the right hand side is a rational number with denominator no larger
than $x_{0}$, the equality in~\eqref{eq:4.25} demands that
\begin{equation}\label{eq:4.27}
v(\zeta ') = v(\zeta ) + o(\rho ^{4}) \mod(2\pi )
\end{equation}
where the term $o(\rho ^{4})$ has the same significance as in~\eqref{eq:4.26}.

To continue, note next that~\eqref{eq:4.27} is consistent with the lower line in~\eqref{eq:4.11}
only in the case that the value of $| r\cos(\tau )| $ at
the two points is bounded by $\kappa \rho ^{4}$ where $\kappa $ again
depends only on $P_{0}$ and $Q_{e}$. On the other hand, as $\beta '$ is
neither 0 nor 1, the top line in~\eqref{eq:4.11} requires that $| r\sin(\tau
)|  > \varepsilon ^{{-}1}\rho ^{4}$. In particular, $r
\ge     \varepsilon ^{{-}1}\rho ^{4}$. This then means that the
value of $| \cos(\tau )| $ at the two points is bounded by
$\kappa \varepsilon $. As such, the values of $\tau $ at these points must
have the form $\tau =\frac{\pi}{2}    \pm     \varepsilon
\cdot \varsigma$, where $\varsigma  > 0$ is bounded solely in terms of
$P_{0}$ and $Q_{e}$. This understood, use of the lower line in~\eqref{eq:4.14} will
establish that the respective values of $w_{e}$ at the two points differ by
an amount that is bounded by $\kappa     \varepsilon $, where $\kappa $ is as
before, a constant that depends only on $P_{0}$ and $Q_{e}$. However, if
such is the case, then small $\varepsilon $ is consistent with~\eqref{eq:4.3} only if
$\hat {x} = 0$. This and~\eqref{eq:4.27} require that $\hat {y} = 0$ also, and so
$\zeta =\zeta '$.

\substep{Step 3}
As is argued subsequently in this step, if
$\varepsilon $ is small and $\rho $ is very small, then the closure of
$K_{e}$ is smooth near any point with an inverse image that lies where $r$ is
defined and has value in $(\frac{1}{2}\rho , 3\rho )$. To see
why this is true, take $\zeta $ and $\zeta '$  to be points in the closed
parametrizing cylinder that are mapped to the same point in $\mathbb{R}
\times  (S^{1} \times  S^{2})$, are such that $\beta ' > 0$, and that
one lies where $r$ is defined and is in the indicated range. In this case, the
top line of~\eqref{eq:4.14} requires that the respective values of
$\beta_{*} \ln(r)$ differ by less than $\varepsilon $ at the two points. Thus, both
points lie where $r$ is greater than $\frac{1}{4}\rho $.
Furthermore, as one of them lies where $r \leq  3\rho $, the top line in~\eqref{eq:4.14} requires that
$1 - \kappa ^{{-}1}\rho ^{2} > \cos(v_{*}) > 1 - \kappa \rho ^{2}$;
at both; here $\kappa      \ge 1$ is a
constant that depends only on $P_{0}$ and $Q_{e}$.

Meanwhile,~\eqref{eq:4.3} and~\eqref{eq:4.4} require~\eqref{eq:4.25}, and as
$\beta '  \ne  0$, so~\eqref{eq:4.26} and~\eqref{eq:4.27} hold here too. However, if such is the case, then both lie
where $\sin(v_{*})     \ge     \kappa ^{{-}1}\rho $ or both lie
where $\sin(v_{*}) < -\kappa ^{{-}1}\rho $. Here again,
$\kappa      \ge  1$ depends only on $P_{0}$ and $Q_{e}$ when $\rho $ is
small. As a result,~\eqref{eq:4.11} and~\eqref{eq:4.14} insure that the respective values of
$w_{e}$ at the two points differ by no more that $\kappa '\rho $ where
$\kappa '$ is again a constant that depends only on $P_{0}$ and $Q_{e}$.
Granted this, then only the $\hat {x} = 0$ case of~\eqref{eq:4.3} can arise when
$\rho $ is small, and this then requires that $\hat {y} = 0 \mod(m)$ also.
Thus, $v(\zeta ') = v(\zeta )$ and the points are one and the same.

\substep{Step 4}
This step proves the following: If
$\varepsilon $ is small, there then exists $\delta ' > 0$ with the following
significance: When $\rho $ is very small, the parametrizing map embeds the
portion of $e'$s parametrizing cylinder where $\beta ' > 1-\delta' $ and
$\beta_{*} = 0$.

To start the proof, suppose that $(\sigma , v)$ and $(\sigma , v')$ are two
points that lie in the indicated portion of the closed parametrizing
cylinder. Suppose, in addition, that $v  \ne  v'$ and that these two points
are sent to the same point in $\mathbb{R}    \times (S^{1} \times  S^{2})$.
This being the case, the equality between the respective values of
$a_{e}$ at the two points requires that $v' = 2\pi  - v \mod(2\pi \mathbb{Z})$.
As the two points $(\sigma , v)$ and $(\sigma , v')$ are distinct, so
$v  \ne     \pi  \mod(2\pi \mathbb{Z}))$. Thus, at the expense of choosing
which to call $v$ and which to call $v'$, as well as a $\mathbb{Z}$ lift of $\hat
{y}$, one can assume that $v$ and $v'$ have respective $\mathbb{R}$--lifts $\hat
{v}$ and $\hat {v}'$ with $\hat {v}     \in (\pi , 2\pi )$ and with
$\hat {v}' = 2\pi -\hat {v}$. Meanwhile,~\eqref{eq:4.3} requires that
\begin{equation}\label{eq:4.28}
\varepsilon \sin(\hat {v}) - x_{0}    \frac{1}{{2\alpha
_{Q_{e' } } }}\beta ' \hat {v}' = -\varepsilon  \sin(\hat {v})
- x_{0}    \frac{1}{{2\alpha_{Q_{e' } } }}\beta ' \hat
{v} - 2\pi m \frac{1}{{\alpha_{Q_e } }}\hat {x}.
\end{equation}
To see the implications of these constraints, it proves convenient to
introduce
\begin{equation}\label{eq:4.29}
\wp \equiv    - \bigg(\hat {x}\frac{{z' p_o - zp_o ' }}{{x_o }}-\frac{\hat {y}}{m}\bigg) .
\end{equation}
As will now be explained, $\wp \in  (0, 1)$. Indeed, as
$\beta '  \ne  0$, so~\eqref{eq:4.26} holds and can be used to write
\begin{equation}\label{eq:4.30}
\hat {v}    = \pi(1+\wp+\hat{x}\hat{u}) \qquad \text{and}\qquad
{\hat {v}}'    = \pi(1-\wp-\hat{x}\hat{u})
\end{equation}
where $\hat{u}$ is a function of $\sigma$ that obeys $| \hat{u}|
\leq     \kappa     \rho ^{4}$ and $| \partial_{\sigma }\hat{u}|
     \leq     \kappa $ with $\kappa $ a constant that is determined by $P_{0}$
and $Q_{e}$.

With $\wp $ as just defined, the constraint in~\eqref{eq:4.28} can be rearranged to
read
\begin{equation}\label{eq:4.31}
\beta '\wp - 2m\frac{\hat {x}}{{x_o }}-\frac{1}{\pi}\varepsilon  \sin(2\pi\wp)+\hat {x}u = 0 ,
\end{equation}
where $u$ is a function of $\sigma$ that vanishes at $\theta_{o}$ and is
determined a priori by $P_{0}$ and $Q_{e}$. In particular $| u|
\leq     \kappa \rho ^{4}$ where $\beta '  \ne  0$ and $| \partial
_{\sigma }u|      \leq     \kappa $ where $\kappa $ is a constant that
is determined a priori by $P_{0}$ and $Q_{e}$. As will now be explained, the
desired bound for $(1-\beta ')$ follows from this last equation. To see why,
note that by virtue of~\eqref{eq:4.30}, small $\rho $ insures that the distance
between $\hat {v}$ from either $\pi$ or $2\pi$ is bounded from zero by a
constant that depends only on $P_{0}$ and $Q_{e}$. In particular, $|\sin(\hat {v})| $
is bounded away from zero by such a positive
constant. As $| \sin(\hat {v})| $ is uniformly bounded away from
zero, and as $\wp $ is a fraction between 0 and 1 whose denominator is $mx_{0}$,
it follows that there are no $(\hat {x}, \hat {y})$ versions of~\eqref{eq:4.31}
when $\varepsilon $ is small, $(1-\beta ') < \varepsilon ^{2}$,
and $\rho $ is very small.
\end{proof}

\step{Part 4}
This part of the discussion identifies the singular points and verifies
that they have positive local intersection number. In this regard, remember
\fullref{lem:4.2} which states that the inverse images of the singular points lie
where $\beta_{*} = 0$ and $\beta '  \in (\delta , (1-\delta))$.

To begin, fix $\hat {x}$ and $\hat {y}$ so that $\wp      \in  (0, 1)$
as defined in~\eqref{eq:4.29}. Then view the small $\varepsilon $ and very small
$\rho $ version of~\eqref{eq:4.31} as an equation for $\sigma      \in (\theta
_{0}-3\rho , \theta_{o})$. Because the values of $\beta '$ range
over $[0, 1]$, this equation has a solution only if  $0 \leq  m\hat {x}
\leq     \frac{1}{2}x_{0}$. Moreover, if a solution exists, then
it is unique. To explain, remember that with $\varepsilon $ fixed and small,
\fullref{lem:4.2} provides a constant $\delta      \equiv     \delta (\varepsilon )
 \in  (0, 1)$ such that a solution to~\eqref{eq:4.31} must occur at a value for
$\sigma$ where $\delta < \beta ' < (1-\delta )$. Thus, there
exists a constant,  $b \equiv  b(\varepsilon )     \in  (0, 1)$ such that
\begin{equation}\label{eq:4.32}
\partial_{\sigma }\beta ' > b \frac{1}{\rho ^4}
\end{equation}
at any value of $\sigma$ where~\eqref{eq:4.31} holds. This last fact implies the
uniqueness of any solution to~\eqref{eq:4.31} in the case that $\rho $ is very small.

The next point to make is that the solutions to the various $(\hat {x},\wp )$
versions of~\eqref{eq:4.31} result in singularities of $K_{e}$ of the
simplest sort: Each singular point is the center of a small radius ball in
$\mathbb{R}    \times  (S^1 \times  S^{2})$ that intersects $K_{e}$ as
the union of two embedded disks meeting only at the origin. Here is why:
According to~\eqref{eq:4.30}, $\hat {v}=\pi (1 + \wp) + o(\rho^{4})$
and As a consequence, when $\rho $ very small,~\eqref{eq:4.4} guarantees
that any pair of versions of~\eqref{eq:4.31} with different values for $\wp $ will
yield disjoint singular points in $K_{e}$. Meanwhile, two versions of~\eqref{eq:4.31}
that are defined using the same choice for $\wp $ but with different choices
for $\hat {x}$ give corresponding singular points in $K_e$ at
distinct $\theta$ values.

With the singular points identified, the next task is to verify that each
self intersection is transversal and has positive intersection number. This
task requires a suitable expression for the push-forward by the
parametrizing map of the vector fields $\partial_{\sigma }$ and $\partial
_{v}$. Such a formula is given in~\eqref{eq:4.33} below. In this regard, the
following notation is used: The pull-back via the parametrizing map of $(1 -
3\cos^{2}\theta )$ is denoted as $c$, and that of $\surd 6 \cos\theta$ as
$c'$. Note also that the push-forwards of $\partial_{\sigma }$ and
$\partial_{v}$ are not notationally distinguished from the originals. In
addition, the label $e$ on $(a_{e}, w_{e})$ is suppressed so that a
subscript on the resulting $a$ or $w$ can be used to indicate the partial
derivatives in the direction labeled by the subscript. The label $e$ is also
suppressed so that the integer pair $Q_{e}$ appears as $(q, q')$. With this
notation set, here are the promised formulae for the push-forwards of
$\partial_{\sigma }$ and $\partial_{v}$:
\begin{equation}\label{eq:4.33}
\begin{split}
\partial_{\sigma} &= (cw)_{\sigma }\partial_{t} +
(c'w)_{\sigma }\partial_{\varphi } + a_{\sigma }\partial_{s} +
\partial_{\theta } .
\\
\partial_{v} &= (q+cw_{v})\partial_{t} + (q'+c'w_{v})\partial
_{\varphi } + a_{v}\partial_{s} .
\end{split}
\end{equation}
Now let $(a, w)$ and $(a', w')$ denote the respective versions of the
parametrizing functions that come from the two sheets that are involved at
the given intersection point. Likewise, use $(\partial_{\sigma },
\partial_{v})$ and $({\partial_{\sigma}}', \partial_{v}')$ to
denote the corresponding versions of~\eqref{eq:4.33}. The convention used below takes
the unprimed pair as the image via the parametrizing map of the point
$(\sigma ,\hat {v})$ with $\hat {v}$ as in~\eqref{eq:4.30}. Meanwhile, the
primed pair is the image of $(\sigma , \hat {v}')$ with $\hat {v}'$ also
from~\eqref{eq:4.30}.

To establish transversality for the self-intersection point and to obtain
the local intersection number, first write $\partial_{\sigma }     \wedge
    \partial_{v}     \wedge     {\partial_{\sigma}}'  \wedge     {\partial_{v}}'$ as
    $\tau (\partial_{s}     \wedge     \partial_{t}     \wedge
    \partial_{\theta }     \wedge     \partial_{\varphi })$ with $\tau
 \in     \mathbb{R}$. As demonstrated below, $\tau $ is non-zero; thus, the
intersection is transversal. Granted this, the sign of $\tau $ is the sign
to take for the local sign of the self intersection point. With regard to
the upcoming expression for $\tau $, note that~\eqref{eq:4.14} finds $a_{\sigma } =
a'_{\sigma } = 0$ and $w_{v} = w'_{v}$ at the intersection point. Here is
$\tau $:
\begin{equation}\label{eq:4.34}
\tau    =    - (a_{v} - a'_{v})\Big[(c(w - w'))_{\sigma
}(q'+c'w_{v}) - (c'(w - w'))_{\sigma }(q+cw_{v})\Big].
\end{equation}
To evaluate~\eqref{eq:4.34}, use~\eqref{eq:4.14} to deduce that when $\rho $ is very small,
then
\begin{equation}\label{eq:4.35}
\begin{split}
a_{v} - a'_{v} &= 2 \varepsilon \sin(\pi (\wp +\hat
{x}\hat{u})) > 0,
\\
w_{\sigma}-w'_{\sigma}&= -\pi\wp x_{0}\frac{1}{{\alpha_{Q_e}(\sigma)}}\partial_{\sigma }\beta '
+ o(1),
\end{split}
\end{equation}
where the term denoted by $o(1)$ is bounded by a constant that depends only on
$P_{0}$ and $Q_{e}$. Granted~\eqref{eq:4.35} and granted that $w - w'$ is of the order
of unity, any very small $\rho $ version of~\eqref{eq:4.34} has the form $\tau  =
2\pi     \varepsilon     \wp  \sin(\pi \wp )x_{0}    \partial_{\sigma
}\beta ' + o(1)$, where the term designated as $o(1)$ is again bounded by a
constant that depends only on $P_{0}$ and $Q_{e}$. As~\eqref{eq:4.32} guarantees that
$\partial_{\sigma }\beta '$ is very large when $\rho $ is very small,
so $\tau  > 0$ as required for a transversal intersection with positive local
intersection number.

\subsection{Parametrizations near boundary circles with a trivalent vertex label}\label{sec:4d}

In this subsection, $o$ denotes a trivalent vertex in $T$ while $e$, $e'$ and $e''$
denote the three incident edges to $o$. In this regard, it is assumed here
that only one of these edges labels a cylinder in $\mathbb{R}\times(S^1 \times  S^{2})$
where $\theta  < \theta_{o}$. The
discussion for the case when $\theta  > \theta_{o}$ on only one of
$K_{e}$, $K_{e'}$ and $K_{e''}$ is not presented since it is identical to
the discussion that follows but for some obvious cosmetic alterations. This
understood, the edges $e$, $e'$ and $e''$ are distinguished as follows: First,
$\theta  < \theta_{o}$ only on $K_{e}$. Second, in the case that
$Q_{e'}$ and $Q_{e''}$ are not proportional (and thus not proportional to
$Q_{e}$), take $e'$ so as to make
\begin{equation}\label{eq:4.36}
q_{e'}{q_{e''}}' - {q_{e'}}'q_{e''} < 0.
\end{equation}
The story starts with a preliminary digression to set the stage. To begin
the digression, assume that $\rho_{e1}=\rho_{e'0}=\rho
_{e''0}$ in what follows, and use $\rho $ denote any of the three. Here
again, take $\rho  \ll \delta $. Now the function $\beta' $ refers to the
version in~\eqref{eq:4.5} with this same value for $\rho $ and with
$\theta_{*}$ set to equal $\theta_{o}$.

Require the $e$, $e'$ and $e''$ versions of $\varepsilon $ in~\eqref{eq:4.1} to be constant
and much less than $\delta $ near the respective circles where $|
\sigma -\theta_{o}|  = 2\rho $, and require that the three
constants agree. Likewise, require that $a^{0}_{e}$, $a^0_{e'}$ and
$a^0_{e''}$ are constant near these circles; the values for these
constants are specified below. In addition, require that $w^{0}_{e}$,
$w^0_{e'}$ and $w^0_{e''}$ vanish near the respective circles where
$| \sigma -\theta_{o}|  = 2\rho $. Finally, require that
each of $v^{0}_{e}$, $v^{0}_{e'}$ and $v^0_{e''}$ is constant where
$| \theta -\theta_{o}|      \sim 2\rho $; the
values of the latter are also specified below.

As $\sigma$ approaches $\theta_{o}$, the functional form of $a$ and $w$ must
be modified for each of the three edges to accommodate the first three
constraints in~\eqreft33. To ease the proliferation of subscripts in the
subsequent discussion, agree now to use $\alpha $, $\alpha' $ and $\alpha''$
to denote the respective $Q = Q_{e}$, $Q_{e'}$ and $Q_{e''}$ versions of
$\alpha_{Q}(\sigma )$. Also, to avoid possible confusion, the
coordinate $v$ is used only to denote the $\mathbb{R}/(2\pi \mathbb{Z})$
coordinate on the parametrizing cylinder for $K_{e}$. The corresponding
coordinates for the $K_{e'}$ and $K_{e''}$ cylinders are denoted below by $v'$
and $v''$. Finally, to avoid a proliferation of primes, when $N = (n, n')$ and $K
= (k, k')$ are pairs of integers (or real numbers), then $[N, K]$ is used below
to denote $nk' - n'k$. For example, the distinction between the edges $e'$ and
 $e''$ that appears in~\eqref{eq:4.36} can now be written as $[Q_{e'}, Q_{e''}] < 0$.

The first step to defining the three versions of $(a, w)$ where $\sigma$ is
near $\theta_{o}$ is to specify each near the point or points on the
$\sigma =\theta_{o}$ boundary of its parametrizing cylinder that will
map to the point in the mutual intersection of the closures of $K_{e}$,
$K_{e'}$ and $K_{e''}$. In this regard, there are two such `singular points'
on the $\sigma =\theta_{o}$ circle in the $e$ version of the
parametrizing cylinder, and one each on the $\sigma =\theta_{o}$
circle in the $e'$ and $e''$ versions of the parametrizing cylinder. Here, the
singular points on the $e$ version of the $\sigma =\theta_{o}$ circle
have $\mathbb{R}/(2\pi \mathbb{Z})$ coordinates 0 and $2\pi
\frac{\alpha '}{\alpha }$, the singular point on the $e'$
version of this circle has $\mathbb{R}/(2\pi \mathbb{Z})$ coordinate 0, while
that on the $e''$ circle has $\mathbb{R}/(2\pi \mathbb{Z})$ coordinate
$2\pi \frac{\alpha'}{\alpha''}$.

To define the versions of $(a, w)$ near these singular points, let $(x, y)$
denote Cartesian coordinates for $\mathbb{R}^{2}$. Divide a small radius,
open disk centered at the origin in the $(x, y)$ plane into four open sets,
the four components of the complement of the locus where $x^{2} = y^{2}$.
The closures of the two components where $| y|  > | x|$
will be identified with respective open neighborhoods of the points $v = 0$
and $v = 2\pi     \frac{\alpha ' }{\alpha }$ on the $\sigma
=\theta_{o}$ circle in the parametrizing cylinder for $K_{e}$.
Meanwhile, the closure of the $x > | y| $ component will be
identified with an open neighborhood of the $v' = 0$ point on the the $\sigma
=\theta_{o}$ circle in the parametrizing cylinder for $K_{e'}$.
Finally, the closure of the $x < -| y| $ portion will be
identified with an open neighborhood of the $v'' = 2\pi \frac{\alpha'}{\alpha''}$ point on the $\sigma =\theta
_{0}$ circle in the parametrizing cylinder for $K_{e''}$. These
identification are made as follows: In all cases,
\begin{equation}\label{eq:4.37}
\sigma    = \theta_{o} + x^{2}- y^{2}.
\end{equation}
Thus, $\sigma =\theta_{o}$ where $x^{2} = y^{2}$. Meanwhile,
\begin{equation}\label{eq:4.38}
\begin{aligned}
v &= \frac{1}{\alpha }xy\quad\text{where } y > |x|, &
v &= \frac{1}{\alpha }xy + 2\pi \frac{\alpha ' }
{  \alpha }\quad\text{ where } y < -|x|,
\\
v' &= \frac{1}   {\alpha ' }xy\quad\text{ where } x > |y|, &
v'' &= \frac{1}{\alpha'' }xy + 2\pi
\frac{\alpha ' }   {\alpha''} \quad\text{ where }x < -|y|.
\end{aligned}
\end{equation}
Here, all assignments are defined modulo $2\pi \mathbb{Z}$.

The next step parametrizes the three versions of $(a, w)$ using the
coordinates $x$ and $y$. This is done as follows: In all cases,
\begin{equation}\label{eq:4.39}
a  \equiv  x .
\end{equation}
Meanwhile,
\begin{equation}\label{eq:4.40}
\begin{aligned}
w_{e}     &\equiv  y  &&\text{ where }y > | x| .
\\
w_{e}     &\equiv y - 2\pi \frac{[Q_{e' } ,Q_{e''} ]}{\alpha } &&\text{ where }y < -| x| .
\\
w_{e'}     &\equiv  y + \frac{[Q_e ,Q_{e' } ]}   {\alpha\alpha'} xy . &&
\\
w_{e''}     &\equiv  y + \frac{[Q_e ,Q_{e''} ]}{\alpha\alpha''} xy +
2\pi     \frac{[Q_e ,Q_{e''} ]}   {\alpha''}. &&
\end{aligned}
\end{equation}
Granted~\eqref{eq:4.37}--\eqref{eq:4.40}, the $e$, $e'$ and $e''$ versions
of~\eqref{eq:3.2} now define a
smooth map to $\mathbb{R}    \times  (S^1 \times  S^{2})$ from each of
the four components of the complement in a small disk about the origin in
the $x$--$y$ plane of the $x^{2} - y^{2} = 0$ locus. As explained below, the
resulting maps fit together across this locus so as to define a smooth,
symplectic embedding of the whole of some smaller radius disk into $\mathbb{R}
\times  (S^1    \times S^{2})$.

The preceding formulae write $(a_{e}, w_{e})$ near the points where $v = 0$
and $v = 2\pi \frac{\alpha '}{\alpha }$ on the $\sigma
= \theta_{o}$ boundary circle of $K_{e}$'s parametrizing cylinder. They
also give $(a_{e'}, w_{e'})$ near the $v' = 0$ point on the $\sigma = \theta_{o}$
in the parametrizing cylinder for $K_{e'}$, and they
give $(a_{e''}, w_{e''})$ near the point where $v'' = 2\pi
\frac{\alpha'}   {\alpha''}$ on the $\sigma
=\theta_{o}$ boundary of the parametrizing domain for $K_{e''}$. The
third step parametrizes the three versions of $(a, w)$ on a neighborhood of
the rest of the relevant $\sigma =\theta_{o}$ circle. Note that in
the equations that follow,
\begin{equation}\label{eq:4.41}
\beta_{*} \equiv \beta \bigg(\frac{1}{\rho} \big( x^{2} +
y^{2}\big)\bigg).
\end{equation}
To start, consider $(a_{e}, w_{e})$. In this case, take
\begin{equation}\label{eq:4.42}
\begin{aligned}
a_{e}&=\beta_{*} x + (1-\beta_{*})    \varepsilon
\bigg(\cos\bigg(v - \pi \frac{\alpha '}{\alpha }\bigg) - \cos\bigg(\pi
\frac{\alpha '}{\alpha }\bigg)\bigg) \\
\text{and }\quad
w_{e}&=\beta_{*} y + \beta '\bigg(\beta \frac{[Q_e,Q_{e' } ]}{\alpha \alpha ' }xy +
\frac{[Q_e,Q_{e' } ]}{\alpha ' } v\bigg) - (1-\beta_{*})
\varepsilon \sin\bigg(v - \pi \frac{\alpha ' }{\alpha}\bigg)
\end{aligned}
\end{equation}
where $0  \leq  v  \leq 2\pi \frac{\alpha '}{\alpha}$ and
$| \sigma -\theta_{o}|  < \rho $. Where $2\pi
\frac{\alpha ' }{\alpha }     \leq  v  \leq 2\pi$ and
$| \sigma -\theta_{o}|  < \rho $, take
\begin{align}\label{eq:4.43}
a_{e}&=\beta_{*}  x + (1-\beta_{*})    \varepsilon  \bigg(\cos\bigg(v - \pi
\frac{\alpha ' }{\alpha }\bigg) - \cos\bigg(\pi
\frac{\alpha ' }{\alpha }\bigg)\bigg)\\
\text{and}\quad
w_{e}&=\beta_{*} y + \beta '\bigg(\beta_{*}\frac{[Q_e ,Q_{e''}]}{\alpha \alpha''} xy - 2\pi
\frac{[Q_e ,Q_{e''} ]}{\alpha''}+
\frac{[Q_e ,Q_{e''} ]}{\alpha''} v\bigg)
\nonumber\\
&\hspace{160pt}-(1-\beta_{*})
  \varepsilon\sin\bigg(v -\pi \frac{\alpha '}{\alpha }\bigg).\nonumber
\end{align}
The $e'$ version of $(a, w)$ requires the introduction of a $\sigma$--dependent
family of diffeomorphisms of the constant $\sigma$ circles in the $e'$
version of the parametrizing cylinder. The latter is denoted by $\phi
_{e'}$. View it as a map from the parametrizing cylinder to $\mathbb{R}/(2\pi \mathbb{Z})$
with the property that its restriction to any fixed
value of $\sigma$ define a diffeomorphism of $\mathbb{R}/(2\pi \mathbb{Z})$.
As such, it need only have the following two properties:
\begin{equation}\label{eq:4.44}
\begin{aligned}
\phi_{e'}(\sigma , v') &= \frac{\alpha '}{\alpha}
(v'-\pi )  &&\text{where both } | \sigma -\theta_{o}|  < 2\rho
^{4} \text{ and } x^{2} + y^{2} > \tfrac{1}{2}\rho.
\\
\phi_{e'}(\sigma , v') &= v'  &&\text{where }| \sigma -\theta_{o}| 
\ge  3\rho ^{4}.
\end{aligned}
\end{equation}
An analogous $e''$ version is denoted by $\phi_{e''}$. The latter should
obey
\begin{equation}\label{eq:4.45}
\begin{aligned}
\phi_{e''}(\sigma , v'') &= \frac{\alpha''}{\alpha }v'' - \pi
\frac{\alpha '}{\alpha}
&&\text{where both } |
\sigma -\theta_{o}|  < 2\rho ^{4} \text{ and } x^{2} +
y^{2} > \tfrac{1}{2}\rho.
\\
\phi_{e''}(\sigma , v') &= v'' &&\text{where }| \sigma -\theta
_{0}|      \ge  3\rho ^{4}.
\end{aligned}
\end{equation}
With $\phi_{e'}$ and $\phi_{e''}$ as above, what follows are the $e'$
and $e''$ versions of $(a, w)$:
\begin{equation}\label{eq:4.46}
\begin{aligned}
a_{e'}&=\beta_{*} x + (1-\beta ') \rho  + (1-\beta
_{*})    \varepsilon  \big(\cos(\phi_{e'}) - \cos(\phi
_{e'}|_{v' = 0})\big).
\\
w_{e'}&=\beta_{*} \bigg(y +\frac{[Q_e ,Q_{e' }]}{\alpha \alpha '} xy \bigg)
 - (1-\beta_{*})    \varepsilon
 \sin(\phi_{e'}).
\end{aligned}
\end{equation}
Meanwhile,
\begin{equation}\label{eq:4.47}
\begin{aligned}
a_{e''}&=\beta_{*} x - (1-\beta ') \rho  + (1-\beta
_{*})    \varepsilon  \big(\cos(\phi_{e''}) - \cos(\phi
_{e''}|_{v'' = 0})\big).
\\
w_{e''}&=\beta_{*} \bigg(y + \frac{[Q_e ,Q_{e'' }]}{\alpha \alpha''} xy\bigg) - (1-\beta_{*})
 \varepsilon  \sin(\phi_{e''}) + 2\pi     \beta'\frac{[Q_e ,Q_{e''}]} {\alpha''}.
\end{aligned}
\end{equation}
Three tasks lie ahead. Here is the first: Verify that the closures of
$K_{e}$, $K_{e'}$ and $K_{e''}$ fit together where they meet along the
$\theta =\theta_{o}$ locus to define a smoothly immersed surface. The
second task is to verify that the singular points are of the simplest sort:
Each centers a small radius ball that intersects the surface as a pair of
embedded disks that meet transversely at a single point. Here is the final
task: Verify the positivity of the local intersection number at each
singular point.

The three tasks are addressed next.

\step{Task 1}
The first step here is to verify that~\eqref{eq:4.37}--\eqref{eq:4.40} in conjunction
with~\eqref{eq:3.2} define maps to $\mathbb{R}    \times  (S^1 \times  S^{2})$ that
fit together across the $x^{2} = y^{2}$ locus to provide a smooth,
symplectic embedding of some small radius disk about the origin in the
$x$--$y$
plane. To be precise here, the resulting map from a small radius disk can be
written so that it sends a pair $(x, y)$ with norm $(x^{2} + y^{2})^{1 /
2} \ll \rho ^{4}$ to the point with coordinates
\begin{multline}\label{eq:4.48}
\Big(s = x,\ t = q_{e}\frac{1}{\alpha (\theta )}xy +
(1-3\cos^{2}\theta ) y,\\ \theta =\theta_{o}+x^{2}-y^{2},
\varphi  = {q_e}'\frac{1}{\alpha (\theta )}xy +
\surd 6\cos(\theta )y \Big)
\end{multline}
Here, both $t$ and $\varphi$ are defined modulo $2\pi \mathbb{Z}$. Note that
the differential at the origin of the map in~\eqref{eq:4.48} sends $\partial_{x}$
to $\partial_{s}$ and $\partial_{y}$ to the Reeb vector field
in~\eqreft16. Thus, it symplectically embeds a small radius disk about the origin.

By definition, the map in~\eqref{eq:4.48} agrees with that given where $y > |x| $
using~\eqref{eq:3.2} and~\eqref{eq:4.37}--\eqref{eq:4.40}. The verification that it agrees
with the maps from~\eqref{eq:3.2} and~\eqref{eq:4.37}--\eqref{eq:4.40}
on the other components of the
complement of the $x^{2} = y^{2}$ locus is left to the reader except for
the following comment: Algebraic manipulations can rewrite the various maps
to $\mathbb{R}    \times  (S^1 \times  S^{2})$ from the other components
so that they appear exactly as depicted in~\eqref{eq:4.48} but for the addition to $t$
and $\varphi$ of some integer multiple of $2\pi$.

The next step is to verify that the maps that are defined using~\eqref{eq:3.2} in
conjunction with~\eqref{eq:4.42}--\eqref{eq:4.47} extend to the $\sigma =\theta_{o}$
circle of each parametrizing cylinder so that the union of the $|
\theta -\theta_{o}|  < 3\rho $ portion of the resulting
images defines the image via a proper immersion of the complement in $S^{2}$
of three pairwise disjoint, closed disks. This is done as follows: Suppose
that $v  \in  (0, 2\pi \frac{\alpha ' }{\alpha })$. As
written, the formulae in~\eqref{eq:4.42} extend the definition of $(a_{e}, w_{e})$
to a small radius disk centered on $(\theta_{o}, v)$ in $(0, \pi )
\times     \mathbb{R}/(2\pi \mathbb{Z})$. Likewise,
when $v  \in  (2\pi \frac{\alpha '}{\alpha}, 2\pi )$, then~\eqref{eq:4.43} extend
the definition of $(a_{e}, w_{e})$ to a small radius disk centered on
$(\theta_{o}, v)$ in $(0, \pi )    \times     \mathbb{R}/(2\pi \mathbb{Z})$.
Meanwhile, if $v' \in  (0, 2\pi )$, then~\eqref{eq:4.44} together with~\eqref{eq:4.46}
extend the definition of $(a_{e'}, w_{e'})$ to a small radius disk about
$(\theta_{o}, v')$. Finally, if $v''  \ne 2\pi \frac{\alpha'}{\alpha''}
\mod(2\pi)$, then~\eqref{eq:4.45} and~\eqref{eq:4.47}
extend the definition of $(a_{e''}, w_{e''})$ to a small radius disk about
$(\theta_{o}, v'')$ in $(0, \pi )    \times     \mathbb{R}/(2\pi \mathbb{Z})$.

With these extensions understood, suppose again that $v \in  (0, 2\pi
\frac{\alpha ' }{\alpha })$. Then the extended $(a_{e}, w_{e})$
define via~\eqref{eq:3.2} an immersion into $\mathbb{R}    \times  (S^1
\times  S^{2})$ of a small radius disk centered at $(\theta_{o}, v)$.
Set $v' \equiv     \frac{\alpha}{\alpha ' }v$. Then~\eqref{eq:4.44}
and~\eqref{eq:4.46} define via~\eqref{eq:3.2} an immersion into
$\mathbb{R}    \times  (S^1\times  S^{2})$ of a small radius disk centered at
$(\theta_{o}, v')$.
Now, let $\sigma  \in  (0, \pi )$ and  $v \in     \mathbb{R}/2\pi$ be
such that $| \theta_{o}-\sigma | \ll\rho ^{4}$ and
such that the extension of $(a_{e}, w_{e})$ is defined at the point
$(\sigma , v)$ and that of $(a_{e'}, w_{e'})$ is defined at $(\sigma
,\frac{\alpha}{\alpha'}v)$. It then follows from~\eqref{eq:4.42},~\eqref{eq:4.44} and~\eqref{eq:4.46} that
\begin{equation}\label{eq:4.49}
a_{e}(\sigma , v) = a_{e'}\bigg(\sigma , \frac{\alpha}{\alpha'}v\bigg)
\qquad\text{and}\qquad
w_{e}(\sigma , v) = w_{e'}\bigg(\sigma ,\frac{\alpha} {\alpha '}v\bigg) +
\frac{1}{\alpha ' }
({q_e}'q_{e'} - q_{e}{q_{e'}}') v
\end{equation}
These last equalities imply that the $e$ and $e'$ versions of the extended maps
parametrize open subsets of a single immersed surface, this the union of the
$\theta  < \theta_{o}+3\rho $ portion of the closure of $K_{e'}$, the
closure of the portion of $K_{e}$ in the image of points $(\sigma , v)$ with
$\sigma  > \theta_{o}-3\rho $ and  $v \in  (0, 2\pi
\frac{\alpha}{\alpha ' })$, and the image via~\eqref{eq:4.48} of a
small radius disk centered at the origin in the $(x, y)$ plane.

A similar argument using~\eqref{eq:4.43},~\eqref{eq:4.45} and~\eqref{eq:4.47}
proves the analogous
statement for the union of the $\theta  < \theta_{o}+3\rho $ portion
of the closure of $K_{e''}$, the closure of the portion of $K_{e}$ that is
in the image of points $(\sigma , v)$ with $\sigma  > \theta
_{o}-3\rho $ and $v  \in  (2\pi \frac{\alpha}{\alpha'}, 2\pi )$, and the image via~\eqref{eq:4.48} of a small radius disk
centered at the origin in the $(x, y)$ plane. The details of the latter
argument are left to the reader.

\step{Task 2}
The task is to describe all of the immersion points. This task is
accomplished in five steps.

\substep{Step 1}
Note that $a_{e'} > 0$ except at $v' = 0 \in     \mathbb{R}/(2\pi \mathbb{Z})$ and $x = 0$, while $a_{e''} < 0$ save at
$v'' = 2\pi \frac{\alpha ' }{\alpha''}$ and
$x = 0$. However, both of these points correspond to the origin in the
$x$--$y$
coordinate disk. Thus, the closures of the portions of $K_{e'}$ and
$K_{e''}$ where $| \sigma -\theta_{o}|  < 3\rho $ are
disjoint save for the image of the origin in the $x$--$y$ coordinate disk. As a
consequence, it is sufficient to focus separately on the singularities in
the respective closures of $K_{e}$, $K_{e'}$ and $K_{e''}$.

As will now be explained, if $\varepsilon $ and $\rho $ are small, then the
closures of the $\theta  < \theta_{o}+3\rho $ portions of $K_{e'}$
and $K_{e''}$ lack singular points. To argue in the case of $K_{e'}$, note
that the variation in $w_{e'}$ is not greater than a multiple of $\rho  +
\varepsilon $, so with the latter very small, only the $\hat {x} = 0$ case
of~\eqref{eq:4.3} and~\eqref{eq:4.4} can appear. Furthermore, no $\hat {x} = 0$ and $\hat {y}
 \ne  0 \mod(m)$ versions of~\eqref{eq:4.3} and~\eqref{eq:4.4} can occur in this case with one
of $v'$ and $v' - 2\pi \frac{\hat {y}}{m}$ very close to 0 in
$\mathbb{R}/(2\pi \mathbb{Z})$. Indeed, such is the case because $a_{e'}$
achieves its unique maximum on any given constant $\sigma$ circle at the
origin in $\mathbb{R}/(2\pi \mathbb{Z})$. This understood, the existence of
any $\hat {x} = 0$ and $\hat {y}     \ne  0 \mod(m)$ solutions to~\eqref{eq:4.3}
and~\eqref{eq:4.4} is precluded by virtue of two facts: First, $\phi_{e'}$ is a
diffeomorphism. Second, if the respective values of the cosine function
agree at two distinct points in $\mathbb{R}/(2\pi \mathbb{Z})$, then the
corresponding values of the sine function do not.

Except for notational changes, the argument just given also proves the
assertion that the $\theta  < \theta_{o}-3\rho $ part of the closure
of $K_{e''}$ is embedded.

\substep{Step 2}
Turn next to the case of $K_{e}$. In this
regard, keep in mind that points $(\sigma , v)$ and $(\sigma ', v')$ in the
parametrizing domain are mapped to the same point if and only if the
conditions in~\eqref{eq:4.2} are obeyed for some integer pair $N$. Equivalently, the
conditions in~\eqref{eq:4.3} and~\eqref{eq:4.4} are obeyed for some pair $(\hat {x}, \hat
{y})     \in     \mathbb{Z}    \times     \mathbb{Z}/(m\mathbb{Z})$, and
\begin{equation}\label{eq:4.50}
 v' = v - 2\pi     \hat {x}\frac{\alpha_Z (\sigma )}   {\alpha
_{Q_e } (\sigma )} - 2\pi \frac{\hat {y}}{m} \mod (2\pi     \mathbb{Z})
\end{equation}
To start the story for $K_{e}$, note that $a_{e} > 0$ when $v  \in  (0,
2\pi \frac{\alpha '}{\alpha })$, while $a_{e} < 0$ in
the case that $v  \in  (2\pi \frac{\alpha ' }{\alpha}, 2\pi )$.
Thus, the respective images of the maps that are defined via~\eqref{eq:3.2}
by~\eqref{eq:4.42} and~\eqref{eq:4.43} are disjoint except where their domains overlap,
where $v = 0$ and $v = 2\pi \frac{\alpha '}{\alpha }$. On
both of these loci, $a_{e} = 0$. In any event, it is sufficient to consider
separately the cases where $v  \in  [0, 2\pi \frac{\alpha ' }{\alpha }]$ and where
$v  \in  [2\pi \frac{\alpha ' }{\alpha }, 2\pi ]$, but taking care not to double count any immersion
points that occur where $v = 0$ or $v = 2\pi \frac{\alpha '}{\alpha }$.

The next point to make is that the values of either $\beta_{*}$ or
$\beta'$ at any point mapping to a singular point must be non-zero when
$\varepsilon $ and $\rho $ are small. Indeed, with a reference to~\eqref{eq:4.42}, an
argument given previously establishes the existence here of a positive lower
bound for $\beta_{*}+    \beta '$ that depends only on
$\theta_{o}$, $Q_{e}$ and $Q_{e'}$. Thus, solutions to~\eqref{eq:4.3} and~\eqref{eq:4.4}
where $\sigma      \in (\theta_{o}-3\rho , \theta_{o}]$ can
occur only where $| \theta_{o}-\sigma |  < 2\rho ^{4}$,
or in the image of a point where $x^{2} + y^{2}     \leq  4\rho ^{2}$
via the map in~\eqref{eq:4.48}.

The subsequent discussion involves the number $\wp_{0}     \in  [0, 1)$
that is defined for each pair $(\hat {x}, \hat {y})     \in     \mathbb{Z}
\times     \mathbb{Z}/(m\mathbb{Z})$ by the condition
\begin{equation}\label{eq:4.51}
\wp_{0}=\hat {x}\frac{\alpha_Z (\theta_o )}{\alpha(\theta_o )}+
\frac{\hat {y}}{m} \mod(\mathbb{Z}).
\end{equation}
Because any small $\varepsilon $ and $\rho $ version of $|w_{e}| $ in~\eqref{eq:4.42} is a priori
bounded by a constant that depends
only on $\theta_{o}$, $Q_{e}$ and $Q_{e'}$, so the set of pairs $(\hat{x}, \hat {y})$
that allow a solution to~\eqref{eq:4.3} and~\eqref{eq:4.4} where $\beta
_{*} > 0$ or where $\beta ' > 0$ has size bounded by $\theta_{o}$, $Q_{e}$ and $Q_{e'}$.
Thus, the set of values of $\wp_{0}$ that
can arise from such pairs has a corresponding upper bound to its size.
Moreover, the possible values for $\wp_{0}$ in this set are determined
a priori by $\theta_{o}$, $Q_{e}$ and $Q_{e'}$.

With $\wp_{0}$ so defined, introduce the function of $\sigma$ given by
\begin{equation}\label{eq:4.52}
\wp \equiv \wp_{0}+\hat {x}\bigg(\frac{\alpha_Z (\sigma )}{\alpha (\sigma )}-
\frac{\alpha_Z (\theta_o )}{\alpha(\theta_o )}\bigg).
\end{equation}
Note that if $\sigma$ is such that $\beta ' > 0$, then $| \wp -\wp
_{0}|      \leq     \kappa     | \hat {x}|     \rho ^{4}$
where $\kappa $ is determined solely by $\theta_{o}$, $Q_{e}$ and
$Q_{e'}$. Note as well that the relation in~\eqref{eq:4.50} between $v'$ and $v$ can be
summarized succinctly by the formula $v' = v - 2\pi \wp  \mod(2\pi \mathbb{Z})$.

\substep{Step 3}
Suppose now that $(\sigma , v)$ is a
solution to some given $(\hat {x}, \hat {y})$ version of~\eqref{eq:4.3} and~\eqref{eq:4.4}
with $v  \in  [\frac{1}{\alpha }\rho ^{2}, 2\pi \frac{\alpha ' }{\alpha }-
\frac{1}{\alpha}\rho ^{2}]$ and $\beta '(\sigma ) > 0$. Granted~\eqref{eq:4.42}, the condition
in~\eqref{eq:4.4} is equivalent to
\begin{equation}\label{eq:4.53}
\cos\bigg(v - \pi \frac{\alpha '}{\alpha } - 2\pi \wp \bigg)
= \cos\bigg(v - \pi \frac{\alpha '}{\alpha }\bigg)
\end{equation}
If $\rho $ is small and $\sigma$ fixed, there are at most two values for $v$
that lie in the indicated range and satisfy~\eqref{eq:4.53}. To elaborate, without
the constraint on the domain of $v$, there are precisely two solutions to
equation in~\eqref{eq:4.53} for the given value of $\sigma$; one is the point $\pi
(\frac{\alpha '}{\alpha }+\wp )$, and the other is
$\pi (\frac{\alpha '}{\alpha }+\wp -1)$. Moreover, if
$\rho $ is sufficiently small and $\frac{\alpha '}{\alpha
} < \frac{1}{2}$, then at most one of these lies in the required
interval. If $\frac{\alpha '}{\alpha} > \frac{1}{2}$, then at least one of the two is in this interval.

In any event, if $\pi (\frac{\alpha ' }{\alpha }+\wp )$
lies in $[\frac{1}{\alpha }\rho ^{2}, 2\pi \frac{\alpha ' }{\alpha }-\frac{1}{\alpha}\rho ^{2}]$,
then so does $\pi (\frac{\alpha ' }{\alpha }-\wp )$; and if $\pi (\frac{\alpha ' }{\alpha}+\wp -1)$
lies in $[\frac{1}{\alpha }\rho ^{2}, 2\pi \frac{\alpha '}{\alpha }-\frac{1}{\alpha}\rho ^{2}]$,
then so does $\pi (\frac{\alpha ' }{\alpha }+1-\wp )$. This said, note that if any given
$(\hat {x}, \hat {y})'$s version of $\wp $ puts $\pi (\frac{\alpha ' }{\alpha }+\wp )$ in
$[\frac{1}{\alpha }\rho ^{2}, 2\pi \frac{\alpha ' }{\alpha }-\frac{1}{\alpha}\rho ^{2}]$,
then the corresponding version of $\wp $ for the pair
$(-\hat {x}, m-\hat {y})$ has $\pi (\frac{\alpha ' }{\alpha }+\wp -1)$ in this same interval.
The converse is also true.
Moreover, this correspondence does not alter the corresponding intersecting
disks in $K_{e}$ since it amounts to switching $v$ with $v'$. Thus, it is enough
to consider the case that $\pi (\frac{\alpha ' }{\alpha}+\wp )$ lies in the desired interval
$[\frac{1}{\alpha }\rho ^{2}, 2\pi \frac{\alpha ' }{\alpha }-\frac{1}{\alpha }\rho ^{2}]$.
Note that when $\rho $ is small, such is the
case if and only if
\begin{equation}\label{eq:4.54}
\wp_{0} < \frac{\alpha ' }{\alpha }|_{\sigma = \theta_o }
\end{equation}
To start the analysis, use~\eqref{eq:4.42} to write~\eqref{eq:4.3} as
\begin{equation}\label{eq:4.55}
\varepsilon \sin(\pi \wp )=\pi  m \frac{1}   {\alpha '}\bigg(\beta '\wp \frac{[Q_e ,Q_{e' } ]}{m} -
\frac{\alpha ' }{\alpha }\hat {x}\bigg).
\end{equation}
This last equation should be viewed as a condition on $\sigma$. In
particular, because $\beta'$ takes values in [0, 1], and because $[Q_{e},
Q_{e'}] = -[Q_{e'}, Q_{e''}]  \ge  0$, any small $\varepsilon $ and
$\rho  < \varepsilon ^{2}$ version of~\eqref{eq:4.55} can be satisfied by some
$\sigma      \in  [\theta_{o}-3\rho , \theta_{o}]$ if and only if
\begin{equation}\label{eq:4.56}
0 < \hat {x}\frac{\alpha ' }{\alpha }|_{\sigma
= \theta_o }  < \wp_{0}\frac{[Q_e ,Q_{e' }]}{m}.
\end{equation}
Moreover, this solution occurs at a value of $\sigma$ where
\begin{equation}\label{eq:4.57}
\begin{aligned}
&\beta '  \in (\kappa , (1-\kappa )).
\\
&\beta '\wp \frac{[Q_e ,Q_{e' } ]}{ m} -
\frac{\alpha ' }{\alpha }\hat {x} > \frac{1}{\kappa }\varepsilon.
\end{aligned}
\end{equation}
Here, $\kappa      \in (0, 1)$ is a constant that depends only on
$\theta_{o}$, $Q_{e}$ and $Q_{e'}$. Finally, if $\varepsilon $ is small,
if $\rho  < \varepsilon ^{2}$, and if~\eqref{eq:4.56} holds, then there is a
unique choice of $\sigma$ that solves~\eqref{eq:4.55}.

There is one last point to make for the cases when $\wp_{0}$ obeys~\eqref{eq:4.54}:
Distinct values for the pair $(\hat {x}, \hat {y})$ produce
disjoint singular points in $K_{e}$. To explain, note first that by virtue
of the fact that the respective values of $a_{e}$ that arise must be equal,
two choices for $(\hat {x}, \hat {y})$ can produce the same singular point
in $K_{e}$ only if the corresponding values for $\wp_{0}$ agree. Granted
this, if the resulting singular points have the same $\theta$ coordinate,
then~\eqref{eq:4.55} demands that the respective values for $\hat {x}$ agree. Thus,
so do the values for $\hat {y}$.

\substep{Step 4}
The story in this step concerns the cases
where~\eqref{eq:4.3} and~\eqref{eq:4.4} hold with a value of $v$ either in $[2\pi
\frac{\alpha '}{\alpha}-\frac{1}{\alpha}\rho ^{2}, 2\pi \frac{\alpha '}{\alpha}]$ or in
$[0, \frac{1}{\alpha}\rho ^{2}]$. As is explained next, no
solutions to~\eqref{eq:4.3} and~\eqref{eq:4.4} with such values for $v$ result in $K_{e}$
singularities singularities if $\varepsilon $ and $\rho $ are small, and if
$\theta_{o}$ is suitably generic.

To begin the explanation, assume for the moment only that $v$ is within
$\frac{1}{\alpha}\rho ^{2}$ of either 0 or $2\pi \frac{\alpha '}{\alpha}$. If $v \in [2\pi
\frac{\alpha '}{\alpha}-\frac{1}{\alpha}\rho ^{2}, 2\pi \frac{\alpha '}{\alpha}]$,
then~\eqref{eq:4.4} requires that $\wp_{0}$ is either 0 or equal to the value of
$\frac{\alpha '}{\alpha}$ at $\theta_{o}$. If $v  \in [0, \frac{1}{\alpha}\rho ^{2}]$, then $\wp_{0}$ is
either zero or equal to the value of $1-\frac{\alpha '}{\alpha}$ at $\theta_{o}$. In this regard, note
that when $\varepsilon $
and $\rho $ are small, the case where $\wp_{0} = 0$ requires $\hat {x} = 0$ and thus $\hat {y} = 0$ as well.
Thus, the $\wp_{0} = 0$ case does not
lead to a singularity in $K_{e}$. Meanwhile, at the risk of replacing the
pair $(\hat {x}, \hat {y})$ with $(-\hat {x}, m-\hat {y})$, it is
sufficient to study the case where $v  \in  [2\pi \frac{\alpha'}{\alpha}-\frac{1}{\alpha}\rho ^{2}, 2\pi
\frac{\alpha '}{\alpha}]$ and where $\wp_{0}$ is
equal to the $\theta_{o}$ value of $\frac{\alpha '}{\alpha}$.

To see what this last condition implies, write $Q_{e'} = [Q_{e'},
Q_{e}] Z + \frac{w}{m}Q_{e}$ where  $w \in     \mathbb{Z}$. Doing
so identifies
\begin{equation}\label{eq:4.58}
\frac{\alpha '}{\alpha}    = [Q_{e'}, Q_{e}]
\frac{\alpha_Z}{\alpha}+\frac{w}{m}
\end{equation}
and so the condition on $\wp_{0}$ requires that
\begin{equation}\label{eq:4.59}
([Q_{e'}, Q_{e}] - \hat {x})\frac{\alpha_Z}{\alpha} = 0 \mod\bigg(\frac{1}{m}\mathbb{Z}\bigg).
\end{equation}
Now, if $\theta_{o}$ is suitably generic, then the ratio
$\unfrac{\alpha_Z}{\alpha}$ will be irrational, and so the only
solution to~\eqref{eq:4.59} is that where $\hat {x} = [Q_{e'}, Q_{e}]$. However,
with $\varepsilon $ and $\rho $ small, a glance at~\eqref{eq:4.42} shows that such a
value for $\hat {x}$ is incompatible with the condition in~\eqref{eq:4.52} unless
$[Q_{e'}, Q_{e}]$ and $\hat {x}$ both vanish. Indeed, such is the case
because
\begin{equation}\label{eq:4.60}
w_{e}(\sigma , v-2\pi \wp )- w_{e}(\sigma , v) = 2\pi
\beta '\frac{[Q_e ,Q_{e' } ]}{\alpha } + o(\varepsilon
+\rho ),
\end{equation}
and this has the same sign as $-\hat {x}$ in the case that $\hat {x} =
[Q_{e'}, Q_{e}]$, neither are zero and both $\varepsilon $ and $\rho $
are small.

To rule out the case that both $[Q_{e'}, Q_{e}] = 0$ and $\hat {x} = 0$,
note that~\eqref{eq:4.42} demands that the resulting singularity in $K_{e}$ is the
image of two points in the $(x,y)$ plane via the map in~\eqref{eq:4.48} where one has
the form $(0, y)$ and the other $(0, -y)$. Moreover,~\eqref{eq:4.4} demands that $y$ obey
\begin{equation}\label{eq:4.61}
\beta_{*}|y|  = -(1-\beta_{*})    \varepsilon
\sin\bigg(\pi \frac{w}{m}\bigg)
\end{equation}
where  $w \in  \{1, \ldots, m-1\}$ is the integer that appears
in~\eqref{eq:4.58}. Since $\pi \frac{w}{m}=\frac{\alpha '}{\alpha}$ in this case and since
$\frac{\alpha '}{\alpha} \in  (0, 1)$, the right hand side of~\eqref{eq:4.61} is non-positive and
the left hand side is non-negative. Since the two sides can not vanish
simultaneously, there are no values of $y$ that make~\eqref{eq:4.61} hold.

In the case that $\theta_{o}$ is special and so there is an $\hat {x}
\ne [Q_{e'}, Q_{e}]$ solution to~\eqref{eq:4.59}, there may well be solutions
to~\eqref{eq:4.3} and~\eqref{eq:4.4} with $v  \in  [2\pi \frac{\alpha '}{\alpha}-\frac{1}{\alpha}\rho ^{2}, 2\pi
\frac{\alpha '}{\alpha}]$. These can be analyzed with
much the same machinery as used for when $v$ is further from either 0 or $2\pi
\frac{\alpha '}{\alpha}$. To keep an already long story
from getting longer, this task is left to the reader, as is the task of
verifying that the resulting singularities of $K_{e}$ are transversal with
local self-intersection number 1.

\substep{Step 5}
This final step characterized the
singularities in the closure of $K_{e}$ that lie where $| \theta
_{0}-\sigma |  < 3\rho $ and $v  \in  [2\pi \frac{\alpha '}{\alpha}, 2\pi ]$. The story here is
much as in \refstep{Step 3} and \refstep{Step 4}. First, when $\theta_{o}$ is sufficiently
generic, there are no solutions that lie in the image via the map in~\eqref{eq:4.48}
of points where $\beta_{*}     \ne  0$. Second, if $\varepsilon $
and $\rho $ are sufficiently small, then all singularities in $K_{e}$ must
lie where $\beta '(\theta )$ is bounded away from zero by a constant that
depends only on $\theta_{o}$, $Q_{e}$ and $Q_{e'}$. Third, if $(\sigma ,v)$ and $(\sigma ', v')$
map to the same point in $K_{e}$, then $\sigma  =\sigma '$ and $v' = v -2\pi \wp  \mod(2\pi \mathbb{Z})$ where $\wp $ is
defined as in~\eqref{eq:4.52}. Here, $\wp_{0}     \in $ [0, 1) is defined from some
pair $(\hat {x}, \hat {y})$ as in~\eqref{eq:4.51}. In addition,~\eqref{eq:4.53} must hold.
As a consequence, if $v$ is taken to be a real number in $[2\pi
\frac{\alpha '}{\alpha}, 2\pi ]$, then $v$ must have
one of two forms: Either $v = \pi (1+\frac{\alpha '}{\alpha}+\wp )$ or else
$v = \pi (\frac{\alpha '}{\alpha}+\wp )$. In the former case, $1-\wp  > \frac{\alpha '}{\alpha}$
and in the latter, $\wp  > \frac{\alpha '}{\alpha}$. For essentially the same reasons as before, it is only
necessary to consider one of these two sorts of cases, and so the discussion
below makes the assumption that $v = \pi (1+\frac{\alpha '}{\alpha}+\wp )$ and that $1-\wp  >
\frac{\alpha '}{\alpha}$. In this regard, note that when $\rho $ is small, then the
latter condition holds if and only if $1-\wp_{0} > \frac{\alpha'}{\alpha}$.

With this last assumption understood, then $v'= \pi (1+\frac{\alpha'}{\alpha}-\wp )$ is also in
$[2\pi \frac{\alpha '}{\alpha}, 2\pi ]$ and $v > v'$. This being the case, a
referral to~\eqref{eq:4.43} finds~\eqref{eq:4.3} equivalent to the condition
\begin{equation}\label{eq:4.62}
-\varepsilon \sin(\wp )=\pi  m\frac{1} {\alpha''}\bigg(\beta '\frac{[Q_e ,Q_{e''}]}{m}\wp
-\frac{\alpha''}{\alpha}\hat {x}\bigg).
\end{equation}
As with its analog in~\eqref{eq:4.55}, this should be viewed as an equation for
$\sigma$. As such, it has a solution if and only if
\begin{equation}\label{eq:4.63}
\frac{[Q_e ,Q_{e''}]}{m}\wp_{0}<    \hat{x}\frac{\alpha''}{\alpha}\bigg|_{\sigma =
\theta_o } < 0.
\end{equation}
Moreover, when $\varepsilon $ is small and $\rho $ is very small, then the
solution is unique and it occurs where
\begin{equation}\label{eq:4.64}
\begin{aligned}
&\beta '  \in (\kappa , 1-\kappa ).
\\
&\beta '\frac{[Q_e ,Q_{e''}]}{m}\wp  -
\frac{\alpha''}{\alpha}\hat {x}< \frac{1}{\kappa}\varepsilon ,
\end{aligned}
\end{equation}
where $\kappa      \in (0,1)$ is a constant that depends only on
$\theta_{o}$, $Q_{e}$ and $Q_{e'}$.

An argument from \refstep{Step 3} also applies here to prove that each singular point
in the $\beta ' > 0$ and $v  \in  [2\pi \frac{\alpha '}{\alpha}, 2\pi ]$ portion of $K_{e}$ lies in a ball whose
intersection with $K_{e}$ is the union of two embedded disks that meet only
at their centers.

\step{Task 3}
The task here is to verify that small $\varepsilon $ and very small $\rho$
guarantees that the singularities in the $\theta  > \theta_{o}-3\rho $ portion of $K_{e}$ are
those of transversally intersecting
disks with positive local intersection number. For this purpose, use $(a, w)$
and $(a', w')$ to denote the respective versions of the parametrizing
functions that come from the two disk that are involved at the given
intersection point. Use $(\partial_{\sigma }, \partial_{v})$ and
$({\partial_{\sigma}}', \partial_{v}')$ to denote the corresponding
versions of the push-forward that are depicted in~\eqref{eq:4.33}. Consider first the
case where the inverse image in the parametrizing cylinder of the singular
point in the unprimed disk is a point $(\sigma , v)$ with $v = \pi
(\frac{\alpha '}{\alpha}+\wp )$ and with $\wp  < \frac{\alpha '}{\alpha}$. In this case, the primed pair
is the image of $(\sigma , v')$ with $v' = \pi (\frac{\alpha '}{\alpha}-\wp )$. Thus, both
$v$ and $v'$ lie in the interval $(0, 2\pi\frac{\alpha '}{\alpha})$.

In the present situation, the intersection is transversal if $\tau $ as
defined in~\eqref{eq:4.34} is non-zero; and then the sign of $\tau $ is the local
sign at the intersection point. In the case at hand, a referral to~\eqref{eq:4.42}
finds that
\begin{equation}\label{eq:4.65}
\begin{aligned}
a_{v} - a'_{v} &= -2\varepsilon  \sin(\pi \wp ) < 0.
\\
w_{\sigma } - w'_{\sigma } &= 2\pi \wp \frac{[Q_e ,Q_{e'}]}{\alpha '}\beta '_{\sigma }+ o(1).
\end{aligned}
\end{equation}
Here, the term that is designated as $o(1)$ is uniformly bounded no matter how
small $\varepsilon $ and $\rho $. By virtue of~\eqref{eq:4.57}, ${\beta '}_{\rho }$
is bounded from below by $\kappa '\rho ^{{-}4}$, with $\kappa ' > 0$
depending solely on $\theta_{o}$, $Q_{e}$ and $Q_{e'}$. Meanwhile,
$[Q_{e}$, $Q_{e'}]$ = -$[Q_{e'}$, $Q_{e''}]$ and thus is positive.
Therefore, $\tau $ is positive and so the local sign at the intersection
point is $+1$.

To continue with the assigned task, the second case to consider are those
intersections where $v = \pi (\frac{\alpha '}{\alpha}+1+\wp )$ and $1-\wp  > \frac{\alpha '}{\alpha}$.
Here, $v' = \pi (\frac{\alpha '}{\alpha}+1-\wp )$.
Thus, the self-intersections that occur in the case at hand occur at values
of $v$ and $v'$ that lie in the interval $(2\pi \frac{\alpha '}{\alpha}, 2\pi )$.
This noted, then referral to~\eqref{eq:4.43} finds
\begin{equation}\label{eq:4.66}
\begin{aligned}
a_{v} - a'_{v} &= 2\varepsilon  \sin(\pi \wp ) > 0.
\\
w_{\sigma } - w'_{\sigma } &= 2\pi \wp     \frac{[Q_e ,Q_{e''}]}{\alpha''}\beta '_{\sigma } + o(1),
\end{aligned}
\end{equation}
where the term designated as $o(1)$ has the same properties as its analog
in~\eqref{eq:4.65}. In this case, the second line in~\eqref{eq:4.66} is very negative when $\rho
$ is small, so $\tau $ is again positive and the local intersection number
is equal to 1.

\subsection{Intersections between distinct cylinders}\label{sec:4e}

The purpose of this next to last subsection is to complete the
proof of \fullref{thm:3.1} in the case that all partition sets
that define $T$ have a single element. In this regard, the task
here is to verify that the intersections between cylinders
$K_{e}$ and $K_{e'}$ when $e \ne  e'$ are distinct edges in the
graph $T$ are transversal with positive local intersection number.

To see how such a guarantee can be made, suppose that $e$ is any given edge.
The immersion constructed in the preceding subsections that defines $K_{e}$
involved the specification of data $\bigl\{\rho_{e0}, \rho_{e1},
\varepsilon_{e}, a^{0}_{e}, w^{0}_{e}, v^{0}_{e}\bigr\}$ where
$\varepsilon_{e}, a^{0}_{e}, w^{0}_{e}$ and $v^{0}_{e}$ are
functions of the coordinate $\sigma$ on the parametrizing cylinder, and
where the other two are constant and positive. Although subsequent
subsections gave upper bounds for $\varepsilon_{e}$ near the boundaries
of the parametrizing cylinder there is no positive lower bound near these
boundaries. With $\varepsilon_{e}$ chosen near boundaries of the
parametrizing cylinder, upper bounds were then specified for the constants
$\rho_{e0}$ and $\rho_{e1}$, but there were no positive lower bounds.
As is explained below, the transversality of all intersections between all
pairs of distinct $K_{e}$ and $K_{e'}$ is guaranteed with careful choices
for the various versions of $\bigl\{\varepsilon_{e}, \rho_{e0}, \rho
_{e1}, a^{0}_{e}, w^{0}_{e}, v^{0}_{e}\bigr\}$. As should be
evident from details below, all new constraints on $\bigl\{\varepsilon_{e},
\rho_{e0}, \rho_{e1}, a^{0}_{e}, w^{0}_{e},
v^{0}_{e}\bigr\}$ are compatible with those given in the previous
subsections.

The discussion in the remainder of this subsection is divided into eight
parts.

\step{Part 1}
There is an immediate issue that arises when edges $e'$ and $e''$ have
monovalent vertices that share an angle assignment. Of concern is to choose
the $e = e'$ and $e = e''$ versions of $\{\varepsilon_{e}, \rho_{e0},
\rho_{e1}, a^{0}_{e}, w^{0}_{e}, v^{0}_{e}\}$ so as to
keep the resulting versions of $K_{e'}$ and $K_{e''}$ disjoint at values of
$\theta$ that approach the common vertex angle assignment. The discussion
of this issue addresses the respective cases where the angle label in
question is 0, $\pi$, and then neither 0 nor $\pi$.

\medskip{\bf The case of angle 0}

Let $o'$ and $o''$ denote the respective vertices on $e'$ and $e''$
with angle label 0. Now, $o'$ has a label from $\hat{A}$, either of the form
$(1,+,\ldots)$, $(1,-,\ldots)$ or simply
$\{1\}$. In the first case, $s  \to     \infty $ on $K_{e'}$ as $\theta
\to 0$, in the second case, the function s has a finite limit as $\theta  \to 0$,
and in the third case, $s  \to -\infty $ on $K_{e'}$ as
$\theta  \to 0$. Of course, $o''$ has one of these three sorts of labels also.

Now, suppose that $o'$ is labeled by an element of the form $(1,+,\ldots)$ from $\hat{A}$ and $o''$
by either an element $\{1\}$ from
 $\hat{A}$ or one of the form $(1,-,\ldots)$. In this case, make
both $a^0_{e'}$ and $a^0_{e''}$ constant at values of $\sigma$
where either the $e'$ or $e''$ version of the relevant~\eqref{eq:4.6} or~\eqref{eq:4.8} holds.
Choose these constants so that $a^0_{e'} \gg a^0_{e''}$; this then
makes $K_{e'}$ disjoint from $K_{e''}$ where either version of which ever
of~\eqref{eq:4.6} or~\eqref{eq:4.8} holds. A similar choice for $a^0_{e'}$ and
$a^0_{e''}$ guarantees this same conclusion when $o'$ is labeled by
 $\{1\}$ and $o''$ by $(1,-,\ldots)$.

Suppose next that $[Q_{e'}$, $Q_{e''}] < 0$ and that both $o'$ and $o''$ are
labeled by $(1,+,\ldots)$ elements from $\hat{A}$. Granted this,
take $\rho_{e'0} \gg \rho_{e'0}$, and take $a^0_{e'}$ and
$a^0_{e''}$ to be constant where either version of~\eqref{eq:4.6} holds with
$a^0_{e'} \gg a^0_{e''}$. This makes $K_{e'}$ and $K_{e''}$
disjoint where either version of~\eqref{eq:4.6} holds. In the case that $[Q_{e'},Q_{e''}] < 0$ and both $o'$ and
$o''$ are labeled by an $(-1,-,\ldots)$ element from $\hat{A}$, now take $\rho_{e'0} \ll\rho_{e''0}$,
and take $a^0_{e'}$ and $a^0_{e''}$ to be constant where either
version of~\eqref{eq:4.6} holds, but keep $a^0_{e'} \gg a^0_{e''}$.

To continue, suppose that $[Q_{e'}$, $Q_{e''}] = 0$. The simplest case has
both $o'$ and $o''$ labeled by (1). Here, it is sufficient to take
$a^0_{e'}$ and $a^0_{e''}$ to be constant where either version
of~\eqref{eq:4.8} holds, but with $|a^0_{e'} - a^0_{e''}|  >1$. Such a choice makes
$K_{e'}$ and $K_{e''}$ disjoint where either version
of~\eqref{eq:4.8} holds.

Here is the story when $[Q_{e'}, Q_{e''}]=0$ and both $o'$ and $o''$ are
labeled by $(1,+,\ldots)$ elements from $\hat{A}$, or both by
$(1,-,\ldots)$ elements from $\hat{A}$. In either case, make
$\rho_{e'0}=\rho_{e''0}$, both $\varepsilon_{e' }$and
$\varepsilon_{e''}$ much less than 1. Then, make both $a^0_{e'}$ and
$a^0_{e''}$ constant where either version of~\eqref{eq:4.6} holds, but with
$|a^0_{e''}|  \gg | $$a_{e''}| > 1$. This then
guarantees that $K_{e'}$ is disjoint from $K_{e''}$ where~\eqref{eq:4.6} holds. and
both $\varepsilon_{e'}$ and $\varepsilon_{e''}$ much less than 1 to
insure that $K_{e'}$ and $K_{e''}$ are disjoint where either version of~\eqref{eq:4.6} holds.

\medskip{\bf The case of angle $\pi$}

Each of the angle 0 subcases just described has a very
evident angle $\pi$ analog and vice versa. The stories for the
corresponding angle 0 and angle $\pi$ subcases are identical save for some
notation and sign changes. This understood, the angle $\pi$ cases are left
to the reader save for the following equations that give the angle $\pi$
versions of~\eqref{eq:4.6} and~\eqref{eq:4.7}:
\begin{equation}\label{eq:4.67}
\begin{aligned}
a_{e}&=\frac{1}{\kappa}    \beta ' \ln(\pi -\sigma ) +
a^{0}_{e} + \big(\varepsilon  (1-\beta ')+ (\pi -\sigma )\beta '\big) \cos(v + v^{0}_{e}).
\\
w_{e} &= (1-\beta ') w^{0}_{e} - \big(\varepsilon  (1-\beta ')+ (\pi -\sigma)\beta '\big)
\sin(v + v^{0}_{e}).
\end{aligned}
\end{equation}
Here, $\varepsilon \equiv \varepsilon_{e}$ and
\begin{equation}\label{eq:4.68}
\kappa  \equiv    -
\frac{q_e ' }{q_e} +
\sqrt {\frac{3}{2}},
\end{equation}
Note that the angle $\pi$ version of~\eqref{eq:4.8} has the same form as the
original.

\medskip{\bf The case of neither 0 nor $\pi$}

Let $e'$ and $e''$ again
denote the two edges that are involved. The first point to make is that
$K_{e'} \cap K_{e''} = \emptyset$ if the vertex $o'$ has the smaller angle
label of the two vertices on $e'$ while $o''$ has the larger of the two angle
labels of the vertices on $e''$. This understood, consider the case where the
angle labels of $o'$ and $o''$ are either both the smaller of the two angle
labels on their incident edges. In this case, take $\rho_{0e'}=\rho
_{0e''}$, with both much less than 1. Likewise, choose $\varepsilon
_{e}$ and $\varepsilon_{e'}$ to be very small. Finally, take both
$a^0_{e'}$ and $a^0_{e''}$ to be constant with $\bigl|
$$a^0_{e'} - a^0_{e''}\bigr|  > 2\pi$ where~\eqref{eq:4.9} is valid.
This makes $K_{e}$ and $K_{e'}$ disjoint where~\eqref{eq:4.9} holds.

\step{Part 2}
There is also an issue to address in the case that two bivalent vertices
in $T$ have the same angle assignment. Let $o$ denote the first and let $e$ and $e'$
denote its incident edges using the usual convention where $\theta_{o}$
is the larger of the angles that are assigned to the vertices on $e$. Let
$\hat{o}$ denote a vertex with $\theta_{\hat{o}}=\theta_{o}$, and let
$\hat{e}$ and $\hat{e}'$ denote the corresponding incident edges. Let $Y_{o}$ denote
the closure of $K_{e}     \cup K_{e'}$ and let $Y_{\hat{o}}$ denote that of
$K_{\hat{e}}     \cup K_{\hat{e}'}$. Since $s  \to     \infty $ on both $Y_{o}$
and $Y_{\hat{o}}$ along certain paths where $\theta$ limits to $\theta
_{0}$, these two subvarieties may well intersect where $\theta$ is near
$\theta_{o}$. The goal is to insure that the intersection points are
transversal with positive intersection number.

For this purpose, keep in mind that both $Y_{o}$ and $Y_{\hat{o}}$ converge in
$S^{1} \times S^{2}$ as multiple covers of $\theta =\theta_{*}$ Reeb orbits.
Let $\gamma_{o}$ and $\gamma_{\hat{o}}$ denote the
latter. Note that $\gamma_{o}$ is determined by the $\theta =\theta
_{0}$ value for the parameter $v^{0}_{e}$, and $\gamma_{\hat{o}}$ is
determined in an analogous fashion by $v^0_{\hat{e}}$. In particular, if
the respective constant values of $v^{0}_{e} = v^{0}_{e'}$ and
$v^0_{\hat{e}} = v^0_{\hat{e}'}$ near the $\sigma =\theta_{o}$
circles in the relevant parametrizing domains are chosen to be unequal and
sufficiently generic, then $\gamma_{o}$ and $\gamma_{\hat{o}}$ will be
distinct Reeb orbits. Choose these two angles to insure that such is the
case.

To continue, take $\rho_{e1} \ll \rho_{\hat{e}1}$ and take
$a^{0}_{e}$ and $a^0_{\hat{e}}$ both constant with $a^0_{\hat{e}} \gg a^{0}_{e}$
at points in the respective parametrizing cylinders where
$\sigma$ is within $3\rho_{\hat{e}}$ of $\theta_{o}$. In
particular, choose $a^0_{\hat{e}} \gg a^{0}_{e} -2\ln (\rho_{e1})$.
Likewise, make $a^0_{e'}$ and $a^0_{\hat{e}'}$ both constant with
$a^0_{\hat{e}'} \gg a^0_{e'}$ at points where $\sigma  < \theta
_{0}+3\rho_{\hat{e}}$ in their respective parametrizing cylinders. If
$\rho_{\hat{e}1}$ is sufficiently small, then these choices have the
following consequences: First, all intersections between $Y_{o}$ and
$Y_{\hat{o}}$ occur at points in $Y_{o}$ at very large s, in particular where
the $o$ version of the coordinates $(r, \tau )$ are defined and where the
corresponding $\beta_{*} = 1$. More to the point, these
intersection points occur where $Y_{o}$ looks very much like a multiple
cover of the $\mathbb{R}$--invariant cylinder $\mathbb{R}    \times     
\gamma_{o}$. Meanwhile, these intersection points occur in $Y_{\hat{o}}$ where the
$\hat{o}$ version of $\beta_{*}$ is zero.

Granted the preceding, keep in mind the following: Let $I$ denote an arc
with compact closure in an orbit of the Reeb vector field. Then $\mathbb{R}
\times  I$ has transversal intersections with the closure of any given
version of $K_{(\cdot )}$, and that these intersection points have positive
local intersection number. As can be verified using~\eqref{eq:4.33}, this is
a consequence of the positivity of the relevant version of the function
$\alpha_{Q}$.

Now, as remarked, if $\rho_{e1} \ll \rho_{\hat{e}1}$ and if
$a^{0}_{e} \ll a^0_{\hat{e}}$, then $Y_{o}$ looks very much like the
cylinder $\mathbb{R}    \times     \gamma_{o}$ where it intersects
$Y_{\hat{o}}$. Meanwhile, neighborhoods in $Y_{\hat{o}}$ of the intersection
points are constant translates along $\mathbb{R}$ in $\mathbb{R}    \times  (S^1 \times  S^{2})$
of a standard embedding. This understood, it
should not come as a surprise that these intersections are also transversal
and have positive intersection number. It is left to the reader as an
exercise with~\eqref{eq:4.14},~\eqref{eq:4.15} and~\eqref{eq:4.33} to verify that such is the case.

\step{Part 3}
This part of the discussion provides an overview of the strategy that is
used below to control the remaining intersections between distinct versions
of $K_{(\cdot )}$.

To start, suppose that $e'$ and $e''$ are edges of $T$, and that
$\theta_{* 0} < \theta_{*1}$ are values of $\theta$ on both
$K_{e'}$ and $K_{e''}$. In addition, suppose that both the $e = e'$ and $e = e''$ versions of the pair
$(a_{e}, w_{e})$ are given by~\eqref{eq:4.1} when $\sigma
     \in [\theta_{*0}, \theta_{*1}]$. It then follows
from~\eqref{eq:3.2} that $K_{e'}$ and $K_{e''}$ are disjoint provided that
\begin{equation}\label{eq:4.69}
| a^0_{e'} - a^0_{e''}|  > \varepsilon_{e'} +
\varepsilon_{e''} \text{ when }\sigma      \in  [\theta_{*0}, \theta
_{*1}].
\end{equation}
The pair $K_{e'}$ and $K_{e''}$ are said below to be `well separated' at a
given value, $\theta_{*}$, of $\theta$ if~\eqref{eq:4.1} describes the $e = e'$ and $e = e''$
versions of $(a_{e}, w_{e})$ and if the inequality
in~\eqref{eq:4.69} holds at $\sigma =\theta_{*}$. So as to avoid repetitive
qualifiers, the respective portions of two versions of $K_{(\cdot )}$ where
$\theta$ has a given range are also deemed `well separated' in the event
that one or both such portions is empty.

The strategy used below keeps the various versions of $K_{(\cdot )}$
pairwise well separated as much as possible. To implement the strategy,
first fix some positive constant $\rho $, smaller than the constant $\delta
$ that was introduced in \fullref{sec:4a}. Thus, $\rho$ is much smaller
than $\frac{1}{1000}$ times the difference between the larger
and smaller of the angles that label the vertices on any given edge of $T$.
Agree to make sure that all versions of $\rho_{e0}$ and $\rho_{e1}$
are much less than $\rho $. The plan is to keep the versions of $K_{(\cdot
)}$ pairwise well separated at angles with distance $2\rho $ or more
from the angles that label $T$'s vertices. To be precise, the various versions
of $a^{0}_{e}$ are taken to be locally constant on the complement in $[0,
\pi ]$ of the points with distance $2\rho $ or less from the finite set of
angles that label the vertices of $T$. Of course, these constant values are
chosen to insure that~\eqref{eq:4.69} is pairwise obeyed.

Granted the preceding, it is worth noting in advance those values of $\theta
$ where well separation must be abandoned. It proves useful for this purpose
to have on hand a particular proper immersion of $T$ into the rectangle $[-1,
1] \times  [0, \pi ]$. To define this immersion, first map the monovalent
and bivalent vertices to the boundary of the rectangle $[-1,1] \times  [0,
\pi ]$ in the following manner: Each vertex whose label from $\hat{A}$ has the
form $(\cdot ,-,\ldots)$ is placed on $\{-1\} \times [0, \pi ]$ by using its angle label for the $[0, \pi ]$
factor. The analogous rule places each vertex from $\hat{A}$ of the form $(\cdot ,+,\ldots)$ on
$\{1\}\times  [0, \pi ]$. Put each monovalent
vertex with label (1) from $\hat{A}$ on $(-1, 1) \times  \{0\}$, and put
each with a $(-1)$ label on $(-1, 1) \times     \pi$. In this regard, if $e$ is
the incident edge to such a vertex, set the horizontal coordinate of the
vertex equal to the value of $\tanh(a^{0}_{e})$ at either $\sigma  = 0$ or
$\sigma =\pi$ as the case may be. If $o$ is any given trivalent vertex,
use $\theta_{o}$ to denote its angle label, and place $o$ on $(-1, 1)
\times \theta_{o}$.

To finish the construction, it is necessary to identify each edge of $T$ with
an arc in the rectangle that runs between the relevant vertices. This is to
be done so that the interior of each arc avoids the boundary of the
rectangle and also avoids all vertices. In addition, the horizontal
coordinate on the rectangle must restrict without critical points to each
arc.

The vertical coordinate on the arc labeled by a given edge $e$ at any given
interior point is written below as $\tanh(s_{e})$ with $s_{e} \in \mathbb{R}$.

If $e'$ and $e''$ are distinct edges and if their representative arcs can be
drawn without interior intersections, then $K_{e'}$ and $K_{e''}$ can be
kept well separated. Indeed, this can be done along the following lines: Use
$\sigma$ to denote the vertical coordinate in the rectangle. As $\sigma$
has nowhere zero derivative on each arc, so the corresponding version of
$s_{(\cdot )}$ can be viewed as a function of $\sigma$. This understood,
identify $a^0_{e'}$ and $a^0_{e''}$ with $s_{e'}$ and $s_{e''}$. If
$\varepsilon_{e'}$ and $\varepsilon_{e''}$ are made very small, the
resulting $K_{e'}$ will then be well separated from $K_{e''}$.

As might be expected, crossing of these edge labeled arcs may be
unavoidable. To identify the necessary arc crossings, draw the edge labeled
arcs by starting at the top edge of the rectangle, $[-1, 1] \times \{\pi \}$,
and proceeding downwards. To conform to what has been said
already, all edge labeled arcs will be drawn as vertical arcs except perhaps
where the horizontal coordinate has distance $2\rho $ or less to an angle
that labels a vertex in $T$. Of course, distinctly labeled vertical arcs will
have distinct horizontal coordinates.

With the arcs drawn in this manner, the following are the only circumstances
that may require one arc to cross another:
\itaubes{4.70}
\textsl{If edges $e$ and $\hat{e}$ have monovalent vertices that share an angle assignment in $(0, \pi )$,
then a crossing of their arcs may be necessary to keep the $a^{0}_{e}$ and $a^0_{\hat{e}}$
assignments compatible with those given already in \refstep{Part 1}.
}

\item \textsl{Let $o$ denote a monovalent vertex with a label $(0,-,\ldots )$ and suppose that $e$ is the incident edge.
Let $\hat{e}$ denote a second edge whose vertices are assigned angles that are distinct from $\theta_{o}$.
The arcs labeled by $e$ and $\hat{e}$ must cross in the case that $s_{e}$
and $s_{\hat{e}}$ are both defined with $s_{e} > s_{\hat{e}}$ at the relevant $\sigma      \in \{\theta
_{0}\pm 2\rho \}$. It follows from~\eqref{eq:4.9} that no crossing is necessary if $s_{e} < s_{\hat{e}}$
at this value of $\sigma$.
}

\item \textsl{Let $o$ denote a bivalent vertex and suppose that $e$ is an edge that is incident to $o$.
The arc labeled by $e$ must cross that labeled by some other edge $\hat{e}$ if $s_{e}$ and $s_{\hat{e}}$
are both defined with $s_{e} < s_{\hat{e}}$ at the relevant $\sigma \in \{\theta_{o}\pm 2\rho \}$.
It follows from~\eqref{eq:4.14} and~\eqref{eq:4.15} that no crossing is necessary if $s_{e} > s_{\hat{e}}$
at this value of $\sigma$.
}

\item \textsl{Let $o$ denote a trivalent vertex that connects by two incident edges to vertices with larger angle label.
Denote these two edges by $e'$ and $e''$. Then the respective arcs labeled by $e'$ and $e''$ cross in the case that
$[Q_{e'}, Q_{e''}] < 0$ and $s_{e'} < s_{e''}$ at $\sigma =\theta_{o}+2\rho $.
It follows from~\eqref{eq:4.46} and~\eqref{eq:4.47} that no crossing is necessary if both $[Q_{e'}, Q_{e''}] < 0$
and $s_{e'} >s_{e''}$ at $\sigma =\theta_{o}+2\rho $, or if $[Q_{e'}, Q_{e''}] = 0$.
}

\item \textsl{Let $o$, $e'$ and $e''$ be as in the previous point. Let $\hat{e}$ denote a third edge and suppose that
$s_{\hat{e}}$ is defined at $\sigma =\theta_{o}+2\rho $ and suppose that it lies between
$s_{e'}$ and $s_{e''}$. Then the arc labeled by $\hat{e}$ must cross either that labeled by
$e'$ or that labeled by $e''$.
}

\item \textsl{Let $o$ and $o'$ both denote vertices with angle label 0. Then the respective arcs that are labeled by
the incident edges to $o$ and $o'$ may have to cross to keep the $a^0_{e'}$ and $a^0_{e''}$
assignments compatible with those given already in \refstep{Part 1} of this discussion.
}
\end{itemize}

\refstep{Part 5}--\refstep{Part 8} below address these various cases.

\step{Part 4}
This part of the story relates two observations that are subsequently
exploited in the case that $K_{e'}$ and $K_{e''}$ can not be kept well
separated.

\substep{Observation 1}
This observation concerns an
example where~\eqref{eq:4.69} holds at $\sigma =    \theta_{*0}$ and
$\sigma =\theta_{*1}$, fails in between, yet
$K_{e'}$ and $K_{e''}$ remain disjoint. In particular, if $Q_{e'}$ is
proportional to $Q_{e''}$, then the respective signs of $a^0_{e'} -a^0_{e''}$ can differ at
$\sigma =\theta_{*1}$ and at
$\sigma =\theta_{*0}$ with $K_{e'}$ still disjoint from
$K_{e''}$.

To explain, note that whether or not $Q_{e'}$ and $Q_{e''}$ are
proportional, a point $(\sigma , v')$ in the parametrizing cylinder for
$K_{e'}$ and a point $(\sigma , v'')$ in the parametrizing cylinder for
$K_{e''}$ are sent via the relevant versions of~\eqref{eq:3.2} to the same point in
$\mathbb{R}    \times  (S^1 \times  S^{2})$ if and only if the following
holds: There is an $\mathbb{R}$--valued lift, $\hat {v}'$, of $v'$, a
corresponding lift, $\hat {v}''$, of $v''$, and an integer pair $N = (n, n')$
such that
\begin{equation}\label{eq:4.71}
\begin{aligned}
\alpha
_{Q_{e' } }
\hat {v}' &= \alpha_{Q_{e''} } \hat {v}'' - 2\pi
\alpha_{N}
\\
a_{e'}(\sigma , v') &= a_{e''}(\sigma , v'')
\\
w_{e''}(\sigma , v') &= w_{e''}(\sigma , v'') - \frac{1}
{\alpha_{Q_{e'}}}[Q_{e'}, Q_{e''}] \hat {v}'' + 2\pi
\frac{1}  {\alpha_{Q_{e'}}}[Q_{e'}, N]
\end{aligned}
\end{equation}
Now, if $Q_{e'}$ is proportional to $Q_{e''}$ then the middle term in the
lowest line above is zero. Such being the case, let $\kappa $ denote the
maximum of the $Q = Q_{e'}$ version of $\alpha_{Q}$ over the interval
$[\theta_{*0}, \theta_{*1}]$. Then the third point
in~\eqref{eq:4.71} can not be met if $0 < | w_{e'} - w_{e''}|  < 2\pi
\frac{1}{\kappa}$ in this interval. This last condition can be
achieved by suitable choices of $w^0_{e'}$ and $w^0_{e''}$ if
$\varepsilon_{e}$ and $\varepsilon_{e''}$ are small. Of course, if the
third condition in~\eqref{eq:4.71} can not be met, then no amount of variation in
$a^0_{e'}$ and $a^0_{e''}$ on $[\theta_{*0}, \theta_{*1}]$ will make $K_{e'}$ intersect $K_{e''}$.

\substep{Observation 2}
Intersections between $K_{e'}$ and
$K_{e''}$ are allowed when transversal with $+1$ local intersection
number. This can always be arranged at values of $\theta$ in $(\theta_{*0}, \theta_{*1})$
if the following three conditions hold: First, $[Q_{e'}, Q_{e''}] < 0$. Second,~\eqref{eq:4.69} holds at
both $\sigma =\theta_{*0}$ and at $\sigma =\theta_{*1}$. Finally, $a^0_{e'} - a^0_{e''}$
is positive at $\sigma  = \theta_{*0 }$ but negative at $\sigma =\theta_{*1}$.

To explain this last claim, note that when both $\varepsilon_{e}$ and
$\varepsilon_{e'}$ are sufficiently small, then~\eqref{eq:4.33} dictates that the
intersections between $K_{e'}$ and $K_{e''}$ where $\theta  \in (\theta_{*0}, \theta_{*1})$
are transversal with sign
that of $(a^{0}_{e' }- a^0_{e''})_{\sigma } [Q_{e'},Q_{e''}]$.
Thus, if the given conditions are satisfied, then the variation
of $a^0_{e'}$ and $a^0_{e''}$ over $[\theta_{*0}, \theta_{*1}]$
can be arranged to guarantee the given conclusions.

\step{Part 5}
This part of the discussion considers the cases in the first and second
points of~\eqreft4{70}. In the situation outlined by the first point, any two
edges that are involved will have respective versions of $Q_{(\cdot )}$ that
are proportional. With this understood, then the first observation in the
preceding \refstep{Part 4} can be used to keep the corresponding versions of
$K_{(\cdot )}$ disjoint in spite of the crossing of the corresponding
arcs in $[-1, 1] \times  [0, \pi ]$.

Turn now to the second point in~\eqreft4{70}. Suppose that $o$ is a monovalent
vertex in $T$ with label $(0,-,\ldots)$ from $\hat{A}$. Suppose
first that $\theta$ takes values that are greater than $\theta_{o}$ on
$K_{e}$. If, as assumed, $a^{0}_{e} > a^0_{\hat{e}}$ at $\sigma  = \theta_{o}+2\rho $,
then a suitable modification of $a^{0}_{e}$ can
guarantee that the $\theta      \in (\theta_{o}, \theta_{o}+2\rho )$
intersection points between $K_{e}$ and $K_{\hat{e}}$ are
transversal and have $+1$ local intersection number.

To explain, take $\rho_{e0}$ to be very small, and let $\theta_{*0} \equiv \theta_{o}+2\rho_{e0}$ and
$\theta_{*1} \equiv     \theta_{o}+2\rho $. The assumption here is that
$a^0_{\hat{e}} - a^{0}_{e}$ is negative at $\theta_{*1}$. The
goal then is to modify $a^{0}_{e}$ inside the interval $(\theta_{*0}, \theta_{*1})$ so that
$a^0_{\hat{e}} - a^{0}_{e}$ is
positive at $\theta_{*0}$. For this purpose, keep in mind that the
pair $Q_{e}$ is equal to $m (p, p')$ with $m$ a positive integer and with $(p,
p')$ the relatively prime pair of integers that $\theta_{o}$ defines via~\eqref{eq:1.7}.
As such, there is a positive number, $\kappa $, such that
\begin{equation}\label{eq:4.72}
\alpha_{Q_{\hat{e}}}(\theta_{o})=\kappa [Q_e, Q_{\hat{e}}]
\end{equation}
Since the right hand side of~\eqref{eq:4.72} is positive, so $[Q_{e}, Q_{\hat{e}}] >
0$. This understood, if $\hat{e}$ and $e$ are respectively renamed as $e'$ and $e''$,
then the final observations in \refstep{Part 4} can be applied here to find the
desired modification of $a^{0}_{e}$.

Consider next the case that $\theta$ takes values on $K_{e}$ that are less
than $\theta_{o}$. In this case, set $\theta_{*0}=\theta
_{0}-4\rho_{e1}$ and set $\theta_{*1}=\theta
_{0}-2\rho $. By assumption, $a^{0}_{e}- a^0_{\hat{e}} > 0$ at
$\sigma =\theta_{*0}$, and the goal is to modify $a^{0}_{e}$
inside the interval $(\theta_{*0}, \theta_{*1})$ so that
$a^{0}_{e}- a^0_{\hat{e}}$ is negative at $\sigma =\theta
_{*1}$. For this purpose, note that $Q_{e}$ in this case is equal to
$-m (p, p')$ with $m$ a positive integer and with $(p, p')$ as before.
Thus,~\eqref{eq:4.72} holds with $\kappa  < 0$ and so $[Q_{\hat{e}}, Q_{e}] > 0$. In this
case, agree to relabel $\hat{e}$ as $e''$ as $e$ as $e'$. This done, then the
observations in \refstep{Part 4} again apply to give the desired modification of
$a^{0}_{e}$.

\step{Part 6}
This part considers the third point in~\eqreft4{70}. Consider here the case that
 $o$ is a bivalent vertex in $T$. Let $e$ and $e'$ denote the edges of $T$ that contain
 $o$ with the convention taken here that $\theta$ takes values that are less
than $\theta_{o}$ on $K_{e}$. Let $\rho_{*}$ denote the equal
values of $\rho_{e1}$ and $\rho_{e'0}$. Now, suppose that $\hat{e}$ is a
third edge and that $\theta_{o}$ is a value of $\theta$ on $K_{\hat{e}}$.
As noted in~\eqreft4{70}, in the case that $a^0_{\hat{e}} < a^{0}_{e}$ where
$\sigma = \theta_{o}-2\rho $ then it is a straightforward
consequence of~\eqref{eq:4.14} and~\eqref{eq:4.15} that the values of $a^0_{\hat{e}}$ can be
modified if necessary where $| \sigma -\theta_{o}|  < 2\rho $ so as to guarantee that
$K_{\hat{e}}$ is disjoint from the $| \theta -\theta_{o}|      \leq 2\rho $ part of the
closure of $K_{e}     \cup K_{e'}$.

On the other hand, if $a^0_{\hat{e}} - a^{0}_{e} > 0$ where $\sigma =\theta_{o}-2\rho $,
then it may not be possible to
modify $a^0_{\hat{e}}$ where $| \sigma -\theta_{o}|
\leq 2\rho $ so that $K_{\hat{e}}$ avoids the $\theta \in [\theta
_{0}-2\rho, \theta_{o}+2\rho ]$ portion of the closure of
$K_{e}     \cup K_{e'}$. However, as is explained next, there are
modifications that guarantee that all intersection points here are
transversal with $+1$ local intersection number.

To see how this such modifications come about, suppose that
$a^0_{\hat{e}} > a^{0}_{e}$ where $\sigma =\theta_{o}-2\rho $. In this case,
keep $a^{0}_{e}$ constant where $\sigma      \in [\theta_{o}-2\rho , \theta_{o}-\rho ]$ but make
$a^0_{\hat{e}}$ a non-decreasing function of $\sigma$ in this interval so
that the result is constant near $\theta_{o}-\rho $ and is such that
$a^0_{\hat{e}} \gg -4\ln(\rho_{*})$ at $\sigma =\theta_{o}-\rho $.
In particular, make this constant value greater
than $-4\ln(\rho_{*})$ plus the supremum of the values of
$a^{0}_{e}$ and $a^0_{e'}$ on the $| \sigma -\theta_{*}|  < 2\rho $
portions of their parametrizing cylinders. This done, keep $a^0_{\hat{e}}$ constant on $[\theta_{o}-\rho ,
\theta_{o}+\rho ]$, thus huge. Now, the larger this constant value for
$a^0_{\hat{e}}$, the closer $K_{e}     \cup K_{e'}$ is to an $\mathbb{R}$--invariant
cylinder where it comes near $K_{\hat{e}}$. With this in mind,
the argument used in \refstep{Part 2}
can be repeated in the case at hand to guarantee
that the intersection points between $K_{\hat{e}}$ and the closure of the
$\theta      \in [\theta_{o}-2\rho , \theta_{o}+2\rho ]$ part
of $K_{e}     \cup K_{e'}$ are transversal with $+1$ local intersection
numbers if the constant value of $a^0_{\hat{e}}$ on the interval $[\theta
_{0}-\rho , \theta_{o}+\rho ]$ is sufficiently large.

\step{Part 7}
This part considers the fourth and fifth points in~\eqreft4{70}. In the case of
the fourth point, the second observation of \refstep{Part 4} can be employed to
defined $a^0_{e'}$ and $a^0_{e''}$ where $\sigma      \in [\theta
_{0}+\rho , \theta_{o}+2\rho ]$ to the following effect: First,
all $\theta      \in [\theta_{o}, \theta_{o}+2\rho ]$
intersections between $K_{e'}$ and $K_{e''}$ occur where
$\theta      \in [\theta_{o}+\rho, \theta_{o}+2\rho ]$, are transversal, and
have +1 local intersection number. Second, $a^0_{e'} - a^0_{e''} > \varepsilon_{e'}+\varepsilon_{e''}$
at $\sigma =\theta_{o}+\rho $.

Consider next the situation that is described in the fifth point of~\eqreft4{70}.
To start, set the convention so that $[Q_{e'}, Q_{e''}] \leq 0$. Use $e$
to denote the third of the incident edges to $o$, and use $\rho_{*}$
to denote the common values of $\rho_{e1}$, $\rho_{e'0}$ and $\rho
_{e''0}$.

There are two cases to consider. The first case has $a^0_{e''} > a^0_{\hat{e}} > a^0_{e'}$
at $\theta =\theta_{o}+2\rho $.
With the first two observations of \refstep{Part 4} in mind, there is no cause for
concern in this case if either $[Q_{\hat{e}}, Q_{e''}] \leq 0$ or
$[Q_{\hat{e}}, Q_{e'}] \ge 0$. An argument that rules out the
possibility that both inequalities fail simultaneously invokes an identity
that concerns a set of four ordered pairs of real number: Denote the four
ordered pairs as $\{A_{k}\}_{k = 0,1,2,3}$, and here is the identity:
\begin{equation}\label{eq:4.73}
[A_{1}, A_{2}] [A_{3}, A_{0}] + [A_{2}, A_{3}] [A_{1},
A_{0}] + [A_{3}, A_{1}][A_{2}, A_{0}] = 0 .
\end{equation}
In this last equation, the bracket between a pair $A = (a, a')$ and another, $B= (b, b')$
is again defined by the rule $[A, B] = ab'-a'b$. This last identity
is now applied with $A_{0}$ equal to the value of $(1-3\cos^{2}\theta ,\surd 6\cos\theta )$ at
$\theta =\theta_{o}+2\rho $, $A_{1} = Q_{e'}$, $A_{2} = Q_{e''}$ and $A_{3} = Q_{\hat{e}}$. With these
assignments,~\eqref{eq:4.73} is equivalent to the assertion that
\begin{equation}\label{eq:4.74}
[Q_{e'}, Q_{e''}]
\alpha_{Q_{\hat{e}}}  + [Q_{e''}, Q_{\hat{e}}]
\alpha_{Q_{e' } }  + [Q_{\hat{e}}, Q_{e'}] \alpha
_{Q_{e''} }  = 0.
\end{equation}
Since the various versions of $\alpha_{Q}$ are positive and since
$[Q_{e'}, Q_{e''}] \leq  0$, this equality rules out the possibility
that both $[Q_{\hat{e}}, Q_{e''}] > 0$ and $[Q_{\hat{e}}, Q_{e'}] < 0$.

In the second case, $a^0_{e'} > a^0_{\hat{e}} > a^0_{e''}$ at
$\sigma =\theta_{o}+2\rho $. If $[Q_{e''}, Q_{\hat{e}}] \leq  0$
or if $[Q_{\hat{e}}, Q_{e'}] \leq  0$, then the second observation in
\refstep{Part 4} above can be applied to suitably modify $a^0_{\hat{e}}$ where $\sigma
 \in [\theta_{o}-2\rho, \theta_{o}+2\rho ]$ so that the
resulting version of $K_{\hat{e}}$ intersects the portion of the closure of
$K_{e}     \cup K_{e'}     \cup K_{e''}$ where $\theta \in [\theta_{o}-2\rho, \theta_{o}+2\rho ]$
transversally with $+1$ local intersection numbers. Of course, it may well be that both of these
inequalities go the wrong way.

What follows explains the story when both $[Q_{e''}, Q_{\hat{e}}] > 0$ and
$[Q_{\hat{e}}, Q_{e'}] > 0$. The first step here is to make $\varepsilon
_{e}, \varepsilon_{e'}$ and $\varepsilon_{e''}$ constant where
$| \sigma -\theta_{o}|  < 2\rho $. Meanwhile, decrease
$\varepsilon_{\hat{e}}$ between $\theta_{o}+2\rho $ and $\theta
_{0}+\rho $ so that the result is constant near $\theta_{o}+\rho$ and very much smaller than
$(\varepsilon_{e}\rho_{*})^{6}$. Extend $\varepsilon_{\hat{e}}$ to
$[\theta_{o}-2\rho_{*}, \theta_{o}+\rho ]$ as this constant. Next, vary
$a^0_{\hat{e}}$ as $\sigma$ decreases from $\theta_{o}+\rho $ to
$\theta_{o}+\frac{1}{2}\rho $ so that the result is
constant and zero near $\theta_{o}+\frac{1}{2}\rho $. Then,
extend $a^0_{\hat{e}}$ to the interval $[\theta_{o}-2\rho_{*}, \theta_{o}+\frac{1}{2}\rho ]$ as zero. It follows
from the vertex $o$ versions of the formulae in~\eqref{eq:4.1},~\eqref{eq:4.46} and~\eqref{eq:4.47} that
$a^0_{e'}$ and $a^0_{e''}$ can be chosen so that $K_{e'}$, $K_{e''}$
and $K_{\hat{e}}$ are pairwise disjoint where $\theta      \in [\theta
_{0}+2\rho_{*}, \theta_{o}+2\rho ]$.

In the case that $[Q_{\hat{e}}, Q_{e}] \ne 0$, vary $w^0_{\hat{e}}$ as
$\sigma$ decreases from $\theta_{o}+2\rho $ to $\theta_{o}+\rho $
so that the result is constant and zero near $\theta_{o}+\rho $. In the
case that $[Q_{\hat{e}}, Q_{e}] = 0$, introduce $\kappa $ to denote the
maximum value of the $Q = Q_{e}$ version of $\alpha_{Q}$ on the $\sigma  \in [\theta_{o}-2\rho, \theta_{o}]$
portion of $K_{e}$'s
parametrizing cylinder. In this case, vary $w^0_{\hat{e}}$ as $\sigma$
decreases from $\theta_{o}+2\rho $ to $\theta_{o}+\rho $ so that
the result is constant near $\theta_{o}+\rho $ and equal to
$\frac{\pi}{2\kappa}$. Extend $w^0_{\hat{e}}$ as a constant
to the interval $[\theta_{o}-2\rho_{*}, \theta_{o}+\rho]$.

Now, if $\varepsilon_{\hat{e}}$ is very much smaller than $(\varepsilon
_{e}\rho_{*})^{6}$ on $[\theta_{o}-2\rho_{*},\theta_{o}+\rho ]$ it then follows from the formulae
in~\eqref{eq:4.42},~\eqref{eq:4.43},~\eqref{eq:4.46} and~\eqref{eq:4.47}
that any intersection between $K_{\hat{e}}$ and the
portion of the closure of $K_{e}     \cup K_{e'}     \cup K_{e''}$ where
$\theta      \in [\theta_{o}-2\rho_{*}, \theta_{o}+2\rho ]$ occurs in the region of the latter surface that is
parametrized via~\eqref{eq:3.2} and~\eqref{eq:4.48} by the radius $\rho_{*}$ disk
centered in the $x$--$y$ plane. In fact, these intersections must occur at points
whose $x$--$y$ coordinates are within $4\varepsilon_{\hat{e}}$ of zero when
$\varepsilon_{\hat{e}}$ is very small.

To continue, note that the origin in the $x$--$y$ plane maps to a point in the
closure of the union $K_{e}     \cup K_{e'}     \cup K_{e''}$ where the
tangent plane is parallel to the tangent plane of $\mathbb{R}    \times \gamma$, where
$\gamma \subset S^{1}    \times S^{2}$ is a small
portion of an integral curve of the Reeb vector field. Now, as an
observation of \refstep{Part 2} recalls, $K_{\hat{e}}$ intersects such a surface
transversely with $+1$ local intersection numbers. This suggests that the
intersections of the small $\varepsilon_{\hat{e}}$ version of $K_{\hat{e}}$
with the $\theta      \in [\theta_{o}-2\rho_{*}, \theta_{o}+2\rho ]$
portion of the closure of $K_{e}     \cup K_{e'}     \cup K_{e''}$ will be transversal with $+1$
local intersection numbers. It is a
left as an exercise with~\eqref{eq:4.33} and~\eqref{eq:4.48} to verify that such is indeed the
case.

There is still more to do because as things stand now, both
$a^0_{\hat{e}}$ and $a^{0}_{e}$ are zero where $\sigma=\theta_{o}-2\rho_{*}$. Note that the
$\theta  \in [\theta_{o}-2\rho, \theta_{o}-2\rho_{*}]$ portions of
the cylinders $K_{e}$ and $K_{\hat{e}}$ will be disjoint as long as
$\varepsilon_{\hat{e}}$ is very much smaller than $\varepsilon $ and both
the pairs $(a^0_{\hat{e}}, w^0_{\hat{e}})$ and $(a^{0}_{e},w^{0}_{e})$ are kept at their
$\sigma =\theta_{o}-2\rho_{*}$ values as $\sigma$ decreases further to $\theta_{o}-2\rho $. Even
so, such an extension is not consistent at $\sigma =\theta_{o}-2\rho $ with~\eqref{eq:4.69}.

The desired extension of $(a^{0}_{e}, w^{0}_{e})$ keeps the latter
constant on $[\theta_{o}-2\rho, \theta_{o}-2\rho_{*}]$.  Meanwhile,
the extension of $(a^0_{\hat{e}}, w^0_{\hat{e}})$ employs the first
and second observations in \refstep{Part 4}.  To be more explicit,
$a^0_{\hat{e}}$ is either increased or decreased from zero as $\sigma$
decreases so that it is constant near $\theta_{o}-2\rho $, but with a
value that obeys $| a^0_{\hat{e}}| > \varepsilon
_{e}+\varepsilon_{\hat{e}}$. In this regard, $a^0_{\hat{e}}$ is
increased in the case that $[Q_{\hat{e}}, Q_{e}] < 0$ and it is
decreased when $[Q_{\hat{e}}, Q_{e}] > 0$. In the case that
$[Q_{\hat{e}}, Q_{e}] =0$, either a decrease or increase is
permissible. It then follows using~\eqref{eq:4.1} and~\eqref{eq:4.33}
that such a version of $a^0_{\hat{e}}$ can be constructed to insure
that $K_{\hat{e}}$ and $K_{e}$ are disjoint at values of $\theta$ near
$\theta_{o}+4\rho $ and that they intersect transversally where
$\theta \in [\theta_{o}-2\rho, \theta_{o}-2\rho_{*}]$ with $+1$ local
intersection numbers. In this regard, note that this can be done in
the case that $[Q_{\hat{e}}, Q_{e}]=0$ without introducing any
intersections between $K_{\hat{e}}$ and $K_{e}$.

\step{Part 8}
This last part of the subsection discusses the final point in~\eqreft4{70}. To
explain the situation here, let $o'$ and $o''$ denote distinct vertices of $T$
with angle label 0, and let $e'$ and $e''$ denote the corresponding incident
edges. Suppose first that the $\hat{A}$ label of $o'$ has the form $(1,+,\ldots)$
while $o''$ has either (1) or a label of the form
$(1,-,\ldots)$ from $\hat{A}$. In this case, there is no need
for an arc crossing if $s_{e'} > s_{e''}$ at $\sigma  = 2\rho $. Such is
also true when $s_{e'} > s_{e''}$ at $\sigma  = 2\rho $ and the $\hat{A}$
label of $o'$ is labeled by (1) while that of $o''$ has the form $(1,-,\ldots)$.
However, in either case, the corresponding arcs must
cross where $\sigma  < 2\rho $ if $s_{e'} < s_{e''}$. Make such a
crossing where $\sigma      \in [\rho , 2\rho ]$ and the second
observation in \refstep{Part 4} can be applied to choose $a^0_{e'}$ and
$a^0_{e''}$ on $[\rho , 2\rho ]$ so as to keep all $\theta  < 2\rho $
intersections between $K_{e'}$ and $K_{e''}$ where $\theta  \in [\rho , 2\rho ]$,
all transversal, and all with $+1$ local intersection number.

The case that both $o'$ and $o''$ have a $(+1)$ label is the simplest of those
where $o'$ and $o''$ have the same sort of label from $\hat{A}$. In this case,
$[Q_{e'}$, $Q_{e''}] = 0$, and so the first observation in \refstep{Part 4} can be
used to keep $K_{e'}$ disjoint from $K_{e''}$ if the corresponding $e'$ and
 $e''$ arcs must cross at some point where $\sigma  < 2\rho $.

In the case that $o'$ and $o''$ both have either a $(1,+,\ldots)$ label or a $(1,-,\ldots)$ label,
there are two subcases
to consider. In the case that the $e'$ and $e''$ arcs must cross where $\sigma  \leq 2\rho $,
then make such a crossing where $\sigma      \in [\rho , 2\rho ]$. In the case that $Q_{e'}$ is not proportional
to $Q_{e''}$, the second observation in \refstep{Part 4} is used to choose
$a^0_{e'}$ and $a^0_{e''}$ on $[\rho , 2\rho ]$ so as to keep all
$\theta  < 2\rho $ intersections between $K_{e'}$ and $K_{e''}$ where
$\theta      \in [\rho , 2\rho ]$, all transversal, and all with local
$+1$ intersection number. In the case that $[Q_{e'}, Q_{e''}] = 0$, then the
first observation in \refstep{Part 4} can be used to choose $(a^0_{e'},w^0_{e'})$ and $(a^0_{e''}, w^0_{e''})$ on
$[\rho , 2\rho]$ so as to keep $K_{e'}$ disjoint from $K_{e''}$ where $\theta  < 2\rho $.

\subsection{The case when $\wp $ has sets with two or more elements}\label{sec:4f}

This last subsection considers now the general case where the graph $T$ is
defined by a partition with sets that have more than one element. In what
follows, $\wp $ will denote such a partition with chosen cyclic orderings of
its subsets. The discussion here is broken into four parts. The first three
parts serve to specify the collection $\{(a_{e}, w_{e})\}$ and the
remaining part verifies that the collection meets all requirements.

\step{Part 1}
The purpose of this first part of the discussion is to construct from $T$ a
canonical moduli space graph to which the constructions in \fullref{sec:4c} apply.
This new graph is denoted by $\hat {T}$. The latter is isomorphic to $T$ as an
abstract graph via an isomorphism that preserves the labels of all edges and
all but the bivalent vertices. The isomorphism also preserves the angles of
the corresponding pairs of bivalent vertices.

Here is how the graphs $T$ and $\hat {T}$ differ: Suppose that $o \in T$  is
a bivalent vertex, and let $\hat{o}  \in \hat {T}$ denote its
partner. The vertex $o$ is labeled by a cyclic ordering of a partition subset,
say $\wp_{o}     \in     \wp $. Meanwhile, $\hat{o}$ is labeled by the data
$(0,+, P_{\hat{o}})$ where $P_{\hat{o}}$ is the sum of the integer pairs from
the elements in $\wp_{0}$.

A referral to \fullref{sec:3a} shows that $\hat {T}$ is a bona fide
moduli space graph. Moreover, the discussion in \fullref{sec:4c}
applies to $\hat {T}$. Let $\{(a_{\hat{e}}, w_{\hat{e}})\}$
denote the resulting data set for $\hat {T}$ as constructed in
the preceding Sections~\ref{sec:4b}--\ref{sec:4e}. The required set
$\{(a_{e}, w_{e})\}$ for $T$ is constructed now either starting directly
from $\{(a_{\hat{e}}, w_{\hat{e}})\}$, or from the
$J$--pseudoholomorphic curve that \fullref{thm:3.1} provides from
$\{(a_{\hat{e}}, w_{\hat{e}})\}$. The former approach is taken
below and the latter is left as an exercise for the reader.

\step{Part 2}
Suppose that $e$ is a given edge of $T$, and let $o'$ and $o$ denote the vertices
from $T$ that lie on $e$ with the convention taken that $\theta_{o'} < \theta_{o}$.
The corresponding edge, $\hat{e}$, in $\hat {T}$ has the
corresponding bounding vertices $\hat{o}$ and $\hat{o}'$ with $\theta_{\hat{o}} = \theta_{o}$ and
$\theta_{\hat{o}'}=\theta_{o'}$. This
understood, the functions $(a_{e}, w_{e})$ on $[\theta_{o'}, \theta_{o}]\times \mathbb{R}/(2\pi \mathbb{Z})$
are set equal to
$(a_{\hat{e}}, w_{\hat{e}})$ at all points except in the case that one of $o$
and $o'$ is a bivalent vertex. In the latter case, the equality still holds
except at values of $(\sigma , v)$ that are very close to those of the
missing point for the $\hat {T}$--parametrization on the relevant boundary
circle. In any event, the required data $\{\varepsilon_{e}, \rho
_{e0}, \rho_{e1}, a^{0}_{e}, w^{0}_{e}, v^{0}_{e}\}$
for $(a_{e}, w_{e})$ are declared equal to their $\hat {T}$ counterparts.

To be more explicit about the differences between $(a_{e}, w_{e})$ and
$(a_{\hat{e}}, w_{\hat{e}})$, suppose for the sake of argument that the vertex
 $o \in T$ is bivalent. Let $\hat{e}$ denote the respective partner to $e$ in
$\hat {T}$. Let $\rho $ denote the constant value of $\rho_{\hat{e}1}$
where $\sigma$ is near $\theta_{o}$. For convenience of notation, assume
that $v^0_{\hat{e}} = 0$ where $\sigma$ is near $\theta_{o}$. This
understood, then the equality between $(a_{e}, w_{e})$ and $(a_{\hat{e}},
w_{\hat{e}})$ holds where $\sigma      \sim     \theta_{o}$ except possibly
at values of $(\sigma , v)$ with distance $\rho ^{8}$ or less from the
missing point on the $\sigma =\theta_{o}$ circle.

\step{Part 3}
To describe $(a_{e}, w_{e})$ near the point where $\sigma =\theta
_{0}$ and $v = 0$, it is necessary to parametrize a neighborhood of the
point $(\sigma =\theta_{o}, v = 0)$ in the parametrizing cylinder for
 $e$ by the coordinates $(r, \tau )$ with $r \leq  3\rho $ and with $\tau
\in  [-\pi , 0]$. For this purpose, it proves necessary to introduce the
complex coordinate $z  \equiv r e^{i\tau }$. Also required is the choice
of a parameter $\delta      \in  (0, \rho ^{7})$.

To obtain the desired parametrization, write $P_{\hat{o}}=m_{\hat{o}} (p, p')$ with $p$ and $p'$
the relatively prime integers defined via~\eqref{eq:1.7}
by $\theta_{o}$ and with $m_{\hat{o}}     \ge  1$. Next, let $n$ denote
the number of elements in $\wp_{o}$ and suppose that $\wp_{o}$ has
been given a linear ordering. Use the latter to label the integer pairs from
its elements as $\{m_{1}(p, p'), \ldots , m_{n}(p, p')\}$ with each
$m_{j}$ a positive integer. Thus, $\sum_{j}m_{j} = m_{\hat{o}}$.

Let $0 = b_{1} < b_{2} < \ldots < b_{n} < \delta \rho ^{8}$ now
denote a chosen set of very small real numbers, and introduce the complex
function
\begin{equation}\label{eq:4.75}
\eta (z) = \beta \bigg(\frac{1}{\delta \rho }r\bigg) \prod_{1 \leq
j \leq n} (z - b_{j})^{m_j } + \bigg(1-\beta \bigg(\frac{1}{\delta\rho }r\bigg)\bigg) z^{m_{\hat{o}} }.
\end{equation}
Note that with $\rho $ small, and any choice for $\delta      \in  (0, \rho^{7})$,
the zeros of $\eta $ consist of the points in the set
$\{b_{j}\}$. An argument for the function $\eta $ is needed on the lower
half plane and also at points on the real axis where  $z \notin \{b_{j}\}$.
To be precise, take the branch that gives
\begin{equation}\label{eq:4.76}
\arg(\eta ) = m_{\hat{o}}    \tau  \text{ at points where } | z|      \ge  2\rho ^{8}.
\end{equation}

With the preceding in hand, here is is how to write $\sigma$ and $v$ in
terms of \mbox{$r$ and $\tau $:}
\begin{equation}\label{eq:4.77}
\begin{aligned}
\sigma  &= \theta_{\hat{o}} + \varepsilon r \sin(\tau ) .
\\
\hat {v}   & = \bigg(1 - \frac{\alpha_{Q_{e'}}{(\sigma )}}
{\alpha_{Q_e } (\sigma )}\bigg)    \frac{1}  {m_{\hat{o} }}\arg(\eta ) +
\frac{1}  {\alpha_{Q_e } (\sigma )} r \cos(\tau ) .
\end{aligned}
\end{equation}
As in the analogous~\eqref{eq:4.11}, the coordinate $\hat {v}$ is
$\mathbb{R}$--valued
and reduces modulo $2\pi$ to $v$. Meanwhile, $e'$ is the second of $o$'s incident
edges.

With $v_{*}$ as in~\eqref{eq:4.13} set
\begin{equation}\label{eq:4.78}
\begin{aligned}
a_{e} &= -\beta_{*}    \frac{1}{m_{\hat{o} }}\ln(|
\eta | ) + a_{\hat{o}}+\varepsilon  \big(\beta_{*} +
(1-\beta_{*}) \cos(v_{*})\big).
\\
w_{e} &= -\varepsilon  (1-\beta_{*}) \sin(v_{*})
\\&\quad+
x_{\hat{o}}    \beta '\bigg(\frac{1}  {\alpha_{Q_e } }\beta_{*}\bigg(\frac{1}{m_{\hat{o} }}\arg(\eta )-\frac{1}
{2\alpha_{Q_{e' } } }r \cos(\tau )\bigg) - \frac{1}  {2\alpha_{Q_e }}(1-\beta_{*}) v_{*}\bigg).
\end{aligned}
\end{equation}
Here, $v$ is viewed as taking values in $[0, 2\pi ]$. In addition, both
$\sigma$ and $v$ are to be viewed where $\beta_{*}>0$ as
functions of $r$ and $\tau $.

To define $(a_{e'}, w_{e'})$ near the point where $\sigma =\theta_{\hat{o}}$
 and $v = 0$ on the $e'$ version of the parametrizing cylinder, first
write the cylinder's coordinates $\sigma$ and $v$ near this point in terms of
 $r \in  (0, 3\rho )$ and $\tau      \in  [0, \pi ]$ using the rule
\begin{equation}\label{eq:4.79}
\begin{aligned}
\sigma    &= \theta_{*} + \varepsilon r \sin(\tau ) .
\\
\hat {v}   & = \bigg(
\frac{\alpha_{Q_e }(\sigma )}{\alpha_{Q_{e' } } (\sigma )} - 1\bigg) \frac{1}{m_{\hat{o} }}\arg(\eta
)+\frac{1}  {\alpha_{Q_{e' } } (\sigma )} r \cos(\tau).
\end{aligned}
\end{equation}
Here, $\arg(\eta )$ is again defined by~\eqref{eq:4.76}. With~\eqref{eq:4.79} set, define
\begin{equation}\label{eq:4.80}
\begin{aligned}
a_{e'} &= -\beta_{*}    \frac{1}{m_{\hat{o}}}
\ln(|
\eta | ) + a_{\hat{o}}+\varepsilon  \big(\beta_{*} +
(1-\beta_{*}) \cos(v)\big).
\\
w_{e'} &= -\varepsilon  (1-\beta_{*}) \sin(v)
\\&\quad+ x_{0}    \beta'
\bigg(\frac{1} {\alpha_{Q_{e'}}}\beta_{*}\bigg(\frac{1}  {m_{\hat{o} }}\arg(\eta )
+
\frac{1}{2\alpha_{Q_e }}r \cos(\tau )\bigg) + \frac{1}  {2\alpha_{Q_e }} (1-\beta_{*}) v\bigg).
\end{aligned}
\end{equation}
Here $v$ is again viewed as taking values in $[0, 2\pi ]$, and it with $\sigma$
are viewed as functions of $r$ and $\tau $ where $\beta_{*} > 0$.

\step{Part 4}
With what has been said in \fullref{sec:3}, \fullref{thm:3.1} now follows from the following claim:

\begin{em}
Very small $\delta $ versions of the definition just given of $\{(a_{e}, w_{e})\}$ satisfy the
criteria in~\eqreft33.
\end{em}

The justification for this claim is given next in three steps.

\substep{Step 1}
To start, let $o$ denote a bivalent vertex
in $T$ and let $e$ and $e'$ denote the incident edges. The first point to note is
that if $\delta $ is very small, then minor modifications of the arguments
from \fullref{sec:4c} prove that the change of variables from $(\sigma , v)$ to
$(r, \tau )$ is invertible on both the $e$ and $e'$ versions of the
parametrizing cylinders.

\substep{Step 2}
The closures of $K_{e}$ and $K_{e'}$ fit
together to define a smoothly immersed surface near points with $\theta
\sim     \theta_{\hat{o}}$ provided that the following is true: Let $v_{1}
\in  (0, 2\pi )$ obey
\begin{equation}\label{eq:4.81}
v_{1}     \notin  \bigg\{\frac{1}{{\alpha_{Q_e } (\theta_o)}}b_{1}, \ldots , \frac{1}{{\alpha_{Q_e } (\theta_o
)}}b_{n}\bigg\}.
\end{equation}
Then, there exists an integer pair, $N = (n, n')$, and extensions of the
definitions of $(a_{e}, w_{e})$ and $(a_{e'}, w_{e'})$ to some
neighborhood in $(0, \pi )    \times     \mathbb{R}/(2\pi \mathbb{Z})$ of
$(\theta_{*}, v_{1})$ so that~\eqref{eq:4.16} holds.

The verification of this condition proceeds just as in \fullref{sec:4c} at points
$v_{1}$ that differ by more than $2\delta \rho $ from either 0 or $2\pi$.
In the case that $v_{1}$ does not obey this condition, the equations
in~\eqref{eq:4.77}--\eqref{eq:4.80} directly give the required extensions of $(a_{e}, w_{e})$
and $(a_{e'}, w_{e'})$. This understood, there are various cases to
consider depending on whether
\begin{equation}\label{eq:4.82}
\begin{split}
v_{1} &< 2\pi, \\
\text{or}\quad \frac{1}{{\alpha_{Q_e } (\theta_{\hat{o}} )}}b_{k} &<
v_{1} < \frac{1}{{\alpha_{Q_e } (\theta_{\hat{o}})}}b_{k + 1}
\quad\text{for some }k  \in  \{1,\ldots, n-1\}, \\
\text{or}\quad v_{1} &> \frac{1}{{\alpha_{Q_e } (\theta_o )}}b_{N}.
\end{split}
\end{equation}
In the left most case, take the integer pair $N = Q_{e'}$. In the $k$'th
version of the middle case in~\eqref{eq:4.82}, take $N = Q_{e'}+\sum_{1 \leq j
\leq k}m_{j} (p, p')$, and in the right most case, take $N = Q_{e}$. Note
that when comparing this last case with the case in \fullref{sec:4c}, the $N = Q_{e}$
version of~\eqref{eq:4.16} is indistinguishable from the $N = 0$ version. It is
left to the reader to confirm that~\eqref{eq:4.16} holds with the values of $N$ as
above. In this regard, note that the functions $\beta_{*}$ and
$\beta'$ that appear in~\eqref{eq:4.13} and~\eqref{eq:4.77}--\eqref{eq:4.80} are equal to 1 at the
relevant points.

\substep{Step 3}
There are three more issues to examine vis
\`{a} vis~\eqreft33. The first is that of the asymptotics as laid out in \fullref{deff:3.2}.
The verification that these are as required is a straightforward task
using~\eqref{eq:3.2} with~\eqref{eq:4.77}--\eqref{eq:4.80}. The second is to verify that the data
$\{(a_{e}, w_{e})\}$ define a moduli space graph that is isomorphic to
the given graph $T$. This is also straightforward and so the details are omitted.

The final issue concerns the singularities that lie in the closure of $ \cup_{e}K_{e}$.
A very small choice for $\delta $ also simplifies the
analysis. To explain, let $o$ denote any given bivalent vertex. Then~\eqref{eq:3.2}
and~\eqref{eq:4.77}--\eqref{eq:4.80} define a smooth map, $\phi_{0}$, from some multiply
punctured version of the $| z|  < 2\delta \rho$ disk
in $\mathbb{C}$ into $\mathbb{R}    \times  (S^1 \times  S^{2})$. If
$\delta $ is very small, then $s$ is huge on the image of each such $\phi
_{0}$. Thus, any singularity in $ \cup_{e}K_{e}$ that is not already
present in its $\hat {T}$ analog is a singularity of the image of some $\phi
_{0}$. However, as explained next, each $\phi_{0}$ is an embedding when
$\delta $ is small. Hence, the closure of $ \cup_{e}K_{e}$ meets all
of~\eqreft33's requirements.

To prove that $\phi_{0}$ is an embedding, note first that points $z$ and $z'$
in the domain of $\phi_{0}$ are mapped to the same point only if they
have the same imaginary part. Indeed, otherwise, the images will have
distinct $\theta$ values. Meanwhile, use of~\eqref{eq:3.2} with~\eqref{eq:4.77}--\eqref{eq:4.80} finds
that~\eqref{eq:4.24} still holds. Granted that this is the case, it then follows that
the real parts of $z$ and $z'$ must agree as well if $\phi_{0}(z) = \phi_{0}(z')$. Thus, $z = z'$.

%
%

\setcounter{theorem}{0}
\section[Proof of Theorem~1.3]{Proof of \fullref{thm:1.3}}\label{sec:5}

The purpose of this last section is to prove \fullref{thm:1.3}. In this regard,
the proof is obtained from \fullref{thm:3.1} by demonstrating that $\hat{A}$ has a
positive line graph if and only if it has a moduli space graph. The
implication from positive line graph to moduli space graph is proved in the
first subsection. The reversed implication is proved in the second.

\subsection{From a positive line graph to a moduli space graph}\label{sec:5a}

Suppose that $\hat{A}$ has a positive line graph, $L_{\hat{A}}$. The goal is to
use the data from $L_{\hat{A}}$ to construct a labeled, contractible graph, $T$,
as described in \fullref{sec:3a}
The construction starts with the graph
$L_{\hat{A}}$ and successively modifies it to obtain $T$. This construction of $T$
occupies the seven parts of this subsection that follow.

\step{Part 1}
The purpose of this part of the subsection is to explain why the edges of
a positive line graph obeys Constraint 2 in \fullref{sec:3a}. Here is a formal
statement to this effect:

\begin{lemma}\label{lem:5.1}
Let $L$ denote a positive line graph for $\hat{A}$, let $e  \in L$ denote an edge, and let
$\theta_{o} < \theta_{1}$
denote the angles that are assigned the vertices on $e$.  Then
\begin{equation*}
{q_e}'(1-3\cos ^{2}\theta )- q_{e}\surd 6\cos (\theta )  \ge  0
\end{equation*}
at all $\theta  \in [\theta_{o}, \theta_{1}]$ with equality if and only if $\theta$ is either $\theta
_{0}$ or $\theta_{1}$ and the corresponding vertex is monovalent with angle in $(0, \pi )$.
\end{lemma}

\begin{proof}[Proof of \fullref{lem:5.1}]

The verification of the claim is given in five steps.

\substep{Step 1}
This first step considers the claim at the
vertex angles on $e$ . To start, remark that the stated inequality holds at any
bivalent vertex angle on $e$  because the final point in~\eqreft1{18} requires
\begin{equation}\label{eq:5.1}
p {q_e}' - p' q_{e} > 0
\end{equation}
when $(p, p')$ defines the angle of the vertex via~\eqref{eq:1.7}. In this regard, keep
in mind that~\eqref{eq:1.7} writes $p$ as a positive multiple of $(1 -3\cos ^{2}\theta )$ and $p'$
as the same multiple of $\surd 6\cos \theta$.

When the vertex is monovalent with angle in $(0, \pi )$, then the condition
in the lemma holds at the vertex angle in as much as the first and third
points in~\eqreft1{18} assert that $(q_{e}, {q_e}')$ is proportional to the pair
that defines the angle via~\eqref{eq:1.7}.

Meanwhile, the required inequality can be seen to hold at $\theta  =\theta_{o}$
when the latter is 0 by using~\eqreft1{14} with the fourth point
in~\eqreft1{18}. Likewise,~\eqreft1{14} and the second point in~\eqreft1{18} imply the lemma's
assertion when $\theta =\theta_{1}$ and $\theta_{1}=\pi$.

\substep{Step 2}
This and the remaining steps verify the Lemma's inequality at the
angles that lie strictly between $\theta_{o}$ and $\theta_{1}$. To
start this process, let $Q = (q, q') \ne (0, 0)$ denote a given ordered
pair of integers. Then the function $\alpha_{Q}(\theta )$ on $[0,
\pi ]$ vanishes only at that angles $\theta_{Q}$ and $\theta_{ - Q}$
that are respectively defined via~\eqref{eq:1.7} by $Q$ and by $-Q$.
In this regard, note that $\theta_{Q}$ can be defined when $q < 0$
provided that $\smash{\bigl|\frac{q'}{q}\bigr|} > \sqrt{\unfrac{3}{2}}$.
Meanwhile,
$\theta_{ - Q}$ can be defined when $q > 0$ provided
$\smash{\bigl|\frac{q'}{q}\bigr|} > \sqrt{\unfrac{3}{2}}$. Thus, at least of
one of $\theta_{Q}$ and
$\theta_{ - Q}$ exists in all cases and both exist only in the case that
$\smash{\bigl|\frac{q'}{q}\bigr|} < \sqrt{\unfrac{3}{2}}$.

To continue, note that the derivative of any given $Q = (q, q')$ version of
$\alpha_{Q}$ is
\begin{equation}\label{eq:5.2}
\surd 6\sin \theta  (q + \surd 6 \cos \theta  q').
\end{equation}
In particular, this derivative is positive at $\theta =\theta_{Q}$
and negative at $\theta =\theta_{ - Q}$. As a consequence, the
desired inequality is satisfied for the given edge $e$  if and only if one of
the conditions listed next hold:
\itaubes{5.3}
\textsl{Both $\theta_{Q_{e} } $ and $\theta_{ -_{Q_{{e}} }} $ are defined, $\theta_{Q_{{e}}
}  < \theta_{ -_{Q_{{e}}} } $, and  both $\theta_{Q_{{e}} }     \le \theta_{o}$ and
$\theta_{o'}    \le    \theta_{ -_{Q_{{e}}}} $.
}

\item \textsl{Both $\theta_{Q_{{e}} } $ and $\theta_{ -_{Q_{{e}}} } $ are defined,
$\theta_{ -_{Q_{{e}}} }  < \theta_{Q_{{e}} } $ and either  $\theta_{Q_{{e}} }     \le
\theta_{o}$ or $\theta_{o'}    \le    \theta_{ -_{Q_{{e}} }} $.
}

\item \textsl{$\theta_{ -_{Q_{{e}}} } $ is not defined and $\theta_{o}    \ge    \theta_{Q_{{e}} } $.
}

\item \textsl{$\theta_{Q_{{e}} } $ is not defined and $\theta_{o'}    \le    \theta_{ -_{Q_{{e}}} }$.
}
\end{itemize}

These last constraints are analyzed with the help of the following
observation: Suppose that $P$ and $Q$ are non-trivial integer pairs and suppose
that both $\theta_{P}$ and $\theta_{Q}$ exist. Then $\theta_{Q} < \theta_{P}$
if and only if one of the following holds:
\itaubes{5.4}
\textsl{$q'  \ge  0$, $p' \le 0$  and at least one is non-zero.
}

\item \textsl{If $p'$ and $q'$ have the same sign, then $q'p - qp' > 0$.
}
\end{itemize}

\substep{Step 3}
This step and the next assume that $\theta_{o} > 0$ and $\theta_{o'} < \pi$. For this purpose, let $(p_{o},
p_{o}')$ and $(p_{o'}, p_{o'}')$ denote the respective integer pairs that
define these angles via~\eqref{eq:1.7}.

This step considers the case when the $Q_{e}$ version of
$\smash{\bigl|\frac{q'}{q}\bigr|}$ is greater than
$\sqrt{\unfrac{3}{2}}$. Thus, the third and fourth options in~\eqreft5{3} are
moot. The first option in~\eqreft5{3} holds when ${q_e}' > 0$ and the second when
${q_e}' < 0$. Suppose first that ${q_e}' > 0$. If $p_{o}'  \le  0$ then
the $\theta_{o}$ requirement is met by virtue of the first point in~\eqreft5{4}.
Meanwhile,~\eqref{eq:5.1} together with the second point in~\eqreft5{4} guarantee
the $\theta_{o}$ requirement in the case that $p_{o}' > 0$. The first
point in~\eqreft5{4} guarantees the $\theta_{o'}$ requirement if $p_{o'}' \ge  0$.
If $p_{o'}' < 0$, then the $\theta_{o'}$ requirement is
guaranteed by the $o'$ version of~\eqref{eq:5.1} using the second point in~\eqreft5{4}.

Now suppose that ${q_e}' < 0$. If $p_{o}' < 0$, then~\eqref{eq:5.1} and the second
point in~\eqreft5{4} guarantee the $\theta_{o}$ requirement. If $p_{o}'  \ge  0$,
then the $\theta_{o}$ requirement fails. In this case, the $\theta_{o'}$
requirement holds due to the second part of the final point in~\eqreft1{18}.
Indeed, it would fail automatically were $p_{o'}' \le  0$, but
this is not allowed. On the other hand, if $p_{o'}' > 0$, then the $\theta
_{o'}$ requirement follows from the $o'$ version of~\eqref{eq:5.1} using the second
point in~\eqreft5{4}.

\substep{Step 4}
Granted that $\theta_{o} > 0$ and
$\theta_{o'} < \pi$, this step assumes that the $Q_{e}$ version of
$\smash{\bigl|\frac{q'}{q}\bigr|}$ is less than $\sqrt{\unfrac{3}{2}}$. In this case, neither of the first two points
in~\eqreft5{3} hold. In the case that $q_{e} > 0$, only the third point is possible
to satisfy. If ${q_e}'$ and $p_{o}'$ have the same sign, then the
requirement is met by virtue of~\eqref{eq:5.1} and the second point in~\eqreft5{4}. The
requirement is also met if $p_{o}' \le  0$ and ${q_e}' \ge  0$. Of
course, the requirement can not be met if $p_{o}'> 0$ and ${q_e}' < 0$. However, as will now be explained,
such signs for $p_{o}'$ and ${q_e}'$
never appear. Indeed, were these the correct signs, then~\eqref{eq:5.1} would demand
$p_{o}$ to be negative and
\begin{equation}\label{eq:5.5}
\sqrt {\frac3{2}}   >  \bigg|
\frac{q_e ' }{{q_e }} \bigg|  \ge  \bigg|
\frac{p_o ' }{{p_o }} \bigg|.
\end{equation}
However, this violates the condition in~\eqref{eq:1.7}.

Suppose next that $q_{e} < 0$ and so only the fourth point in~\eqreft5{3} is
relevant. In the case that $p_{o'}' \ge  0$ and ${q_e}' \ge  0$, then
the desired inequality is insured by the first point in~\eqreft5{4}. If
$p_{o'}'$ and ${q_e}'$ have different signs, the desired inequality is
insured by~\eqref{eq:5.1} and the second point in~\eqreft5{4}. Meanwhile, the case where
both $p_{o'}' < 0$ and ${q_e}' < 0$ can not occur because~\eqref{eq:5.1} again
requires that $p_{o}$ is negative while satisfying~\eqreft5{3}.

\substep{Step 5}
This step considers the case that $\theta_{o}$ = 0. The argument for the case when $\theta_{o'}=\pi$ is
omitted because it is identical to that given here save for some cosmetic
changes. In the $\theta_{o} = 0$ case, it is necessary to verify either
the second or fourth of the options in~\eqreft5{3}. In this regard, the first
point is that $-Q_{e}$ in all cases defines an angle via~\eqref{eq:1.7}. Indeed,
this is a consequence of the fact noted in \refstep{Step 1} that \fullref{lem:5.1}'s
inequality holds at $\theta  = 0$.

In the case that $Q_{e}$ defines an angle via~\eqref{eq:1.7}, then \fullref{lem:5.1}'s
inequality at $\theta  = 0$ requires that ${q_e}'<-(\sqrt{\unfrac3{2}})|q_{e}|$
and so the angle defined by $Q_{e}$ via~\eqref{eq:1.7} is greater than that defined by $-Q_{e}$.
Moreover, neither is less
than $\theta_{o'}$. Indeed, the angle defined by $Q_{e}$ via~\eqref{eq:1.7} must
be greater than $\theta_{o'}$ since the condition in \fullref{lem:5.1} holds
near $\theta  = 0$. Since ${q_e}' < 0$, this last point, the $o'$ version of~\eqref{eq:5.1}
and the second point in~\eqreft5{4} are consistent only if $p_{o'}'>0$.
Given that $p_{o'}' > 0$, then the $o'$ version of~\eqref{eq:5.1} and the second
point in~\eqreft5{4} establish the claim.

In the case that $Q_{e}$ does not define an angle via~\eqref{eq:1.7}, then $q_{e }<0$
and the absolute value of the ratio of ${q_e}'$ to $q_{e}$ is less than
$\sqrt {\unfrac3{2}} $. If ${q_e}'$ and $p_{o}'$ have opposite
signs, or if both are positive, then~\eqreft5{4} guarantees the conditions for the
fourth option in~\eqreft5{3}. On the other hand, in no case can both ${q_e}'$ and
$p_{o}'$ be negative when $q_{e}$ is negative. Here is why: Were all three
negative, then the $o'$ version of~\eqref{eq:5.1} would require $p_{o' } < 0$ also. As
such, this same version of~\eqref{eq:5.1} would declare the ratio of $p_{o'}'$ to
$p_{o'}$ to be less than that of ${q_e}'$ to $q_{e}$. By assumption the
latter is less than $\sqrt {\unfrac3{2}} $, and thus the former
would be less than $\sqrt {\unfrac3{2}} $. But this conclusion
with $p_{o'} < 0$ violates the given fact that $(p_{o'}, p_{o'}')$ defines
an angle via~\eqref{eq:1.7}.
\end{proof}

\step{Part 2}
Suppose that the maximal angle on $L_{\hat{A}}$ is less than $\pi$. Let
$\theta_{o}$ denote this angle. This step describes a modified version of
$L_{\hat{A}}$, a graph that is isomorphic to $L_{\hat{A}}$ except perhaps at
angles that are very close to $\theta_{o}$. This new graph has some
number of added monovalent vertices, all with angle $\theta_{o}$, these
labeled by the $(0,-,\ldots)$ elements from $\hat{A}$ whose
integer pairs define $\theta_{o}$ via~\eqref{eq:1.7}. The modification of
$L_{\hat{A}}$ is denoted below as $T_{1}$.

To start the description, let $o  \in L_{\hat{A}}$ denote the monovalent
vertex with the largest angle on $L_{\hat{A}}$. Let $\hat{e}$ denote the incident
edge to $o$. In the case that the element $(0,-, -Q_{\hat{e}})$ is in $\hat{A}$, no
modification occurs and $T_{1} = L_{\hat{A}}$. If this 4--tuple is is not in
$\hat{A}$, then $-Q_{{\hat{e}} }$ is equal to a sum of some $n > 1$ pairs, $P_{1} +\cdots + P_{n}$,
where each such pair is a positive
multiple of the relatively prime pair that defines $\theta_{o}$ via~\eqref{eq:1.7}
and where each $(0,-, P_{k})$ is in $\hat{A}$.

To proceed in this case, choose $n-1$ angles
$\theta_{1} < \theta_{2}\cdots  \theta_{n - 1} < \theta_{o}$ that are all
greater than the smallest angle on $e$.

Modify $L_{\hat{A}}$ so that the resulting graph has $n-1$ trivalent vertices at
these chosen angles. Label the incident edges to the $k'$th such vertex as $e$,
 $e'$, and $e''$ using the convention that $e$  connects the vertex o to a vertex
with smaller angle, while $e'$ and $e''$ connect to vertices with larger angle.
Also take the convention that any given $k  \le  n-2$ version of the edge
 $e''$ is the same as the $(k+1)'$st version of the edge $e$ . In particular, $e'$ is
capped with a $\theta =\theta_{o}$ monovalent vertex. This is also
the case for $e''$ when $k = n-1$.

Here are the integer pair assignments: In the case $k = 1$, the edge integer
pair assignments are $Q_{e} = Q_{{\hat{e}}}$, $Q_{e'} = -P_{1}$ and
$Q_{e''} = Q_{{\hat{e}}}+P_{1}$. In the case where $k > 1$, these pair
assignments are $Q_{e} = Q_{{\hat{e}} }+\sum_{j < k } P_{j}$, while
$Q_{e'} = -P_{k}$ and $Q_{e''} = Q_{{\hat{e}} }+\sum_{j \le k} P_{j}$.

By virtue of \fullref{lem:5.1}, this graph obeys the moduli space graph constraints
where it differs from $L$, thus at angles that are less than the minimal angle
on $\hat{e}$. In this regard, Constraint 2 in \fullref{sec:3a} is obeyed for $T_{1}$
because any given $T_{1}$ version of $\alpha_{Q}$ is a positive multiple
of a corresponding $L_{\hat{A}}$ version that obeys the constraint in \fullref{lem:5.1}
for the relevant interval. Moreover, the $\theta =\theta_{o}$
monovalent vertices on $T_{1}$ are in 1--1 correspondence with the subset of
$(0,-,\ldots)$ elements in $\hat{A}$ whose integer pair
component defines $\theta_{o}$ via~\eqref{eq:1.7}.

\step{Part 3}
This part of the construction describes the analogous operation on
$L_{\hat{A}}$ when its largest angle is $\pi$. The resulting version of
$T_{1}$ is isomorphic to $L_{\hat{A}}$ except at angles near $\pi$ where it
may have some trivalent vertices and more than one $\theta =\pi$
monovalent vertex. In particular, the labels of its $\theta =\pi$
monovalent vertices account for the $(-1,\ldots)$ elements
in $\hat{A}$.

To start, let $n$ and $n'$ denote the respective numbers of $(-1,-,\ldots)$ and
$(-1,+,\ldots)$ elements in $\hat{A}$. If $n > 0$,
label the $(-1,-,\ldots)$ elements from $\hat{A}$ from $1$ to $n$,
and if $n' > 0$, label the $(-1,+,\ldots)$ elements from $1$ to
$n'$. Let $\{P^{ - }_{k}\}_{1 \le k \le n}$ and $\{P^{ +}_{k}\}_{1 \le k \le n'}$
denote the corresponding set of integer pair components.

Two trivial cases can be dispensed with straight away; that where $\hat{A}$ has
but a single $(-1,\ldots)$ element and $c_{ - } = 0$, and
that where $\hat{A}$ has no $(-1,\ldots)$ elements and $c_{ - }=1$.
No modification of $L_{\hat{A}}$ is necessary in either of these cases.
Thus, $T_{1}$ is equal to $L_{\hat{A}}$ in both of these cases.

In the general case, the modification adds $n+n'+c_{ - }-1$ trivalent
vertices with successively larger angles, all near $\pi$. To be precise
here, the incident edges to any given vertex can be designated by $e$ , $e'$ and
 $e''$ so that $e$  connects the vertex in question to one with a smaller angle
while $e'$ and $e''$ connect to vertices with larger angles. In all cases, the
 edge designated as $e'$ is capped by a monovalent vertex with angle $\theta  = \pi$.
 Such is also the case for the version of $e''$ that attaches to the
trivalent vertex with the largest angle.

The edge labels for the incident edges to the trivalent vertices are
obtained via an inductive process using the following rules: For the
trivalent vertex with the smallest angle label, set $Q_{e}=\sum_{k}P^{ + }_{k}-\sum_{k} P^{ - }_{k}- (0, c_{ - })$.
Now, granted this, label this vertex as number 1 and label the remaining
trivalent vertices by consecutive integers starting from 2 in order of
increasing angle. Granted this numbering system, the first $c_{ - }$ of the
trivalent vertices have $Q_{e'} = (0, -1)$. If $n' = 0$, then the remaining
$n-1$ have $Q_{e'} = P_{j}$ for the vertex numbered $c_{ - }+j$ when  $1 \le  j  \le  n-1$.
If $n' > 0$, then $Q_{e'} = P_{j}$ for the vertex numbered
$c_{ - }+j$ when $1  \le  j  \le  n$. Use $Q_{e'} = -P_{j}$ for the
vertex numbered $c_{ - }+n+j$ with $1  \le  j  \le  n'-1$. With regard to
these assignments, note that the convention that $[Q_{e'}, Q_{e''}] \le  0$ is not necessarily observed.

By virtue of \fullref{lem:5.1}, the graph $T_{1}$ obeys the moduli space graph
conditions where it differs from $L_{\hat{A}}$, thus at the angles that are
achieved on $L_{\hat{A}}$'s largest angled edge. Moreover, this new graph has
the desired property: Its $\theta =\pi$ monovalent vertices account for
all of the $(-1,\ldots)$ elements in $\hat{A}$ plus $c_{ - }$
 elements with label $(-1)$.

By the way, the positivity requirement in \fullref{cons:2} in \fullref{sec:3a} can be
deduced from the following observations: If $\theta$ is less than its value
at the first trivalent vertex, then the relevant $Q$ is that of an edge in
$L_{\hat{A}}$ whose version of \fullref{cons:2} in \fullref{sec:3a} holds for the same
value of $\theta$. To argue for this constraint in the case that $\theta$
is near $\pi$, note that if the relevant version of $Q = (q, q')$ obeys
$q' <(\sqrt {\unfrac3{2}}) q$,
then $\alpha_{Q}(\theta ) > 0$ if
$\theta$ is nearly $\pi$ since~\eqreft1{14} guarantees its positivity at $\theta
=\pi$. This is to say that if the trivalent vertex angles are very near
$\pi$, it is enough that ${q_e}' < (\sqrt{\unfrac3{2}}) q_{e}$
for each incident edge to each trivalent vertex. This last requirement is
met by virtue of the conditions in~\eqreft1{14}.

\step{Part 4}
Let $\theta_{o}$ now denote the minimal angle on $L_{\hat{A}}$. Of
course, this is also the minimal angle on $T_{1}$. If $\theta_{o} > 0$,
then an upside down version of the discussion in \refstep{Part 2} (a verbatim
repetition save for some evident cosmetic changes) modifies $T_{1}$ by adding
trivalent vertices with angles just slightly greater than $\theta_{o}$
and monovalent vertices at $\theta_{o}$ to construct a new graph,
$T_{2}$, with the following property: First there exists some $\delta  > 0$
such that $T_{2}$ obeys the moduli space graph conditions at angles
$\theta     \in [\theta_{o}, \theta_{o}+\delta ]$. Moreover, it
has only trivalent vertices at angles in $(\theta_{o}, \theta_{o}+\delta )$
and it has as many $\theta =\theta_{o}$ monovalent
vertices as there are $(0,-,\ldots)$ elements in $\hat{A}$ whose
integer pairs define $\theta_{o}$ via~\eqref{eq:1.7}. Moreover, these elements
label the $\theta =\theta_{o}$ monovalent vertices in $T_{2}$.
Meanwhile, $T_{2}$ is isomorphic to $T_{1}$ at angles $\theta  >\theta_{o}+\frac1{2}\delta $.

In the case that $\theta_{o} = 0$, the upside down version of the
discussion of \refstep{Part 3} modifies $T_{1}$ by adding only trivalent vertices
at angles near 0 and monovalent vertices with angle equal to 0. This version
of $T_{2}$ obeys the moduli space graph conditions where it differs from
$T_{1}$ and thus where it differs from $L_{\hat{A}}$. Moreover, the set of
$\theta  = 0$ vertices in $T_{2}$ has a two subset partition: The first
subset accounts for the $(1,\ldots)$ elements in $\hat{A}$, and
the second contains $c_{ + }$ vertices with the label (1).

\step{Part 5}
Suppose now that $o$ is a bivalent vertex in $L_{\hat{A}}$ whose angle is
defined via~\eqref{eq:1.7} by one or more integer pairs from the collection of
$(0,-,\ldots)$ elements in $\hat{A}$. In this regard, consider
only the case when there are no pairs from $(0,+,\ldots)$
 elements in $\hat{A}$ that define $\theta_{o}$. Described here is a
modification to $T_{2}$ at angles very close to $\theta_{o}$ that
replaces the bivalent vertex $o$ with one or more trivalent and monovalent
vertices that account for those $(0,-,\ldots)$ elements in
$\hat{A}$ with integer pair giving $\theta_{o}$. To begin, let $e$  denote the
incident edge to $o$ on which $\theta_{o}$ is maximal, and let $e'$ denote
the incident edge on which $\theta_{o}$ is minimal. Fix some very small
and positive number, $\delta $. The modification proceeds in two steps.

The first step constructs a graph, $\hat {T}$, that lacks a bivalent vertex
at $\theta_{o}$, having one trivalent vertex at $\theta_{o}-\delta $
and one monovalent vertex at $\theta_{o}$. To be more explicit, let $\hat{o} \in    \hat {T}$
denote its trivalent vertex at $\theta_{o}-\delta $
and let $\hat{e}$, $\hat{e}'$ and $\hat{e}''$ denote its three incident edges. The
labeling convention here is such that  $\hat{o}$  has the largest angle of the two
vertices on $\hat{e}$, and the smallest angle of the two vertices on $\hat{e}'$ and
on $\hat{e}''$. In this regard, $\hat{e}''$ contains the added monovalent vertex
with angle $\theta_{o}$. In addition, as the integer pair assigned to
$\hat{e}''$ is $Q_{e}-Q_{e'}$, so $-Q_{{\hat{e}}''}$ is the sum of the integer
pairs that define $\theta_{o}$ via~\eqref{eq:1.7}.

To describe the rest of $\hat {T}$, agree to designate the three components
of $\hat {T}-\hat{o}$  as $\hat {T}_{{\hat{e}}}$, $\hat {T}_{{\hat{e}}'}$
and $\hat {T}_{{\hat{e}}''}$ so that the subscript indicates that the
component contains the interior of its labeling edge. Let $T_{2e}$ and
$T_{2e'}$ denote the analogously labeled components of $T_{2}-o$. Then
$\hat {T}_{{\hat{e}}}$ is isomorphic to $T_{2e}$ and $\hat
{T}_{{\hat{e}}'}$ to $T_{2e}$. With regards to such isomorphisms, the
convention taken here and subsequently is that an isomorphism between
labeled graphs with some open edges must preserve all labeling of vertices
and edges, but it need not match the angles of any `absent' vertices.

If $\hat{A}$ has a single $(0,-,\ldots)$ element whose integer
pair defines $\theta_{o}$, then the graph $T_{3}$ is set equal to $\hat
{T}$. If there is more than one such element, the graph $\hat {T}$ is
further modified by employing the construction in \refstep{Part 2} with the edge
$\hat{e}''$ playing the role of $\smash{L_{\hat{A}}}$. Thus, the modification replaces the
 edge $\hat{e}''$ with a subgraph whose monovalent vertices account for all of
the $(0,-,\ldots)$ elements in $\hat{A}$ with integer pairs that
define $\theta_{o}$ via~\eqref{eq:1.7}. This subgraph has one less trivalent
vertex than it has monovalent vertices. These can be assigned distinct
angles, all between $\theta_{o}$ and $\theta_{o}-\delta $.

As will now be explained, any sufficiently small $\delta $ version of the
graph $T_{3}$ obeys the moduli space graph conditions where it differs from
$T_{2}$ and thus where it differs from $L_{\hat{A}}$. To begin, remark that
the positivity of the $Q = Q_{e}$ and $Q = Q_{e'}$ versions of the function
$\alpha_{Q}$ imply that these functions are positive for small $\delta $
on the edges $\hat{e}$ and $\hat{e}'$ of $\hat {T}$. If the $Q = Q_{{\hat{e}}}$ version
of $\alpha_{Q}$ is positive on $[\theta_{o}-\delta , \theta
_{o})$ and vanishes at $\theta_{o}$, then the arguments from \refstep{Part 2}
settle the claim that $T_{3}$ obeys the moduli space graph conditions where
it differs from $T_{2}$. Granted this, remark that the $Q = Q_{{\hat{e}}}$
version of $\alpha_{Q}$ is zero at $\theta_{o}$ because $Q_{{\hat{e}}} = Q_{e}-Q_{e'}$,
and according to the fifth point in~\eqreft1{18}, this pair is $-1$ times a pair that defines
$\theta_{o}$ via~\eqref{eq:1.7}. Moreover, as
 explained subsequent to~\eqref{eq:5.2}, the derivative of $\alpha_{Q}$ at its zero
is negative when the angle of the zero is $\theta_{ - Q}$. Since this is
the case at hand, the $Q = Q_{{\hat{e}}}$ version of $\alpha_{Q}$ is positive
on the half open interval $[\theta_{o}-\delta , \theta_{o})$ when
$\delta $ is small.

\step{Part 6}
Suppose next that $o$ is a bivalent vertex whose angle is defined via~\eqref{eq:1.7}
by an integer pair from some $(0,+,\ldots)$ element in
$\hat{A}$. Consider first the case when no integer pairs from $(0,-,\ldots)$ elements
in $\hat{A}$ define this angle. In this case, the
modification to $T_{2}$ amounts to adding some data to the label of the
bivalent vertex $o$ so as to make the label that of a bivalent vertex in a
moduli space graph. In this regard, $o'$s label must be a partition subset for
some partition of the set of $(0,+,\ldots)$ elements whose
integer pairs define $\theta_{o}$ via~\eqref{eq:1.7}. Take the one set partition
and assign $o$ this set.

Consider now the case where $\theta_{o}$ is also defined via~\eqref{eq:1.7} by
integer pairs from both $(0,+,\ldots)$ elements in $\hat{A}$ and
$(0,-,\ldots)$ elements in $\hat{A}$. Let $P_{ + }$ denote the
sum of the integer pairs from the former set and let $P_{ - }$ denote the
sum of those from the latter. Note that both $P_{ + }$ and $P_{ - }$ define
$\theta_{o}$ via~\eqref{eq:1.7}. What follows describes a modification of $T_{2}$
so as to obtain a graph, $T_{3}$, with one bivalent vertex with angle
$\theta_{o}$, one or more monovalent vertices with angle $\theta_{o}$,
and some trivalent vertices with angles near $\theta_{o}$. This graph
$T_{3}$ will satisfy the moduli space graph conditions where it differs from
$T_{2}$ and its bivalent and monovalent vertices will account for all of the
$(0,\ldots)$ elements in $\hat{A}$ whose integer pair component
defines $\theta_{o}$ via~\eqref{eq:1.7}. The modification here is very similar to
that described in \refstep{Part 5}. In particular, there are two steps, with the first
modifying $T_{2}$ by adding a single trivalent vertex at an angle just less
than $\theta_{o}$ and adding a monovalent vertex with angle $\theta
_{o}$. This preliminary modification also has a bivalent vertex with angle
$\theta_{o}$. Let $\hat {T}$ denote this new graph. If $\delta $ is
positive but very small, then $\hat {T}$ can be constructed so that it has a
trivalent vertex,  $\hat{o}$ , with angle $\theta_{o}-\delta $. The three
incident edges, $\hat{e}$, $\hat{e}'$, and $\hat{e}''$ are such that  $\hat{o}$  has the
larger of the angles of the vertices on $\hat{e}$. As before, the component
$\hat {T}_{{\hat{e}}}    \subset    \hat {T}-\hat{o}$  is isomorphic to
the component $T_{2e}    \subset  T_2-o$. Meanwhile, $\hat{e}'$ connects
 $\hat{o}$  to the bivalent vertex at angle $\theta_{o}$ in $\hat {T}$ while
$\hat{e}''$ connects  $\hat{o}$  to the monovalent vertex with angle $\theta_{o}$.
The label for $\hat{e}''$ is $-P_{ - }$, while that for $\hat{e}'$ is $Q_{e}+P_{
- }$. Note that the open graph $\hat {T}_{{\hat{e}}'}-{\hat{e}}'$ is
isomorphic to $T_{2e'}$.

The graph $\hat {T}$ must now be modified so that the result, $T_{3}$, obeys
the moduli space graph conditions where it differs from $T_{2}$. First of
all, this involves replacing $\hat{e}''$ by a subgraph with some number of
monovalent vertices and one less number of trivalent vertices, with the
subgraph chosen so that its monovalent vertices have angle $\theta_{o}$
and account for those $(0,-,\ldots)$ elements in $\hat{A}$ whose
integer pair defines $\theta_{o}$ via~\eqref{eq:1.7}. This procedure is exactly
that used in the previous step to go from the latter's $\hat {T}$ to
$T_{3}$. Note that \fullref{cons:2} in \fullref{sec:3a} is obeyed on all of the edges
in this subgraph. Indeed, the argument for this is a verbatim repetition of
the one that proves the analogous claim in \refstep{Part 5}. The final task in the
construction of $T_{3}$ is to grant a label to the bivalent vertex at angle
$\theta_{o}$. In this case, the label must be a partition of the set of
those $(0,+,\ldots)$ elements in $\hat{A}$ whose integer pair
defines $\theta_{o}$ via~\eqref{eq:1.7}. As before, take the 1--set partition. Note
that this is forced by the fact that $Q_{{\hat{e}}'} - Q_{e'} = P_{ + }$
which is the sum of the integer pairs from this same set of elements. By the
way, note that the $\hat{e}'$ version of \fullref{cons:2} in \fullref{sec:3a} is obeyed
when $\delta $ is small by virtue of two facts: First, the $Q = Q_{e}$
version of $\alpha_{Q}$ is bounded away from zero on $[\theta_{o}-\delta , \theta_{o}]$.
Second, the $Q = P_{ + }$ version of
$\alpha_{Q}$ is zero at $\theta_{o}$ and so is very small on this
interval when $\delta $ is small.

\step{Part 7}
Apply the constructions in \refstep{Part 5} and \refstep{Part 6} simultaneously to all of the
bivalent vertices. The result is a moduli space graph for $\hat{A}$.

\subsection{From moduli space graph to positive line graph}\label{sec:5b}

Now suppose that $\hat{A}$ has a moduli space graph, $T_{\hat{A}}$. The goal is to
obtain from $T_{\hat{A}}$ a positive line graph for $\hat{A}$. This is
accomplished in a sequential fashion using the various `moves' that are
described in \refstep{Part 1}, below. These moves are used to eliminate trivalent
vertices. To picture this process, imagine a trivalent vertex as the point
in a partially unzipped zipper where two edges are joined as one. The
modifications amount to closing in a sequential fashion all of these
zippers. \refstep{Part 2} of the subsection provides the details for how these moves
are used.

The modifications to $T_{\hat{A}}$ will result in graphs that are not moduli
space graphs. Even so, these graphs have labeled edges and vertices that
obey certain constraints. These constraints are listed below, and a graph
that obeys them is deemed a `positive graph'.

A positive graph, $T$, is a connected, contractible graph with at least one
 edge and with labeled vertices and edges. The vertices of $T$ are either
monovalent, bivalent or trivalent. Each is labeled with an angle in $[0, \pi]$.
These angles are constrained as follows:
\itaubes{5.6}
\textsl{The two vertices on any given edge have distinct angles.}

\item \textsl{The angle of any given multivalent vertex is neither the largest nor the smallest of the angles
of the vertices on its incident edges.
}

\item \textsl{Any vertex angle in $(0, \pi )$ is defined via~\eqref{eq:1.7} by an integer pair.
}
\end{itemize}

Each edge of $T$ is labeled by an integer pair. If $e$  its an edge, then $Q_{e} = (q_{e}$, ${q_e}')$
denotes its integer pair. These are constrained as follows:
\itaubes{5.7}
\textsl{Let  $o \in  T$ denote a monovalent vertex with angle in $(0, \pi )$ and let $e$ denote its incident edge.
Then $\pm Q_{e}$ defines $\theta_{o}$ via~\eqref{eq:1.7} with the $+$ sign taken if and only if $\theta
_{o}$ is the smaller of the two angles of the vertices.
}

\item \textsl{Let $o  \in T$ denote a bivalent vertex and let $e$ and $e'$ denote its incident edges. If $Q_{e} \ne Q_{e'}$,
then either $Q_{e }- Q_{e'}$ or $Q_{e' }-Q_{e}$ defines $\theta_{o}$ via~\eqref{eq:1.7}.
}

\item \textsl{Let $o  \in  T$ denote a trivalent vertex, and let $e$, $e'$ and $e''$ denote its incident edges.
Then $Q_{e} = Q_{e'} + Q_{e''}$ with the convention that $\theta_{o}$ lies between the angle of
the vertex opposite $o$ on $e$ and the angles of the vertices opposite $o$ on both $e'$ and $e''$.
}

\item \textsl{Let $e$  denote any given edge of $T$ and let $\theta_{o} < \theta_{1}$ denote the angles that are
assigned the vertices on $e$. Then
\begin{equation*}
{q_e}'(1-3\cos ^{2}\theta )- q_{e}\surd 6\cos (\theta )    \ge  0
\end{equation*}
at all $\theta     \in [\theta_{o}, \theta_{1}]$ with equality if and only if $\theta$ is either $\theta
_{0}$ or $\theta_{1}$, the angle in question is in $(0, \pi )$, and the corresponding vertex is monovalent.
}
\end{itemize}

Each positive graph that appears below is related to the asymptotic data set
$\hat{A}$ in a manner that is described momentarily. For this purpose, it is
necessary to assign an integer pair to each vertex with angle in $(0, \pi )$. If $o$ is such a vertex,
then $P_{o}$ is used to denote its integer pair.
Here are the assignments: If $o$ is monovalent, then $P_{o}=\pm Q_{e}$,
where $e$  denotes $o$s incident edge and where the $+$ sign is taken if and only
if $\theta_{o}$ is the smaller of the two angles of $e'$s vertices. If $o$ is
bivalent, then $P_{o} = Q_{e}-Q_{e'}$ where $e$  and $e'$ are $o$s incident
 edges with the convention here that $\theta_{o}$ is the larger of the two
angles of the vertices on $e$. If $o$ is a trivalent vertex, set $P_{o} = 0$.

What follows describes how $\hat{A}$ enters the picture:
\itaubes{5.8}
\textsl{The sum of the integer pairs that are associated to the edges with a $\theta =\pi$
vertex is obtained from $\hat{A}$ by the following rule: First, subtract the sum of the
integer pairs from the $(-1,-,\ldots)$ elements in $\hat{A}$ from the sum of those from the
$(-1,+,\ldots)$ elements, and then subtract $(0, c_{ - })$ from the result.
}

\item \textsl{The sum of the integer pairs that are associated to the edges with a $\theta  = 0$ vertex is obtained from
$\hat{A}$ by the following rule: First, subtract the sum of the integer pairs from the $(1,+,\ldots)$
elements in $\hat{A}$ from the sum of those from the $(1,-,\ldots)$ elements and then subtract
$(0, c_{ + })$ from the result.
}

\item \textsl{Let $\theta \in (0, \pi )$. Then, the sum of the integer pairs that are associated to the
bivalent vertices at angle $\theta$ minus the sum of those pairs that are associated to the monovalent
vertices at angle $\theta$ is obtained from $\hat{A}$ by the following rule:
Subtract the sum of the integer pairs from the $(0,-,\ldots)$ elements in $\hat{A}$ that define
$\theta$ via~\eqref{eq:1.7} from the sum of the integer pairs from the
$(0,+,\ldots)$ elements in $\hat{A}$ that defined $\theta$ via~\eqref{eq:1.7}.
}
\end{itemize}

A positive graph that obeys~\eqreft5{8} is called a `positive graph for $\hat{A}$'.
According to \fullref{lem:5.1}, a positive line graph for $\hat{A}$ is a linear positive
graph for $\hat{A}$, that is, one with no trivalent vertices. \fullref{lem:5.2} below
proves the converse. Note that $T_{\hat{A}}$ itself is a positive graph for $\hat{A}$.

\step{Part 1}
To set the stage here and in \refstep{Part 2}, let $T$ now denote a given positive
graph. Let $o$ denote a trivalent vertex in $T$ and let $e$ , $e'$ and $e''$ denote the
three incident edges to $o$ with the usual convention taken to distinguish $e$.
This is to say that the angle $\theta_{o}$ lies between the angle of the
vertex opposite $o$ on $e$  and both the angle of the vertex opposite $o$ on $e'$ and
that of the angle opposite $o$ on $e''$. The edges $e'$ and $e''$ are distinguished
when $Q_{e'}$ is not proportional to $Q_{e''}$ by making
$[Q_{e'},Q_{e''}] \equiv q_{e'}{q_{e''}}' - {q_{e'}}'q_{e''}$ negative.
Also, keep in mind the following two conventions from the previous
subsection that concern the three components of $T-o$: The first is
with regards to their labeling, this as $T_{e}$, $T_{e'}$ and $T_{e''}$ with
the labeling such that the closure of any one of the three contains its
labeling edge. The other convention concerns the notion of an isomorphism
between one of these components and some other non-compact graph with
labeled vertices and edges. In particular, the isomorphism must send
vertices to vertices and edges to edges so as to respect the labeling.
However, such an isomorphism has no need to respect the angle of the absent
vertex on the open edge.

With these conventions set, what follows in this \refstep{Part 1} are the moves that
are used to modify a given positive graph for $\hat{A}$ so as to eliminate the
trivalent vertices. In all cases, the modified graph is a positive graph for
$\hat{A}$. There are two versions to each move listed below, one for the case
that $e$  connects the given vertex to a vertex with a larger angle, and one
for the case that the connection is to a vertex with a smaller angle. Only
the first version is presented since the two versions differ only
cosmetically.

Note that the first three moves modify the original graph so as to
\underline{decrease} the angle that is assigned to the given trivalent
vertex. (In the omitted version where $e$  connects to a vertex with smaller
angle, the corresponding moves will increase the angle of the given
trivalent vertex.) The remaining four moves modify the graph so as to
eliminate the given trivalent vertex.

To set the stage for the moves, agree to let $o$ denote the trivalent vertex
in question and let $\theta_{o}$ denote its original angle assignment. In
what follows, $\theta_{\hat{o}}$ denotes the larger of the two angles that
label the vertices that lie opposite $o$ on $e'$ and $e''$. Keep in mind that
$\theta_{\hat{o}}$ is less than $\theta_{o}$. A distinguishing feature
of the geometry here is that neither the $Q = Q_{e'}$ nor $Q_{e''}$ versions
of $\alpha_{Q}$ can vanish on $(\theta_{\hat{o}}, \theta_{o})$.
Since $Q_{e} = Q_{e' }+ Q_{e''}$, the $Q = Q_{e}$ version of $\alpha
_{Q}$ is also positive on $(\theta_{\hat{o}}, \theta_{o})$. As a
result, $T$ can be modified without either compromising the positive graph
conditions or changing its topology by giving $o$ any angle in
$(\theta_{\hat{o}}, \theta_{o})$.

With the preceding understood, the first three moves describe cases where $T$
is modified so that the result has a trivalent vertex with angle just less
than $\theta_{\hat{o}}$.

\substep{Move 1}
Assume here that $\theta_{\hat{o}}$
labels just one vertex on $e' \cup e''$ and that this vertex is bivalent.
In this case $T$ is modified to produce a new positive graph for $\hat{A}$, this
denoted by $T_{*}$. The graph $T_{*}$ has a trivalent vertex,
$o_{*}$, with angle $\theta_{*}$ just less than $\theta_{\hat{o}}$ and a
bivalent vertex with angle $\theta_{\hat{o}}$. The
integer pair component of the latter's label is the same as that of the
$\theta_{\hat{o}}$ labeled vertex on $e' \cup e''$.

To continue the description, note that $o_{*}$ has incident edges
 $e_{*}$, $e_{*}'$, $e_{*}''$ where $e_{*}$ is the
only one of the three that connects $o_{*}$ to a vertex with a larger
angle. The latter is the aforementioned bivalent vertex with angle $\theta_{\hat{o}}$.
Write the components of $T_{*}-o_{*}$ as
$T_{*e}$, $T_{*e'}$ and $T_{*e''}$. These graphs are
related to $T_{e}$, $T_{e'}$ and $T_{e''}$ as follows: In the case that $e''$
has the $\theta_{\hat{o}}$ labeled vertex, then, $T_{e''}-e''$ and
$T_{*e''}$ are isomorphic as non-compact graph with labeled vertices
and edges. Meanwhile $T_{e'}$ and $T_{*e'}$ are likewise isomorphic,
as are the pair $T_{e}$ and $T_{*e}-e_{*}$. The
analogous isomorphisms hold when $e'$ has the $\theta_{\hat{o}}$ bivalent
vertex.

\substep{Move 2}
This move is relevant to the case that
both $e'$ and $e''$ have bivalent vertices with angle $\theta_{\hat{o}}$. In
this case, $T$ is again modified to produce a new positive graph for $\hat{A}$.
This graph has a trivalent vertex at angle just less than $\theta_{\hat{o}}$
and a single bivalent vertex at angle $\theta_{\hat{o}}$ that
sits on the incident edge $e_{*}$ to $o_{*}$. Here, the
notational convention for the incident edges to $o_{*}$ are as in Move~1.
The integer pair component of the label for this bivalent vertex is the
sum of the integer pairs that label the bivalent vertices on $e'$ and $e''$. In
this regard, the components $T_{*e'}$ and $T_{*e''}$ of $T_{*}-o_{*}$ are respectively isomorphic
to $T_{e'}-e'$ and $T_{e''}-e''$ from $T-o$. Meanwhile, $T_{*e}-e_{*}$ is isomorphic to $T_{e}$.
The verification that the version of $T_{*}$ with $\theta_{*}$ nearly   
$\theta_{\hat{o}}$  is a positive graph for $\hat{A}$ requires 
only the verification of the fourth
point in~\eqreft5{7} for the edges that touch $e_{*}$. In this regard, the
positivity of the relevant versions of $\alpha_{Q}$ follow from the
positivity at $\theta_{\hat{o}}$ of the $Q_{e'}$ and $Q_{e''}$ versions.

\substep{Move 3}
This move is relevant to when there is a
single vertex  $\hat{o} \in e' \cup e''$ with angle $\theta_{\hat{o}}$,
that this vertex is trivalent, and that it has a single incident edge that
connects it to a vertex with angle less than $\theta_{\hat{o}}$. Agree to
relabel the edge between $o$ and  $\hat{o}$  as $e_{0}$. Now, label the other two
incident edges to $o$ as $e_{1}$ and $e_{2}$ with the convention that $e = e_{1}$,
while labeling the other two incident edges to  $\hat{o}$  as $e_{3}$ and
 $e_{4}$ with the convention that $e_{4}$ connects  $\hat{o}$  to a vertex with
angle less than $\theta_{\hat{o}}$.

Let $T_{*}$ denote the modified graph. It has a trivalent vertex,
$o_{*}$, at angle just less than $\theta_{\hat{o}}$, and another,
 $\hat{o}_{*}$ at angle just greater than $\theta_{o}$. These two are
connected by an edge, $e_{*0}$. The remaining two incident edges to
$o_{*}$ connect the latter to vertices with smaller angles, while the
remaining two incident edges to  $\hat{o}_{*}$ connect it to vertices
with larger angles. The integer pair assigned to $e_{*0}$ is the sum
of those assigned to $e_{2}$ and $e_{4}$, this being also the sum of those
assigned to $e_{1}$ and $e_{3}$. Meanwhile, $T_{*}-e_{*0}$ is isomorphic to $T-e_{0}$.

The fact that $T_{*}$ is a positive graph for $\hat{A}$ follows directly
with the verification of the fourth condition in~\eqreft5{7}. And, the latter
follows when $o_{*}$ has angle nearly $\theta_{\hat{o}}$ and
 $\hat{o}_{*}$ has angle nearly $\theta_{o}$ from the fact that the
inequality is strictly obeyed by $e_{0}$ on $[\theta_{\hat{o}}, \theta_{o}]$,
and by the other incident edges to $o$ and  $\hat{o}$  on the relevant
intervals in $[0, \pi ]$.

The remaining moves describe modifications to $T$ that either remove a given
trivalent vertex, or replace it with either one monovalent vertex or one
bivalent vertex.

\substep{Move 4}
Suppose here that only one vertex on $e'\cup e''$ has angle $\theta_{\hat{o}}$, and that the latter is monovalent.
This move explains how $T$ is modified so as to replace the trivalent and
monovalent vertices with a single bivalent vertex.

To begin the story, remark that $\theta_{\hat{o}}$ must be greater than 0
as $e'\cup e''$ has a vertex with a smaller angle. Moreover, because the
 $e'$ version of $\alpha_{Q}$ is positive at $\theta_{\hat{o}}$ and because
$[Q_{e'}, Q_{e''}] < 0$, the vertex on $e' \cup e''$ with angle $\theta_{\hat{o}}$
must sit on $e''$. The graph $T$ is modified at angles near $\theta_{\hat{o}}$ by
removing $e''-o$ so as to replace $o$ with a bivalent
vertex in the modified graph. To elaborate, let $T_{*}$ denote the new
graph. It has a bivalent vertex, $o_{*}$, at angle $\theta_{\hat{o}}$. Use $e_{*}$ and $e_{*}'$
to denote its incident edges with the convention that $e_{*}$ connects $o_{*}$ to a
vertex with angle less than $\theta_{\hat{o}}$. Label the components of
$T_{*}-o_{*}$ as $T_{*e}$ and $T_{*e'}$
with the convention that $T_{*e}$ contains the interior of $e_{*}$. Then $T_{*e}$ is isomorphic to the
component $T_{e'}$ of $T-o$, and $T_{*e'}$ is isomorphic to the component $T_{e}$. Because
the $Q = Q_{e'}$ version of $\alpha_{Q}$ is positive at $\theta_{\hat{o}}$, this is also
the case for the $Q = Q_{e}$ version. This then
implies that $T_{*}$ obeys the fourth constraint in~\eqreft5{7}. Thus,
$T_{*}$ is a positive graph since it also obeys the first three
conditions in~\eqreft5{7}. Meanwhile, the $T_{*}$ version of~\eqreft5{8} is obeyed
by virtue of the fact that the integer pair for the vertex $o_{*}$ is $-Q_{e''}$.

\substep{Move 5}
This move is relevant when both $e'$ and $e''$
have vertices with angle $\theta_{\hat{o}}$ with one bivalent and the other
monovalent. In this regard, note that $Q_{e'}$ and $Q_{e''}$ can not lie on
the same line in $\mathbb{R}^{2}$ in this case. Thus, with the $[Q_{e'},Q_{e''}] < 0$
convention, the bivalent vertex must lie on $e'$. In this case,
the graph $T$ is modified by eliminating both the trivalent vertex and the
monovalent vertex on $e''$. To elaborate here, let $T_{*}$ again denote
the new graph. It has a bivalent vertex with angle $\theta_{\hat{o}}$. Let
$o_{*}$ denote the latter, and let $e_{*}$ and $e_{*}'$
denote its incident edges with the convention that $e_{*}$ connects
$o_{*}$ to a vertex with a smaller angle label. Then the component
$T_{*e}$ of $T_{*}-o_{*}$ is isomorphic to the
component $T-e'$ that contains vertices with angles less that $\theta_{\hat{o}}$.
Meanwhile, the component $T_{*e'}$ of $T_{*}-o_{*}$ is isomorphic to the component $T_{e}$ of $T-o$.

With the labeling as describe, $T_{*}$ is a positive graph for $\hat{A}$.
Indeed, the only substantive issue here is that raised by the fourth point
in~\eqreft5{7}. In this regard, the $e_{*}'$ version of this inequality
holds because it is strictly obeyed by the $Q = Q_{e}$ version of $\alpha
_{Q}$ at $\theta_{*}$. Meanwhile, the $e_{*}$ version of
the inequality holds because it holds for the version that is labeled by the
 edge that connects the bivalent vertex on $e'$ to a smaller angled vertex.

\substep{Move 6}
This and the remaining moves are relevant
to the cases where $\theta_{\hat{o}}$ is the angle of a monovalent vertex
on $e'$ and a monovalent vertex on $e''$. This move considers the case where the
angle is in $(0, \pi )$.

In this case, a new graph, $T_{*}$, is obtained from $T$ by removing $(e'  \cup e'')-o$,
thus replacing $o$ by a monovalent vertex with angle
$\theta_{\hat{o}}$. To elaborate, the graph $T_{*}$ has a monovalent
vertex, $o_{*}$, with angle $\theta_{\hat{o}}$. In addition,
$T_{*}-o_{*}$ is isomorphic to $T_{e}$. In this regard, keep
in mind that both $Q_{e'}$ and $Q_{e''}$ define $\theta_{\hat{o}}$ via~\eqref{eq:1.7}.
Thus, they are positive multiples of each other. This understood,
then $Q_{e}$ must also define $\theta_{\hat{o}}$ via~\eqref{eq:1.7}. The fourth
point in~\eqreft5{7} holds on $T_{*}$ because it holds on $T$ and because the
 $Q = Q_{e}$ version of $\alpha_{Q}$ vanishes at $\theta_{\hat{o}}$ and
is positive on $(\theta_{\hat{o}}, \theta_{o}]$.

\substep{Move 7}
This considers the case that $\theta_{\hat{o}} = 0$. In all of these cases, $T_{*}$ has a $\theta = 0$
monovalent vertex whose complement is isomorphic to $T_{e}$. The
verification that $T_{*}$ is a positive graph for $\hat{A}$ is
straightforward and so left to the reader.

\step{Part 2}
This last part of the subsection explains how the preceding moves can be
used to construct a positive line graph from $\hat{A}$ given its original moduli
space graph $T_{\hat{A}}$. In this regard, keep in mind that $T_{\hat{A}}$ is a
positive graph for $\hat{A}$. The forthcoming \fullref{lem:5.2} asserts that a linear
positive graph for $\hat{A}$ is a positive line graph. This understood, the task
at hand is to modify $T_{\hat{A}}$ using \refstep{Move 1}--\refstep{Move 7} so as to obtain a positive
graph for $\hat{A}$ that lacks trivalent vertices.

\begin{lemma}\label{lem:5.2}
A linear, positive graph for $\hat{A}$ is a positive line graph for $\hat{A}$.
\end{lemma}

\begin{proof}[Proof of \fullref{lem:5.2}]
The only substantive issue here is that raised
by the fourth point in~\eqreft1{18}. In this regard, suppose that $L$ is a linear,
positive graph for $\hat{A}$ and that $e$  is an edge in L. The condition on the
positivity of $p{q_e}'-p'q_{e} > 0$ in the case that $(p, p')$ is an integer
pair that defines the angle of a bivalent vertex on $\hat{e}$ follows directly
from fourth constraint in~\eqreft5{7}. Now, suppose that ${q_e}' < 0$. Let $o$ and
 $o'$ denote the two vertices on $e$  with the convention that o's angle is less
than that of $o'$. Assume first that both these vertices have angles in $(0, \pi )$.
If such is the case, then $p_{o'}' > 0$ requires $p_{o}' > 0$
because both angles must be smaller than $\frac\pi{2}$ if the larger is.

For the sake of argument, suppose that $p_{o}' > 0$ but that $p_{o'}'< 0$. The $Q = Q_{e}$
version of $\theta_{Q}$ must be greater than
$\theta_{o}$ since the former is greater than $\frac\pi{2}$
and the latter less than $\frac\pi{2}$. However, according to
the second point in~\eqreft5{4}, this same $\theta_{Q}$ is less then $\theta
_{o'}$. Thus, $\alpha_{Q}$ vanishes between $\theta_{o}$ and $\theta
_{o'}$ and this violates the last item in \eqreft57.
The other
possibility is that where either $\theta_{o}$ is 0 or $\theta_{o'} = \pi$.
In the case that $\theta_{o} = 0$, then the argument just given
has the $Q = Q_{e}$ version of $\theta_{Q}$ less than $\theta_{o'}$
when it exists. If available, the $Q = -Q_{e}$ version of this angle is also
less than $\theta_{o'}$ by virtue of the first point in~\eqreft5{4}. In either
case, this means that the $Q = Q_{e}$ version of $\alpha_{Q}$ has a zero
in $(0, \theta_{o'})$. In the case that $\theta_{o'}=\pi$, the $Q = Q_{e}$ version of
$\theta_{Q}$ is greater than $\theta_{o}$ when it
exists. This is another consequence of the first point of~\eqreft5{4}. When
available, the $Q = -Q_{e}$ version of $\theta_{Q}$ is also greater than
$\theta_{o}$; this a consequence of the second point in~\eqreft5{4}.
Thus the $Q = Q_{e}$ version of $\alpha_{Q}$ has a zero in $(\theta_{o}, \pi)$.

The remainder of this section describes an algorithm that uses \refstep{Move 1}--\refstep{Move 7}
from \refstep{Part 1} to change $T_{\hat{A}}$ into a linear, positive graph for $\hat{A}$
and thus produce the desired positive line graph for $\hat{A}$. The algorithm
has four steps.\end{proof}

\substep{Step 1}
Suppose that $T$ is any given positive graph
for $\hat{A}$. Let $V_{ + }$ denote the set of trivalent vertices with only one
incident edge that contains a larger angle vertex. Likewise, define $V_{ -}$
to be the set of trivalent vertices with only one incident edge that
contains a smaller angle vertex. Let $n_{T + }$ denote the number of
 elements in $V_{ + }$ and let $n_{T - }$ denote the corresponding number in
$V_{ - }$. If $V_{ + }$ is empty, go to \refstep{Step 3}. If not, let $o \in V_{ +}$
be a vertex whose angle is the smallest of those from the vertices in
$V_{ + }$. \refstep{Move 1}--\refstep{Move 7} can now be used to successively modify $T$ so that the
result, $T'$, is a positive graph for $\hat{A}$ with $n_{T' + } = n_{T + }-1$
and with $n_{T' - }= n_{T - }$. Indeed, \refstep{Move 1}--\refstep{Move 3} successively
decrease the angle of the relevant trivalent vertex, this by an amount that
is bounded uniformly away from zero. Thus, only finitely many applications
of \refstep{Move 1}--\refstep{Move 3} are possible. The subsequent moves all eliminate a trivalent
vertex. In any event, when $T'$ is produced, go to \refstep{Step 2}.

\substep{Step 2}
Repeat \refstep{Step 1} using $T'$ now instead of $T$.

Note that \refstep{Step 1} and \refstep{Step 2} ultimately result in a positive graph for $\hat{A}$
whose version of $V_{ + }$ is empty and whose version of $V_{ - }$ has the
same number of elements as does that of $T_{\hat{A}}$.

\substep{Step 3}
The input to this step is a positive line
graph for $\hat{A}$ whose version of $V_{ + }$ is empty. Let $T$ now denote the
latter. If $V_{ - }$ is also empty, then stop because $T$ is the desired
linear graph. If $V_{-} \ne \emptyset$, let $o$ denote the trivalent graph
with the largest angle. Successively apply the up side down versions of
\refstep{Move 1}--\refstep{Move 7} to $o$. The result is a new, positive graph for $\hat{A}$ with no
 elements in its version of $V_{ + }$ and one less trivalent vertex than $T$.
Denote this graph by $T'$. Go to \refstep{Step 4}.

\substep{Step 4}
Repeat \refstep{Step 3} using $T'$ now instead of $T$.

\bibliographystyle{gtart}
\bibliography{link}
\nocite{HWZ1-corr}
\end{document}